\documentclass[a4paper,11pt]{article}

 \usepackage{multicol}
\usepackage{amsmath,amssymb,graphics,epsfig,color,cite}
\usepackage{mathrsfs,bbm}
\usepackage{amsthm} 
\usepackage{bm}
\usepackage[toc,page]{appendix}
\usepackage{tikz}
\usetikzlibrary{arrows}
\usetikzlibrary{decorations.pathreplacing,decorations.markings}
\tikzset{arrow data/.style 2 args={%
      decoration={%
         markings,
         mark=at position #1 with \arrow{#2}},
         postaction=decorate}
      }%

 \usepackage{enumitem}

\usepackage{hyperref}
 
\usepackage{geometry}
\definecolor{light}{gray}{.10}

\DeclareMathAlphabet{\mathpzc}{OT1}{pzc}{m}{it}

\newcounter{savecntr}

\newcommand{\mc}[1]{\mathcal{#1}}

\newcommand{\mbf}[1]{\mathbf{#1}}

\DeclareMathOperator{\range}{Ran}

\newcommand{\lp}{\langle}
\newcommand{\rp}{\rangle}

\newcommand{\ve}{\varepsilon}

\DeclareMathOperator*{\argmin}{arg\,min}
\DeclareMathOperator*{\argmax}{arg\,max}

\def\Hess{{\rm Hess\,}}

\def\supp{\mathop{\rm supp} \nolimits} 
\def\Ran{{\rm Ran}}
\def\sspan{{\rm Span}}

\def\and {{\rm \; and \;}}

\def\dim {{\rm \; dim  \;}}

\newcommand{\ft}[1]{\mathsf{#1}} 
\newcommand {\pa}{\partial}

\newtheorem{theorem}{Theorem}
\newtheorem{proposition}{Proposition}
\newtheorem{definition}[proposition]{Definition}
\newtheorem{lemma}[proposition]{Lemma}
\newtheorem{corollary}[proposition]{Corollary}
\newtheorem{remark}[proposition]{Remark}

\usepackage{titlesec}
\titleformat*{\section}{\Large\bfseries}
\titleformat*{\subsection}{\large\bfseries}
\titleformat*{\subsubsection}{\normalsize\bfseries}
\titleformat*{\paragraph}{\normalsize\bfseries}
\titleformat*{\subparagraph}{\normalsize\bfseries}

\usepackage[hang,font={normalsize,it},labelfont=sf]{caption}

\title{The exit from a metastable state: concentration of  the exit point distribution on the low energy saddle points}
\author{Giacomo Di Ges\`u\setcounter{savecntr}{\value{footnote}}\thanks{Current affiliation: Institut f\"{u}r Analysis und Scientific Computing, E101-TU Wien, Wiedner Hauptstr. 8, 1040 Wien, Austria. E-mail: \{giacomo.di.gesu,boris.nectoux\}@asc.tuwien.ac.at}$\ $\footnotemark[2]\, , Tony Leli\`evre\setcounter{savecntr}{\value{footnote}}\thanks{CERMICS, \'Ecole des Ponts, Universit\'e  Paris-Est, INRIA, 77455 Champs-sur-Marne, France. E-mail:  tony.lelievre@enpc.fr}\,, Dorian Le Peutrec\setcounter{savecntr}{\value{footnote}}\thanks{Laboratoire de Math\'ematiques d'Orsay, Univ. Paris-Sud, CNRS, Universit\'e Paris-Saclay, 91405 Orsay, France. E-mail: dorian.lepeutrec@math.u-psud.fr}$\, \ $and Boris Nectoux\footnotemark[1]\, \footnotemark[2]}

\begin{document} 

 \maketitle

 \begin{abstract}
We consider  the first exit point distribution from a bounded domain~$\Omega$ of the stochastic process~$(X_t)_{t\ge 0}$ solution to the overdamped Langevin dynamics
$$d X_t = -\nabla f(X_t) d t + \sqrt{h} \ d B_t$$ 
starting from the quasi-stationary distribution in~$\Omega$. In the small temperature regime ($h\to 0$) and under rather general assumptions on $f$  (in particular, $f$ may have several critical points in $\Omega$), it is proven that the support of the distribution of the first exit point concentrates on some points realizing the minimum of~$f$ on $\pa \Omega$.  The proof relies on     tools to study  tunnelling effects in   semi-classical analysis. Extensions of the results to more general initial distributions than the quasi-stationary distribution are also presented.  

\end{abstract}

\tableofcontents


\section{Introduction and main results}\label{sec:intro}

\subsection{Setting and motivation}

\subsubsection{Overdamped Langevin dynamics}
 We are interested in the overdamped Langevin dynamics
\begin{equation}\label{eq.langevin}
d X_t = -\nabla f(X_t) d t + \sqrt{h} \ d B_t,
\end{equation}
where $X_t\in \mathbb R^d$ is a vector in $\mathbb R^d$,~$f: \mathbb R^d \to 
\mathbb R$ is a $C^\infty$ function, $h$ is a positive parameter and 
$(B_t)_{t\geq 0}$ is a standard $d$-dimensional Brownian motion. Such a 
dynamics is prototypical of models used for example in computational 
statistical physics to simulate the evolution of a molecular system at a 
fixed temperature, in which case $f$ is the potential energy function 
and $h$ is proportional to the temperature. It admits as an invariant 
measure the Boltzmann-Gibbs measure (canonical ensemble) $Z^{-1} e^{ 
-\frac 2h f}$ where $Z=\int_{\mathbb R^d}  e^{ 
-\frac 2h f} < \infty$. In the small 
temperature regime $h \to 0$, the stochastic process $(X_t)_{t \ge 0}$ 
is typically metastable: it stays for a very long period of time in a subset of 
$\mathbb R^d$ (called a metastable state) before hopping to another metastable 
state. In the context of statistical physics, this behavior is expected 
since the molecular system typically jumps between various 
conformations, which are indeed these metastable states. For modelling 
purposes as well as for building efficient numerical methods, it is thus 
crucial to be able to precisely describe the exit event from a metastable 
state, namely the law of the first exit time and the first exit point.

The main objective of this work is to address the following question: 
given a metastable domain $\Omega\subset \mathbb R^d$, what are the exit 
points in the small temperature regime $h\to 0$? Compared to the work~\cite{di-gesu-le-peutrec-lelievre-nectoux-16}, we here only identify the support of the first exit 
point distribution, and the relative likelihood of the points in this 
support, whereas in~\cite{di-gesu-le-peutrec-lelievre-nectoux-16}, we also study the exit through 
points which occur with exponentially small probability in the limit $h 
\to 0$. The results here are thus less precise than in~\cite{di-gesu-le-peutrec-lelievre-nectoux-16}, 
but we also work under much more general assumptions on the function $f$ which are made precise in the next section.
 
 \subsubsection{Exit point distribution and purpose of this work}
   
Let us consider a domain~$\Omega \subset \mathbb R^d$ and the associated exit
event from $\Omega$. More precisely, let us introduce
\begin{equation}\label{eq.tau}
\tau_{\Omega}=\inf \{ t\geq 0 | X_t \notin \Omega      \}
\end{equation}
\label{page.tau}
the first exit time from $\Omega$.  The concentration of the law of $X_{\tau_{\Omega}}$  on a subset of $\pa \Omega$ is defined as follows. 

\begin{definition}
Let $\mathcal Y\subset \pa \Omega$.   
The law of $X_{\tau_{\Omega}}$ concentrates on~$\mathcal Y$  in the limit $h\to 0$   if for every  neighborhood $\mathcal  V_{\mathcal Y}$ of ${\mathcal Y}$ in~$\pa \Omega$  
  $$\lim \limits_{h\to 0}\mathbb P  \left [ X_{\tau_{\Omega}} \in \mathcal  V_{\mathcal Y}\right]=1,$$ 
  and if for all $x\in \mathcal Y$ and for all neighborhood $\mathcal  V_x$ of $x$ in~$\pa \Omega$ 
  $$\lim \limits_{h\to 0}\mathbb P  \left [ X_{\tau_{\Omega}} \in \mathcal  V_x\right]>0.$$
    In other words, $\mathcal Y$ is the support of the law of $X_{\tau_{\Omega}}$ in the limit $h\to 0$. 
  \end{definition} 
 
\noindent
\textbf{Previous results on the  behaviour  of the law of $X_{\tau_\Omega}$ when $h\to 0$}. 
They are mainly three kinds of approaches to study where and how the  law of $X_{\tau_\Omega}$ concentrates on $\pa \Omega$ when $h\to 0$.  We refer to \cite{Day} for a comprehensive review of the literature. \\
The first approach  one is based on   formal computations: the concentration of the  law of $X_{\tau_\Omega}$ on  $\argmin_{\pa \Omega}f$ in the small temperature regime ($h\to 0$) has been studied  in~\cite{MS77}  when $\pa_n f >0$ on $\pa \Omega$  and in~\cite{schuss90, MaSc} when considering also the case when $\pa_n f=0$  on  $ \pa \Omega$. \\
The second  approach    is based on rigorous  techniques developed for partial differential equations. 
When it holds,
\begin{equation}\label{eq.prev-res1}
\partial_nf>0 \text{ on } \partial \Omega,
\end{equation}
where $\partial_nf$ is the normal derivative of $f$ on $\pa \Omega$, and 
\begin{equation}\label{eq.prev-res2}
\{x\in \Omega, \vert \nabla f(x)\vert =0\}=\{x_0\} \text{ with } f(x_0)=\min_{\overline \Omega} f \text{ and } \text{det Hess} f(x_0)>0,
 \end{equation}
 the  concentration of the law of~$X_{\tau_{\Omega}}$ in the limit $h\to 0$ on  $\argmin_{\pa \Omega}f$ has been obtained  in~\cite{kamin1979elliptic,Kam,Per,day1984a,day1987r}, when $X_0=x\in \Omega$.\\
Finally, the last approach   is based on techniques developed in large deviation theory. When~\eqref{eq.prev-res1} and~\eqref{eq.prev-res2} hold, and $f$ atteins its minimum on $\pa \Omega$ at one single point $y_0$, it is proved in~\cite[Theorem 2.1]{FrWe} that the law of~$X_{\tau_{\Omega}}$ in the limit $h\to 0$ concentrates on $y_0$, when $X_0=x\in \Omega$. 
 In \cite[Theorem 5.1]{FrWe}, under more general assumptions on $f$, for $ \Sigma\subset \pa \Omega$,   the limit of $h\ln \mathbb P  \left [ X_{\tau_{\Omega}} \in \Sigma\right]$ when $h\to 0$ is related   to a minimization problem  involving  the quasipotential  of the process~\eqref{eq.langevin}.  
%
Let us mention two limitations when applying~\cite[Theorem 5.1]{FrWe} in order to obtain some information on the first exit point distribution. First, this theorem requires to be able to compute the quasipotential in order to get useful information: this is trivial under the assumptions~\eqref{eq.prev-res1} and~\eqref{eq.prev-res2} but more complicated under more general assumptions on $f$ (in particular when $f$ has several critical points in $\Omega$). Second, even when the quasi potential is analytically known, this result only gives the subset of $\partial \Omega$ where exit will not occur on an exponential scale in the limit $h\to 0$. It does not allow to exclude exit points with probability which goes to zero polynomially in $h$ (this indeed occurs, see~Section~\ref{sec.A4}), and it does not give the relative probability to exit through exit points with non-zero probability in the limit $h\to 0$.

 Let us mention that \cite{kamin1979elliptic,Kam,Per,day1984a,day1987r,FrWe} also cover the case of non reversible diffusions.  
 
\medskip

\noindent
\textbf{Purpose of this work: the case when $f$ has several critical points in $\Omega$}. As explained above,  the concentration of the  law of~$X_{\tau_\Omega}$ on  $\argmin_{\pa \Omega}f$ was obtained when~\eqref{eq.prev-res1} and~\eqref{eq.prev-res2} hold (which imply in particular that~$f$ has only one critical point in $\Omega$). Our work  aims at generalizing in the reversible case  the results~\cite{kamin1979elliptic,Kam,Per,day1984a,day1987r} and~\cite[Theorem 2.1]{FrWe}, when $f$ has several critical points in $\Omega$. In particular, we exhibit    more general assumptions on $f$ in which  the law of~$X_{\tau_\Omega}$ concentrates on   points belonging to $\argmin_{\pa \Omega}f$ and we compute the relative probabilities to leave through each of them.    For instance,   we do not assume that $\partial_nf>0$ on~$\partial \Omega$ (i.e. we drop the assumption~\eqref{eq.prev-res1}),  we have no restriction on the number of critical points of~$f$ in~$\Omega$ (i.e. we do not assume~\eqref{eq.prev-res2}) and $f$ is allowed to have critical points in~$\Omega$ (e.g. saddle points or local minima) with  larger energies than  $\min_{\partial \Omega}f$ (however we do not consider the case when~$f$  has critical  points on~$\partial
\Omega$). Here are examples of outputs of our work.

First, we show for example the following result: if   $\{y\in \Omega, \, f(y)<\min_{\pa \Omega}f\}$ is connected and contains all the critical points of~$f$  in~$\Omega$  together with  $\pa_nf>0$ on $\argmin_{\pa \Omega}f$, then,   
 the exit point distribution concentrates on $\argmin_{\partial 
\Omega} f$ when $X_0$ is distributed according to  the quasi-stationary distribution~$\nu_h$ of the process~\eqref{eq.langevin} in~$\Omega$ (see~Definition~\ref{defqsd} below) or $X_0=x\in \{y\in \Omega, \, f(y)<\min_{\pa \Omega}f\}$. This extends the results of~\cite{kamin1979elliptic,Kam,Per,day1984a,day1987r} and~\cite[Theorem 2.1]{FrWe} to a more general geometric 
setting.

Second, we also study situations where critical points of $f$ in $\Omega$ 
are larger in energy than $\min_{\partial \Omega} f$. In such a case, 
again for $X_0$ distributed according to $\nu_h$ or $X_0=x\in K$ 
(where $K$ is a compact subset  of $\Omega$ to be made precise below), the law 
of $X_{\tau_\Omega}$ concentrates when $h\to 0$ on a subset of $\partial \Omega$ which 
can be  strictly included in $\argmin_{\partial \Omega} f$. In particular, we 
show that the following phenomena can occur:
 \begin{itemize}
   \item[(i)]  There exist points $z\in \argmin_{\pa \Omega} f$, $C>0$ and $c>0$, such that  for all sufficiently small neighborhood  $ \Sigma_z$ of $z$ in $\pa \Omega$,  in the limit $h\to 0$: $\mathbb P  \left [ X_{\tau_{\Omega}} \in \Sigma_z\right]\le C\, e^{-\frac ch}$ (see~\eqref{eq.t1} in~Theorem~\ref{thm.main} and the discussion  after the statement of Theorem~\ref{thm.main}).   
  \item[(ii)]  There exist points $z\in \argmin_{\pa \Omega} f$ and $C>0$, for all sufficiently small neighborhood  $ \Sigma_z$ of $z$ in $\pa \Omega$,  $\mathbb P  \left [ X_{\tau_{\Omega}} \in \Sigma_z\right]=C\sqrt{h}\, (1+o(1))$. This is explained in  Section~\ref{sec.A4}. 
  \end{itemize} 
In particular, motivated  by the desire to analyse the metastability of the exit event from $\Omega$, we exhibit explicit  assumptions on $f$  which aim at ensuring  the two following properties:
\begin{itemize}[leftmargin=1.6cm,rightmargin=1.3cm]
\item[\textbf{[P1]}] When $X_0$ is initially distributed according to the quasi-stationary distribution $\nu_h$ of the process~\eqref{eq.langevin} in~$\Omega$ (see~Definition~\ref{defqsd} below), the law of $X_{\tau_\Omega}$ concentrates in the limit $h\to 0$ on some global minima  of~$f$ on~$\pa \Omega$.
\item[\textbf{[P2]}] There exists  a connected component of $\{f<\min_{\pa \Omega}f\}$ such that when $X_0=x$ and $x$ belongs to  this connected component, the law of $X_{\tau_\Omega}$ concentrates in the limit $h\to 0$ on the same points of    $\pa \Omega$ as it does when $X_0\sim \nu_h$. 
\end{itemize} 
Finally,  we give sharp asymptotic estimates when $h\to 0$ on the principal eigenvalue   and the principal eigenfunction of the infinitesimal generator of the diffusion~\eqref{eq.langevin} associated with Dirichlet boundary conditions on $\pa \Omega$, see Section~\ref{se.interm}.    
Let us mention that a simplified version of the results of this work is  presented in~\cite{IHPLLN}. 
 \medskip

 \noindent
\textbf{Organization of the introduction}.  In Section~\ref{qsd}, we introduce the quasi-stationary distribution associated with $\Omega$ and the process~\eqref{eq.langevin}, and we explain why it is relevant to study the exit event from  a metastable domain~$\Omega$ assuming that the process~\eqref{eq.langevin} is initially distributed according to the quasi-stationary distribution. 
In Section~\ref{nota-hypo}, we  introduce  assumptions on $f$  which will be used throughout this paper  and we state   the main result of this work (see Theorem~\ref{thm.main}). Finally, in Section~\ref{discussion-hyp}, we discuss the necessity of the assumptions related to obtain the results of Theorem~\ref{thm.main}.

 \subsection{Metastability and the quasi-stationary distribution} \label{qsd}

The quasi-stationary distribution  is the cornerstone of our analysis. Here and in the following, we assume that~$\Omega
\subset \mathbb R^d$ is smooth, open, bounded and connected. 
Let us give  the definition of the quasi-stationary distribution associated with the overdamped Langevin process \eqref{eq.langevin} and $\Omega$:
 
\begin{definition} \label{defqsd}
Let $\Omega\subset \mathbb R ^d$ and consider the dynamics \eqref{eq.langevin}.  A quasi-stationary distribution is a probability measure $\nu_h$ supported in~$\Omega$ such that for all measurable sets $A\subset \Omega$ and for all $t\ge 0$\label{page.qsd}
\begin{equation}\label{eq.vhA} \nu_h(A)=\frac{\displaystyle\int_{\Omega} \mathbb P_x \left[X_t \in A,   t<\tau_{\Omega}\right]  \nu_h(dx ) }{\displaystyle \int_{\Omega}   \mathbb P_x \left[t<\tau_{\Omega} \right] \nu_h(dx )}.
\end{equation}
\end{definition}
\noindent
Here and in the following, the subscript $x$ indicates that the
stochastic process starts from $x \in \mathbb R^d$: $X_0=x$.
 In words,~\eqref{eq.vhA} means that if
$X_0$ is distributed according to  $\nu_h$, then for all $ t>0$,  $X_t$
is still distributed according to  $\nu_h$ conditionally on 
$X_s \in \Omega$ for all $s \in (0,t)$.  We have the following results from~\cite{le2012mathematical}:
\begin{proposition}\label{pr.tv-conv}
Let $\Omega\subset \mathbb R ^d$ be a bounded domain and consider the
dynamics~\eqref{eq.langevin}.  Then, there exists a probability measure
$\nu_h$ with support in~$\Omega$ such that, whatever the law of the
initial condition $X_0$ with support in~$\Omega$, it holds:
\begin{equation}\label{eq.cv_qsd}
\lim_{t \to \infty} \| {\rm Law}(X_t| t < \tau_\Omega) - \nu_h \|_{TV} = 0.
\end{equation}
\end{proposition}
\noindent
Here,~${\rm Law}(X_t| t < \tau_\Omega)$ denotes
the law of $X_t$ conditional to the event $\{t<\tau_\Omega\}$.
A corollary of this proposition is that the quasi-stationary distribution
$\nu_h$ exists and is unique.
For a given initial distribution of the process~\eqref{eq.langevin}, if  the convergence in~\eqref{eq.cv_qsd} is much
quicker than the exit from $\Omega$, the exit from the domain~$\Omega$ is said to be metastable.   When the exit from $\Omega$ is metastable,  it is thus relevant to study the exit event from~$\Omega$ assuming that the process \eqref{eq.langevin} is initially distributed according to the quasi-stationary distribution $\nu_h$. 


 
Let us introduce the infinitesimal generator
of the dynamics (\ref{eq.langevin}), which is the differential operator
\begin{equation}\label{eq.L-0}
L^{(0)}_{f,h}=-\nabla f \cdot \nabla  + \frac{h}{2}  \ \Delta.
\end{equation}
\label{page.lofh}
  In the notation $L^{(0)}_{f,h}$, the superscript $(0)$ indicates
that we consider an operator on functions, namely $0$-forms. 
 The basic
observation to define our functional framework is that the operator
$L^{(0)}_{f,h}$ is self-adjoint on the weighted $L^2$ space
$$L^2_w(\Omega)=\left\{u:\Omega \to \mathbb R 
,  \int_\Omega u^2 
 e^{-\frac{2}{h} f } \,   < \infty\right\}$$
(the weighted Sobolev spaces $H^k_w(\Omega)$ are defined
similarly). 
Indeed, for any smooth test functions $u$ and $v$ with compact supports in~$\Omega$, one has
$$\int_\Omega (L^{(0)}_{f,h}u) v\,  e^{-\frac{2}{h} f} = \int_\Omega
(L^{(0)}_{f,h}v) u \, e^{-\frac{2}{h} f} = - \frac{h}{2} \int_\Omega
\nabla u\cdot \nabla v\, e^{-\frac{2}{h} f}.$$
This gives a proper framework to introduce the Dirichlet realization
$L^{D,(0)}_{f,h}$ on $\Omega$ of the operator~$L^{(0)}_{f,h}$: \begin{proposition} \label{fried}
The Friedrich's extension associated with  the quadratic form 
$$\phi \in C^{\infty}_c(\Omega)\mapsto \frac{h}{2} \int_{\Omega}\left
  \vert \nabla \phi  \right\vert^2 e^{-\frac{2}{h}f }  $$ 
is denoted
 by $-L^{D,(0)}_{f,h}$. It is
a  non negative unbounded self adjoint operator on
$L^2_w(\Omega)$ with
domain~$$D\left(L^{D,(0)}_{f,h}\right)=H^1_{w,0}(\Omega)\cap
H^2_w(\Omega),$$
where $H^1_{w,0}(\Omega)=\{u \in H^1_w(\Omega), \, u=0 \text{ on
} \partial \Omega\}$. 
\end{proposition}
\label{page.ldofh}
The compact injection  $H^1_w(\Omega)\subset L^2_w(\Omega)$ implies that the operator $L^{D,(0)}_{f,h}$ has a compact
resolvent and its spectrum is consequently purely discrete. Let us introduce $\lambda_h >0$ the smallest eigenvalue
 of $-L^{D,(0)}_{f,h}$:
  \begin{equation} \label{eq.lh}
 \lambda_h=\inf \sigma\big (-L^{D,(0)}_{f,h}\big).
\end{equation}
\label{page.lambdah}
The eigenvalue $\lambda_h$ is called the principal eigenvalue of $-L^{D,(0)}_{f,h}$. 
  From standard results on
 elliptic operator (see for example~\cite{MR1814364, Eva}),~$\lambda_h$ is non degenerate and its associated eigenfunction $u_h$
 has a sign on~$\Omega$. Moreover,~$u_h \in C^{\infty} (\overline \Omega)$.
Without loss of generality, one can then assume that:  
 \begin{equation} \label{eq.norma}
 u_{h}>0\ \text{on}\ \Omega\ \ \text{and}\ \ \int_{\Omega}u_{h}^{2} e^{-\frac 2hf} =1.
\end{equation}
\label{page.uh}
The eigenvalue-eigenfunction pair $(\lambda_h,u_h)$ satisfies:
\begin{equation} \label{eq:u-bis}
\left\{
\begin{aligned}
 -L^{(0)}_{f,h}\, u_h &=  \lambda_h u_h    \ {\rm on \ }  \Omega,  \\ 
u_h&= 0 \ {\rm on \ } \partial \Omega.
\end{aligned}
\right.
\end{equation}

The link between the quasi
stationary distribution $\nu_h$ and the function $u_h$
is given by the following proposition (see for example \cite{le2012mathematical}):
\begin{proposition} \label{uniqueQSD}
The unique quasi-stationary distribution $\nu_h$ associated with the
dynamics~\eqref{eq.langevin} and the domain~$\Omega$ is given by:
\begin{equation} \label{eq.expQSD}
\nu_h(dx)=\frac{\displaystyle  u_h(x) e^{-\frac{2}{h}  f(x)}}{\displaystyle \int_\Omega u_h(y) e^{-\frac{2}{h}  f(y)}dy}\,  dx.
\end{equation}
\end{proposition}
\label{page.qsd}

The next proposition (which can also be found in \cite{le2012mathematical}) characterizes
the law of the exit event from $\Omega$.
\begin{proposition}\label{indep1}
Let us consider the dynamics~\eqref{eq.langevin} and the quasi
stationary distribution $\nu_h$ associated with the domain~$\Omega$. If
$X_0$ is distributed according to $\nu_h$, the random variables
$\tau_{\Omega}$ and $X_{\tau_{\Omega}}$ are independent. Furthermore $\tau_{\Omega}$ is exponentially distributed with parameter $\lambda_h$ and the law of $X_{\tau_{\Omega}}$ has a density with respect to the Lebesgue measure on $\partial \Omega$ given by
\begin{equation}\label{eq.dens}
z\in \partial \Omega \mapsto - \frac{h}{2\lambda_h} \frac{ \partial_n
  u_h(z) e^{-\frac{2}{h} f(z)}}{\displaystyle \int_\Omega u_h(y) e^{-\frac{2}{h} f(y)} dy}.
\end{equation}
\end{proposition}
Here and in the following,~$\partial_{n}=n\cdot \nabla$ stands for the normal derivative and $n$ is the unit outward normal on~$\partial \Omega$.

\subsection{Hypotheses and main results}
\label{nota-hypo}
This section is dedicated to the statement    of the main result of this work.

\subsubsection{Hypotheses and notation}
 
 In the following, we consider a setting that is more general than the one  of Section~\ref{qsd}: 
 $\overline\Omega$ is a $C^{\infty}$ oriented compact and  connected Riemannian manifold of dimension $d$ with boundary $\partial \Omega$.    
 
 \noindent
The following notation will be used:  for $a\in \mathbb R$,
$$\{f<a\}=\{x\in \overline \Omega, \ f(x)<a\},\ \ \{f\le a\}=\{x\in \overline \Omega,\  f(x)\le a\},$$
and\label{page.fa}
$$\{f=a\}=\{x\in \overline \Omega, \ f(x)=a\}.$$
Let us recall the definition of the  domain of attraction of a subset $D$ of $\Omega$ for the $-\nabla f$ dynamics. Let $f : \overline \Omega \to \mathbb R$ be a $C^{\infty}$ function.  Let $x\in \Omega$ and denote by~$\varphi_{t}(x)$ the solution to the ordinary differential equation 
\begin{equation}\label{hbb}
  \frac{d}{dt}\varphi_{t}(x)=-\nabla f(\varphi_{t}(x)) \text{ with } \varphi_{0}(x)=x,
  \end{equation}
on the interval $t\in [0,t_x]$, where 
$$t_x=\inf  \{t\ge 0, \ \varphi_{t}(x)\notin \Omega\}>0.$$
Let $x\in \Omega$ be such that $t_x=+\infty$. The $\omega$-limit set of $x$, denoted by $\omega(x)$, is defined by
$$\omega(x)=\{y\in \overline \Omega, \, \exists (s_n)_{n\in \mathbb N} \in (\mathbb R_+)^{\mathbb N}, \,\lim_{n\to \infty}s_n=+ \infty, \,\lim_{n\to \infty}\varphi_{s_n}(x)=y \}.$$
Let us recall that the $\omega$-limit set $\omega(x)$ is included in the set of the critical points of $f$ in $\overline \Omega$. 
Moreover, when $f$ has a finite number of critical points in $\overline \Omega$,  
$$\exists y\in \overline \Omega, \  \omega(x)=\{y\}.$$
Let $D$ be a subset of $\Omega$. The domain of attraction of a subset $D$ of $\Omega$ is  defined by\label{page.AD}
\begin{equation}\label{eq.ad}
 \mathcal A(D)=\{ x\in \Omega, \,t_x=+\infty \text{ and } \omega(x)\subset D\}.
 \end{equation}
 Let us now introduce the basic assumption which is used throughout  this work:
%
\begin{equation}
\tag{\textbf{A0}}\label{H-M}
 \left.
    \begin{array}{ll}
        &\text{The function $f : \overline \Omega \to \mathbb R$ is a $C^{\infty}$ function.}   \\
 &\text{For all $x\in \pa \Omega$, $\vert \nabla f(x)\vert \neq 0$.}\\
 &\text{The functions $f$ and $f|_{
 \partial \Omega}$ are Morse functions.}\\
   &\text{Moreover,~$f$   has  at least one local minimum in  $\Omega$.}
    \end{array}
\right \}
 \end{equation}


 \noindent
  A function \label{page.HM} $\phi: \overline \Omega \to \mathbb R$ is a  Morse function if all its  critical points are non degenerate (which implies in particular that $\phi$ has a finite number of critical points since $\overline \Omega$ is compact and a non degenerate critical point is isolated from the other critical points). 
Let us recall that a critical point  $z\in \overline\Omega$ of~$\phi$  is non degenerate if the hessian matrix of~$\phi$ at $z$, denoted by    $\Hess \phi(z)$,  is invertible.
  We refer for example to~\cite[Definition 4.3.5]{Jost:2293721} for a definition of the hessian matrix on a manifold.  
 A non degenerate critical point $z\in \overline\Omega$ of~$\phi$ is  said to have index $p\in\{0,\dots,d\}$
if   $\Hess \phi(z)$ has precisely $p$
negative eigenvalues (counted with multiplicity). In the case $p=1$,   $z$  is called a saddle point.
  
   \medskip
      
   \noindent
For any local minimum $x$ of~$f$ in~$\Omega$, one defines\label{page.hfx}
\begin{equation}\label{eq.Hfx}  
\ft H_f (x):=\,  \inf_{\substack{ \gamma \in C^0([0,1], \overline \Omega)\\ \gamma(0)=x\\ \gamma(1)\in \pa \Omega}} \ \, \max_{t\in [0,1]}\,  f\big ( \gamma(t)\big),  
\end{equation}
where $C^0([0,1], \overline \Omega)$ is the set of continuous paths from $[0,1]$ to $\overline \Omega$.
 In Section~\ref{ft-Ci}, another equivalent definition of $\ft H_f$ is given (see indeed~\eqref{eq.c3} and~\eqref{eq.=delta}). 
Let us now define a set of assumptions which will ensure that   \textbf{[P1]} and \textbf{[P2]} are satisfied  (see indeed Theorem~\ref{thm.main} and Section~\ref{discussion-hyp} for a discussion on these assumptions):
\begin{itemize}[leftmargin=1.3cm,rightmargin=1.1cm]
\item~\eqref{H-M} holds and 
\begin{equation}\tag{\textbf{A1}}\label{eq.hip1}
 \exists ! \ft C_{\ft{max}}  \in \mathcal C \text{ such that } \max\limits_{\ft C\in \mathcal C}  \,  \Big \{  \max_{\overline{\ft C}}f-\min_{\overline{\ft C}} f   \Big\} =  \max_{\overline{\ft C_{\ft{max}} }}f-\min_{\overline{\ft C_{\ft{max}} }} f
\end{equation}
\label{page.cmax}
\noindent
where 
\begin{align}
\label{mathcalC-def}
&\mathcal C:=\big \{ \ft C (x), \, x  \text{ is  a local minimum of $f$ in }\Omega \big \},
\end{align}\label{page.c}
with, for a local minimum $x$ of $f$ in $\Omega$, 
\begin{equation}\label{eq.Cdef2}
\begin{aligned}
\ft C(x)  \text{ is the connected component of } \{f< \ft H_f(x)\} \text{ containing } x. 
\end{aligned}
\end{equation}\label{page.cx}


\item     \eqref{eq.hip1} holds and\label{page.hypo}
\begin{equation}\tag{\textbf{A2}} \label{eq.hip2}
\pa \ft C_{\ft{max}} \cap \pa \Omega\neq \emptyset.
\end{equation}
\item     \eqref{eq.hip1} holds and
  \begin{equation}\tag{\textbf{A3}} \label{eq.hip3}
\pa  \ft C_{\ft{max}} \cap \pa \Omega\subset \argmin_{\pa \Omega} f.
\end{equation}  
 \end{itemize}
More precisely, the assumptions~\eqref{H-M},~\eqref{eq.hip1},~\eqref{eq.hip2}, and~\eqref{eq.hip3} ensure that when $X_{0}\sim\nu_h$ or $X_0=x\in  \ft C_{\ft{max}}$,   the law of  $X_{\tau_{\Omega}}$ concentrates on the set $ \pa \ft C_{\ft{max}} \cap \pa \Omega $, see items~1 and~2 in Theorem~\ref{thm.main}. Finally, let us  introduce the following assumption:
 
 \begin{itemize}[leftmargin=1.3cm,rightmargin=1.1cm]
\item[]     \eqref{eq.hip1} holds and 
 \begin{equation} \tag{\textbf{A4}} \label{eq.hip4}
\pa \ft C_{\ft{max}}  \cap \Omega \text{ contains no  \textit{separating saddle point} of $f$,}
 \end{equation}
where the definition  of a separating saddle point of $f$   is  introduced  below in item 1 in Definition~\ref{de.SSP}. 
 \end{itemize}
 \begin{sloppypar}
 \noindent 
  The assumption~\eqref{eq.hip4} together with~\eqref{H-M},~\eqref{eq.hip1},~\eqref{eq.hip2}, and~\eqref{eq.hip3},  ensures that the probability  that the process~\eqref{eq.langevin} (starting from the quasi-stationary distribution $\nu_h$ or from $x\in  \ft C_{\ft{max}}$) leaves $\Omega$ through any sufficiently small    neighborhood of $z\in \pa \Omega\setminus \pa \ft C_{\ft{max}}$ in~$\pa \Omega$  is exponentially small when $h\to 0$, see indeed item 3 in Theorem~\ref{thm.main}. 
  \medskip
  
  \noindent 
 In Figure~\ref{fig:okay}, one has represented a one-dimensional case where~\eqref{eq.hip1},~\eqref{eq.hip2},~\eqref{eq.hip3}  and~\eqref{eq.hip4} are satisfied. 
In Section~\ref{discussion-hyp}, the    assumptions~\eqref{eq.hip1},~\eqref{eq.hip2},~\eqref{eq.hip3}, and~\eqref{eq.hip4}   are discussed. In particular, it is shown that if one of the assumptions among~\eqref{eq.hip1},~\eqref{eq.hip2}, or~\eqref{eq.hip3}  does not hold, then there exists a function $f$ for which  either \textbf{[P1]} or \textbf{[P2]} is not satisfied.  
Equivalent formulations of the assumptions \eqref{eq.hip1},~\eqref{eq.hip2},~\eqref{eq.hip3}, and~\eqref{eq.hip4}   will be given in Section~\ref{sec:hip}. 
 \end{sloppypar}

\begin{remark}
It is proved in Proposition~\ref{pr.p1} that when \eqref{H-M} holds, for all local minima $x$ of $f$ in $\Omega$, one has $\ft C (x)\subset \Omega$ (see~\eqref{eq.Cdef2}). This implies that for all $y\in \ft C (x)$, $t_y=+\infty$  and  then, $\ft C (x)\subset \mathcal A(\ft C (x))$.  
\end{remark}

  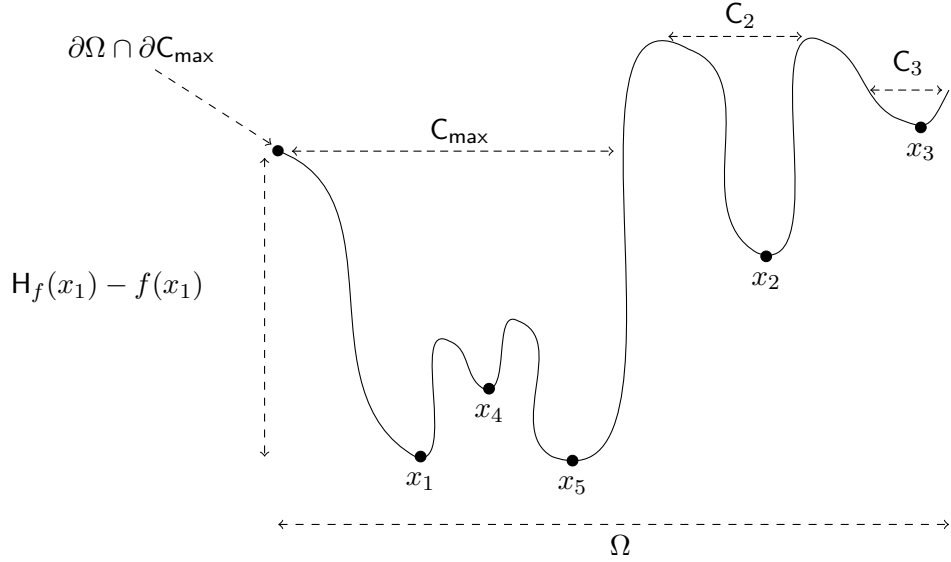
\begin{figure}[h!]
\begin{center}
\begin{tikzpicture}[scale=0.9]
\coordinate (b1) at (0,4.5);
\coordinate (b2) at (2 , 0);
\coordinate (b2a) at (2.5 , 1.7);
\coordinate (b2b) at (3 , 1);
\coordinate (b2c) at (3.5 , 2);
\coordinate (b4) at (4,0);
\coordinate (b5) at (6,6);
\coordinate (b6) at (7,3);
\coordinate (b7) at (8,6.1);
\coordinate (b8) at (9,5);
\coordinate (b9) at (10,6);
  \draw (b8) ..controls (9.5,4.8)   .. (9.8,5.4) ;
\draw [black!100, in=150, out=-10, tension=10.1]
  (b1)[out=-20]   to  (b2)  to (b2a) to (b2b) to (b2c)   to (b4) to (b5) to (b6) to (b7) to (b8);
     \draw [dashed, <->]   (0,-1) -- (9.8,-1) ;
      \draw (5,-1.3) node[]{$\Omega$};
   \draw [dashed, <->]   (-0.2,4.4) -- (-0.2,0) ;
      \draw (2.65,4.8) node[]{$ \ft C_{\ft{max}}$};
         \draw [densely  dashed, <->]  (0.2,4.5) -- (4.9,4.5) ;
          \draw (6.75,6.5) node[]{$ \ft C_2$};
             \draw [densely  dashed, <->]  (5.7,6.19) -- (7.65,6.19) ;
             \draw [densely  dashed, <->]  (8.65,5.4) -- (9.7,5.4) ;
              \draw (9.2,5.8) node[]{$ \ft C_3$};
          \draw (-2.5,2.5) node[]{$ \ft H_f(x_1)- f(x_1)$};
               \draw (-2,6) node[]{$ \pa \Omega \cap \pa \ft C_{\ft{max}}$};
\draw [dashed, ->]   (-1.8,5.7) -- (-0.1,4.6) ;
      \tikzstyle{vertex}=[draw,circle,fill=black,minimum size=4pt,inner sep=0pt]
        \draw (0.0,4.5) node[]{$\bullet$}; 
\draw (2.08,0) node[vertex,label=south: {$x_1$}](v){}; 
\draw (7.13,2.95) node[vertex,label=south: {$x_2$}](v){}; 
\draw (9.39,4.85) node[vertex,label=south: {$x_3$}](v){}; 
\draw (4.3,-0.06) node[vertex,label=south: {$x_5$}](v){}; 
\draw (3.08 , 1) node[vertex,label=south: {$x_4$}](v){}; 
    \end{tikzpicture}
\caption{A one-dimensional case where~\eqref{eq.hip1},~\eqref{eq.hip2},~\eqref{eq.hip3} and~\eqref{eq.hip4} are satisfied.  
On the figure, $f(x_1)=f(x_5)$, $\ft H_f(x_1)=\ft H_f(x_4)=\ft H_f(x_5)$,~$\mathcal C=\{\ft C_{\ft{max}}, \ft C_2,\ft C_3\}$ (where $\mathcal C$ is defined by~\eqref{mathcalC-def}),  $\pa  \ft C_2\cap \pa  \ft C_{\ft{max}}=\emptyset$ and $\pa \ft C_3\cap \pa  \ft C_{\ft{max}}=\emptyset$. Therefore, the assumption~\eqref{eq.hip4} is indeed satisfied. }
 \label{fig:okay}
 \end{center}
\end{figure}

 
\subsubsection{Notation for the local minima and  saddle points of the function~$f$}
\label{se.def-zj}

 The main purpose of this section is to introduce the  local minima   and the  generalized  saddle points of~$f$. These elements of $\overline \Omega$ are used extensively throughout this work and  play a crucial role in our analysis. Roughly speaking, the generalized  saddle points of~$f$ are the saddle points $z\in \overline \Omega$ of the extension of~$f$ by~$-\infty$ outside $\overline \Omega$.  Thus, when the function $f$ satisfies the assumption \eqref{H-M}, a generalized  saddle point of~$f$ (as introduced in~\cite{HeNi1}) is either  a saddle point $z\in \Omega$ of~$f$ or a local minimum $z\in \pa \Omega$ of $f|_{\pa \Omega}$ such that $\pa_nf(z)>0$. \\ 
 Let us assume that the function $f$ satisfies the assumption \eqref{H-M}. 
Let us denote by\label{page.u0omega} 
\begin{equation}\label{mo-omega}
\ft U_0^{\Omega}=\{x_1,\dots,x_{\ft m_{0}^{\Omega}}   \}\subset \Omega 
\end{equation}
the set of local minima of~$f$ in~$\Omega$ where $\ft m_{0}^{\Omega}\in \mathbb N$ is the number of  local  minima of~$f$ in~$\Omega$. Notice that since $f$ satisfies  \eqref{H-M},~$\ft m_0^\Omega\ge 1$.

\medskip

\noindent
The set of saddle points  of~$f$ of index $1$  in~$\Omega$  is denoted by~$\ft U_1^{\Omega}$ and its cardinality by~$\ft m_1^\Omega $\label{page.u1omega}. Let us define 
$$\label{page.u1paomega} 
\ft U_1^{\pa \Omega}:=\{z \in \pa \Omega, \,  z \text{ is a local minimum of } f|_{\pa \Omega} \text{ but }  \text{not a local minimum of~$f$ in~$\overline \Omega$ }\}.
$$
 Notice that an equivalent definition of $\ft U_1^{\pa \Omega}$ is 
 \begin{align}
\label{eq.mathcalU1_bis}
\ft U_1^{\pa \Omega}=\{z \in \pa \Omega, \,  z \text{ is a local minimum of } f|_{\pa \Omega}  \, \text{ and }  \, \pa_nf(z)>0\},
\end{align}
which follows from the fact that $\nabla f(x) \neq 0$ for all $x\in \pa \Omega$. Let us introduce 
\begin{equation}
\label{eq.m1-pa}
\ft m_1^{\pa \Omega}:={\rm Card}( \ft U_1^{\pa \Omega}).
\end{equation}
In addition, one defines:
$$\label{page.u1overlineomega} 
\ft U_1^{\overline \Omega}:=  \ft U_1^{\pa \Omega} \cup \ft U_1^{ \Omega} \ \,  \text{ and } \, \ \ft m_1^{\overline \Omega}:={\rm Card}(\ft U_1^{\overline \Omega})=\ft m_1^{\pa \Omega}+\ft m_1^{ \Omega}.$$
The set $\ft U_1^{\overline \Omega}$ is 
the set of the generalized saddle points of~$f$. If $\ft U_1^{\pa \Omega}$ is not empty, its elements are denoted by:
\begin{equation}\label{eq.U1paOmega}
\ft U_1^{\pa \Omega}=\{z_1,\ldots,z_{\ft m_1^{\pa \Omega}}\}\subset \pa \Omega,
\end{equation}
and if $\ft U_1^{\Omega}$ is not empty, its elements are labeled as follows:
\begin{equation}\label{eq.U1Omega}
\ft U_1^{ \Omega}=\{ z_{\ft m_1^{\pa \Omega}+1},\ldots,z_{\ft m_1^{\overline \Omega}}\}\subset \Omega.
\end{equation}
Thus, one has:
$$\ft U_1^{\overline \Omega}=\{z_1,\ldots,z_{\ft m_1^{\pa \Omega}}, z_{\ft m_1^{\pa \Omega}+1},\ldots,z_{\ft m_1^{\overline \Omega}}\}.$$
We assume that the elements of $\ft U_1^{\pa \Omega}$ are ordered such that:\label{page.k1paomega}
\begin{equation}\label{eq.z11}
\{z_1,\ldots,z_{\ft k_1^{\pa \Omega}}\}=\ft U_1^{\pa \Omega}\cap \argmin_{\pa \Omega} f.
\end{equation}
Notice that  $\ft k_1^{\pa \Omega}\in \{0,\ldots, \ft m_1^{\pa \Omega}\}$. 
 
Let us assume that the assumptions~\eqref{eq.hip1},~\eqref{eq.hip2}, and~\eqref{eq.hip3} are satisfied.  
In this case, let us recall that $\ft C_{\ft{max}}$ is defined by~\eqref{eq.hip1}. Moreover, in this case, one has  $\ft k_1^{\pa \Omega}\ge 1$ and
$$\pa\ft C_{\ft{max}} \cap  \pa \Omega \subset \{z_1,\ldots,z_{\ft k_1^{\pa \Omega}}\}.$$
Indeed, by assumption $ \pa \ft C_{\ft{max}} \cap  \pa \Omega\subset \{f=\min_{\pa \Omega} f\}$ (see~\eqref{eq.hip3}) and there is no local minima of~$f$ in~$\overline \Omega$  on~$\pa \ft C_{\ft{max}}$ (since $\ft C_{\ft{max}}$ is a sublevel set of~$f$). 
 We assume lastly  that the set $\{z_1,\ldots,z_{\ft k_1^{\pa \Omega}}\}$ is  ordered such that:\label{page.k1pacmax}
\begin{equation}\label{eq.k1-paCmax}
\{z_1,\ldots,z_{\ft k_1^{\pa \ft C_{\ft{max}} }}\}=\{z_1,\ldots,z_{\ft k_1^{\pa \Omega}}\} \cap \pa  \ft C_{\ft{max}}.
\end{equation}
Notice that $\ft k_1^{\pa \ft C_{\ft{max}}}\in \mathbb N^*$ and $\ft k_1^{\pa \ft C_{\ft{max}}}\le \ft k_1^{\pa \Omega}$.  
We provide an example in Figure~\ref{fig:shema_nota} to illustrate the notations introduced in this section.  \\

As introduced in~\cite[Section 5.2]{HeNi1},~$\ft U_0^{  \Omega}$  is the set of generalized critical points of~$f$ of index $0$ for  the Witten Laplacian acting on functions   with Dirichlet boundary conditions on $\pa \Omega$, and $\ft U_1^{\overline \Omega}$  is the set of generalized critical points of~$f$ of index $1$ for  the Witten Laplacian acting on $1$-forms  with tangential Dirichlet boundary conditions on~$\pa \Omega$. We refer to Section~\ref{sec:LD} for the definition of these Witten Laplacians.
 \begin{remark}
The assumption \eqref{H-M} implies that~$f$ does not have any saddle point (i.e critical point of index $1$) on $\pa \Omega$. Actually, under \eqref{H-M}, the points~$(z_i)_{i=1,\ldots,\ft m_1^{\pa \Omega}}$ play geometrically the role of saddle points. Indeed, zero Dirichlet boundary conditions are consistent with extending $f
$ by $- \infty$ outside $\overline \Omega$, in which case the point $(z_i)_{i=1,\ldots,\ft m_1^{\pa \Omega}}$ are
geometrically saddle points of $f$ (i.e. $z_i$ is a local minimum of $f|_{\pa \Omega}$ and a  local maximum of~$f|_{D_i}$, where $D_i$ is  the straight line passing through~$z_i$ and orthogonal to  $\pa \Omega$   at~$z_i$).

 \end{remark}
\begin{figure}[h!]
\begin{center}
\begin{tikzpicture}[scale=1]
\tikzstyle{vertex}=[draw,circle,fill=black,minimum size=4pt,inner sep=0pt]
\tikzstyle{ball}=[circle, dashed, minimum size=1cm, draw]
\tikzstyle{point}=[circle, fill, minimum size=.01cm, draw]
\draw [rounded corners=10pt] (1,0.5) -- (-0.25,2.5) -- (1,5) -- (5,6.5) -- (7.6,3.75) -- (6,1) -- (4,-0.3) -- (2,0) --cycle;
\draw [thick, densely dashed,rounded corners=10pt] (1.5,0.5) -- (.25,1.5) -- (0.5,2.5) -- (0.09,3.5) -- (2.75,3.75) -- (3.5,3) -- (2.4,2) -- (2.9,1.5) --cycle;
\draw [thick, densely dashed,rounded corners=10pt]    (3,3.9)  -- (3.4,3) -- (4.8,3.3) --(5.5,2.9)--(6.5,4) --(6.5,5)  -- (5.3,6) -- (3,4.58)  --cycle;
\draw [thick, densely dashed,rounded corners=10pt] (5.9,0.78)--(6,1.7) -- (5,2.3) -- (3.7,1.5) -- (3.4,0.6) --cycle;
 \draw (1.4,1.3) node[]{$\ft C_{\ft{max}}$};
  \draw (5.4,5.5) node[]{$\ft C_2$};
    \draw (4.3,1) node[]{$\ft C_3$};
     \draw  (2.4,4.9) node[]{$\Omega$};
    \draw  (7.8,3) node[]{$\pa \Omega$};
\draw (3.3 ,3.2) node[vertex,label=north east: {$z_5$}](v){};
\draw  (5.9,0.99) node[vertex,label=south east: {$z_4$}](v){};
\draw (1.7 ,2.5) node[vertex,label=north west: {$x_1$}](v){};
\draw (4.9 ,4.4) node[vertex,label=north : {$x_2$}](v){};
\draw (3.1,1.2) node[vertex,label=south : {$z_6$}](v){};
\draw (0.38,1.45) node[vertex,label=south west: {$z_1$}](v){};
\draw (6.2,5.2) node[vertex,label=north east: {$z_3$}](v){};
\draw (0.17,3.4) node[vertex,label=north west: {$z_2$}](v){};
\draw(5,1.5)  node[vertex,label=north: {$x_3$}](v){};
\draw (5.3,2.56)  node[vertex,label=east: {$z_7$}](v){};
\draw (3.8,2.3) node[vertex,label=north : {$y_m$}](v){};
\end{tikzpicture}

\begin{tikzpicture}[scale=0.85]
\tikzstyle{vertex}=[draw,circle,fill=black,minimum size=5pt,inner sep=0pt]
\tikzstyle{ball}=[circle, dashed, minimum size=1cm, draw]
\tikzstyle{point}=[circle, fill, minimum size=.01cm, draw]

\draw [dashed] (-5.4,-0.6)--(6,-0.6);
\draw [dashed,->] (-5.4,-0.6)--(-5.4,3);
\draw [dashed] (6,-0.6)--(6,3);
\draw [densely dashed,<->] (-5.3,-1.3)--(5.8,-1.3);
 \draw (0,-1.6) node[]{$\pa \Omega$};
 \draw (-5.9,2.8) node[]{$f|_{\pa \Omega}$};
 \draw[thick] (-5.4,2) ..controls  (-5.2,1.96).. (-5,1.6);
\draw[thick] (-5,1.6) ..controls  (-3.7,-1.4).. (-2,1.6);
\draw[thick] (-1.5,1.6) ..controls  (-1,-1.6).. (0.2,2.8) ;
\draw[thick] (-2,1.6) ..controls  (-1.87,1.9) and (-1.63,2.1) .. (-1.5,1.6);
\draw [thick] (0.2,2.8)  ..controls  (0.3,3.2) and (0.67,3.4).. (0.8,3);
\draw[thick] (0.8,3) ..controls  (1.8,-1.6).. (2.7,2.4);
\draw[thick] (2.7,2.4) ..controls (2.8,2.7) and (3.1,2.7)..    (3.3,2.4) ;
\draw[thick] (3.3,2.4) ..controls  (4.3,0).. (5.5,1.5);
\draw[thick] (5.5,1.5) ..controls  (5.83,2).. (6,2);

\draw (1.8,-0.58) node[vertex,label=south: {$z_3$}](v){};
\draw (-3.7,-0.6) node[vertex,label=south: {$z_1$}](v){};
\draw (-0.89,-0.6)  node[vertex,label=south: {$z_2$}](v){};
\draw  (4.45,0.45)  node[vertex,label=south: {$z_4$}](v){};
\end{tikzpicture}

\caption{Schematic representation of     $\mathcal C$ (see~\eqref{mathcalC-def}) and $f|_{\pa \Omega}$ when the assumptions~\eqref{H-M}, ~\eqref{eq.hip1},~\eqref{eq.hip2} and~\eqref{eq.hip3} are satisfied. 
  In this representation, $x_1\in \Omega$ is the global minimum of $f$ in $\overline \Omega$ and  the other local minima of $f$ in $\Omega$ are  $x_2$  and $x_3$ (thus $\ft U_0^{ \Omega}=\{x_1,x_2,x_3\}$ and $\ft m_0^\Omega=3$). Moreover, $\min_{\pa \Omega} f=f(z_1)=f(z_2)=f(z_3)=\ft H_f(x_1)=\ft H_f(x_2)<\ft H_f(x_3)=f(z_4)$, $\{f<\ft H_f(x_1)\}$ has two connected components: $\ft C_{\ft{max}}$ (see~\eqref{eq.hip1}) which contains $x_1$ and  $\ft C_2$ which contains $x_2$.  Thus, one has $\mathcal C=\{\ft C_{\ft{max}},\ft C_2, \ft C_3\}$. In addition, $\ft   U_1^{\pa \Omega}=\{z_1,z_2,z_3,z_4\}$ ($\ft m_1^{  \pa \Omega}=4$), $\{z_1,z_2,z_3\}=\argmin_{\pa \Omega} f$ ($\ft k_1^{  \pa \Omega}=3$), $\ft  U_1^{ \Omega}=\{z_5,z_6,z_7\}$ where $\{z_5\}=\overline{\ft C_{\ft{max}}}\cap\overline{\ft  C_2}$ ($\ft m_1^\Omega=3$ and~\eqref{eq.hip4} is not satisfied)  and $\min(f(z_6),f(z_7))>f(z_4)$, $\pa \ft C_{\ft{max}}\cap \pa \Omega=\{z_1,z_2\}$ ($\ft k_1^{\pa \ft C_{ \ft{max}}}=2$). Finally, one has $\ft m_1^{\overline  \Omega}=7$. The point $y_m\in \Omega$ is a local maximum of $f$ with $f(y_m)>f(z_i)$ for all $i\in \{1,\ldots,7\}$.   
}
 \label{fig:shema_nota}
 \end{center}
\end{figure}
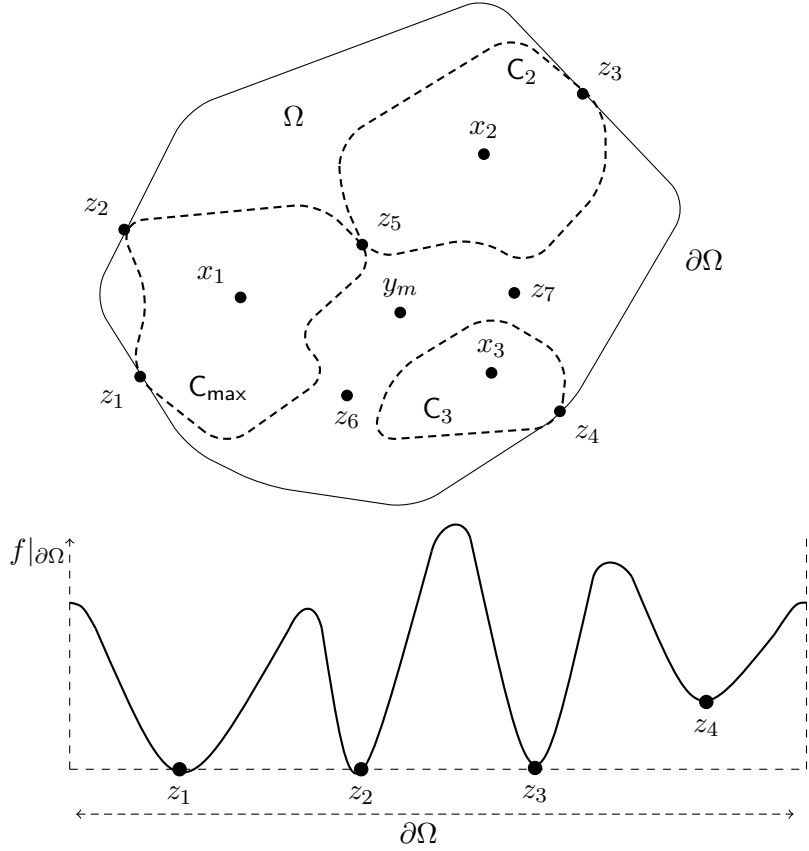

\subsubsection{Main results on the exit point distribution} 
The main result of this work   is the following. 

\begin{theorem}\label{thm.main}
Let us assume  that the assumptions \eqref{H-M},~\eqref{eq.hip1},~\eqref{eq.hip2}, and~\eqref{eq.hip3} are satisfied. 
  Let  $F\in L^{\infty}(\partial \Omega,\mathbb R)$ and $(\Sigma_{i})_{i\in\{1,\dots,\ft k_{1}^{\pa \Omega}\}}$ be a family  of disjoint open subsets of~$\pa \Omega$ such that 
  $$\text{for all } i\in\big \{1,\dots,\ft k_{1}^{\pa \Omega}\big \},  \ z_{i}\in \Sigma_{i},$$ 
  where we recall that $\big \{z_1,\dots,z_{\ft k_{1}^{\pa\Omega}}\big \}= \ft U_1^{\pa \Omega}\cap \argmin_{\pa \Omega} f$ (see~\eqref{eq.z11}).   Let~$K$ be a compact subset of $\Omega$ such that $ K\subset \mathcal A(\ft C_{\ft{max}})$ (see~\eqref{eq.hip1} and~\eqref{eq.ad}). Let~$\mu_0$ be a probability distribution which is  either supported in~$K$ or equals  to the quasi-stationary distribution~$\nu_h$ of the process~\eqref{eq.langevin}  in $\Omega$ (see Definition~\ref{defqsd} and~\eqref{eq.expQSD}).  
  Then:
  \begin{enumerate}[leftmargin=1.3cm,rightmargin=1.3cm]
  \item 
There exists $c>0$ such that in the limit  $h\to 0$:
\begin{equation} \label{eq.t1}
\mathbb E_{\mu_0} \left [ F\left (X_{\tau_{\Omega}} \right )\right]=\sum \limits_{i=1}^{\ft k_1^{\pa \Omega}}
\mathbb E_{\mu_0} \left [ \mathbf{1}_{\Sigma_{i}}F\left (X_{\tau_{\Omega}} \right )\right]  +O\big (e^{-\frac ch}\big )
 \end{equation}
 and
 \begin{equation} \label{eq.t2}
\sum \limits_{i=\ft k_1^{\pa \ft C_{\ft{max}} }+1}^{\ft k_1^{\pa \Omega}}
\mathbb E_{\mu_0} \left [ \mathbf{1}_{\Sigma_{i}}F\left (X_{\tau_{\Omega}} \right )\right]  =O\big (h^{\frac14} \big ),
 \end{equation}
 where we recall that $\big \{z_1,\ldots,z_{\ft k_1^{\pa \ft C_{\ft{max}} }}\big \}=\pa \ft C_{\ft{max}}\cap \pa \Omega$ (see~\eqref{eq.k1-paCmax}). 
\item When for some $i\in\big \{1,\dots,\ft k_{1}^{\pa \ft C_{\ft{max}} }\big \}$ the function  $F$ is $C^{\infty}$ in a neighborhood  of $z_{i}$, one has when $h\to 0$:
\begin{equation} \label{eq.t3}
 \mathbb E_{\mu_0} \left [ \mathbf{1}_{\Sigma_{i}}F\left (X_{\tau_{\Omega}} \right )\right]=F(z_i)\,a_{i} +O(h^{\frac14}),
\end{equation}
 where\label{page.ai}
 \begin{equation} \label{ai}
 a_i=\frac{  \partial_nf(z_i)      }{  \sqrt{ {\rm det \ Hess } f|_{\partial \Omega}   (z_i) }  } \left (\sum \limits_{j=1}^{\ft k_1^{\pa \ft C_{\ft{max}} }} \frac{  \partial_nf(z_j)      }{  \sqrt{ {\rm det \ Hess } f|_{\partial \Omega}   (z_j) }  }\right)^{-1}. \end{equation}
 \item 
When~\eqref{eq.hip4} is satisfied 
 the remainder term  $O( h^{\frac14})$ in \eqref{eq.t2}
 is   of the order  $O\big (e^{-\frac{c}{h}}\big )$ for some $c>0$  
and
the remainder term  $O\big (h^{\frac14}\big )$  in \eqref{eq.t3} is of the order $O(h)$  and admits a full asymptotic expansion in~$h$ (as defined in Remark~\ref{eq.asymptoO(h)} below). 
\end{enumerate} 
Finally,   the constants involved in  the remainder terms in~\eqref{eq.t1},~\eqref{eq.t2}, and~\eqref{eq.t3}  are uniform with respect to the probability distribution $\mu_0$ supported in~$K$. 
\end{theorem}

\begin{remark}\label{eq.asymptoO(h)}  Let us recall that for $\alpha>0$,~$(r(h))_{h>0}$ admits a full asymptotic expansion in~$h^\alpha$   if there exists a sequence $(a_k)_{k\geq 0}\in \mathbb R^{\mathbb N}$ such that for any $N\in \mathbb N$, it holds in the limit $h\to 0$: \begin{equation*}
r(h)=\sum_{k=0}^Na_kh^{\alpha k}+O\big (h^{\alpha(N+1)}\big ).
\end{equation*}
\end{remark}

\noindent
According to~\eqref{eq.t1}, when the function $F$ belongs to $C^{\infty}(\pa \Omega,\mathbb R)$ and $x\in \mathcal A(\ft C_{\ft{max}})$, one has  in the limit $h\to 0$:
$$
\mathbb E_{x} \left [ F\left (X_{\tau_{\Omega}} \right )\right] =
\sum_{i=1}^{\ft k_{1}^{\pa \ft C_{\ft{max}} } } a_i F(z_i) + O(h^{\frac14})= \frac{\sum \limits_{i=1}^{\ft k_{1}^{\pa \ft C_{\ft{max}} } }  \displaystyle{\int_{\Sigma_i} F \partial_nf \, e^{-\frac{2}{h}f }} }{\sum \limits_{i=1}^{\ft k_{1}^{\pa \ft C_{\ft{max}} } }  \displaystyle{\int_{\Sigma_i}  \partial_nf \, e^{-\frac{2}{h} f }  } }  + o_h(1), 
$$ 
where the order in $h$ of the remainder term $o_h(1)$  depends on the support of $F$ and on whether or not  the assumption~\eqref{eq.hip4} is satisfied.  
This is  reminiscent  of  previous  results  obtained  in~\cite{kamin1979elliptic,Kam,Per,day1984a,day1987r}.

Theorem~\ref{thm.main} implies that in the limit $h\to 0$, when $X_{0}\sim\nu_h$ or $X_0=x\in  \mathcal A(\ft C_{\ft{max}})$,    the law of  $X_{\tau_{\Omega}}$  concentrates on the set $\{z_{1},\dots, z_{\ft k_1^{\ft C_{\ft{max}} }}\}=\pa \Omega \cap \pa \ft C_{\ft{max}}$ with  explicit formulas  for the probabilities  to exit through each of the $z_i$'s. Therefore, \textbf{[P1]} and \textbf{[P2]} are satisfied when the assumptions~\eqref{eq.hip1},~\eqref{eq.hip2}, and~\eqref{eq.hip3} holds. 

Another consequence of   Theorem~\ref{thm.main} is the following.  The probability to exit through a global minimum $z$ of $f|_{\partial \Omega}$ which satisfies $\partial_nf(z)<0$ is exponentially small in the limit $h\to 0$ (see~\eqref{eq.t1}) and when  assuming~\eqref{eq.hip4}, the probability to exit through $ z_{\ft k_1^{\ft C_{\ft{max}}}+1},\dots, z_{\ft k_1^{\pa \Omega}} $ is also exponentially small even though  all these points belong to $\argmin_{\pa\Omega}f$.

%
%

Let us now give   two crucial results used  in the proof of \textbf{[P2]} in Theorem~\ref{thm.main}. The first   result shows that, when the assumptions \eqref{H-M}  and \eqref{eq.hip1} are satisfied, and $\min_{\ft   C_{\ft{max}}}f= \min_{\overline \Omega}f$ (which is automatically the case when \eqref{eq.hip1}, \eqref{eq.hip2}, and \eqref{eq.hip3}~hold, see Lemma~\ref{le.11}),     the quasi-stationary distribution $\nu_h$ (see Proposition~\ref{uniqueQSD})  concentrates in neighborhoods of the global minima of $f$ in $\ft   C_{\ft{max}}$.  This is stated in the following proposition. 
 \begin{proposition}\label{pr.con}
Assume that the assumptions \eqref{H-M}  and \eqref{eq.hip1}  are satisfied.  Furthermore, let us assume that 
$$\min_{\overline{\ft   C_{\ft{max}}}}f= \min_{\overline \Omega}f,$$
where we recall that  $\ft C_{\ft{max}}$ is introduced in~\eqref{eq.hip1}. Let $\ft O$ be an open subset of $\Omega$. 
 Then, if  $\ft O\cap \argmin_{\ft C_{\ft{max}}}f\neq \emptyset$,
 one has in the limit $h\to 0$:
$$
\nu_h\big( \ft O\big)  = \frac{   \sum_{x\in\ft O\cap \argmin_{\ft C_{\ft{max}}}f}  \big( {\rm det \ Hess } f   (x)   \big)^{-\frac12}  }{  \sum_{x\in \argmin_{\ft C_{\ft{max}}}f}  \big( {\rm det \ Hess } f   (x)   \big)^{-\frac12}  }\  \big(1+O(h) \big).
$$
When $\overline{\ft O}\cap \argmin_{\ft C_1}f= \emptyset$,  
 there exists $c>0$ such that  when $h \to 0$:
$$
\nu_h\big( \ft O\big) =O\big ( e^{-\frac{c}{h}} \big).
$$
\end{proposition}
Proposition~\ref{pr.con} is a direct consequence of~\eqref{eq.expQSD} and Proposition~\ref{pr.masse} below (see the beginning of  Section~\ref{section-4}).\\
 The second  result used in the proof of Theorem~\ref{thm.main}  connects the law of $X_{\tau_\Omega}$ when $X_0\sim \nu_h$ and $X_0=x\in \mathcal A(\ft C_{\ft{max}})$ in the limit $h\to 0$.  This is stated in the following proposition.      
 \begin{proposition}\label{pr.exp-qsd-x}
Assume that the assumptions \eqref{H-M}  and \eqref{eq.hip1}  are satisfied.  Let us moreover assume that 
$$\min_{\overline{\ft   C_{\ft{max}}} }f= \min_{\overline \Omega}f,$$
where we recall that  $\ft C_{\ft{max}}$ is introduced in~\eqref{eq.hip1}.  
  Let $K$ be  a compact subset of $\Omega$ such that $K\subset \mathcal A(\ft   C_{\ft{max}})$ and let $F\in C^{\infty}(\partial \Omega,\mathbb R)$. Then, there exists $ c>0$ such that for all $x\in K$:
$$
\mathbb E_{\nu_h}  \left [ F\left (X_{\tau_{\Omega}}\right)\right ]=\mathbb E_{x}  \left [ F\left (X_{\tau_{\Omega}}\right)\right ]  +O\big (e^{-\frac{c}{h} }\big )
$$
in the limit $h \to 0$ and uniformly in~$x \in K$.
\end{proposition}
\noindent
Proposition~\ref{pr.exp-qsd-x} is a direct consequence of Lemma~\ref{le.exp-qsd-x} below (see Section~\ref{se3}). It gives sufficient conditions to ensure that \textbf{[P2]} is satisfied. 
\medskip

\noindent
Let us end this section with the following  theorem dealing with the case when   $X_0=x\in \mathcal A(\ft C)$, when $\ft C\in \mathcal C$ (see~\eqref{mathcalC-def})    is not necessarily $\ft C_{\ft{max}}$.

\begin{theorem}\label{thm.2}
Let us assume that  \eqref{H-M} holds. Let $\ft C\in \mathcal C$ (see~\eqref{mathcalC-def}). Let us assume that 
\begin{equation}\label{eq.cc1}
\pa \ft C\cap \pa \Omega\neq \emptyset \  \text{ and } \  \vert \nabla f\vert \neq 0 \text{ on } \pa \ft C.
\end{equation}
 Recall that  $\pa \ft C\cap \pa \Omega \subset \ft U_1^{\pa \Omega}$ (see~\eqref{eq.mathcalU1_bis} and~\eqref{eq.U1paOmega}). 
For all $z \in \pa \ft C\cap \pa \Omega$, let  $\Sigma_{z}$ be   an open subset of~$\pa \Omega$ such that $ z\in \Sigma_{z}$. 
Let $K$ be a compact subset of $\Omega$ such that $K\subset \mathcal A(\ft C)$. Then, there exists $c>0$ such that for  $h$ small enough,
$$\sup_{x\in K}\mathbb P_x\Big [X_{\tau_\Omega}\in  \pa \Omega \setminus  \bigcup_{z\in \pa \ft C\cap \pa \Omega} \Sigma_{z}\Big ]\le e^{-\frac ch}.$$
Assume moreover that   the sets $(\Sigma_{z})_{z\in \pa \ft C\cap \pa \Omega}$  are two by two disjoint.  Let $z\in \pa \ft C\cap \pa \Omega$. Then,   it holds for all $x\in K$, 
$$\mathbb P_{x}[X_{\tau_\Omega}\in \Sigma_z]= \frac{  \partial_nf(z)      }{  \sqrt{ {\rm det \ Hess } f|_{\partial \Omega}   (z) }  } \left (\sum \limits_{y\in\pa \ft C \cap \pa \Omega  } \frac{  \partial_nf(y)      }{  \sqrt{ {\rm det \ Hess } f|_{\partial \Omega}   (y) }  }\right)^{-1}(1+O(h)),$$
in the limit $h \to 0$ and uniformly in~$x \in K$.
\end{theorem}

\noindent
Theorem~\ref{thm.2} implies that when $ \ft C \in \mathcal C$ satisfies~\eqref{eq.cc1} (for instance, this is the case for~$\ft C_3$ on Figures~\ref{fig:okay} and~\ref{fig:shema_nota}),    the law of $X_{\tau_\Omega}$ when $X_0=x\in \mathcal A(\ft C)$ concentrates when $h\to 0$ on  $\pa \ft C \cap \pa \Omega$. Let us  mention that the proof of Theorem~\ref{thm.2} is based on the use of Theorem~\ref{thm.main} with a suitable subdomain of $\Omega$ containing  $\ft C$.

When ${\ft  C}\in \mathcal C$ and does not satisfy~\eqref{eq.cc1},  it is much harder to   exhibit explicit assumptions  on $\ft C$ to give the most probable places of exit from $\Omega$  of the process~\eqref{eq.langevin} when $h\to 0$  (as suggested by one dimensional-examples, see Appendix B).

\subsubsection{Intermediate results on the smallest eigenvalues  and on the principal eigenfunction of  $-L^{D,(0)}_{f,h}$}\label{se.interm}
Let us recall that from~\eqref{eq.dens}, one has:
$$\mathbb E_{\nu_h}\left [ F\left (X_{\tau_{\Omega}} \right )\right]=- \frac{h}{2\lambda_h} \frac{ \displaystyle \int_{\pa \Omega} F\,  \partial_n
  u_h  e^{-\frac{2}{h} f }}{\displaystyle \int_\Omega u_h  e^{-\frac{2}{h} f }  }.$$
Therefore, to obtain  the asymptotic estimates on $\mathbb E_{\nu_h}\left [ F\left (X_{\tau_{\Omega}} \right )\right]$ stated
 in Theorem~\ref{thm.main} when $h\to 0$, i is sufficient to study  the asymptotic behaviour of the quantities 
$$\lambda_h, \ \,  \int_\Omega u_h e^{-\frac 2h f},\ \, \text{ and}\ \,  \pa_nu_h. $$
The intermediate results we obtain  are the following. 
\begin{enumerate}[leftmargin=1.3cm,rightmargin=1.3cm]
\item In Theorem~\ref{pp}, one gives for  $h\to 0$ small enough, a   lower and an upper bound for all the $\ft m_0^\Omega$ small  eigenvalues of   $-L^{D,(0)}_{f,h}$ when~\eqref{H-M} is satisfied.  
\item  In Theorems~\ref{thm-big0} and~\ref{re-specific-case}, one gives a sharp asymptotic equivalent in the limit $h \to 0$  of the smallest eigenvalue $\lambda_h$ of $-L^{D,(0)}_{f,h}$ when~\eqref{H-M} and~\eqref{eq.hip1} are satisfied.   

\item  In Proposition~\ref{pr.masse}, when~\eqref{H-M},~\eqref{eq.hip1}  and 
$\min_{\overline{\ft C_{\ft{max}}}}f= \min_{\overline \Omega}f$ hold, one shows that  $u_h\,e^{-\frac 2h f}$ concentrates in the   $L^1(\Omega)$-norm       on the global minima of~$f$ in~$\ft C_{\ft{max}}$ in the limit $h\to 0$. 

\item \begin{sloppypar}
In Theorem~\ref{thm-big-pauh}, one studies the   concentration  in the limit $h \to 0$ of the normal derivative  of the principal eigenvalue $u_h$ of $-L^{D,(0)}_{f,h}$ on $\pa \Omega$ when~\eqref{H-M},\eqref{eq.hip1}, \eqref{eq.hip2}, and~\eqref{eq.hip3} are satisfied.  In particular, one computes sharp asymptotic equivalents  of $\pa_n u_h$ in neighborhoods of  $\pa \ft C_{\ft{max}}\cap \pa \Omega$  in   $\pa \Omega$. 
\end{sloppypar}
\end{enumerate}


\subsection{Discussion of the hypotheses}
\label{discussion-hyp}
In this section,  we discuss the necessity of the assumptions~\eqref{eq.hip1}, \eqref{eq.hip2}, and \eqref{eq.hip3} to obtain \textbf{[P1]} and \textbf{[P2]}.  We also discuss the necessity of the  assumption \eqref{eq.hip4} in order to get item 3 in Theorem~\ref{thm.main}.

\subsubsection{On the assumption~(\ref{H-M})}
\label{sec.on-A0}

The results of this work actually still hold under a weaker assumption 
than~\eqref{H-M}, namely by simply assuming that $f : \{x \in  \pa \Omega, \pa_n f(x) > 0\} \to \mathbb R$ 
is a Morse function instead of $f|_{ \pa \Omega }$ is a Morse function. Indeed, as 
mentioned in~\cite[Section 7.1]{HeKoSu}, the statement of Lemma~\ref{ran1} (which is the only 
place where we use the fact that $f$ is a Morse function on $\{x \in  \pa \Omega, \pa_n f(x) \le 0\}$, relying on~\cite[Section 3.4]{HeNi1}) still 
holds under this weaker assumption, see Appendix A.


\subsubsection{On the assumption~(\ref{eq.hip1})}
In this section, we discuss the necessity of the assumption \eqref{eq.hip1} in order to obtain the results of Theorem~\ref{thm.main} (or equivalently  \textbf{[P1]} and \textbf{[P2]}). More precisely, one first exhibits a case where  \eqref{eq.hip1} and  \textbf{[P2]} are not satisfied. Then,  one shows that there are cases where \textbf{[P1]} and  \textbf{[P2]} are satisfied but not \eqref{eq.hip1}. Finally, one explains why it is more difficult to analyse  \textbf{[P1]} and  \textbf{[P2]} when  \eqref{eq.hip1} does no hold. 
 \medskip


\noindent
\textbf{An example where \eqref{eq.hip1}  and \textbf{[P2]} are not satisfied}.
\medskip

\noindent
Let us consider $
 z_1>0$,~$ z_2:=-z_1$,~$z=0$
and $f\in C^{\infty}([z_1,z_2],\mathbb R)$ a Morse function such that  
$$
f \text{ is an even function}, \  \{x\in [z_1,z_2], f'(x)=0\}=\{x_1,z,x_2\},
$$
where 
$$
z_1<x_1<z<x_2<z_2, \, f(z_1)=f(z_2), \, f(x_1)=f(x_2)<f(z_1)< f(z).
$$
\begin{sloppypar}
\noindent
Notice that in this case $x_1=-x_2$,~$x_1$ and $x_2$  are the two global minima of  $f$ on $[z_1,z_2]$,~$z$~is the global maximum of~$f$ on $[z_1,z_2]$ and $\ft H_f (x_1)=\ft H_f (x_2)=f(z_1)$. Such a function~$f$ is represented in Figure~\ref{fig:hip1}. 
For such functions $f$, the assumption \eqref{eq.hip1} is not satisfied since $\argmax\,  \big \{ \ft H_f (x)-f(x), \ x \text{ is local minimum of } f \text{ in } \Omega \big\} =\{x_1,x_2\}$ and  $x_1$ belongs to a connected component of $\{f<\ft H_f(x_1)\}$ which differs from the connected component of $\{f<\ft H_f(x_1)\}$ which contains $x_2$.  
\end{sloppypar}

 \begin{figure}[h!]
\begin{center}
\begin{tikzpicture}
  \draw [domain=1.37:3.56*pi, scale=0.7, samples=300, smooth] plot (\x,{2*cos(\x r) });
   \tikzstyle{vertex}=[draw,circle,fill=black,minimum size=5pt,inner sep=0pt]
      \draw [dashed] (1,0.26)--(10,0.26); 
   \draw (12,0.23) node[]{$\{f=\min_{\pa \Omega }f\}$};   
   
    \draw (1,0.26) node[vertex,label=west: {$z_1$}](v){}; 
        \draw (7.8,0.26) node[vertex,label=north: {$z_2$}](v){};
        \draw (4.4,1.4) node[vertex,label=north: {$z$}](v){};  
 \draw (2.17,-1.4) node[vertex,label=south: {$x_1$}](v){}; 
  \draw (6.63,-1.4) node[vertex,label=south: {$x_2$}](v){};  
    \draw [dashed] (2.17,-1.4)--(10,-1.4); 
   \draw (12,-1.4) node[]{$\{f=\min_{  \Omega }f\}$};   
 
    \end{tikzpicture}
    \caption{A one-dimensional example when~\eqref{eq.hip1} and \textbf{[P2]} are not satisfied.}
    \label{fig:hip1}
    \end{center}
\end{figure}
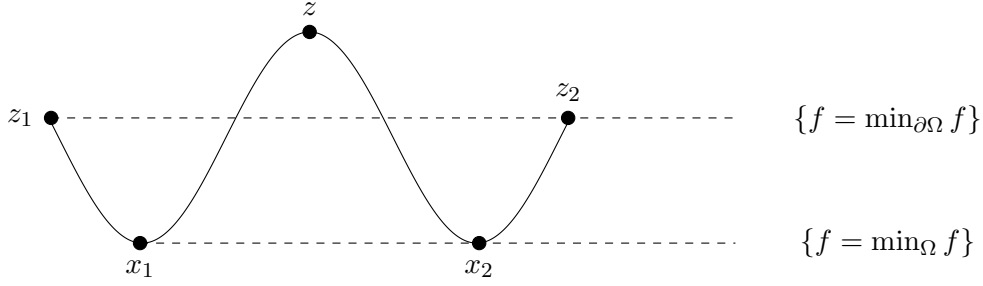

\noindent
Since for $x\in (z_1,z_2)$ and $h> 0$, $\nu_h(x)=\nu_h(-x)$ and $\mathbb P_{x}  [X_{\tau_{(z_1,z_2)}}=z_1 ]=\mathbb P_{-x}  [X_{\tau_{(z_1,z_2)}}=z_2 ]$, one has for all $h> 0$:
$$
\mathbb P_{\nu_h}  [X_{\tau_{(z_1,z_2)}}=z_1 ]=\frac 12  \ \text{ and }\  \mathbb P_{\nu_h}  [ X_{\tau_{(z_1,z_2)}}=z_2 ]=\frac 12. 
$$
However, from~\eqref{exp} below (see the Appendix) together with  Laplace's method,   for $x\in (z_1,z)$, there exists $c>0$ such that in the limit  $h\to 0$:
\begin{equation}\label{discussion1}
\mathbb P_{x}  [X_{\tau_{(z_1,z_2)}}=z_1 ]=1+O( e^{-\frac ch}), \ \text{ and }\  \mathbb P_{x}  [ X_{\tau_{(z_1,z_2)}}=z_2 ]=O( e^{-\frac ch}),
\end{equation}
and for $x\in (z,z_2)$, there exists $c>0$ such that in the limit  $h\to 0$:
$$
\mathbb P_{x}  [X_{\tau_{(z_1,z_2)}}=z_1 ]=O( e^{-\frac ch}), \ \text{ and }\  \mathbb P_{x}  [ X_{\tau_{(z_1,z_2)}}=z_2 ]=1+O( e^{-\frac ch}).
$$
Therefore, in this example, the assumption \textbf{[P2]} is not satisfied.
The domain~$\Omega$ is not metastable  (see Section~\ref{qsd})  for any    deterministic initial conditions $X_0=x\in [z_1,z_2]\setminus\{z\}$.  
\medskip

\noindent
\textbf{There are cases when  \textbf{[P1]} and  \textbf{[P2]} are satisfied but not~\eqref{eq.hip1}.}
\medskip 

 \noindent
In the symmetric case depicted in Figure~\ref{fig:hip1}  the quasi-stationary distribution $\nu_h$ concentrates  in the two wells $(z_1,z)$ and $(z,z_2)$ (see~\cite{LPN2018}): i.e. for any $a_1<b_1$ such that  $(a_1,b_1)\subset (z_1,z)$ and $x_1\in (a_1,b_1)$, and  $a_2<b_2$ such that $(a_2,b_2)\subset (z,z_2)$ and $x_2\in (a_2,b_2)$, it holds
$$\lim_{h\to 0}\nu_h\big( (a_1,b_1)\big)=\frac 12 \ \text{ and }\ \lim_{h\to 0}\nu_h\big((a_2,b_2)\big)=\frac 12.$$ 
However, it is proved in~\cite{LPN2018}, that this equal repartition of~$\nu_h$ when $h\to 0$ is  unstable with respect to perturbations.
Indeed, changing a little bit  the value of the determinant of the hessian matrix  at~$x_1$ or~$x_2$, or the normal derivative  at~$z_1$ or $z_2$ of  the symmetric potential~$f$ depicted in Figure~\ref{fig:hip1} (while keeping the fact that~\eqref{eq.hip1} is not satisfied) makes~$\nu_h$  concentrates in the limit $h\to 0$ in only one of the two  wells $(z_1,z)$ or $(z,z_2)$, and  \textbf{[P1]} and \textbf{[P2]} then also hold.

\begin{remark}  
 The main goal of~\cite{LPN2018} is to study the repartition of~$\nu_h$ when $h\to 0$ in the double-well case (where~\eqref{eq.hip1} does not hold). In particular, in this case, it is shown   that the asymptotic behaviour when $h\to 0$ of $\nu_h$  which generically  happens is the following:~$\nu_h$ concentrates in only one of the two wells in the limit $h\to 0$, and     \textbf{[P1]} and \textbf{[P2]} hold. 
\end{remark}
\medskip 

 \noindent
 \textbf{On the analysis of~\textbf{[P1]} and~\textbf{[P2]} when~\eqref{eq.hip1} does not hold.}
 

\medskip

\noindent
To analyse whether   \textbf{[P1]} or  \textbf{[P2]} is satisfied when \eqref{eq.hip1} does not hold, one needs in particular to  have access to the repartition of $\nu_h$ in neigborhoods  of the local minima of $f$ in $\Omega$ when $h\to 0$.  

When~\eqref{eq.hip1} is not satisfied,   the analysis of the repartition of $\nu_h$  is  tricky. This can be explained as follows. When~\eqref{eq.hip1}  is not satisfied, one has    from~Theorem~\ref{pp} below (see Section~\ref{sec.e2}),
$$\lim_{h\to 0} h\ln \lambda_h=\lim_{h\to 0} h \ln  \lambda_{2,h},$$
where $\lambda_{2,h}$ is the second smallest eigenvalue of $-L^{D,(0)}_{f,h}$. It is difficult to measure the quality of the approximation of $u_h$ by 
projecting an ansatz on ${\rm Span}(u_h)$, since the error is related to the 
ratio of $\lambda_{h}$ over $\lambda_{2,h}$ (see Lemma~\ref{quadra}).
 Moreover, when~\eqref{eq.hip1} is not satisfied,  it is difficult to predict in which well $\nu_h$ concentrates when it does, as explained in~\cite{LPN2018}.   This is again due to the fact that this prediction relies on  a very accurate  comparison between   $\lambda_{h}$ and~$\lambda_{2,h}$.

On the contrary, when  the assumption~\eqref{eq.hip1} is satisfied, one can more easily   obtain  an approximation of  $u_h$ (see~\eqref{eq.uh=} below) since in that case  $\lim_{h\to 0} h\ln \lambda_h<\lim_{h\to 0} h \ln  \lambda_{2,h}$ and thus Lemma~\ref{quadra}  provides a sufficiently accurate error estimate of the 
approximation of $u_h$ by a simple ansatz (namely a cut-off function), see indeed Proposition~\ref{pr.con} above and Proposition~\ref{pr.masse} below. 


%
%
%


\subsubsection{On the assumption~(\ref{eq.hip2})}
\label{sec.A2}
In this section, we discuss the assumption~\eqref{eq.hip2} to obtain the results stated in Theorem~\ref{thm.main}.  To this end, let us consider the  following one-dimensional example. Let $z_1<z_2$ and $f$: $[z_1,z_2]\to \mathbb R$ be a $C^{\infty}$ Morse function. Let us assume that $\{x\in [z_1,z_2], f'(x)=0\}=\{x_1,x_2,c,d\}$ with $z_1<x_1<c<x_2<d<z_2$,~$f(x_2)<f(x_1)<f(z_1)<f(z_2)<f(d)<f(c)$ (see Figure~\ref{fig:hip2}). This implies that $f'(z_2)<0$,~$f(d)-f(x_2)>f(z_1)-f(x_1)$. Moreover, it holds
$$\ft H_f(x_1)=f(z_1),\, \ft H_f(x_2)=f(d),\ f(z_1)=\min_{\pa \Omega}f, \ \ft C_{\ft{max}}\subset (c,d) \, \text{ and } \, \pa \ft C_{\ft{max}}\cap \pa \Omega=\emptyset.$$ 
The assumption~\eqref{eq.hip1} is satisfied but not~\eqref{eq.hip2}.
From~\eqref{discussion2-b} below (see Appendix B), there exists $c>0$ such that in the limit $h\to 0$:
\begin{equation}\label{discussion2}
   \mathbb P_{\nu_h}  [ X_{\tau_{(z_1,z_2)}}=z_2 ]= 1+O ( e^{- \frac{c}{h} }).\end{equation}
Therefore, in the small temperature regime and starting from the quasi-stationary distribution, the process~\eqref{eq.langevin} leaves $\Omega=(z_1,z_2)$    through $z_2$ when $h\to 0$. Notice that $z_2$  is not the global minimum of $f|_{\partial \Omega}$ and is even not a generalized critical point of index $1$. Consequently, the condition~\textbf{[P1]} is not satisfied.


\begin{figure}[h!]
\begin{center}
\begin{tikzpicture}
\draw (-0.082,-0.49) ..controls (-0.2,-1)   .. (-1,-0.49) ;
  \draw [domain=-0.17:2.9*pi, scale=0.5, samples=300, smooth] plot (\x,{6*sin(\x r)*exp(-0.2*\x)});
    \draw[densely dashed, <->] (1.5,0.65)-- (3.6,0.65);
    \draw (2.5,0.92) node[]{$\ft C_{\ft{max}}$};
  \draw (2.2,-1.5) node[]{$x_2$};
  \draw (2.3,-1.19) node[]{$\bullet$};
    \draw (-0.2,-1.2) node[]{$x_1$};
  \draw (-0.25,-0.87) node[]{$\bullet$};
   \draw (-1.26,-0.5) node[]{$z_1$};
  \draw (-1,-0.48) node[]{$\bullet$};  
  \draw (4.9,0) node[]{$z_2$};
  \draw (4.51,0.2) node[]{$\bullet$};
  \draw (3.8,1.1) node[]{$d$};
  \draw (3.8,0.65) node[]{$\bullet$};
      \draw (0.73,2.5) node[]{$c$};
  \draw (0.73,2.23) node[]{$\bullet$};
    \end{tikzpicture}
    \caption{A one-dimensional example when~\eqref{eq.hip1} is satisfied  but not the assumption~\eqref{eq.hip2}. In this example,  \textbf{[P1]} is  not satisfied. }
    \label{fig:hip2}
    \end{center}
\end{figure}
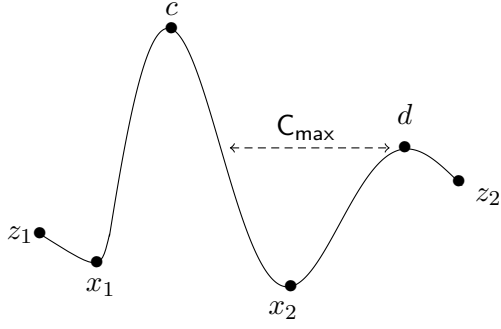


\subsubsection{On the assumption~(\ref{eq.hip3})}
\label{sec.A3}
In this section, we discuss the assumption~\eqref{eq.hip3} to obtain the results of Theorem~\ref{thm.main}. To this end, let us consider the following one-dimensional case. Let $z_1<z_2$ and $f$: $[z_1,z_2]\to \mathbb R$ be a $C^{\infty}$ Morse function. Let us assume that $\{x\in [z_1,z_2], f'(x)=0\}=\{x_1,z,x_2\}$ where $z_1<x_1<z<x_2<z_2$ $f(x_2)<f(x_1)<f(z_1)<f(z_2)<f(z)$ (see Figure~\ref{fig:hip3}). This implies    $f(z_1)-f(x_1)<f(z_2)-f(x_2)$,~$f'(z_1)<0$,~$f'(z_2)>0$,~$x_2$ is the global minimum of~$f$ in~$[z_1,z_2]$,~$x_1$ is a local minimum of~$f$   and $z$ the global maximum of~$f$ in  $[z_1,z_2]$. Then it holds,
$$\ft H_f(x_1)=f(z_1),\, \ft H_f(x_2)=f(z_2),\ f(z_1)=\min_{\pa \Omega}f, \ \pa \ft C_{\ft{max}}\cap \pa \Omega=\{z_2\},     $$
and $\ft C_{\ft{max}}\subset (z,z_2)$. 
The assumptions~\eqref{eq.hip1} and \eqref{eq.hip2} are satisfied but not \eqref{eq.hip3}.
From~\eqref{eq.discussion3-b} below (see  Appendix B),  there exists $c>0$ such that  in the limit $h\to 0$:
\begin{equation}\label{discussion3}
\mathbb P_{\nu_h}  [ X_{\tau_{(z_1,z_2)}}=z_2 ]=1+O(e^{-\frac{c}{h}}).
\end{equation}
Therefore, when $X_0\sim \nu_h$, the law of $X_{\tau_\Omega}$ concentrates  on $z_2$ in the limit $h\to 0$. Since $f(z_2)>\min_{\pa \Omega}f$, the property~\textbf{[P1]} is not satisfied.


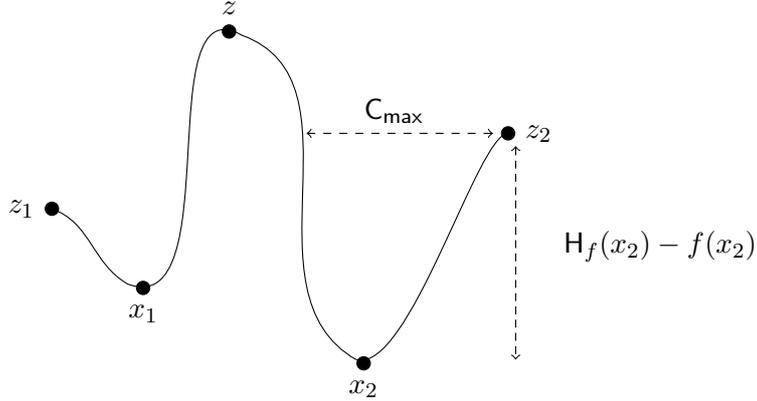
\begin{figure}[h!]
\begin{center}
\begin{tikzpicture}
\coordinate (b1) at (0,3);
\coordinate (b2) at (-2 , 0);
\coordinate (b2a) at (-3.5 , 4.3);
\coordinate (b2b) at (-5 , 1);
\coordinate (b2c) at (-6 , 2);
  \draw (b2) ..controls (-1.3,-0.2) and  (-0.4,3)   ..  (b1)  ;
 
\draw [black!100, in=150, out=-10, tension=10.1]
  (b2c)[out=-20]  to (b2b) to  (b2a) to (b2);
  \draw [dashed, <->]   (-0.2,3) -- (-2.65,3) ;
 \draw (-1.5,3.3) node[]{$ \ft C_{\ft{max}}$};
   \draw [densely  dashed, <->]  (0.1,2.83)-- (0.1,0) ;
 \draw (2,1.5) node[]{$ \ft H_f(x_2)- f(x_2)$};
     \tikzstyle{vertex}=[draw,circle,fill=black,minimum size=5pt,inner sep=0pt]
       \draw (b1) node[vertex,label=east: {$z_2$}](v){}; 
     \draw  (b2c)  node[vertex,label=west: {$z_1$}](v){}; 
     \draw (-1.9,-0.05) node[vertex,label=south: {$ x_2$}](v){}; 
     \draw (-4.8 , 0.95) node[vertex,label=south: {$x_1$}](v){};
          \draw (-3.67 , 4.35) node[vertex,label=north: {$z$}](v){};
    \end{tikzpicture}
\caption{a one-dimensional case where~\eqref{eq.hip1} and~\eqref{eq.hip2} are satisfied but not the assumption~\eqref{eq.hip3}. In this example,  \textbf{[P1]} is  not satisfied.}
 \label{fig:hip3}
 \end{center}
\end{figure}

\subsubsection{On the assumption~(\ref{eq.hip4})}
\label{sec.A4}
In this section, one gives an example to show that when~\eqref{eq.hip4} is 
 not satisfied,  the remainder term  $O( h^{\frac14})$ in \eqref{eq.t2} is not of the order $O(e^{-\frac{c}{h}})$ for some $c>0$. 
To this end, let us consider the following one-dimensional case. Let $z_1<z_2$ and $f$: $[z_1,z_2]\to \mathbb R$ be a $C^{\infty}$ Morse function. Let us assume that $\{x\in [z_1,z_2], f'(x)=0\}=\{x_1,z,x_2\}$ with $z_1<x_1<z<x_2<z_2$ and $f(x_1)<f(x_2)<f(z)=f(z_1)=f(z_2)$ (see Figure~\ref{fig:hip4}). This implies $f'(z_1)<0$,~$f'(z_2)>0$,~$x_1$ is the global minimum of~$f$ in~$[z_1,z_2]$,~$x_2$ is a local minimum of~$f$   and $z$ is a local maximum of~$f$. In this example, it holds:  
$$\ft H_f(x_1)=f(z_1)=\min_{\pa \Omega}f,\  \ft C_{\ft{max}}=(z_1,z),  \, \pa \ft C_{\ft{max}}\cap \pa \Omega=\{z_1\},$$
and 
$$\ft C=(z, z_2),$$
where $ \ft C\neq  \ft C_{\ft{max}}$ is the other connected component of $\{f< \ft H_f(x_1)   \}$. The assumptions~\eqref{eq.hip1}, \eqref{eq.hip2}, and~\eqref{eq.hip3} are satisfied whereas, since $ z\in \pa \ft C_{\ft{max}} \cap \Omega$ is  a separating saddle point of $f$  (see item 1 in Definition~\ref{de.SSP} below), the hypothesis~\eqref{eq.hip4} is not satisfied. 
From~\eqref{exp} (in Appendix B) together with  Laplace's method, for $x\in \ft C_{\ft{max}}$, one has in the limit $h\to 0$:
\begin{equation}\label{discussion4}
\mathbb P_{x}  [ X_{\tau_{(z_1,z_2)}}=z_2 ]=\frac{\sqrt{\vert f''(z)\vert}}{2\vert f'(z_1)\vert \sqrt \pi } \sqrt h  +O(h).
\end{equation}
Moreover, a similar result holds starting from $\nu_h$ (using~Proposition~\ref{pr.exp-qsd-x} above): in the limit $h\to 0$:
$$\mathbb P_{\nu_h}  [ X_{\tau_{(z_1,z_2)}}=z_2 ]=\frac{\sqrt{\vert f''(z)\vert}}{2\vert f'(z_1)\vert \sqrt \pi } \sqrt h  +O(h).$$
In this case, the exit through $z_2$ when $h\to 0$ is not exponentially small but is exactly of the order $\sqrt h$ even though  $z_2$ is a generalized critical point of~$f$ on $\pa \Omega$ (i.e $f(z_2)\in \ft U_1^{\pa \Omega}$, see~\eqref{eq.mathcalU1_bis})  and $f(z_2)=\min_{\pa \Omega}f$. In conclusion,    the remainder term  $O( h^{\frac14})$ in \eqref{eq.t2} is in general not of the order $O(e^{-\frac{c}{h}})$  and is actually exactly of the order $O(\sqrt h)$ in this example.

\begin{remark}\label{re.intro1}This can be generalized to higher-dimensional settings. 
In~\cite[Proposition C.40, item 3]{BN2017}, one shows with some higher-dimensional cases for which the assumption~\eqref{eq.hip4}  does not hold, that the remainder terms $O\big (h^{\frac 14}\big )$ in~\eqref{eq.t2}  and~\eqref{eq.t3} are of the order $O(\sqrt h)$.   We moreover expect that the reminder terms $O\big (h^{\frac 14}\big )$ in~\eqref{eq.t2}  and~\eqref{eq.t3} are of the order $O(\sqrt h)$ in the setting considered in Theorem~\ref{thm.main}. Proving this fact would require some substantially  finer analysis. 
\end{remark}

\begin{figure}[h!]
\begin{center}
\begin{tikzpicture}
  \draw [domain=1.7:2*pi, scale=0.5, samples=300, smooth] plot (\x,{2*cos(\x r) });
    \draw [domain=2*pi:3.5*pi, scale=0.5, samples=300, smooth] plot (\x,{1.6*cos(\x r) +0.4});
 
 \draw (0.3,1)--(0.867,-0.167); 
\draw (5.5,0.2)--(6,1); 
\draw [dashed] (0.3,1)--(7.5,1); 
  \draw (10.6,1) node[]{$\{f=\min_{\pa \Omega}f\}=\{f= \ft H_f(x_1)  \}$};
  \draw (0.06,1.3) node[]{$z_1$};
  \draw (0.3,1) node[]{$\bullet$};
  \draw (6.25,1.3) node[]{$z_2$};
  \draw (6,1) node[]{$\bullet$};
  \draw (3.2,1.4) node[]{$z$};
  \draw (3.2,1) node[]{$\bullet$};
   \draw (1.58,-1.45) node[]{$ x_1$};
  \draw (1.58,-1) node[]{$\bullet$};
    \draw (4.8,-0.89) node[]{$x_2$};
  \draw (4.7,-0.62) node[]{$\bullet$};

   \draw[densely dashed, <->] (0.5,1.2)-- (3, 1.2);
      \draw (1.5,1.5) node[]{$\ft C_{\ft{max}}$};
       \draw[densely dashed, <->](3.3, 1.2)-- (5.8, 1.2) ;
      \draw (4.7,1.5) node[]{$\ft C$};
    \end{tikzpicture}
    \caption{A one-dimensional case where~\eqref{eq.hip1},~\eqref{eq.hip2},  and~\eqref{eq.hip3} hold but not~\eqref{eq.hip4}.}
    \label{fig:hip4}
    \end{center}
\end{figure}
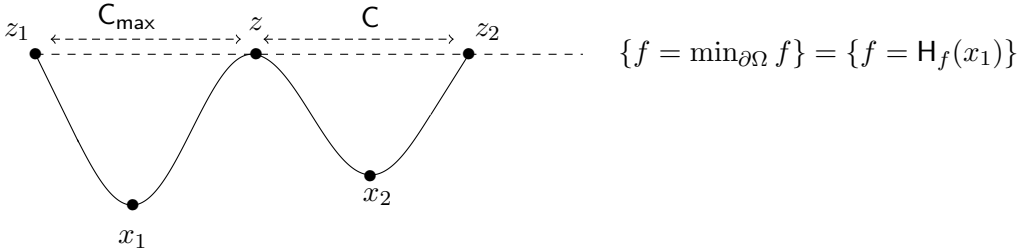


\subsection{Organization of the paper and outline of the proof }
 
 \label{expl-preuve}
The aim of this section is to give an overview of the strategy of the
proof of Theorem~\ref{thm.main}. From~\eqref{eq.dens} and in order to obtain an asymptotic estimate of $\mathbb E_{\nu_h}\left [ F\left (X_{\tau_{\Omega}} \right )\right]$,  we study the asymptotic behaviour when $h\to 0$ of the quantities 
$$\lambda_h, \ \,  \int_\Omega u_h e^{-\frac 2h f}\ \, \text{ and}\ \,  \pa_nu_h,$$
where $\lambda_h$ is defined by~\eqref{eq.lh} and $u_h$ by~\eqref{eq:u-bis}.   \\
To study  $\lambda_h$ and $\pa_nu_h$, the first key point is to notice that the gradient of any eigenfunction  associated with an eigenvalue of $-L^{D,(0)}_{f,h}$ is also a solution to an eigenvalue problem for the same eigenvalue. Let us be more precise. Let $v$ be an eigenfunction associated with $ \lambda\in \sigma(-L^{D,(0)}_{f,h})$. 
The eigenvalue-eigenfunction pair~$(\lambda,v)$ satisfies:
$$
\left\{
\begin{aligned}
 -L^{(0)}_{f,h}\, v &=  \lambda v    \ {\rm on \ }  \Omega,  \\ 
v&= 0 \ {\rm on \ } \partial \Omega.
\end{aligned}
\right.
$$
 By differentiating this relation,
we observe that $\nabla v$
satisfies
\begin{equation}\label{eq:L1_eig}
\left\{
\begin{aligned}
-L^{(1)}_{f,h} \nabla v &=  \lambda \nabla v \text{ on $\Omega$,}\\
\nabla_T v& = 0 \text{ on $\partial \Omega$,}\\
\left(\frac{h}{2} {\rm div} - \nabla f \cdot \right) \nabla v & = 0 \text{ on $\partial \Omega$,}\\
\end{aligned}
\right.
\end{equation}
where
\begin{equation}\label{eq.L-1}
L^{(1)}_{f,h}= \frac{h}{2} \Delta - \nabla f \cdot \nabla - {\rm
  Hess} \, f
\end{equation}
is an operator acting on $1$-forms (namely on vector fields).
Therefore, the  vector field~$\nabla v$ is an eigen-$1$-form 
of  the operator $-L^{D,(1)}_{f,h}$ which is the operator  $-L^{(1)}_{f,h}$ with tangential Dirichlet boundary
conditions (see~\eqref{eq:L1_eig}), associated with the eigenvalue~$\lambda$. \\
The second key point   (see for
example~\cite{HeNi1})  is  that, when~\eqref{H-M} holds,~$-L^{D,(0)}_{f,h} $ admits exactly $\ft m_0^\Omega$ 
eigenvalues smaller than~$\frac{\sqrt{h}}{2}$ (where we recall that~$\ft m_0^\Omega$ is the number of local minima of~$f$ in~$\Omega$)  and that~$-L^{D,(1)}_{f,h} $  admits exactly $\ft m_1^{\overline \Omega}$
eigenvalues smaller than $\frac{\sqrt{h}}{2}$ (where, we recall that~$\ft m_1^{\overline \Omega}$ is the
number generalized saddle points of~$f$ in~$\overline \Omega$).  Actually, all these small
eigenvalues are exponentially small in the regime~$h \to 0$ (namely they are bounded from above by $e^{-\frac ch}$ for some $c>0$),
the other eigenvalues being bounded from below by a constant in this regime. This implies in particular that $\lambda_h$ is an exponentially small eigenvalue of $-L^{D,(1)}_{f,h} $. Let us denote by~$\pi_h^{(0)}$ (resp. $\pi_h^{(1)}$) the projector onto the vector space spanned by the eigenfunctions (resp. eigenforms) associated with the 
$\ft m_0^\Omega$ (resp. $\ft m_1^{\overline \Omega}$)  smallest eigenvalue of $-L^{D,(0)}_{f,h} $ (resp. of $-L^{D,(1)}_{f,h} $). \\
To obtain an asymptotic estimate on $\lambda_h$ when $h\to 0$,  the strategy consists in studying the smallest singular values of  the matrix 
of the gradient operator~$\nabla$ which maps~$\Ran \,\pi_h^{(0)} $, equipped  with the scalar product of $L^2_w(\Omega)$,  to~$\Ran \,\pi_h^{(1)} $. Indeed, from Proposition~\ref{fried}, the squares of the smallest singular values of this matrix are the  smallest eigenvalues of $-\frac 2h L^{D,(0)}_{f,h}$.  Working with    the matrix of~$\nabla|_{\Ran \,\pi_h^{(0)}}$ gives more flexibility than directly working with the matrix of $-L^{D,(0)}_{f,h}|_{\Ran \,\pi_h^{(0)}}$.
To this end, the idea is then to construct an appropriate basis (with so called
quasi-modes) of~$\Ran \,\pi_h^{(0)} $ and~$\Ran \,\pi_h^{(1)} $.  
Moreover, from~\eqref{eq:L1_eig},~$\nabla u_h\in \Ran \,\pi_h^{(1)} $ and thus, to study the asymptotic behaviour of $\pa_nu_h$ on $\pa \Omega$ when $h\to 0$,  one decomposes $\nabla u_h$ along the basis of $\Ran \,\pi_h^{(1)} $. The terms in the decomposition are approximated using quasi-modes. 

\medskip
The paper is organized as follows.  In Section~\ref{sec:j-tilde-j}, one constructs two maps $\mbf{ j}$ and~$\mbf{\widetilde j}$ which will be extensively used  in Section~\ref{section-2}.  These maps are useful in order to understand the different  timescales of the process~\eqref{eq.langevin}   in $\Omega$. 
Section~\ref{section-2} is dedicated to the construction of   quasi-modes for $-L^{D,(0)}_{f,h} $  and $-L^{D,(1)}_{f,h} $. In Section~\ref{section-3}, we study the asymptotic behaviors of the smallest eigenvalues of $-L^{D,(0)}_{f,h} $ (see Theorem~\ref{pp}) and we give an   asymptotic estimate of $\lambda_h$ when $h\to 0$, see Theorem~\ref{thm-big0}. In Section~\ref{section-4},  we give  asymptotic estimates  for $\int_\Omega u_h e^{-\frac 2h f}$ and for $\pa_nu_h$ on $\pa \Omega$ when $h\to 0$ (see Proposition~\ref{pr.masse} and Theorem~\ref{thm-big-pauh}). Finally, Section~\ref{section-5}  is dedicated to the proof of Theorem~\ref{thm.main}.
  \medskip

\noindent
For the ease of the reader, a list of the main notation used in this work is provided at the end of this work.

\section{Association of the local minima of~$f$ with saddle points of~$f$}\label{sec:j-tilde-j}

This section is dedicated to  the construction of two maps: the map $\mbf{j}$ which associates each local minimum of~$f$ with an ensemble  of saddle points of~$f$ and the map $\mbf{\widetilde j}$ which associates each local minimum of~$f$ with a connected component of a sublevel set of~$f$. These maps are useful to define the quasi-modes in  Section~\ref{section-2}. \\
This section is organized as follows. 
In Section~\ref{ft-Ci}, one introduces a set of connected components which play a crucial role in our analysis. 
 The constructions of the maps~$\mbf{ j}$ and~$\mbf{\widetilde j}$  require two preliminary results:  Propositions~\ref{pr.p1} and~\ref{pr.p2} which are respectively  introduced in  Section~\ref{sec.ssp} and Section~\ref{sec.topo}. Then,  the maps $\mbf{j}$ and $\mbf{\widetilde j}$ are defined in Section~\ref{sec:labeling}. Finally, in Section~\ref{sec:hip},  one rewrites the assumptions  (\ref{eq.hip1})-(\ref{eq.hip4})  with the help of the map~$\mbf j$.

\subsection{Connected components associated with the elements of  $\ft U_0^\Omega$  }
\label{ft-Ci}
The aim of this section is to define for each  $x\in \ft U_0^\Omega$,  the connected component  of $\{f< \ft H_f(x)\}$ which contains $x$ (where $\ft H_f(x)$ is defined by~\eqref{eq.Hfx}). For that purpose, let us introduce the following definitions. 

 \begin{definition}
 Let us assume that the assumption \eqref{H-M} holds. For all $x\in \ft U_0^\Omega $ and $\lambda>f(x)$, one defines\label{page.clambdax}
\begin{equation} \label{eq.c1}
\ft C(\lambda,x)\ \ \text{as the connected component of $\{f<\lambda\}$ in~$\overline\Omega$
containing $x$} 
\end{equation}
and \label{page.c+lambdax}
\begin{equation} \label{eq.c2}
\ft C^+(\lambda,x)\ \ \text{as the connected component of $\{f\le\lambda\}$ in~$\overline\Omega$
containing $x$}. 
\end{equation}
Moreover, for all $x\in \ft U_0^\Omega $, one defines\label{page.lambdax}
\begin{equation} \label{eq.c3}
 \lambda(x):=\sup\{\lambda>f(x)\ \text{s.t.}\  \ft C(\lambda,x)\cap\pa\Omega=\emptyset \} \ \ 
\text{ and }\ \ \ft C(x):=\ft  C(\lambda(x),x).
\end{equation}
\end{definition}
\noindent

A direct consequence of Lemma~\ref{le.=delta} below is that for all $x\in \ft U_0^\Omega$,  $\ft C(x)$ defined in~\eqref{eq.c3} coincides with   $\ft C(x)$ introduced in~\eqref{eq.Cdef2} and thus\label{page.c2} 
 \begin{equation}\label{mathcalC}
 \mathcal C=\big \{\ft  C(x),\, x\in \ft U_0^\Omega \big \},
 \end{equation}
where $\mathcal C$ is defined by~\eqref{mathcalC-def} and for $x\in \ft U_0^\Omega$, $\ft C(x)$ is defined by~\eqref{eq.c3}.

Notice that under \eqref{H-M}, for all $x\in \ft U_0^\Omega\subset \Omega$,   $\lambda(x)$ is well defined. Indeed, for all $x\in \ft U_0^\Omega$,  $ \{\lambda>f(x)\ \text{s.t.}\  \ft C(\lambda,x)\cap\pa\Omega=\emptyset \}$ is bounded by~$\sup_{\overline \Omega}f+1$ and nonempty because   for $\beta>0$ small enough $\ft C(f(x)+\beta,x)$ is included in~$\Omega$ (since $x\in \Omega$ and $f$ is Morse).

One has the following results which permits to give another definition of $\ft H_f$ (see~\eqref{eq.Hfx}) which will be easier to handle in the sequel.

\begin{lemma} 
\label{le.=delta}
Let us assume that \eqref{H-M} holds. Then, for 
all $x\in  \ft U_0^\Omega$
\begin{equation}\label{eq.=delta}
\ft H_f(x)=\lambda(x), 
\end{equation}
where $\ft H_f(x)$ is defined by~\eqref{eq.Hfx} and $\lambda(x)$  is defined by~\eqref{eq.c3}.   
\end{lemma}
 \begin{proof}
 Let $x\in  \ft U_0^\Omega$. 
By definition of $\ft H_f(x)$ (see~\eqref{eq.Hfx}), for all  $\ve >0$, there exists  $\gamma\in C^0([0,1], \overline \Omega)$ such that $\gamma(0)=x$,~$\gamma(1)\in \pa \Omega$ and 
 $$\ft H_f(x)\le \max_{ t\in [0,1]} f(\gamma(t))  < \ft H_f(x)+\ve.$$
 Therefore,~$\ft C(\ft H_f(x)+\ve,x)\cap \pa \Omega\neq \emptyset$. 
  Then, by definition of $\lambda(x)$ (see~\eqref{eq.c3}) $\lambda(x)\le \ft H_f(x)+\ve$ which implies, letting $\ve \to 0^+$,~$\lambda(x)\le \ft H_f(x)$. To prove that $\lambda(x)=\ft H_f(x)$, we argue by contradiction and we assume that $\lambda(x)<\ft H_f(x)$. Let us consider $\alpha\in (0,\ft H_f(x)-\lambda(x))$. By definition of $\lambda(x)$,~$\ft C(\lambda(x)+\alpha,x)\cap \pa \Omega\neq \emptyset$. Thus,   there exists  $\gamma\in C^0([0,1], \overline \Omega)$ such that $\gamma(0)=x$,~$\gamma(1)\in \pa \Omega$ and for all $t\in [0,1]$, $f(\gamma(t))<\lambda(x)+\alpha$. This implies that,~$\ft H_f(x) <\lambda(x)+\alpha$ which contradicts the definition of $\alpha$. Therefore $\lambda(x)=\ft H_f(x)$.  This concludes the proof of Lemma \ref{le.=delta}. \end{proof}

\begin{definition} 
\label{de.1}
Let us assume  that the assumption  \eqref{H-M} holds. The integer $ \ft N_{1}$ is defined by\label{page.N1}:
\begin{equation}\label{eq.N11}
 \ft N_{1}:={\rm Card} (\mathcal C) ={\rm Card}\big(\{\ft  C(x),\, x\in \ft U_0^\Omega \}\big)\in  \{1,\dots,\ft m_{0}^{\Omega}\},
\end{equation}
where we recall that  $\ft m_0^\Omega ={\rm Card}\, (\ft U_0^\Omega)$ (see~\eqref{mo-omega}), $\ft  C(x)$ is defined by~\eqref{eq.c3} and $\mathcal C=\big \{\ft  C(x),\, x\in \ft U_0^\Omega \big \}$ (see~\eqref{mathcalC-def} and~\eqref{mathcalC}).
Moreover,  the   elements of $\mathcal C=\{  \ft  C(x) ,x\in \ft U_0^\Omega \}$ 
are denoted by~$  \ft C_{1},\dots,   \ft  C_{ \ft N_{1}}$. Finally, for all $\ell\in \{1,\ldots,\ft N_1\}$,~$\ft C_k$ is denoted by
\begin{equation}\label{eq.E1i}
\ft E_{1,\ell}:=\ft C_\ell.
\end{equation}
\end{definition} 
\noindent
For example, on Figure~\ref{fig:okay}, one has $\ft m_{0}^{\Omega}=4$ and $\ft N_1=3$. The notation~\eqref{eq.E1i} will be useful when constructing the maps $\mbf j$ and $\widetilde{\mbf j}$ in Section~\ref{sec:labeling} below. 
 
\subsection{Separating saddle points}\label{sec.ssp}
 
This section is devoted to the proof of Proposition~\ref{pr.p1} below which will be needed when constructing   the maps $\mbf{j}$ and $\mbf{\widetilde j}$   in Section~\ref{sec:labeling}. Let us first  prove the following lemma which will be used in the  proof of Proposition~\ref{pr.p1}.

\begin{lemma}
\label{le.comp-conn}   Let us assume that the function $f : \overline \Omega \to \mathbb R$ is a $C^{\infty}$ function. 
 Let 
  $x\in \ft U_0^\Omega$. For all $\mu>f(x)$,  it holds:
\begin{equation}
\label{eq.conn-equality1}
\ft C(\mu,x)=\bigcup_{\lambda<\mu }\ft C(\lambda,x)
\end{equation} 
and 
\begin{equation}
\label{eq.conn-equality2}
\ft C^{+}(\mu,x)=\bigcap_{\lambda>\mu}\ft C^{+}(\lambda,x),
\end{equation}
where $\ft C(\mu,x)$ and $\ft C^{+}(\mu,x)$  are respectively defined in~\eqref{eq.c1} and~\eqref{eq.c2}.
\end{lemma}

\begin{proof} The proof is divided into two steps. \\
\textbf{Step 1}. Proof of~\eqref{eq.conn-equality1}.  
\medskip

\noindent
Since $\{f<\lambda\}$ is open in the locally connected space $\overline\Omega$,
the set $\ft C(\lambda,x)\subset\overline\Omega$ is open
for all  $\lambda\in (f(x),\mu)$.
Since moreover $\ft C(\lambda,x)\subset  \ft C(\mu,x)$
for all  $\lambda\in (f(x),\mu)$,
 the union  $\cup_{\lambda\in (f(x),\mu)}\ft C(\lambda,x)$
is an open subset of $\ft C(\mu,x)$. Therefore, since $\ft C(\mu,x)$ is connected,  to obtain~\eqref{eq.conn-equality1},  it is  enough to prove
that the set $\cup_{\lambda\in (f(x),\mu)}\ft C(\lambda,x)$ is  closed in 
$\ft C(\mu,x)$.
To this end, let us show that the complement of $\cup_{\lambda\in (f(x),\mu)}\ft C(\lambda,x)$
in~$\ft C(\mu,x)$ is open. It is obviously the case if it is empty. If $\cup_{\lambda\in (f(x),\mu)}\ft C(\lambda,x)$ is not empty, let us choose  
$$y\in \ft C(\mu,x)\setminus \cup_{\lambda\in (f(x),\mu)}\ft C(\lambda,x).$$ Then, since $y\in \ft C(\mu,x)$,  one has
  $f(y)<\mu$ and  thus  
$y\in \ft C(\lambda,y)\cap \ft C(\mu,x)$ for all $\lambda\in(f(y),\mu)$. Therefore,   it holds $\ft C(\lambda,y)\subset \ft C(\mu,x)$ and $\ft C(\lambda,y)\cap \ft C(\lambda,x)=\emptyset$
for all $\lambda\in(\max\{f(x),f(y)\},\mu)$.
Hence, the open set $\ft C(\lambda,y)$ is included in~$\ft C(\mu,x)$ 
and disjoint from the set $\cup_{\lambda\in (f(x),\mu)}\ft C(\lambda,x)$
 for all $\lambda\in(f(y),\mu)$. This proves that 
 $\cup_{\lambda\in (f(x),\mu)}\ft C(\lambda,x)$ is  closed in 
$\ft C(\mu,x)$. This concludes the proof of~\eqref{eq.conn-equality1}.\medskip

\noindent \textbf{Step 2}. Proof of  \eqref{eq.conn-equality2}.
\medskip

\noindent
 Since for all $\lambda$,~$\ft C^{+}(\lambda,x)$ is a connected component of $\{f\leq \lambda\}$,
it is closed  in this closed set of $\overline\Omega$  and thus it is closed  in~$\overline\Omega$.
It follows that the set 
$\cap_{\lambda>\mu}\ft C^{+}(\lambda,x)$ is connected as a decreasing intersection
 of  compact connected sets. Since $\cap_{\lambda>\mu}\ft C^{+}(\lambda,x)$ is also obviously included in~$\{f\leq \mu\}$
 and contains $x$, it is then  included in~$\ft C^{+}(\mu,x)$ by definition of $\ft C^{+}(\mu,x)$.
The reverse inclusion follows from the fact that  $\ft C^{+}(\mu,x)\subset \ft C^{+}(\lambda,x)$ for all $\lambda>\mu$. This proves~\eqref{eq.conn-equality2} and ends the proof of Lemma~\ref{le.comp-conn}.
 \end{proof}

The constructions of the maps~$\mbf{j}$ and $\mbf{\widetilde j}$ made in Section~\ref{sec:labeling} are based on the notions of {separating saddle points}  and of {critical components}
as introduced  in  \cite[Section 4.1]{HeHiSj} for a case without boundary. Let us define and slightly adapt theses two notions to our setting. 
To this end, let us  first recall that according to~\cite[Section 5.2]{HeNi1}, for any non critical point $z\in\Omega$,  for  $r>0$ small enough
\begin{equation}\label{eq.2-conn1}
\{f<f(z)\}\cap B(z,r) \text{ is connected},
\end{equation}
and  for any critical point $z\in\Omega$ of index $p$ of the Morse function $f$,  for  $r>0$ small enough, 
one has the three possible cases:
\begin{equation}\label{eq.2-conn}
\left\{
    \begin{array}{ll}
        &\text{either }  p=0 \text{ (}z\text{ is a local minimum of }f\text{) 		and  } \{f<f(z)\}\cap B(z,r)=\emptyset,  \\
        &\text{or } p=1 \text{ and  } \{f<f(z)\}\cap B(z,r) \text{ has exactly two connected components}, \\
        &\text{or } p\ge 2 \text{ and  } \{f<f(z)\}\cap B(z,r) \text{ is connected}, 
    \end{array}
\right.
\end{equation}
where $B(z,r):=\{x\in \overline \Omega \ \text{s.t.} \ |x-z|<r\}$. 
The {separating saddle points} of~$f$  and the  {critical components} of~$f$ are defined as follows. 
 \begin{definition}
 \label{de.SSP} 
Assume  \eqref{H-M}. Let $\mathcal C=\{ \ft  C_{1},\dots, \ft  C_{ \ft N_{1}}\}$ be the set of  connected sets  introduced  Definition~\ref{de.1}. 
 \begin{enumerate}[leftmargin=1.3cm,rightmargin=1.3cm]
 \item A point $z\in \ft U_1^{\overline \Omega}$
 is a {\em separating saddle point}  if
 \begin{itemize}
 \item either 
 $z\in \ft U_1^{\Omega}\cap\cup_{i=1}^{ \ft N_{1}}\overline{ \ft  C_{i}}$ and  for $r>0$ small enough, the two connected components of $\{f<f(z)\}\cap B(z,r)$ 
   are contained in different connected components  of
 $\{ f< f(z) \}$, 

 \item or $z\in \ft U_1^{\pa\Omega}  \cap \cup_{i=1}^{\ft  N_{1}}\pa  \ft  C_{i}$. 
 \end{itemize}
 Notice that in the former  case $z\in \Omega$ while in the latter case $z\in \pa \Omega$. 
The set of
 {\em separating saddle points} is denoted by~$\ft U_1^{\ft{ssp}}$\label{page.u1ssp}.
 \item   For any $\sigma \in \mathbb R$, a connected component $\ft E$ of the sublevel set  
 $\{ f< \sigma  \}$ in~$\overline \Omega$ is called {\em a critical connected component}   if $\pa \ft E\cap \ft U_1^{\ft{ssp}}\neq \emptyset$.
 The family of {\em critical connected components} is denoted by~$\mathcal C_{crit}$\label{page.ccrit}.
 \end{enumerate}
 \end{definition}
 \begin{remark}
 It is natural to define generally a separating saddle point of a Morse function $f$ as follows: $z$ is a separating saddle point if for any sufficiently small connected neighborhood $\mathcal V_z$ of $z$, $\mathcal V_z \cap \{f < f(z)\}$ has two connected components included in two  connected components of $\{f<f(z)\}$.
Our definition of separating saddle point is consistent with this general definition when the function  $f$ is extended  by $-\infty$ outside $\overline \Omega$. 
To be more precise, let us introduce some new nonempty set $ X'  $ and let us define the topological 
space $X$ as the disjoint union $ X= \overline \Omega \cup X' $ whose open sets are $$\mathcal O(X)=\{X'\} \cup \{ \text{open sets of } \Omega\} \cup \{U \cup X', \text{ $U$ is an open set of $\overline \Omega$ s.t. $U \cap \partial \Omega \neq \emptyset$} \}.$$
%
%
%
Note that it follows from this definition that $X'$ is connected and that $\partial X' = \partial \Omega$. We denote by $B_X(z,r)$ the ball in $X$: $B_X(z,r)=B(z,r)$ if $B(z,r) \cap \partial \Omega = \emptyset$ and $B_X(z,r)=B(z,r) \cup X'$ if $B(z,r) \cap \partial \Omega \neq \emptyset$.
Moreover, extending $f$ by $-\infty$ on $X'$, the following holds for any $z\in \overline \Omega$
and $r>0$ small enough : $ B_X(z,r) \cap \{f < f(z)\}$ has at least two connected components in $X$ 
iff $z\in \ft U_1^{\overline  \Omega}$, in which case $ B_X(z,r) \cap \{f < f(z)\}$ has precisely two connected 
components in $X$.
Lastly, for $z\in \ft U_1^{\overline \Omega}$, the above two connected components of $ B(z,r) \cap \{f < f(z)\}$ 
in $X$ are contained in different connected components of $\{f < f(z)\}$  in $X$ iff $z\in \ft  U_1^{\ft{ssp}}$.
 \end{remark}

In Figure~\ref{fig:nonSSP}, one gives an example of a saddle point $z$ which is not a {separating saddle point} as introduced in Definition~\ref{de.SSP}. 
 \begin{figure}[h!]
\begin{center}
\begin{tikzpicture}[scale=0.9]
\tikzstyle{vertex}=[draw,circle,fill=black,minimum size=5pt,inner sep=0pt]
\tikzstyle{vertex}=[draw,circle,fill=black,minimum size=6pt,inner sep=0pt]
\draw[thick] (0,0) circle(4.4);
\draw[thick] (-2,2)--(2,-2);
\draw[thick] (-2,-2)--(2,2);
   \draw[thick ] (-2,2) ..controls (-3.3,3.3)  and   (3.3,3.3)  .. (2,2) ;
   \draw[thick ] (-2,-2) ..controls (-3.3,-3.3)  and   (3.3,-3.2)  .. (2,-2) ;
\draw (0,0) node[vertex,label=north: {$z$}](v){};
         \draw (5,5) node[]{$ \{f=f(z) \}$};
          \draw[dashed, ->] (4.8,4.8)--(3.2,3.2);
          \draw[dashed, ->] (4.8,4.8)--(1.8,2.8);
   \draw (0.2,3.7) node[]{ $ f<f(z) $};
   \draw (0.2,2.4) node[]{ $ f>f(z) $};
   \draw (0.2,-2.3) node[]{ $f>f(z) $};
    \draw (-2.8,4.9) node[]{ $f>f(z) $};
 	 \draw (0,-1.5) node[vertex,label=east: {$y_2$}](v){};
	 \draw (0.3,1.6) node[vertex,label=east: {$y_1$}](v){};
	  \draw (-2.4,0) node[vertex,label=east: {$x_1$}](v){};
	   \draw (2.4,0) node[vertex,label=east: {$x_2$}](v){};	   
	   \draw[thick, densely dotted] (0,0) circle(1.2);
	    \draw[dashed, ->] (-5,3)--(-0.7,0);
	    \draw[dashed, ->] (-5,3)--(0.7,0.2);
	    \draw  (0.1,-0.7)node[]{$r$};
	    \draw (-6,3.7) node[]{The two connected components }; 
	     \draw  (-6.65,3.2 )node[]{of $\, \{f<f(z)\}\cap B(z,r)$};
 \draw[->,thick, densely dotted] (0,0)--(0.5,-1);
  
  \draw [thick, red,arrow data={0.25}{stealth},
           arrow data={0.75}{stealth}] (-0.4,0)--(-2.2,-1.5) ;
  \draw [thick, red,arrow data={0.25}{stealth},
           arrow data={0.75}{stealth}] (2.2,-1.5)--(0.6,0);
  \draw [thick, red,arrow data={0.25}{stealth},
           arrow data={0.5}{stealth},
           arrow data={0.75}{stealth}] (-2.2,-1.5) ..controls (-4.3,-4)  and   (4.3,-4)  .. (2.2,-1.5) ;
   \end{tikzpicture}
\caption{Representation of a non {separating saddle point} $z$ in dimension $2$. The points $x_1$ and~$x_2$ are two local minima of~$f$,  and the points $y_1$ and $y_2$ are two local maxima of~$f$.  The two connected components of $\{f<f(z)\}\cap B(z,r)$ 
   are contained in  the same  connected components  of
 $\{ f< f(z) \}$: any two points of these  two connected components can be joined by a   path with values in~$\{ f< f(z) \}$ (see the red path on the figure).}
 \label{fig:nonSSP}
 \end{center}
\end{figure}
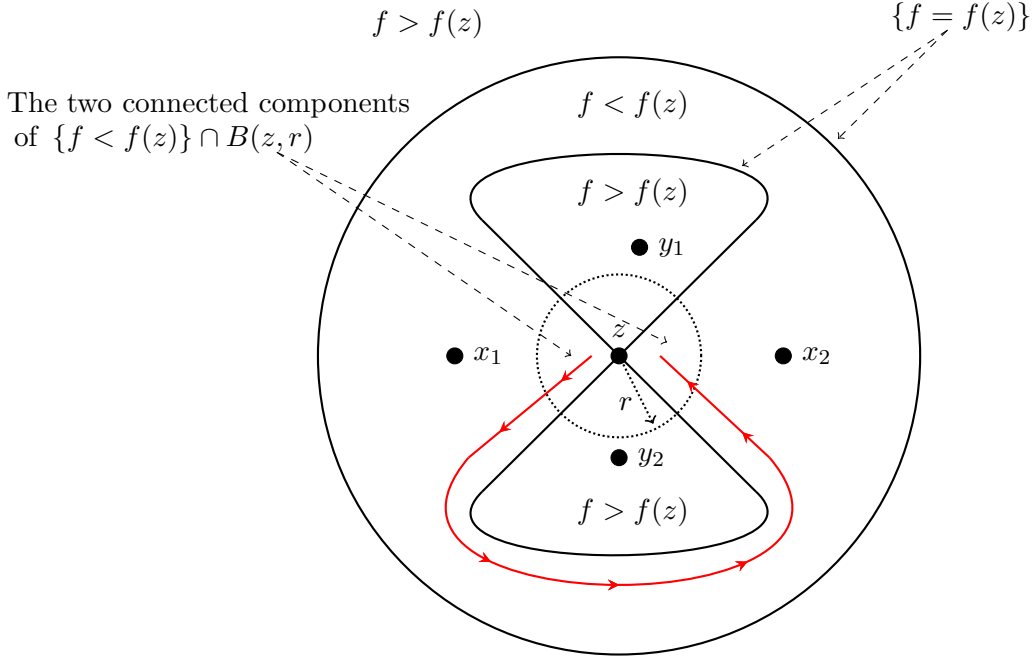

Let us now study the properties of $ \ft  C_{1},\dots, \ft  C_{\ft  N_{1}}$. The following proposition will be used in the first step of the construction of the map  between points in  $\ft U_0^\Omega$ and subsets of~$\ft U_1^{\overline \Omega}$.
\begin{proposition} 
\label{pr.p1}
Let us assume that  \eqref{H-M} holds.  Let $\mathcal C=\{ \ft  C_{1},\dots, \ft  C_{\ft  N_{1}}\}$ be the set of  connected sets  introduced  Definition~\ref{de.1} and let $(k,\ell)\in \{1,\dots, \ft N_{1}\}^2$ with $k\neq \ell$. Then,   
\begin{equation}\label{eq.ck-omega} 
\ft C_k \, \text{ is an open subset of }\,  \Omega \, \text{ and } \, \ft C_k \cap \ft C_\ell=\emptyset.
\end{equation}
In addition, one has 
\begin{equation}\label{eq.ck-omega1} 
\pa \ft C_k\cap \pa\Omega\subset \ft U_1^{\ft{ssp}} \cap \pa \Omega
\ \text{ and } \   \pa  \ft  C_{k}\cap \pa  \ft  C_{\ell} \subset \ft U_1^{\ft{ssp}} \cap \Omega,
 \end{equation}
where the set $\ft U_1^{\ft{ssp}}$ is introduced in item 1 in Definition~\ref{de.SSP}.
Finally,  $\pa  \ft C_{k} \cap \ft U_1^{\ft{ssp}}\neq \emptyset$. 
\end{proposition}

\begin{proof} 
The proof of  Proposition~\ref{pr.p1} is divided into $5$ steps. 
\medskip

\noindent  
\textbf{Step 1.} For $k\in\{1,\dots,\ft N_{1}\}$, let us  show
that $\ft C_{k}$ is an open subset of~$\Omega$. To this end, let us first prove that $\ft C_k\subset \Omega$.   From Definition~\ref{de.1} and~\eqref{eq.c1}, there exists $x_k\in \ft U_0^\Omega\cap \ft  C_{k} $ such that  
 $\ft  C_{k}=\ft C(\lambda(x_k),x_k)$ with $\lambda(x_k)>f(x_k)$.  From  Lemma~\ref{le.comp-conn}, it holds
\begin{equation*}
\ft C(\lambda(x_k),x_k)=\bigcup_{\lambda\in (f(x_k),\lambda(x_k))}\ft C(\lambda,x_k).
\end{equation*}
Moreover, since $\lambda\mapsto \ft C(\lambda,x_k) $
is increasing on $(f(x_k),+\infty)$,  one has,  by definition of $\lambda(x_k)$ (see~\eqref{eq.c1}),  
 that $\ft C(\lambda,x_k)\cap\pa\Omega=\emptyset$
for all $\lambda\in(f(x_k),\lambda(x_k))$. Therefore    
$\ft  C(\lambda(x_k),x_k)\subset\Omega$ and thus  $\ft C_{k}\cap\pa\Omega=\emptyset$. Thus,  $\ft C_k\subset \Omega$. Then, the fact that  $\ft C_k$ is  an open subset of $\Omega$ follows from the fact that $\ft C_k$ is   open in $\overline \Omega$. Indeed,   $\overline\Omega$ is locally connected and $\ft  C_{k}$ is   a connected component of the open set $\{f<\lambda(x_k)\}$. 
\medskip

\noindent 
\textbf{Step 2.} Let us now show that the $\ft  C_{k}$'s are two by two disjoint. To this end,
let  $(k,\ell)\in\{1,\dots,\ft N_{1}\}^2$ with $\ell \neq k$ and $\ft  C_{k}\cap \ft  C_{\ell}\neq \emptyset$.  Therefore, since   for $q\in\{k,\ell\}$, there exists $x_q\in \ft U_0^\Omega \cap \ft C_q$ such that $\ft C_q=\ft C(\lambda(x_q),x_q)$ is a connected component of $\{f<\lambda(x_q)\}$ (see Definition~\ref{de.1} and~\eqref{eq.c1}), it holds $\ft C_k=\ft  C_\ell$ if $\lambda(x_k)=\lambda(x_l)$. Let us prove that 
$\lambda(x_k)=\lambda(x_l)$ by contradiction and assume that, without loss of generality,~$\lambda(x_k)<\lambda(x_l)$. Since $\ft  C_{k}\cap \ft  C_{\ell}\neq \emptyset$, this implies  $\ft C_k\subset \ft  C_\ell$. Therefore, for any $\ve \in (0,\lambda(x_l)-\lambda(x_k))$,~$\ft C(\lambda(x_k)+\ve, x_k)\subset\ft  C_\ell$ and by definition of~$\lambda(x_k)$ (see~\eqref{eq.c1}),~$\ft C(\lambda(x_k)+\ve, x_k)$ intersects~$\pa \Omega$. This is in contradiction with the fact that~$\ft  C_\ell\subset \Omega$. Therefore,~$\lambda(x_k)=\lambda(x_l)$ and thus  $\ft C_k=\ft  C_\ell$. 

\medskip

\noindent 
\textbf{Step 3.} Let us prove that for $k\in\{1,\dots,\ft N_{1}\}$,~$\pa \ft C_k\cap \pa\Omega\subset \ft U_1^{\ft{ssp}}\cap\pa  \Omega$ which is equivalent, according to Definition~\ref{de.SSP}, to $\pa \ft  C_{k}\cap \pa\Omega\subset \ft U_1^{\pa \Omega}$ (where $ \ft U_1^{\pa \Omega}$ is defined in Section~\ref{se.def-zj}). If $\pa \ft  C_{k}\cap \pa\Omega\neq \emptyset$, let us  consider $y\in \pa \ft  C_{k}\cap \pa\Omega$.
According to~\cite[Section 5.2]{HeNi1},  if $y$ is not a critical point of $f|_{\pa \Omega}$, then the hypersurfaces $\{f=f(y)\}$ and $\pa \Omega$ interstects transversally in a neighborhood of $y$.  This implies that  for $r>0$ small enough,~$\{f<f(y)\}\cap B(y,r)$ is connected  and~$\{f<f(y)\}\cap B(y,r) \cap \pa \Omega\neq \emptyset$. Therefore, since $\ft C_k$ is a connected component of $\{f<f(y)\}$, one has $\{f<f(y)\}\cap B(y,r)=\ft C_k\cap B(y,r)$ and thus, $\ft C_k\cap \pa \Omega\neq\emptyset$. This is impossible since $\ft C_k\subset \Omega$. Therefore $y$ is a critical point of $f|_{\pa \Omega}$ and according to~\cite[Section 5.2]{HeNi1}, there are three possible different cases:
\begin{enumerate}
\item either $y$ is local minimum of~$f$ 
\item or $y$ is a local minimum of $f|_{\pa \Omega}$ and $\pa_nf(y)>0$, 
\item or for $r>0$ small enough,~$\{f<f(y)\}\cap B(y,r)$ admits one or two connected components with nonempty intersection with $\pa \Omega$. 
\end{enumerate}
The first case is not possible in our setting since $y\in \overline {\ft C_k}$ implies that $y$ is not  a local minimum of~$f$. The  third case is also not possible since $\ft C_k\subset \Omega$. Therefore $y$ is a local minimum of~$f|_{\pa \Omega}$ and $\pa_nf(y)>0$. This proves that~$\pa \ft C_k\cap \pa\Omega\subset \ft U_1^{\ft{ssp}}\cap\pa  \Omega$. 


\medskip

\noindent
\textbf{Step 4.} Let us prove that  for all $(k,l)\in \{1,\dots, \ft N_{1}\}^2$ with $k\neq \ell$,  $\pa  \ft  C_{k}\cap \pa  \ft  C_{\ell} \subset \ft U_1^{\ft{ssp}} \cap \Omega$ or equivalently (see item 1 in Definition~\ref{de.SSP}) that $\pa \ft  C_{k}\cap \pa\ft  C_{\ell}\subset \ft U_1^{\Omega} $ (where $\ft U_1^{\Omega}$  is the set of saddle points of~$f$ in~$\Omega$, see Section~\ref{se.def-zj}). To this end, 
let us assume that $\pa \ft  C_{k}\cap \pa \ft  C_{\ell}\neq \emptyset$ for some $(k,l)\in \{1,\dots, \ft N_{1}\}^2$ with $k\neq \ell$.
First, since   for $q\in\{k,\ell\}$, there exists $x_q\in \ft U_0^\Omega \cap \ft C_q$ such that $\ft C_q$ is a connected component of $\{f<\lambda(x_q)\}$, one has necessarily $\lambda(x_\ell)=\lambda(x_k) $.
Moreover, it holds $\pa \ft  C_{k}\cap \pa \ft  C_{\ell}\subset \Omega$. Indeed,
if there exists  $z\in \pa \ft  C_{k}\cap\pa\Omega $, we  know from the analysis above that $z\in\ft U_1^{\pa \Omega}$.
It follows that for
$r>0$ small enough,~$B(z,r)\cap\{f<\lambda(x_k) \}$ is connected and therefore  $B(z,r)\cap\{f<\lambda(x_k) \}\subset \ft  C_{k}$. This implies that   $B(z,r)\cap \ft  C_{\ell}=\emptyset$ (since we proved that for $k\neq l$,~$\ft C_k\cap \ft  C_\ell=\emptyset$)  and hence $z\notin \overline{\ft  C_{\ell}}$. 
Lastly,
if there exists $z\in \pa \ft  C_{k}\cap \pa \ft  C_{\ell}\cap\Omega$, then one deduces from
\eqref{eq.2-conn} that $z\in \ft U_1^{\Omega}$.
Indeed, for all $r>0$ small enough,~$B(z,r)\cap \{f<\lambda(x_k)\}$ has   two connected components
respectively included in~$\ft  C_{k}$ and~$\ft  C_{\ell}$.\medskip

\noindent 
\textbf{Step 5.} To conclude the proof of Proposition~\ref{pr.p1}, it remains to show that  for all $k\in\{1,\dots,\ft N_{1}\}$,  $\pa  \ft C_{k}\cap   \ft U_1^{\ft{ssp}} \neq \emptyset$. Let us argue by contradiction and assume that  $\pa  \ft C_{k}\cap \ft U_1^{\ft{ssp}}= \emptyset$. 
Since $(\pa  \ft C_{k}\cap \pa \Omega)\cap  \ft U_1^{\ft{ssp}}=\emptyset$, one has  
$\overline{\ft  C_{k}}\subset \Omega$ (indeed we proved above that  $ \ft C_k\subset \Omega $ and  $\pa \ft  C_{k}\cap \pa \Omega\subset \ft U_1^{\ft{ssp}}$). 
Let us recall that $\ft  C_{k}=\ft C(\lambda(x_k),x_k)$ for some $x_k\in \ft U_0^\Omega\cap \ft C_k$ (see   Definition~\ref{de.1},~\eqref{eq.c1},  and~\eqref{eq.c3}). 
Then,  using the fact that 
\begin{equation}\label{eq.lambda==y}
\pa \ft  C_{k}\subset\{f=\lambda(x_k)\}\subset \Omega,
\end{equation}
 and the fact that   the function $f$ is Morse, for all $z\in \pa \ft  C_{k}$, 
\begin{itemize} 

\item either $z\notin \ft U_1^{\Omega}$ in which case there exists $r_z>0$ such that  
$B(z,r_z)\cap \{f <\lambda(x_k)\}$ is connected  (see~\eqref{eq.2-conn1} together with the fact that $f(z)=\lambda(x_k)$) and thus $B(z,r_z)\cap \{f <\lambda(x_k)\}$ is included in~$\ft  C_{k}$, 

\item or $z\in \ft U_1^{\Omega}$
 in which case  there exists $r_z>0$ such that $B(z,r_z)\cap \{f <\lambda(x_k)\}$ has two connected components  both included in
$\ft  C_{k}$  (because we assume that there is no {separating saddle point} on $\pa \ft  C_{k}$).
\end{itemize} 
In all cases, one can assume in addition, choosing $r_z$ smaller, that  $B(z,r_z)\subset \Omega$ and  $B(z,r_z) \cap \ft U_0^{\Omega}=\emptyset$. 
Let us now consider
$$\ft V_k:=\Big(\cup_{z\in \pa \ft  C_{k} } B(z,r_{z})\Big) \bigcup \ft  C_{k}.$$
 Then, the set $\ft V_k$ is an open subset of $\Omega$
 such that~$\overline{ \ft  C_{k}}\subset \ft V_k$ and~$\ft V_k\cap \{f\leq \lambda(x_k)\}=\overline{\ft  C_{k}}$.
Therefore, the connected set $\overline{\ft  C_{k}}$  is  closed and open in~$\{f\leq \lambda(x_k)\}$,  and thus
 \begin{equation}\label{eq.cc-c-lambda}
  \overline{\ft  C_{k}} \text{  is a connected component of } \{f\leq \lambda(x_k) \}.
 \end{equation}
 Let us now denote by~$\ft C^{+}_k(\mu)$, for $\mu\geq \lambda(x_k)$,   the connected component of $\{f\leq \mu \}$ 
 containing $\overline{\ft  C_{k}}$. It then holds, according to 
 Lemma~\ref{le.comp-conn},
 $$
 \bigcap_{\mu>\lambda(x_k)}\ft C_k^{+}(\mu)=\ft C^{+}_k(\lambda(x_k))=\overline{\ft  C_{k}}.
 $$ 
 \begin{sloppypar}
 \noindent
Moreover, by definition of $\lambda(x_k)$ (see~\eqref{eq.c1}),~$\ft C^{+}_k(\mu)$ meets $\pa\Omega$ for all $\mu>\lambda(x_k)$. 
 Hence,~$\cap_{\mu>\lambda(x_k)}\ft C^{+}_k(\mu)\cap\pa\Omega=\overline{\ft  C_{k}}\cap\pa\Omega$ is
 nonempty as
  a decreasing intersection of nonempty compact sets. This contradicts the fact that $\overline{\ft  C_{k}}\subset \Omega$.  In conclusion $\pa  \ft C_{k} \cap \ft U_1^{\ft{ssp}}\neq \emptyset$. This concludes the proof of Proposition~\ref{pr.p1}.
   \end{sloppypar}
   \end{proof}
%
   We end this section with the following lemma which will be needed in the sequel. 
   \begin{lemma}\label{le.oubli}
Let us assume that~\eqref{H-M} is satisfied.  Let $\mathcal C=\{ \ft  C_{1},\dots, \ft  C_{ \ft N_{1}}\}$ be the set of  connected sets  introduced  Definition~\ref{de.1}. Let us consider $\{j_1,\ldots,j_k\}\subset  \{1,\ldots,\ft N_1\}$ with $k\in \{1,\ldots,\ft N_1\}$ and  $j_1<\ldots<j_k$ such that  $\cup_{\ell=1}^{k}\overline{\ft C_{j_\ell}}$ is connected and such that for all $q\in \{1,\ldots,\ft N_1\}\setminus\{j_1,\ldots,j_k\}$, $\overline{\ft C_q}\cap \cup_{\ell=1}^{k}\overline{\ft C_{j_\ell}}=\emptyset$. Then, there exists $z\in \ft U_1^{\ft{ssp}}$  and $\ell_0 \in \{1,\ldots,k\}$ such that 
$$z\in \pa \ft C_{j_{\ell_0}}  \setminus \Big (    \cup_{\ell=1, \ell\neq \ell_0}^{k}\pa {\ft C_{j_\ell}}\Big).$$ 

\end{lemma}
\begin{proof} 
There are two cases to consider: either $\cup_{\ell=1}^{k}\overline{\ft C_{j_\ell}}\cap \pa \Omega \neq \emptyset$ or $\cup_{\ell=1}^{k}\overline{\ft C_{j_\ell}}\cap \pa \Omega = \emptyset$. Let us consider the case $\cup_{\ell=1}^{k}\overline{\ft C_{j_\ell}}\cap \pa \Omega \neq \emptyset$. Using~\eqref{eq.ck-omega} and~\eqref{eq.ck-omega1}, the result stated in Lemma~\ref{le.oubli} follows from the fact that for all $\ell\in \{1,\ldots,k\}$,  the sets  $\overline{\ft C_{j_\ell}}\cap \pa \Omega= \pa {\ft C_{j_\ell}}\cap \pa \Omega$ are two by two disjoint,  together with the definition of $ \ft U_1^{\ft{ssp}}$ (see the second point of item~1 in Definition~\ref{de.SSP}). Let us now consider the case   $\cup_{\ell=1}^{k}\overline{\ft C_{j_\ell}}\cap \pa \Omega = \emptyset$. From~\eqref{eq.ck-omega1}, one has $\cup_{q,\ell=1, q\neq \ell }^{k}\pa {\ft C_{j_\ell}}\cap  \pa {\ft C_{j_q}}\subset \ft U_1^{\ft{ssp}} \cap  (\cup_{\ell=1  }^{k}\pa {\ft C_{j_\ell}} )$ and this inclusion is an equality if the statement of Lemma~\ref{le.oubli} is not satisfied. 
To prove Lemma~\ref{le.oubli}, let us argue by contradiction, i.e.  let us assume that 
\begin{equation}\label{eq.contra-1}
\bigcup_{q,\ell=1, q\neq \ell }^{k}\pa {\ft C_{j_\ell}}\cap  \pa {\ft C_{j_q}}= \ft U_1^{\ft{ssp}} \cap \Big (\bigcup_{\ell=1  }^{k}\pa {\ft C_{j_\ell}} \Big).
\end{equation} 
Notice  that 
there exists $x\in \ft U_0^\Omega$ such that for all $\ell \in \{1,\ldots,k\}$, $\ft C_{j_\ell}$ is a connected component of $\{f<\lambda(x)\}$. 
Let us prove, using the same arguments as those used to prove~\eqref{eq.cc-c-lambda},   that  $\cup_{\ell=1}^{k}\overline{\ft C_{j_\ell}}$ is a connected component of $\{f\le \lambda(x)\}$.
To this end, let us consider $z\in \cup_{\ell=1  }^{k}\pa {\ft C_{j_\ell}}$. If $z$ is not a separating  saddle point,  there exists $r_z>0$ such that $B(z,r_z)\subset \Omega$ and  $B(z,r_z)\cap \{f <\lambda(x)\}$ is included in  $\cup_{\ell=1}^{k} {\ft C_{j_\ell}}$. Else,  $z$ is  a   separating saddle point and thus, from~\eqref{eq.contra-1},   there exists $(\ell,q) \in \{1,\ldots,k\}^2$, $\ell\neq q$, such that $z\in \pa \ft C_{j_\ell}\cap \pa \ft C_{j_q}$. Thus, again,  there exists $r_z>0$ such that  $B(z,r_z)\subset \Omega$ and $B(z,r_z)\cap \{f <\lambda(x)\}$ is included in  $\cup_{\ell=1}^{k} {\ft C_{j_\ell}}$. Therefore, the same arguments as those used to prove~\eqref{eq.cc-c-lambda}  imply  that  $\cup_{\ell=1}^{k}\overline{\ft C_{j_\ell}}$ is a connected component of $\{f\le \lambda(x)\}$.
 Using in addition~\eqref{eq.conn-equality2} together with the  fact that for all $\mu> \lambda(x)$, the connected component of $\{f\le \mu\}$ which contains $\cup_{\ell=1}^{k}\overline{\ft C_{j_\ell}}$ intersects $\pa \Omega$ (by definition of $\lambda(x)$), one obtains that $\cup_{\ell=1}^{k}\overline{\ft C_{j_\ell}}\cap \pa \Omega$ is not empty. This is a contradiction. Therefore, the set  $\bigcup_{q,\ell=1, q\neq \ell }^{k}\pa {\ft C_{j_\ell}}\cap  \pa {\ft C_{j_q}}$ is strictly included in $\ft U_1^{\ft{ssp}} \cap  (\cup_{\ell=1  }^{k}\pa {\ft C_{j_\ell}} )$. This concludes the  proof of Lemma~\ref{le.oubli}. 
\end{proof}

\subsection{A topological result under the assumption  (\ref{H-M})}
 \label{sec.topo}
This section is devoted to the proof of Proposition~\ref{pr.p2} which will be needed when constructing   the maps $\mbf{j}$ and $\mbf{\widetilde j}$   in Section~\ref{sec:labeling}.   

\begin{proposition}
\label{pr.p2}
Let us assume that the assumption \eqref{H-M} is satisfied. Let us consider~$\ft C_q$ for $q\in \{1,\ldots,\ft N_1\}$ (see~Definition~\ref{de.1}). From~\eqref{eq.c1} and~\eqref{eq.c3}, there exists  $x_q\in \ft U_0^\Omega\cap \ft C_q$ such that $\ft C_q=\ft C(x_q,\lambda(x_q))$. 
Let $\lambda\in(\min_{\overline{\ft C_q}}f,\lambda(x_q)]$ and  $\ft C$ be a 
 connected component of $\ft C_q\cap \{f<\lambda\}$.
Then, 
\begin{equation}
\label{eq.C-U1}
\big( \ft C\cap \ft U_1^{\ft{ssp}} \neq \emptyset\big)\ \  \text{ iff }\ \ \ 
\text{$\ft C\cap\ft U_{0}^{\Omega}$ contains more than one point}. 
\end{equation}  
Moreover, let us define 
$$  \sigma:=\max_{y\in \ft C\cap \ft U_1^{\ft{ssp}}}f(y)$$
 with the convention
$\sigma=\min_{\overline{\ft C}} f$ when $\ft C\cap \ft U_1^{\ft{ssp}} = \emptyset$. Then, the following assertions hold.
\begin{enumerate}
\item For all $\mu\in(\sigma,\lambda]$, the set $\ft C\cap \{f<\mu\}$ is a connected component  of $\{f<\mu\}$.
\item If $\ft C\cap \ft U_1^{\ft{ssp}} \neq \emptyset$, one has 
$\ft C\cap\ft U_{0}^{\Omega}\subset\{f<\sigma\} $
and  
the connected components of  $\ft C\cap \{f<\sigma\}$   belong to~$\mathcal C_{crit}$.
\end{enumerate}
\end{proposition}
\begin{proof}  Notice that from~\eqref{eq.ck-omega}, the set~$\ft C$ is an open subset of $\Omega$.  The proof of Proposition~\ref{pr.p2} is divided into three steps.
 
 \medskip
 \noindent
\textbf{Step 1}.  Proof of \eqref{eq.C-U1}. 
\medskip

\noindent
The fact that
\begin{equation}
\label{eq.sens1-C}
\big( \ft C\cap \ft U_1^{\ft{ssp}} \neq \emptyset\big)\ \ \ \text{implies}\ \ \ 
\text{$\ft C\cap\ft U_{0}^{\Omega}$ contains more than one point} 
\end{equation}  
is straightforward. Indeed, let $z\in \ft C\, \cap \, \ft U_1^{\ft{ssp}}\subset \Omega$. Then,  $z\in \ft U_1^\Omega\cap\ft C_q$ and for $r>0$ small enough,  the two connected components of $\{f<f(z)\}\cap B(z,r)$ 
   are contained in different connected components  of
 $\{ f< f(z) \}$ (see item~1 in Definition~\ref{de.SSP}). 
 Then, since the set~$\ft C$ is a  connected component  of
 $\{ f< \lambda \}$,~$\ft C$ contains at least two   open connected components $\ft A_{1}$ and $\ft  A_{2}$
of $\{f<f(z)\}$. Moreover, for $k\in\{1,2\}$,~$\pa \ft  A_{k}\subset\{f=f(z)\}$. Thus, for $k\in\{1,2\}$,  
the global minimum of~$f$ on $\overline {A_{k}}$ is reached in 
$A_{k}$ and hence at some $y_{k}\in A_{k}\cap \ft U_{0}^{ \Omega}$.
This implies that~$\ft C\cap\ft U_{0}^{\Omega}$ contains at least two elements,~$y_{1}$ and~$y_{2}$.\\
Let us now prove the reverse implication in~\eqref{eq.sens1-C}. 
To this end, let us assume  that there exist two points $x\neq y$ in 
$C\cap\ft U_{0}^{\Omega}$. 
One can assume without loss of generality that   
$$x\in\argmin_{\overline{\ft C}} f=\argmin_{ \ft C} f\in \ft U_{0}^{\Omega} \ \text{ and } \ y \in \ft  C\cap \ft U_{0}^{\Omega}\setminus\{x\}.$$
Let us recall that (see~\eqref{eq.c1}) $
\ft C(\mu,y)$ is  the connected component of $\{f<\mu\}$ containing~$y$. Let us define 
\begin{equation}\label{eq.lambda-y}
\lambda_x(y):=\sup\{\mu>f(y)\ \text{s.t.}\  x\notin \ft C(\mu,y) \}.
\end{equation}
Let us show that
\begin{equation}\label{eq.lambda-y-lambda-bis}
f(y)<\lambda_x(y)<\lambda.
\end{equation}
Notice first that $\lambda_x(y)$ is well defined since $\{\mu>f(y)\ \text{s.t.}\  x\notin C(\mu,y) \}$ is nonempty and bounded. Indeed, since $y$ is a non degenerate local minimum of~$f$, for $\beta>0$ sufficiently small,~$f(w)>f(y)$ for all $w\in \ft C(f(y)+\beta,y)$,~$w\neq y$. Therefore,~$x\notin \ft C(f(y)+\beta,y)$ (because $x\neq y$ and $f(x)\le f(y)$). Moreover   for all $\eta\in \{\mu>f(y)\ \text{s.t.}\  x\notin \ft C(\mu,y) \}$,~$\eta< \lambda$ (because $x,y\in \ft C$  implies $\ft C(\lambda,y)=\ft C$ since~$\ft C(\lambda,y)$ and $\ft C$ are both connected components of $\{f<\lambda\}$). Therefore,~$\lambda_x(y)$ is well defined and satisfies $f(y)<\lambda_x(y) \le \lambda$ (which proves the first inequality in~\eqref{eq.lambda-y-lambda-bis}).
Since $\mu\mapsto \ft  C(\mu,y) $
is increasing on $(f(y),+\infty)$, it holds 
 $x\notin \ft C(\mu,y)$
for all $\mu\in(f(y),\lambda_x(y))$ by definition of~$\lambda_x(y)$. Thus, since
according to Lemma~\ref{le.comp-conn} (see~\eqref{eq.conn-equality1}),
$$
\ft C(\lambda_x(y),y)=\bigcup_{\mu\in(f(y),\lambda_x(y))}\ft C(\mu,y),
$$
the set 
$
\ft C(\lambda_x(y),y)$
does not contain~$x$ and hence $\lambda_x(y)<\lambda$. This proves~\eqref{eq.lambda-y-lambda-bis}. Notice that~\eqref{eq.lambda-y-lambda-bis} implies 
  \begin{equation}\label{eq.lambda-y-inclusion}
 \overline{\ft C(\lambda_x(y),y)} \subset \ft C   
 \end{equation}
Let us now prove that $\pa \ft C(\lambda_x(y),y) \cap \ft U_1^{\ft{ssp}}\neq \emptyset$ which will conclude the proof  of~\eqref{eq.sens1-C}. Let us prove it by contradiction and let us assume that~$\pa \ft  C(\lambda_x(y),y) \cap \ft U_1^{\ft{ssp}}=\emptyset$.
Then, using in addition the fact the function $f$ is Morse and the fact that 
\begin{equation}\label{eq.lambda==y}
\pa \ft C(\lambda_x(y),y) \subset \{f=\lambda_x(y)\}\subset   \Omega 
\end{equation}
the same arguments 
as those used to prove~\eqref{eq.cc-c-lambda} apply and lead to the fact that $\overline{\ft C(\lambda_x(y),y)}$ is  a closed and open connected set in~$\{f\leq \lambda_x(y)\}$. Thus, 
$$
 \overline{\ft C(\lambda_x(y),y)} \ \text{ is a connected component of } \ \{f\leq \lambda(y)\}.
$$
For $\mu\geq \lambda(y)$, let us now denote by~$\ft C^{+}_y(\mu)$ the connected component of $\{f\leq \mu\}$ 
 containing $\overline{\ft C(\lambda_x(y),y)}$. It then holds, according to 
 Lemma~\ref{le.comp-conn}  (see~\eqref{eq.conn-equality2}),
\begin{equation}
\label{eq.muu}
 \bigcap_{\mu>\lambda_x(y)}\ft C^{+}_y(\mu)=\ft C^{+}_y(\lambda_x(y))=\overline{\ft C(\lambda_x(y),y)}.
\end{equation}
In addition,  for all $\mu>\lambda(y)$,~$x\in \ft C(\mu,y)$ by definition of $\lambda_x(y)$ (see~\eqref{eq.lambda-y}), and $\ft C(\mu,y) \subset \ft C^{+}_y(\mu)$.
Thus, using~\eqref{eq.muu},~$x\in \overline{\ft C(\lambda_x(y),y)}$ and hence, since $x\notin \ft C(\lambda_x(y),y)$, it holds~$f(x)=\lambda_x(y)>f(y)$. This contradicts the fact that   that $x\in\argmin_{\overline C} f$. Therefore, we have proven that 
\begin{equation}
\label{eq.pa-lambda-y}
  \pa \ft C(\lambda(y),y) \cap \ft U_1^{\ft{ssp}}\neq \emptyset.
\end{equation}
Using~\eqref{eq.lambda-y-inclusion}, this implies that~$\ft C\cap \ft U_1^{\ft{ssp}} \neq \emptyset$ which concludes the proof of  the reverse implication in~\eqref{eq.sens1-C} and thus the proof of~\eqref{eq.C-U1}.
 
  \medskip
 \noindent
\textbf{Step 2}. Proof of item 1  in Proposition~\ref{pr.p2}.\\
\noindent 
  Let us first deal with the case $\ft U_1^{\ft{ssp}}\cap \ft C=\emptyset$. In that case, the set $\ft C\cap \ft U_{0}^{\Omega}$ 
 is reduced to one element. This implies that for all $\lambda\in (\min_{\overline{\ft C}}f, \lambda]$,
the set~$\ft C\cap\{f<\lambda\}$ is connected since each of its connected components necessarily contains at least one element of
$\ft U_{0}^{\Omega}$. \\
Let us now deal with the case $\ft U_1^{\ft{ssp}}\cap \ft C\neq\emptyset$.  Let us then consider $x\in \argmin_{\ft C}f\subset \ft U_{0}^{\Omega}$ and, for every
$y\in \ft C\cap \ft U_{0}^{\Omega}\setminus\{x\} $, let $\lambda_x(y)$  be as defined in~\eqref{eq.lambda-y}. 
Let us also define 
\begin{equation}\label{eq.sigma=def}
\tilde \sigma:=\max\big \{\lambda_x(y),\,  y\in \ft C\cap \ft U_{0}^{\Omega}\setminus\{x\}\big \},
\end{equation}
which is well defined since the set $ \{y\in \ft C\cap \ft U_{0}^{\Omega}\setminus\{x\}\}$ is nonempty (by \eqref{eq.C-U1}) and contains a finite number of elements (since $f$ is Morse). 
Then, from~\eqref{eq.lambda-y-lambda-bis},~\eqref{eq.pa-lambda-y} and the first inclusion in~\eqref{eq.lambda==y}, one has  
\begin{equation}\label{sigma-ssp}
\tilde \sigma\in (\min_{\overline{\ft C}}f,\lambda)\bigcap \big \{f(z),\, z\in \ft U_1^{\ft{ssp}}\big \}.
\end{equation}
Then,  since for all $\mu\in (\tilde \sigma, \lambda]$ and  for all~$y\in \ft  C\cap \big(\ft U_{0}^{\Omega}\setminus\{x\} \big)$,~$x\in \ft  C(\mu,y)$ (because $\mu> \lambda_x(y)$ and by definition of $\lambda_x(y)$, see~\eqref{eq.lambda-y})  and since the $\ft C(\mu,y)$'s are connected components of $\{f<\mu\}$,   one obtains that $\ft C(\mu,y)= \ft C(\mu,w)$ for all $y,w\in \ft C\cap \big(\ft U_{0}^{\Omega}\setminus\{x\} \big)$. Thus, one has
\begin{equation}\label{C-mi-uo}
\text{ for all } \,  \mu\in (\tilde\sigma, \lambda] \, \text{ and for all} \ y\in \ft C\cap \big(\ft U_{0}^{\Omega}\setminus\{x\} \big), \ \ \ft C\cap \ft U_{0}^{\Omega} \subset \ft C(\mu,y).
\end{equation}
This implies
that  for any $y\in\ft  C\cap \big(\ft U_{0}^{\Omega}\setminus\{x\} \big)$,~$\{f<\mu\}\cap \ft C$ is equal to $\ft C(\mu,y)$ (since every connected component of~$\{f<\mu\}\cap \ft C$ contains at least one element of $ \ft U_{0}^{\Omega}$). Therefore, one has 
\begin{equation}\label{eq.C-connecte}
 \forall \mu\in (\tilde  \sigma, \lambda],\  \{f<\mu\}\cap \ft C  \, \text{ is connected.} 
\end{equation}
Moreover,  it holds  
\begin{equation}\label{eq.sigma=max}
\tilde \sigma=\max_{x\in \ft U_1^{\ft{ssp}}\cap \ft C}f(x),
\end{equation}
and thus $\tilde \sigma= \sigma$, where we  recall that  $\sigma$ is defined in Proposition~\ref{pr.p2}. 
Indeed, if it is not the case, then from~\eqref{sigma-ssp}, one has $\tilde  \sigma<\max_{x\in \ft U_1^{\ft{ssp}}\cap \ft C}f(x)$ and thus~$\ft C$ contains  at least two  connected components
 of $\{f<\max_{\ft U_1^{\ft{ssp}}\cap \ft C}f\}$ with $\max_{\ft U_1^{\ft{ssp}}\cap \ft C}f>\tilde \sigma$,  which  
 contradicts~\eqref{eq.C-connecte}. 
The fact that $\ft C \cap \{f<\mu\}$ is connected follows from~\eqref{eq.C-connecte} and~\eqref{eq.sigma=max}.   
\medskip
 
 \noindent
Let us now prove that $\ft C \cap \{f<\mu\}$ is a connected component of $\{f<\mu\}$ for all $\mu \in (\sigma, \lambda]$. Since $ \ft C \cap \{f<\mu\}$ is connected, one can consider  the connected component~$\ft V$~of~$\{f<\mu\}$ which contains $\ft C \cap \{f<\mu\}$. Then, since $\ft C$ is a connected  component of $\{f<\lambda\}$ and $\mu\le \lambda$, it holds $\ft V\subset \ft C$ and thus, $\ft V\subset \ft C \cap \{f<\mu\}$. Therefore,  one has $\ft V=  \ft C \cap \{f<\mu\}$ is a connected component of $\{f<\mu\}$.
This concludes the proof of item 1  in Proposition~\ref{pr.p2}.  
 \medskip
 
 \noindent
\textbf{Step 3.} Proof of item 2  in Proposition~\ref{pr.p2}.
\medskip

\noindent
 Let us assume that $\ft U_1^{\ft{ssp}}\cap \ft C\neq\emptyset$. Then, using~\eqref{eq.C-U1},~$\ft C\cap  \ft U_{0}^{\Omega}$ contains at least two elements. 
  Let $x\in \argmin_{\ft C}f$ and  
 $y\in \ft C\cap \ft U_{0}^{\Omega}\setminus\{x\}$. Then $f(x)\leq f(y)$ and according to~\eqref{eq.lambda-y-lambda-bis}, it holds $ f(y)<\lambda_x(y)$. From~\eqref{eq.sigma=max} and~\eqref{eq.sigma=def}, it holds moreover
 
 $$\sigma:=\max_{\ft U_1^{\ft{ssp}}\cap \ft C}f\geq \lambda(y).$$
Therefore $f(x)\le f(y)< \sigma$ and thus 
\begin{equation}\label{eq.inclu-CC}
\ft C\cap\ft U_{0}^{\Omega}\subset\{f<\sigma\}.
\end{equation}
 This proves the first statement of item 2 in  Proposition~\ref{pr.p2}. \\
 Let us now prove that each connected component
 of $\ft C\cap\{f<\sigma\}$ is  a {critical connected component} (as introduced in item 2 in Definition~\ref{de.SSP}). Let us first notice that $\ft C\cap\ft U_{0}^{\Omega}\subset\{f<\sigma\} $ implies 
\begin{equation}\label{eq.cup-sigma}
\ft C\cap \{f<\sigma\} =\bigcup_{w\in \ft  U_{0}^{\Omega}\cap \ft C }\ft C(\sigma,w),
\end{equation}
where $\ft C(\sigma,w)$ is defined in~\eqref{eq.c1} (since every connected component of~$\ft C\cap \{f<\sigma\}$ contains at least one element of $ \ft U_{0}^{\Omega}$). 
Let us  consider a connected  component of $\ft C\cap\{f<\sigma\}$. From~\eqref{eq.cup-sigma}, this  component has  the form $\ft C(\sigma,y)$
 for some
 $y\in  \ft U_{0}^{\Omega}$.  
 Since $\sigma\in f(\ft U_1^{\ft{ssp}})$ (see~\eqref{sigma-ssp}),~$\ft C\cap \{f<\sigma\} $ contains at least two connected components, and thus 
 $
 (\ft C\cap \ft U_{0}^{\Omega})\setminus \ft C(\sigma,y) \neq\emptyset$. Let $w\in (\ft C\cap \ft U_{0}^{\Omega})\setminus \ft C(\sigma,y)$.  
Let us  assume that $\ft C(\sigma,y)$ is not a {critical connected component}, i.e that $\pa \ft C(\sigma,y) \cap \ft U_1^{\ft{ssp}}=\emptyset$. Then, the  arguments used to prove~\eqref{eq.cc-c-lambda} imply  that $\overline{\ft C(\sigma,y)}$ is a connected component of $\{f\leq \sigma\}$. Thus, using  Lemma~\ref{le.comp-conn}  (see~\eqref{eq.conn-equality2}), it holds
 $$
 \bigcap_{\lambda>\sigma}\ft C^{+}(\lambda,y)=\overline{\ft C(\sigma,y)}.$$
Moreover, since for all $\lambda>\sigma$,~$w\in \ft C(\lambda,y)$ (see~\eqref{C-mi-uo})  and $\ft C(\lambda,y)\subset  \ft C^{+}(\lambda,y)$, it holds
 $w\in \overline{\ft C(\sigma,y)}$. Therefore, since $f(w)<\sigma$ (see~\eqref{eq.inclu-CC}), one has $w\in \ft C(\sigma,y) $ which contradicts the fact that $w\in (\ft C\cap \ft U_{0}^{\Omega})\setminus \ft C(\sigma,y)$.  This ends the proof of Proposition~\ref{pr.p2}.   
\end{proof}

\subsection{Constructions of the maps $\mbf{j}$ and $\mbf{\widetilde j}$}
\label{sec:labeling}

In this section we construct, under  \eqref{H-M}, two maps $\mbf{j}$ and $\mbf{\widetilde j}$. These maps are  constructed using an association between the local minima of~$f$ and the  (generalized) saddle points $\ft U_1^{\overline \Omega}$. Such maps have been introduced in \cite{BEGK,BGK} and  \cite{HKN,HeHiSj} in the boundaryless case in order to give sharp asymptotic estimates of the eigenvalues of the involved operators. This   has been generalized in \cite{HeNi1} to the boundary case (where the authors introduced the notion of generalized saddle points for $\Delta^{D,(0)}_{f,h}$).\\
Let us recall (see~Lemma~\ref{ran1} below), that~$L^{D,(0)}_{f,h}$ has exactly $\ft m_0^{ \Omega}$ eigenvalues smaller than $\sqrt h$ for sufficiently small $h$.  Actually, from~\cite{HeSj4,HeNi1}, it can be shown that these $\ft m_0^{  \Omega}$ eigenvalues are exponentially small. The goal of the map $\mbf j$ is to associate each local minimum $x$ of~$f$ with a set of generalized saddle points $\mbf j(x)\subset  \ft U_1^{\overline \Omega}$ such that  $f$ is constant over $\mbf j(x)$ and such that, for sufficiently small $h$, there exists at least one eigenvalue of $-L^{D,(0)}_{f,h}$ whose exponential rate of decay is $2\big ( f(\mbf j(x))-f(x)\big )$ i.e.
$$\exists \lambda \in \sigma\big (-L^{D,(0)}_{f,h}\big) \, \text{ such that }\, \lim_{h\to 0}h\ln \lambda=-2 \big (f(\mbf j(x))-f(x) \big ).$$
The map $\widetilde {\mbf j}$   associates  each local minimum $x$ of~$f$ with the connected component of $\{f<f(\mbf j(x))\}$ which contains $x$. 
To construct the maps $\mbf{j}$ and $\mbf{\widetilde j}$, the procedure  relies on the  analysis made in Section~\ref{sec.topo} and on the analysis 
of the sublevel sets of~$f$  following the general analysis of the sublevel sets of
a Morse function on  a manifold without boundary of~\cite[Section 4.1]{HeHiSj} which generalizes  the procedure described  in~\cite{HKN}. 
 To build the maps $\mbf{j}$ and $\mbf{\widetilde j}$, one considers  the connected components of $\{f< \lambda\}\cap \ft U_1^{\ft{ssp}}$ appearing when  $\lambda$ is decreasing from $\max_{  \cup_{k=1}^{\ft N_1}\overline{\ft C_k}}f$ to $-\infty$. Each time a new connected component appears in $\cup_{k=1}^{\ft N_1}\ft C_k$, one picks arbitrarily a local minimum in it and then, one associates this local minimum with the {separating saddle points} on the boundary of this new connected component. 
\\
 
 Let assume that the assumption  \eqref{H-M} holds\label{page.j}. 
 The  constructions of the maps~$\mbf{j}$ and~$\mbf{\widetilde j}$  are made  recursively   as follows:
\begin{enumerate}[leftmargin=1.3cm,rightmargin=1.3cm]

\item  \textbf{Initialization ($q=1$).} We consider $\ft E_{1,\ell}=\ft C_\ell$ for $\ell\in\{1,\dots,\ft  N_1 \}$ (see~\eqref{eq.E1i}).
\medskip

\noindent
For each~$\ell\in\{1,\dots,\ft  N_1 \}$, 
$x_{1,\ell}$
 denotes one point 
in~$\argmin_{\overline{\ft E_{1,\ell}}}f=\argmin_{\ft  E_{1,\ell}}f$.
Then we define, for all~$k\in\{1,\dots,\ft  N_1 \}$,\label{page.xkl}
\begin{equation}\label{jx1}
\sigma_{1,\ell}:= \max_{\overline{\ft  E_{1,\ell}  } } f,\ \  \widetilde{ \mathbf j}  (x_{1,\ell}):=\ft E_{1,\ell},\ \  \text{and}\ \ 
\mathbf j (x_{1,\ell}):= \pa \ft  E_{1,\ell}\cap\ft   U_1^{\ft{ssp}}.
\end{equation}
Notice  that according to Proposition~\ref{pr.p1} and item 2 in Definition~\ref{de.SSP}, it holds
$$\mathbf j(x_{1,\ell})\neq \emptyset,\ \ \pa \ft  E_{1,\ell}\subset\{f=\sigma_{1,\ell}\},\  \widetilde{ \mathbf j}  (x_{1,\ell}) \in \mathcal C_{crit},$$
and 
$$
 \bigcup_{\ell=1}^{\ft  N_1 } \mathbf j (x_{1,\ell})\cap \pa\Omega\subset \ft  U_1^{ \pa\Omega}.
$$
Moreover, one has from Proposition~\ref{pr.p1} (and more precisely the second inclusion in~\eqref{eq.ck-omega1}),
\begin{equation}\label{eq.inter-E}
\forall \ell\neq q\in \{1,\dots,\ft N_1\},\ \pa \ft E_{1,\ell} \cap \pa \ft E_{1,q} \subset \ft U_1^{\ft{ssp}}\cap {\Omega}.
\end{equation}

\item  
 \textbf{First step ($q=2$)}. If $\ft N_1< \ft m_0^\Omega $, we consider $\{f<\lambda\}\bigcap \Big(\cup_{\ell=1}^{\ft N_1}\ft E_{1,\ell}\Big)$ for $\lambda<\max_{ \ell\in\{1,\dots,\ft N_1\}}\sigma_{1,\ell}$.\\
\begin{sloppypar}
\noindent
From Proposition~\ref{pr.p2}, for each $\ell\in\{1,\dots,\ft N_1\}$, 
$ \ft E_{1,\ell} \cap \ft U_0^\Omega\neq \{x_{1,\ell}\}$ if and only if
$\ft U_1^{\ft{ssp}}\cap \ft E_{1,\ell}\neq~\emptyset$. As a consequence, one has: 
 $$\ft U_1^{\ft{ssp}}\bigcap \Big(\cup_{\ell=1}^{\ft N_1}\ft E_{1,\ell}\Big)\neq \emptyset \ \text{ iff }\ 
 \{x_{1,1},\dots,x_{1,\ft N_1}\}\neq \ft U_0^\Omega.$$
If $\ft U_1^{\ft{ssp}}\bigcap \Big(\cup_{\ell=1}^{\ft N_1}\ft E_{1,\ell}\Big)= \emptyset $ (then $\ft N_1=\ft m_0^\Omega$), the constructions of the maps $\widetilde{ \mathbf j}$ and ${ \mathbf j}$ are finished and one goes   to item 4 below. If $\ft U_1^{\ft{ssp}}\bigcap \Big(\cup_{\ell=1}^{\ft N_1}\ft E_{1,\ell}\Big)\neq \emptyset $, one defines
  $$
 \sigma_{2}:= \max_{x\in \ft U_1^{\ft{ssp}}  \bigcap \big(\cup_{\ell=1}^{\ft N_1}\ft E_{1,\ell}\big) } f(x) \ \in\ \Big(\min_{\cup_{\ell=1}^{\ft N_1}\ft E_{1,\ell} }f,  \max_{\ell\in\{1,\dots,\ft N_1\}} \sigma_{1,\ell}  \Big).
 $$ 
The  set $$  \bigcup_{\ell=1}^{\ft N_1}\Big (\ft E_{1,\ell}\cap \{ f<  \sigma_{2} \}\Big)$$ is then the union of 
 finitely many connected components. We denote by  $\ft E_{2,1},\dots,\ft  E_{2,\ft N_2}$\label{page.ekl2} (with $\ft N_2\geq 1$\label{page.n2}) the  connected components of $\bigcup_{\ell=1}^{\ft N_1}\big (\ft E_{1,\ell}\cap \{ f<  \sigma_{2} \}\big)$ which do not contain any of the minima~$ \{x_{1,1},\dots,x_{1,\ft N_1}\}$. 
 From items 1 and 2 in Proposition~\ref{pr.p2} (applied for each $ \ell\in \{1,\ldots,\ft N_1\}$  with  $\ft C= \ft E_{1,\ell}\cap \{ f<  \sigma_{2} \}$ there),
$$\forall \ell \in \{1,\ldots,\ft N_2\}, \ \ft E_{2,\ell} \in \mathcal C_{crit}.~$$
Notice that  the other connected components (i.e. those 
 containing the~$x_{1,\ell}$'s) may be not   {critical}. 
 Let us associate with each~$\ft E_{2,\ell}$,~$1\leq \ell\leq\ft  N_{2}$, 
 one point~$x_{2,\ell}$ arbitrarily chosen in~$\argmin_{\overline{\ft E_{2,\ell}}} f=\argmin_{\ft E_{2,\ell}} f$ (the last equality follows from the fact that  $\pa\ft  E_{2,\ell}\subset \{f=\sigma_{2}\}$).
For~$\ell\in \{1,\ldots, \ft N_{2}\}$, let us define:
 $$\widetilde{ \mathbf j}(x_{2,\ell}):=\ft E_{2,\ell}
 \ \ \text{and}\ \ \mathbf j(x_{2,\ell}):=\pa \ft E_{2,\ell}\cap \ft U_1^{\ft{ssp}}  \subset \{f=\sigma_{2}\}. $$

\item  
 \textbf{Recurrence ($q\ge 3$)}. 
 \medskip
 
 \noindent
If all the local minima of $f$ in $\Omega$ have been labeled at the end of the previous step above ($q=2$), i.e. if $\cup_{j=1}^{2}\{x_{j,1},\ldots,x_{j,\ft N_j}\}= \ft U_0^\Omega$ (or equivalently  if $\ft N_1+\ft N_2= \ft m_0^\Omega $), the constructions of the maps $\widetilde{ \mathbf j}$ and ${ \mathbf j}$ are finished, all the local minima of $f$ have been labeled and one goes  to item 4 below. 
If it is not the case, from Proposition~\ref{pr.p2}, there exists $m \in \mathbb N^*$  such that
\begin{equation}\label{eq.mm}
\text{for all $q\in \{2,\ldots,m+1\}$}, \ \ft U_1^{\ft{ssp}}\bigcap \bigcup_{\ell=1}^{\ft N_1}\Big (\ft E_{1,\ell}\cap \{ f<  \sigma_{q}\}\Big) \neq\ \emptyset
\end{equation}
 where one defines recursively  the decreasing sequence $( \sigma_{q})_{q=3,\ldots,m+2}$  by
$$
 \sigma_{q}:= \max_{x\in \ft U_1^{\ft{ssp}}\bigcap \bigcup_{\ell=1}^{\ft N_1}\big (\ft E_{1,\ell}\cap \{ f<  \sigma_{q-1}\}\big)     } f(x)\ \in\ \Big(\min_{\cup_{\ell=1}^{\ft N_1}\ft E_{1,\ell} }f, \, \sigma_{q-1}  \Big),
  $$
   for $q\in \{3,\ldots,m+2\}$.
   Let us now consider $m^*\in \mathbb N^*$ the larger integer among the integers $m\in \mathbb N^*$ such that~\eqref{eq.mm} holds. 
   Notice that $m^*$ is well defined since the cardinal of $\ft U_0^\Omega$ is finite. By definition of $m^*$, one has:
\begin{equation}\label{eq.m*}
 \ft U_1^{\ft{ssp}}\bigcap \bigcup_{\ell=1}^{\ft N_1}\Big(\ft E_{1,\ell}\cap \{ f<  \sigma_{m^*+2}\}\Big) =\ \emptyset.
 \end{equation}
Then, one  repeats   recursively $m^*$ times  the procedure described above defining $\big (\ft E_{2,\ell},\mbf j(x_{2,\ell}), \widetilde{\mbf j}(x_{2,\ell})\big )_{ 1\le \ell\le \ft N_2  }$ : 
for  $q\in \{2,\ldots,m^*+1\}$, one defines $(\ft E_{q+1,\ell})_{\ell\in \{1,\ldots,\ft N_{q+1}\}}$ as the set of  connected components of 
$$  \bigcup_{\ell=1}^{\ft N_1}\Big (\ft E_{1,\ell}\cap \{ f<  \sigma_{q+1} \}\Big)$$ 
which do not contains any of the local minima $\cup_{j=1}^{q}\{x_{j,1},\ldots,x_{j,\ft N_j}\}$ of~$f$ in~$\Omega$ which have been previously chosen.  
From items 1 and 2 in Proposition~\ref{pr.p2} (applied for each $ \ell\in \{1,\ldots,\ft N_{1}\}$  with  $\ft C= \ft E_{1,\ell}\cap \{ f<  \sigma_{q+1} \}$ there),
$$\forall \ell \in \{1,\ldots,\ft N_{q+1}\}, \ \ft E_{q+1,\ell} \in \mathcal C_{crit}.$$
For~$\ell\in \{1,\ldots,\ft  N_{q+1}\}$, we  associate with each~$\ft E_{q+1, \ell }$, 
 one point~$x_{q+1, \ell }$ arbitrarily chosen in~$\argmin_{\ft E_{q+1, \ell }} f$ .
For~$  \ell \in \{1,\ldots,\ft N_{q+1}\}$, let us define:
 $$\widetilde{ \mathbf j}(x_{q+1, \ell }):=\ft E_{q+1, \ell }
 \ \ \text{and}\ \ \mathbf j(x_{q+1, \ell }):=\pa \ft E_{q+1, \ell }\cap \ft U_1^{\ft{ssp}}  \subset \{f=\sigma_{q+1}\}. $$

From~\eqref{eq.m*} and Proposition~\ref{pr.p2}, $\ft U_0^\Omega=\cup_{j=1}^{m^*+2}\{x_{j,1},\ldots,x_{j,\ft N_j}\}$ and thus, all the local minima of $f$ in $\Omega$ are labeled. The constructions of the maps of the maps $\widetilde{ \mathbf j}$ and ${ \mathbf j}$ are finished and one goes to item 4 below. 

\end{sloppypar}

\item  \textbf{Properties of the  maps $\widetilde{ \mathbf j}$ and  ${ \mathbf j}$.}
\medskip

\noindent
 Let us now give   important features of the map $\mbf j$ which follows directly from its construction and which are used in the sequel. 
Two maps have been defined:
 \begin{equation}
 \label{de.j}
 \widetilde{ \mathbf j}\ :\ \ft U_0^\Omega\longrightarrow \mathcal C_{crit} \quad\text{and}
 \quad \mathbf j\ :\ \ft U_0^\Omega\longrightarrow\mathcal P(\ft U_1^{\overline \Omega}) 
 \end{equation}
 which are clearly injective.
  Notice that the $\mathbf j(x)$,~$x\in \ft U_0^\Omega$, are not   disjoint
in general. 
  For all $x\in\ft U_0^\Omega$,~$f(\mathbf j(x))$
 contains exactly one value, which will be denoted by~$f(\mathbf j(x))$. Moreover,
 since $\cup_{\ell=1}^{\ft N_1}\ft E_{1,\ell}\subset \Omega$ (see the first statement in~\eqref{eq.ck-omega}), one has for all $x\in \ft U_0^\Omega$, 
 $$\widetilde{ \mathbf j}(x)\subset \Omega.$$
 Moreover, it holds
  \begin{equation}
 \label{eq.j-x}
 \forall x  \in \ft U_0^\Omega\setminus \{x_{1,1},\dots,x_{1,\ft N_1}\}, \ \ \mathbf j(x)\subset \Omega\cap \ft U^{\ft{ssp}}_1.
  \end{equation}
 Finally, for all $x\in\ft U_0^\Omega$,  
 $$f(\mathbf j(x))-f(x)>0$$
  and for all $  x\in\ft U_0^\Omega\setminus\{x_{1,1},\dots, x_{1,\ft N_1}\}$,
  \begin{equation}
 \label{eq.strict-sep}
   f(\mathbf j(x))-f(x)  < \min_{ \ell=1,\ldots,\ft N_1  } f(\mathbf j(x_{1,\ell}))-f(x_{1,\ell}).
 \end{equation}
%
%
 \end{enumerate} 
\noindent
 In Figure~\ref{fig:construction-1}, one gives the constructions of the maps $\mbf j$ and $\widetilde{\mbf j}$ for a one-dimensional example. 
 Since,  one can pick   a minimum or  another in a {critical connected component} at each step of the construction of $\mbf j$ and~$\widetilde{\mbf j}$, the maps are not uniquely defined if over one of the connected components $\ft E_{k,\ell}$ ($k\ge 1, \, \ell\in \{1,\ldots,\ft N_k\}$), $\argmin f$ contains more than one point. As will be clear below, this non-uniqueness has no  influence on the results proven hereafter (in particular Theorem~\ref{thm.main}).  
 In Figure~\ref{fig:construction-arbitrary}, we give an example for which two  constructions of the   maps~$\mbf j$ and~$\widetilde{\mbf j}$ are possible. 
 
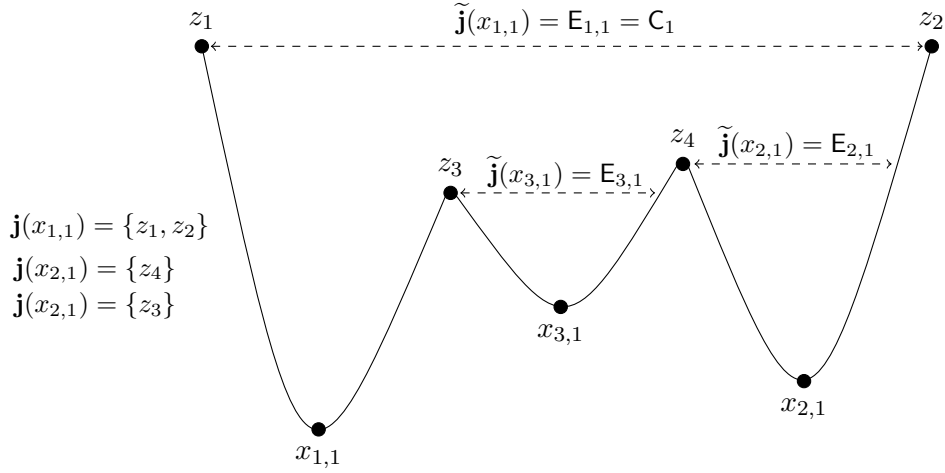
\begin{figure}[!h]
   \begin{center}
  \begin{tikzpicture}[scale=0.8]
\tikzstyle{vertex}=[draw,circle,fill=black,minimum size=5pt,inner sep=0pt]
\draw (-6,4) ..controls  (-4.3,-4).. (-2,1.5)  ;
\draw (-1.8,1.5) ..controls  (0,-1).. (1.8,2)  ;
\draw (6,4)  ..controls  (4,-3).. (2,2)  ;
\draw (-2,1.5) ..controls  (-1.91,1.6).. (-1.8,1.5)   ;
\draw (2,2) ..controls  (1.91,2.1).. (1.8,2)   ;

\draw (-6,4) node[vertex,label=north: {$z_1$}](v){};
\draw (6,4) node[vertex,label=north: {$z_2$}](v){};

\draw (-1.91,1.57)  node[vertex,label=north: {$z_3$}](v){}; 
\draw (1.91,2.05)  node[vertex,label=north: {$z_4$}](v){};
\draw (3.9,-1.55) node[vertex,label=south: {$x_{2,1}$}](v){};
\draw (-4.08,-2.35) node[vertex,label=south: {$x_{1,1}$}](v){};
\draw (-0.1,-0.32)  node[vertex,label=south: {$x_{3,1}$}](v){};
 \draw[dashed, <->]  (-5.85,4) -- (5.85,4);
  \draw (0,4.44) node[]{\small{$\widetilde{\mbf{j}}(x_{1,1})=\ft E_{1,1}=\ft C_1$}};
 \draw [dashed, <->]  (2.1,2.05) -- (5.34,2.05) ;
   \draw (3.8,2.4) node[]{\small{$\widetilde{\mbf{j}}(x_{2,1})=\ft E_{2,1}$}};
  \draw [dashed, <->]  (-1.75,1.57) -- (1.45,1.57) ;
     \draw (0,1.9) node[]{\small{$ \widetilde{\mbf{j}}(x_{3,1})=\ft E_{3,1}$}};
     \draw (-7.5,1) node[]{\small{$\mbf{j}(x_{1,1})=\{z_1,z_2\}$}};
       \draw (-7.75,0.3) node[]{\small{$\mbf{j}(x_{2,1})=\{z_4\}$}};
       \draw (-7.75,-0.30) node[]{\small{$\mbf{j}(x_{2,1})=\{z_3\}$}};

  \end{tikzpicture}
 \caption{ The maps~$\mbf j$ and $\widetilde{\mbf   j}$ in a one-dimensional example. Here the  the maps are uniquely defined and the construction requires three steps.  }
 \label{fig:construction-1}
 \end{center}
 \end{figure}

  \begin{remark}
  In the case when for all local minima $x$  of~$f$,~$\mbf j(x)$ is a single point,~$\mbf j(x)\cap \mbf j(y)=\emptyset$ for all $x\neq y$ and when all the heights $(f(\mbf j(x))-f(x))_{x\in \ft U_1^{\overline \Omega}}$ are distinct, the map $\mbf j$ is exactly the one constructed in~\cite{HeNi1}.  
   \end{remark}  

\begin{figure}[!h]
   \begin{center}
  \begin{tikzpicture}[scale=0.8]
\tikzstyle{vertex}=[draw,circle,fill=black,minimum size=5pt,inner sep=0pt]

\draw (-7,2) ..controls  (-5.8,3.5).. (-5,2) ;
\draw (-5,2) ..controls  (-3.6,-0.9).. (-2,1 )  ;
\draw (-1.5,1 ) ..controls  (0,-0.9).. (1.8,2)  ;
\draw (5,2.05)  ..controls  (4,-3).. (2,2)  ;
\draw (-2,1 ) ..controls  (-1.75,1.33).. (-1.5,1)   ;
\draw (2,2) ..controls  (1.91,2.1).. (1.8,2)   ;

\draw  (5,2.05) node[vertex,label=north: {$z_1$}](v){};
\draw (-7,2)  node[vertex,label=north: { }](v){}; 
\draw (-1.74,1.25)  node[vertex,label=north: {$z_3$}](v){}; 
\draw (1.91,2.05)  node[vertex,label=north: {$z_2$}](v){};
\draw (3.9,-1.72) node[vertex,label=south: {$x_{1,1}$}](v){};
\draw (-3.48,-0.3) node[vertex,label=south: {$x_{2,1}$}](v){};
\draw (-0.1,-0.3)  node[vertex,label=south: {$x_{1,2}$}](v){};
 \draw[dashed, <->]  (2.2,2.05) -- (4.8,2.05);
 \draw (3.6,2.3) node[]{\small{$\ft E_{1,1}=\ft C_1$}};
  \draw[dashed, <->]  (-4.85,2.05) -- (1.6,2.05);
 \draw (-1.6,2.3) node[]{\small{$\ft E_{1,2}=\ft C_2$}};
   \draw [dashed, <->]  (-4.5,1.25) -- (-1.95,1.25)  ;
 \draw (-3,1.5) node[]{\small{$\ft  E_{2,1}$}};
  \end{tikzpicture}
   \begin{tikzpicture}[scale=0.8]
\tikzstyle{vertex}=[draw,circle,fill=black,minimum size=5pt,inner sep=0pt]

\draw (-7,2) ..controls  (-5.8,3.5).. (-5,2) ;
\draw (-5,2) ..controls  (-3.6,-0.9).. (-2,1 )  ;
\draw (-1.5,1 ) ..controls  (0,-0.9).. (1.8,2)  ;
\draw (5,2.05)  ..controls  (4,-3).. (2,2)  ;
\draw (-2,1 ) ..controls  (-1.75,1.33).. (-1.5,1)   ;
\draw (2,2) ..controls  (1.91,2.1).. (1.8,2)   ;

\draw  (5,2.05) node[vertex,label=north: {$z_1$}](v){};
\draw (-7,2)  node[vertex,label=north: { }](v){}; 
\draw (-1.74,1.25)  node[vertex,label=north: {$z_3$}](v){}; 
\draw (1.91,2.05)  node[vertex,label=north: {$z_2$}](v){};
\draw (3.9,-1.72) node[vertex,label=south: {$x_{1,1}$}](v){};
\draw (-3.48,-0.3) node[vertex,label=south: {$x_{1,2}$}](v){};
\draw (-0.1,-0.3)  node[vertex,label=south: {$x_{2,1}$}](v){};
 \draw[dashed, <->]  (2.2,2.05) -- (4.8,2.05);
 \draw (3.6,2.3) node[]{\small{$\ft E_{1,1}=\ft C_1$}};
  \draw[dashed, <->]  (-4.85,2.05) -- (1.6,2.05);
 \draw (-1.6,2.3) node[]{\small{$\ft E_{1,2}=\ft C_2$}};
   \draw [dashed, <->]    (-1.6,1.25)-- (1.2,1.25)   ;
 \draw (0,1.5) node[]{\small{$\ft  E_{2,1}$}};

  \end{tikzpicture}

 \caption{A one-dimensional example for which two  constructions of the  maps~$\mbf j$ and~$\widetilde{\mbf   j}$ are possible. This is due to the fact that  two choices for~$x_{1,2}$ can be made in~$\ft E_{1,2}$ at the first step of the construction since $f$ admits two global minima  in $\ft E_{1,2}$ ($f(x_{2,1})=f(x_{1,2})$).  Both  constructions require two steps.}
 \label{fig:construction-arbitrary}
 \end{center}
 \end{figure}
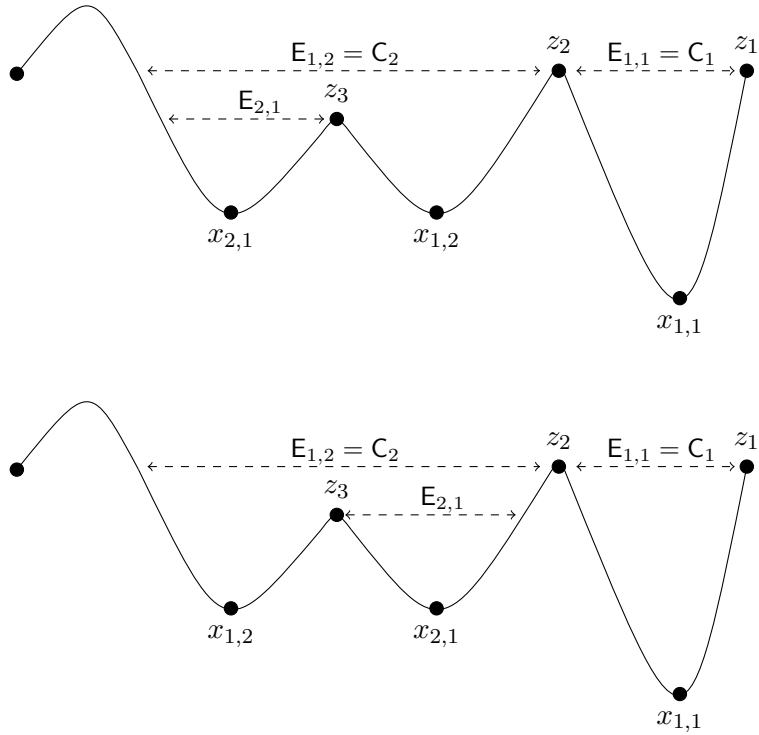

\noindent
The next definition will be used  in Section~\ref{sec:QM-L} to construct   the quasi-modes.

\begin{definition}\label{de.ve-E}
Let us assume that the assumption  \eqref{H-M} is satisfied. Let $\ve$ be such that 
\begin{equation}\label{eq.ve-1}
0\le \ve < \min_{k\ge 1, \, \ell\in \{1,\ldots,\ft N_k\} }\Big( \max_{\overline{\ft E_{k,\ell}} } f-\max_{\ft U_1^{\ft{ssp}} \cap \ft E_{k,\ell}} f\Big),
\end{equation}
where the family $(\ft E_{k,\ell})_{ k\ge 1, \, \ell\in \{1,\ldots,\ft N_k\} }$ is defined  in the construction of the map $\mbf j$ above. 
For~$k\ge 1$ and $\ell\in \{1,\ldots,\ft N_k\}$, one defines 
\begin{equation}\label{eq.ve-E}
\ft E_{k,\ell}(\ve)=\ft E_{k,\ell}\cap  \big \{f< \max_{\overline{\ft E_{k,\ell}} } f   -\ve\big \},
\end{equation}
which is   a connected  component of $\big \{f< \max_{\overline{\ft E_{k,\ell}} } f   -\ve\big \}$ according to item 1 in  Proposition~\ref{pr.p2}. 
\end{definition}

\subsection{Rewriting the assumptions  (\ref{eq.hip1})-(\ref{eq.hip4})  in terms of  the map~$\mbf j$} 
\label{sec:hip}
In this section, one rewrites the assumptions  \eqref{eq.hip1}, \eqref{eq.hip2}~\eqref{eq.hip3},  and~\eqref{eq.hip4}  with the map~$\mbf j$ constructed in Section~\ref{sec:labeling}. To this end, let us prove the following lemma.

\begin{lemma}\label{equiv-hipo}
Let us assume that the hypothesis \eqref{H-M} is satisfied. Then,  the assumption  \eqref{eq.hip1}   is equivalent to the fact that there exists $\ell\in \{1,\ldots,\ft N_1\}$ such that 
 for all $k\in \{1,\ldots,\ft N_1\}\setminus\{ \ell\}$,
$$
f(\mbf j(x_{1,k })) - f( x_{1,k })<f(\mathbf j(x_{1,\ell}))-f(x_{1,\ell}).
$$
Thus, when \eqref{eq.hip1}  holds,  the elements of $\mathcal C=\{ \ft  C_{1} ,\dots,\ft  C_{\ft N_{1}} \}$ (see Definition~\ref{de.1})  are ordered such that $\ell=1$, i.e for all $k\in \{2,\ldots,\ft N_1\}$:
\begin{equation}\label{hs12-b}
f(\mbf j(x_{1,k })) - f( x_{1,k })<f(\mathbf j(x_{1,1}))-f(x_{1,1}).
\end{equation}
Moreover, under  \eqref{eq.hip1}  (or equivalently \eqref{hs12-b}), one has  $\ft C_1=\ft C_{\ft{max}}$, 
where $\ft C_{\ft{max}}$ is defined by~\eqref{eq.hip1}.
\end{lemma}

\begin{proof} Assume that the hypothesis \eqref{H-M} is satisfied.  Let us recall that the set $\mathcal C$ (defined by~\eqref{mathcalC-def}) satisfies  from~\eqref{mathcalC} and Definition~\ref{de.1}: 
 \begin{equation}
 \mathcal C=\{\ft  C(x),\, x\in \ft U_0^\Omega \}=\{\ft C_1,\ldots,\ft C_{\ft N_1}\}.
 \end{equation}
Let $\ft C\in \mathcal C$ and let   $k\in \{1,\ldots,\ft N_1\}$, such that $\ft C=\ft C_k$.   Then, from~\eqref{eq.=delta}  and  the first step of the construction of $\mbf j$ in Section~\ref{sec:labeling}, one has for all $q\in\{1,\ldots,\ft N_1\}$: $\ft H_f(x_{1,q})=\lambda(x_{1,q})=f(\mbf j(x_{1,q}))=\sup_{\ft C_q}f$ and $f(x_{1,q})=\min_{\overline{\ft C_q}}f$. Thus,  it holds 
$$\sup_{\ft C}f-\min_{\overline{\ft C}}f=f(\mbf j(x_{1,k })) - f( x_{1,k }).$$
This implies the results stated in  Lemma~\ref{equiv-hipo}. 
  \end{proof}

\noindent
In view of Lemma~\ref{equiv-hipo} and by construction of the map $\mbf j$ (see the first step in Section~\ref{sec:labeling}), one can rewrite the assumptions \eqref{eq.hip1}, \eqref{eq.hip2}~\eqref{eq.hip3},  and~\eqref{eq.hip4}  with the map $\mbf j$ as follows:  
\begin{sloppypar}
\begin{itemize}[leftmargin=1.3cm,rightmargin=1.3cm]

\item  The assumption~\eqref{eq.hip1} is equivalent to the fact that, up to reordering the elements of $\mathcal C=\{ \ft  C_{1} ,\dots,\ft  C_{\ft N_{1}} \}$ (see Definition~\ref{de.1})  such that~\eqref{hs12-b},
it holds   for all $x\in \{x_{1,2},\ldots,x_{1,\ft N_1}\}$:
\begin{equation} \tag{\textbf{A1j}}
\label{eq.hip1-j}
 f(\mathbf j(x ))-f(x )<f(\mathbf j(x_{1,1}))-f(x_{1,1}).
\end{equation}
Furthermore, in this case,~$\ft C_1=\ft C_{\ft{max}}$, 
where $\ft C_{\ft{max}}$ is defined by~\eqref{eq.hip1}.  

\item  The assumption~\eqref{eq.hip2} rewrites when \eqref{eq.hip1-j} holds,
  \begin{equation} \tag{\textbf{A2j}} 
\label{eq.hip2-j}
\pa\ft  C_{1}\cap\pa\Omega\neq \emptyset.
\end{equation}

\item  The assumption~\eqref{eq.hip3} rewrites when \eqref{eq.hip1-j} holds,
\begin{equation} 
\tag{\textbf{A3j}} 
\label{eq.hip3-j}
\pa\ft  C_{1}\cap \pa \Omega\subset \argmin_{\pa\Omega}f.
\end{equation}

\item  When~\eqref{eq.hip1-j} holds, the assumption~\eqref{eq.hip4} is equivalent to
\begin{equation} 
\tag{\textbf{A4j}} 
\label{eq.hip4-j}
\mbf j(x_{1,1})\subset \pa\Omega.
\end{equation}
This equivalence follows from~\eqref{eq.hip1-j} together with the fact that $\mbf j(x_{1,1})=\pa \ft C_1\cap \ft U_1^{\ft{ssp}}$ (see~\eqref{jx1}) and by definition of a separating saddle point.  
  \end{itemize}
  \end{sloppypar}
 \noindent
 From now on, we work with the formulations~\eqref{eq.hip1-j},~\eqref{eq.hip2-j},~\eqref{eq.hip3-j},  and~\eqref{eq.hip4-j}  of 
 the assumptions \eqref{eq.hip1}, \eqref{eq.hip2},~\eqref{eq.hip3},  and~\eqref{eq.hip4}.
 \medskip
 
 \noindent
 Notice that under~\eqref{eq.hip1-j}, it holds from~\eqref{eq.strict-sep}, for all $x\in \ft U_0^\Omega\setminus \{x_{1,1}\}$, 
\begin{equation}
\label{eq.sj}
 f(\mathbf j(x ))-f(x )<f(\mathbf j(x_{1,1}))-f(x_{1,1}).
\end{equation}
In Figure~\ref{fig:A1}, ones gives an example when \eqref{eq.hip1-j} holds but not \eqref{eq.hip2-j}, \eqref{eq.hip3-j} and \eqref{eq.hip4-j}. 
  In Figure~\ref{fig:H2}, one gives an example when \eqref{eq.hip1-j}  and \eqref{eq.hip2-j} hold but not \eqref{eq.hip3-j} and \eqref{eq.hip4-j}. 
   In Figure~\ref{fig:eH3+}, one gives a case when \eqref{eq.hip1-j}, \eqref{eq.hip2-j} and \eqref{eq.hip3-j} hold but not \eqref{eq.hip4-j}. In  Figure~\ref{fig:H3}, one   gives a case when~\eqref{eq.hip1-j}, \eqref{eq.hip2-j}, \eqref{eq.hip3-j},  and~\eqref{eq.hip4-j} hold.
\medskip

\noindent   
When \eqref{eq.hip1-j} and \eqref{eq.hip2-j} are satisfied,    from Definition~\ref{de.SSP} and Proposition~\ref{pr.p1} (see the first inclusion in~\eqref{eq.ck-omega1} and~\eqref{jx1}), one has
\begin{equation}\label{eq.k1-paC1-0}
\pa \Omega\cap \mbf j(x_{1,1})=\pa \Omega\cap  \pa \ft C_1   =\ft  U_1^{\pa \Omega}\cap \pa \ft C_1.
\end{equation}
In that case, we assume that the elements $\{z_1,\ldots,z_{\ft m_1^{\pa \Omega}}\}$ of $\ft U_1^{\pa \Omega}$ (see~\eqref{eq.U1paOmega}) are ordered such that 
\begin{equation}
\label{eq.k1-paC1}
\pa \Omega\cap  \pa \ft C_1=\{z_1,\ldots,z_{\ft k_1^{\pa\ft C_1}}\}
\end{equation}
where $\ft k_1^{\pa\ft C_1}\in \mathbb N^*$ satisfies $\ft k_1^{\pa \ft C_1}\le \ft m_1^{\pa \Omega}$  (see~\eqref{eq.U1paOmega}). Notice that from Lemma~\ref{equiv-hipo}, this labeling implies   when \eqref{eq.hip3-j} is satisfied:
\begin{equation}
\label{eq.k1-paC-m}
\ft k_1^{\pa \ft C_1}=\ft k_1^{\pa \ft C_{\ft{max}}},
\end{equation}
where $\ft k_1^{\pa \ft C_{\ft{max}}}$ is defined by~\eqref{eq.k1-paCmax}.
%
%
%
%
%
%
%
%
Let us finally prove the following result which will be used in the sequel. 

\begin{lemma}\label{le.11}
Let us assume that the assumptions \eqref{H-M}, \eqref{eq.hip1-j}, \eqref{eq.hip2-j} and \eqref{eq.hip3-j} are satisfied. Then, one has 
$$
 \min_{\Omega}f= \min_{\overline\Omega}f<\min_{\pa \Omega}f
$$
and 
\begin{equation}\label{hs2-c1}
 \argmin_{\ft C_1}f=\argmin_{\overline \Omega}f.
\end{equation}
\end{lemma}
\begin{proof}
The fact that $\min_{\overline\Omega}f<\min_{\pa \Omega}f$ is obvious. Let us prove~\eqref{hs2-c1}. 
Let $k\in \{1,\ldots,\ft N_1\}$ and let us recall that from~Definition~\ref{de.1},  there exists $x\in \ft U_0^\Omega\cap \ft C_k$ such that $\ft  C_{k} =\ft  C(\lambda(x),x)$. Let us assume that  $x\in\argmin_{\Omega}f$. Then, by  definition of the map $\mbf j$ and by definition of $\lambda(x)$ (see~\eqref{eq.c3}) together with the fact that~\eqref{eq.hip1-j}, \eqref{eq.hip2-j} and \eqref{eq.hip3-j} hold,  one has   $f(\mathbf j(x_{1,k}))=\lambda(x)\geq \min_{\pa\Omega}f=f(\mathbf j(x_{1,1}))$. Thus, if $f(x_{1,k})=f(x)\le f(x_{1,1})$, it holds
 $$\lambda(x)-f(x)=f(\mathbf j(x_{1,k}))-f(x_{1,k})\geq f(\mathbf j(x_{1,1}))-f(x_{1,1}).$$
  This implies $\ft C_{k}=\ft C_{1}$ from the assumption \eqref{eq.hip1-j}.  This concludes the proof of~\eqref{hs2-c1}. 
  \end{proof}

\begin{figure}[h!]
\begin{center}
\begin{tikzpicture}
  \draw [domain=5:2*pi, scale=0.5, samples=300, smooth] plot (\x,{2*cos(\x r) });
    \draw [domain=2*pi:3.5*pi, scale=0.5, samples=300, smooth] plot (\x,{1.3*cos(\x r) +0.7});
\draw (5.43,0.26)--(6.8,2); 
    \draw (6.8,2) ..controls (7.1,2.4) and  (7.6,2.25) ..  (8,1.6);
       \draw (-0.9,1) node[]{$\mbf{j}(x_{1,1})=\{z_1\}$};
  \draw (2.9,1.25) node[]{$z_{1}$};
  \draw (3.2,1) node[]{$\bullet$};

    \draw (4.8,-0.58) node[]{$x_{1,1}$};
  \draw (4.8,-0.3) node[]{$\bullet$};
 
       \draw[densely dashed, <->](3.5, 1)-- (5.9, 1) ;
      \draw (4.7,1.3) node[]{\small{$\ft E_{1,1}=\ft C_1=\ft C_{\ft{max}}$}};
    \end{tikzpicture}
    \caption{A one-dimensional example when \eqref{eq.hip1-j} holds but not \eqref{eq.hip2-j}, \eqref{eq.hip3-j} and \eqref{eq.hip4-j}. }
    \label{fig:A1}
    \end{center}
\end{figure}
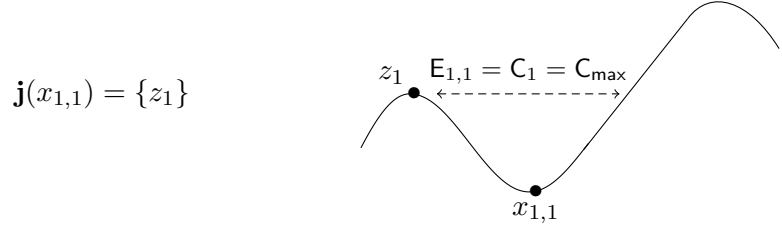
 
\begin{figure}[h!]
\begin{center}
\begin{tikzpicture}
   \draw[densely dashed, <->] (0.7,0.4) -- (0.7,-0.93);
     \draw (-1.1,-0.5) node[]{$f(\mbf{j} (x_{1,2}))-f(x_{1,2})$};
     \draw[densely dashed, <->] (8.6,2.2) -- (8.6,-1.8);
     \draw (10.4,-0) node[]{$f(\mbf{j}(x_{1,1}))-f(x_{1,1})$};
     
    \draw  (4.89,2) ..controls (3,-2) and  (2,-1)  ..  (1.05,0.5) ;
    \draw  (5.38,2) ..controls (6.5,-3) and  (7.2,-3) .. (8,2.2) ;
    \draw  (4.89,2) ..controls (5.02,2.26) and  (5.3,2.25) ..  (5.38,2);
   \draw (-0.3,3.5) node[]{$\mbf{j}(x_{1,1})=\{z_1,z_2\}$};
   \draw (-0.55,2.8) node[]{$\mbf{j}(x_{1,2})=\{z_3\}$};
  \draw (0.7,0.65) node[]{$z_3$};
  \draw (1.05,0.5) node[]{$\bullet$};
  
  \draw (2.6,-1.1) node[]{$x_{1,2}$};
    \draw (2.6,-0.83) node[]{$\bullet$};
    
     \draw (6.85,-2) node[]{$x_{1,1}$};
    \draw (6.8,-1.75)node[]{$\bullet$};
    
      \draw (8,2.2) node[]{$\bullet$};
      \draw (8.4,2.5)  node[]{$z_1$};
      \draw (5.2,2.19) node[]{$\bullet$};
      \draw (4.9,2.4) node[]{$z_2$};
  
   \draw[densely dashed, <->] (1.22,0.6)-- (4.1,0.6);
 \draw (2.6,1) node[]{$\ft E_{1,2}=\ft C_2$};
 \draw[densely dashed, <->] (5.38,2.2)-- (7.9,2.2);
     \draw (6.62,2.5) node[]{\small{$\ft E_{1,1}=\ft C_1=\ft C_{\ft{max}}$}};
    \end{tikzpicture}
    \caption{A one-dimensional example  when \eqref{eq.hip1-j} and \eqref{eq.hip2-j} hold. 
       In this example, \eqref{eq.hip3-j}  and \eqref{eq.hip4-j}  do not hold. }
    \label{fig:H2}
    \end{center}
\end{figure}
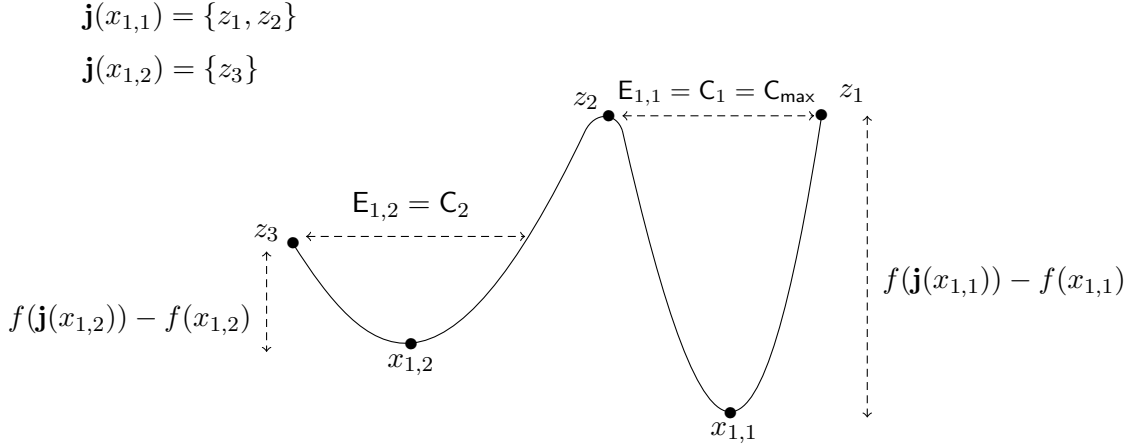

\begin{figure}[h!]
\begin{center}
\begin{tikzpicture}
  \draw [domain=1.7:2*pi, scale=0.5, samples=300, smooth] plot (\x,{2*cos(\x r) });
    \draw [domain=2*pi:3.5*pi, scale=0.5, samples=300, smooth] plot (\x,{1.3*cos(\x r) +0.7});
\draw (0.3,1)--(0.867,-0.167); 
\draw (5.43,0.26)--(6.8,2); 
    \draw (6.8,2) ..controls (7.1,2.4) and  (7.6,2.25) ..  (8,1.6);
       \draw (-0.3,3.5) node[]{$\mbf{j}(x_{1,1})=\{z_1,z_2\}$};
   \draw (-0.51,2.8) node[]{$\mbf{j}(x_{1,2})=\{z_2\}$};
\draw [dashed] (-0.5,1)--(9.3,1); 
  \draw (10.9,1) node[]{$\{f=\min_{\pa \Omega}f\}$};
  \draw (0,1.3) node[]{$z_1$};
  \draw (0.3,1) node[]{$\bullet$};  
  \draw (3.2,1.25) node[]{$z_{2}$};
  \draw (3.2,1) node[]{$\bullet$}; 
   \draw (1.58,-1.26) node[]{$x_{1,1}$};
  \draw (1.58,-1) node[]{$\bullet$};
  
    \draw (4.8,-0.58) node[]{$x_{1,2}$};
  \draw (4.8,-0.3) node[]{$\bullet$};
   \draw[densely dashed, <->] (0.5,1.2)-- (2.8, 1.2);
      \draw (1.7,1.5) node[]{\small{$\ft E_{1,1}=\ft C_1=\ft C_{\ft{max}}$}};
       \draw[densely dashed, <->](3.5, 1.2)-- (5.8, 1.2) ;
      \draw (4.7,1.5) node[]{$\ft E_{1,2}=\ft C_2$};
    \end{tikzpicture}
    \caption{A one-dimensional example when \eqref{eq.hip1-j}, \eqref{eq.hip2-j} and \eqref{eq.hip3-j} hold but not \eqref{eq.hip4-j}. }
    \label{fig:eH3+}
    \end{center}
\end{figure}
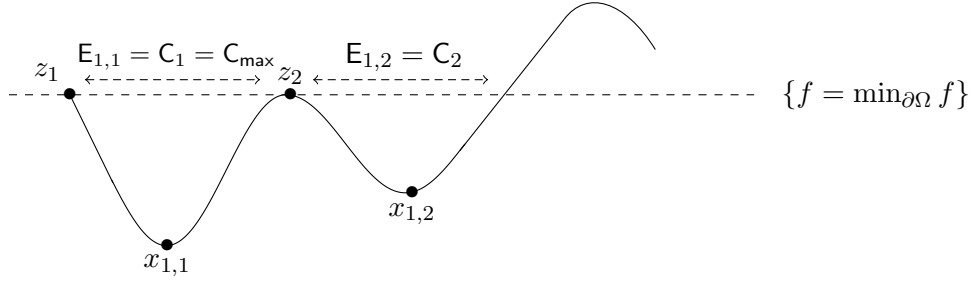

\begin{figure}[h!]
\begin{center}
\begin{tikzpicture}
   \draw[densely dashed, <->] (0.7,0.4) -- (0.7,-1.7);
     \draw (-1.1,-0.5) node[]{$f(\mbf{j} (x_{1,1}))-f(x_{1,1})$};
     \draw[densely dashed, <->] (7.8,2.2) -- (7.8,1.2);
     \draw (9.6,1.7) node[]{$f(\mbf{j} (x_{1,2}))-f(x_{1,2})$};    
    \draw  (4.89,2) ..controls (3,-3) and  (2,-2)  ..  (1.05,0.5) ;
    \draw (5.38,2) ..controls (5.7,1.6) and  (6.3,0.5) .. (7.5,2.2) ;
    \draw (4.89,2) ..controls (5.02,2.26) and  (5.21,2.25) ..  (5.38,2);
  \draw [dashed] (1.05,0.5)--(9,0.5); 
   \draw (10.5,0.5) node[]{$\{f=\min_{\pa \Omega }f\}$};
   \draw (-0.3,3.5) node[]{$\mbf{j} (x_{1,1})=\{z_1\}$};
   \draw (0,2.8) node[]{$\mbf{j} (x_{1,2})=\{z_2,z_3\}$};
  \draw (0.45,0.7) node[]{$z_1$};
  \draw (1.05,0.5) node[]{$\bullet$};
  
  \draw (2.6,-1.86) node[]{$x_{1,1}$};
    \draw (2.6,-1.58) node[]{$\bullet$};
    
     \draw (6.4,1) node[]{$x_{1,2}$};
    \draw (6.3,1.25)node[]{$\bullet$};
    
      \draw (7.5,2.2) node[]{$\bullet$};
      \draw (7.68,2.5)  node[]{$z_3$};
      \draw (5.1,2.19) node[]{$\bullet$};
      \draw (4.8,2.4) node[]{$z_2$};
  
   \draw[densely dashed, <->] (1.25,0.7)-- (4.3,0.7);
 \draw (2.6,1) node[]{\small{$\ft E_{1,1}=\ft C_1=\ft C_{\ft{max}}$}};
 \draw[densely dashed, <->] (5.3,2.2)-- (7.35,2.2);
     \draw (6.3,2.5) node[]{$\ft E_{1,2}=\ft C_2$};
    \end{tikzpicture}
    \caption{A one-dimensional example   when \eqref{eq.hip1-j}, \eqref{eq.hip2-j}, \eqref{eq.hip3-j},  and~\eqref{eq.hip4-j} are satisfied.  }
    \label{fig:H3}
    \end{center}
\end{figure}
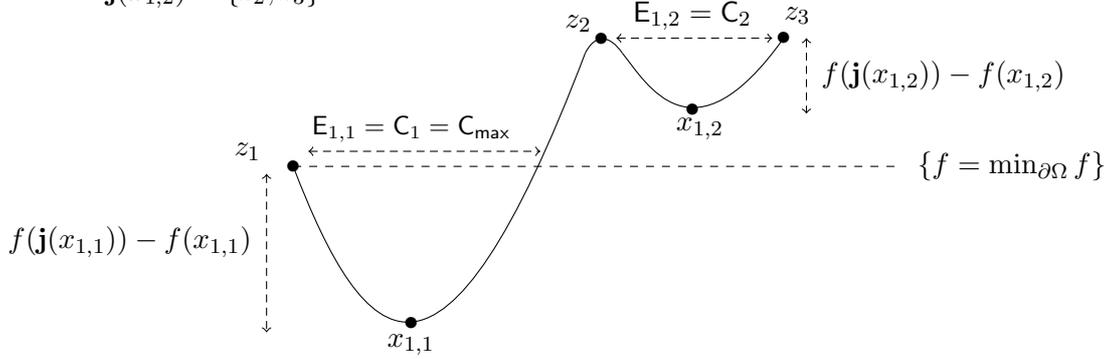
 
\section{Construction of the quasi-modes}
\label{section-2}
This section is dedicated to  the construction of two families of quasi-modes: 
a family of functions which aims at approximating the vector space spanned by the eigenfunctions associated with the  smallest eigenvalues of $-L^{D,(0)}_{f,h}$  and a family of $1$-forms which aims at approximating the vector space spanned by the eigenforms associated with the   smallest eigenvalues of $-L^{D,(1)}_{f,h}$. This construction is made  using the maps ~$\mbf j$ and~$\widetilde{\mbf j}$ previously constructed. 

This section is organized as follows. In Section~\ref{sec.nota}, we introduce   the notations    used throughout this paper for operators, and   the properties   of Witten Laplacians and  of the operators $L_{f,h}^{D,(p)}$ needed in our analysis. 
The maps ~$\mbf j$ and~$\widetilde{\mbf j}$ constructed in the previous section are then used to build the quasi-modes in Section~\ref{sec:QM-L}.

\subsection{Notations and Witten Laplacian}
\label{sec.nota}
 
In  Section~\ref{sec:nota}, one introduces the notations for the Sobolev spaces which are used in this paper.   Section~\ref{sec:LD} is dedicated to the properties   of Witten Laplacians and  of the operators~$L_{f,h}^{D,(p)}$ needed in our analysis.



\subsubsection{Notation for Sobolev spaces}
\label{sec:nota}

For $p\in\{0,\ldots,d\}$, one denotes by
$\Lambda^pC^{\infty}(\overline{\Omega})$ the space of $C^{\infty}$ $p$-forms on
$\overline\Omega$\label{page.cinfty}. Moreover,~$\Lambda^pC^{\infty}_T(\overline\Omega)$ is the set of
$C^{\infty}$ $p$-forms $v$ such that $\mbf tv=0$ on $\partial \Omega$, where $\mathbf t$ denotes the tangential trace on forms\label{page.cinftyt}. 
For $p\in \{0, \ldots,d\}$ and $q\in \mathbb N$\label{page.wsobolevq}, one denotes by
$\Lambda^pH^q_w(\Omega)$ the weighted Sobolev spaces of $p$-forms with
regularity index $q$, for the weight  function~$e^{-\frac{2}{h} f(x)} $ on
$\Omega$: $v\in \Lambda^pH^q_w(\Omega)$ if and only if for all multi-index $\alpha$ with $\vert \alpha \vert \le q$, the $\alpha$ derivative of $v$ is in~$\Lambda^pL^2_w(\Omega)$ where $\Lambda^pL^2_w(\Omega)$ is the completion of the space $\Lambda^pC^{\infty}(\overline{\Omega})$ for the norm 
$$w\in \Lambda^pC^{\infty}(\overline{\Omega})\mapsto \sqrt{\int_{\Omega} \vert w\vert^2e^{-\frac 2h f}}.$$
The subscript $w$ in the notation $L^2_w(\Omega)$ refers to the
fact that the weight  function~$x\in \Omega \mapsto e^{-\frac{2}{h} f(x)} $ appears in the inner product. 
 See for example~\cite{GSchw} for an introduction to Sobolev spaces on manifolds with boundaries. For $p\in\{0,\ldots,d\}$ and $q> \frac{1}{2}$, the set $\Lambda^pH^q_{w,T}(\Omega)$ is defined by 
$$\Lambda^pH^q_{w,T}(\Omega):= \left\{v\in \Lambda^pH^q_w(\Omega)
  \,|\,  \mathbf  tv=0 \ {\rm on} \ \partial \Omega\right\}\label{page.wsobolevqt}.$$
Notice that $\Lambda^pL^2_w(\Omega)$ is the space
$\Lambda^pH^0_w(\Omega)$, and that $\Lambda^0H^1_{w,T}(\Omega)$ is
the space $H^{1}_{w,0}(\Omega)$ than we introduced in Proposition~\ref{fried}. 
We will denote by~$\Vert . \Vert_{H^q_w}$  the norm on the weighted space $\Lambda^pH^q_w
(\Omega)$. Moreover $\langle
\cdot , \cdot\rangle_{L^2_w}$ denotes the scalar product in~$\Lambda^pL^2_w (\Omega)$\label{page.psl2w}.

Finally, we will also use the same notation without the index $w$ to
denote the standard Sobolev spaces defined with respect to the
Lebesgue measure on $\Omega$\label{page.psl2}.


\subsubsection{The Witten Laplacian and the infinitesimal generator of the diffusion~(\ref{eq.langevin}) }
\label{sec:LD}
In this section, we  recall some basic properties of Witten Laplacians, as well as the link between those and the operators $L_{f,h}^{(p)} $ introduced above (see~\eqref{eq.L-0} and~\eqref{eq.L-1}).\\
For $p\in\{0,\dots,n\}$,   one defines  the distorted  exterior derivative \`a la Witten   $d^{(p)}_{f,h}:\Lambda^{p} \,C^{\infty}(\Omega)\to \Lambda^{p+1} \,C^{\infty}(\Omega)$  and   its formal adjoint:
$d_{f,h}^{(p)*}:\Lambda^{p+1} \,C^{\infty}(\Omega)\to \Lambda^{p} \,C^{\infty}(\Omega)$  by 
$$d_{f,h}^{(p)} := e^{-\frac{1}{h}f}\, h\, d^{(p)}\, e^{\frac{1}{h}f} \, \text{ and }  d_{f,h}^{(p)*} := e^{\frac{1}{h}f}\,h\,d^{(p)*}\,e^{-\frac{1}{h}f}.$$
The Witten Laplacian, firstly introduced in \cite{Wit},
is then defined similarly as the Hodge Laplacian $\Delta^{(p)}_H:=(d+ d^*)^2:\Lambda^{p} \,C^{\infty}(\Omega)\to \Lambda^{p} \,C^{\infty}(\Omega)$ by
$$
\Delta_{f,h}^{(p)} := (d_{f,h}+ d_{f,h}^*)^2= d_{f,h}d^{*}_{f,h}+d^{*}_{f,h}d_{f,h}\ :\ \Lambda^{p} \,C^{\infty}(\Omega)\to \Lambda^{p} \,C^{\infty}(\Omega).\label{page.wlaplacien}
$$
The Dirichlet   realization of $\Delta_{f,h}^{(p)}$ on $\Lambda^{p} L^{2}(\Omega)$ is denoted by~$\Delta_{f,h}^{D,(p)}$ and its domain is   
$$D\left (\Delta^{D,(p)}_{f,h}\right )=\left\{  w \in    \Lambda^pH^2\left (\Omega\right )  | \ \mathbf  tw=0, \ \mathbf  t d^*_{f,h}w=0   \right \}.$$
The operator $\Delta^{D,(p)}_{f,h}$ is self-adjoint\label{page.wlaplaciend}, nonnegative, and its associated quadratic form is given by 
$$
\phi\in \Lambda^{p}H^1_{T}(\Omega)\, 
\mapsto\,  \Vert d^{(p)}_{f,h}\phi \Vert_{ L^2 }^2 +\Vert d^{(p)*}_{f,h}\phi\Vert_{L^2}^2,
$$
where
$$\Lambda^{p}H^1_{T}(\Omega)=\left\{  w \in    \Lambda^pH^1\left (\Omega\right )  | \ \mathbf  tw=0 \right \}.$$
We refer in particular to \cite[Section 2.4]{HeNi1} for a comprehensive definition of Witten Laplacians with Dirichlet tangential boundary conditions and statements on their properties.
The link between the Witten Laplacian and the infinitesimal generator $L^{(0)}_{f,h}$ of the diffusion~\eqref{eq.langevin} is   the following:
since
\begin{equation}\label{eq.Witten0}
L^{(0)}_{f,h} = -\nabla f \cdot \nabla - \frac{h}{2}  \ \Delta_H^{(0)}\, \text{ and }\,
\Delta^{(0)}_{f,h}= h^2 \Delta_H^{(0)} + \vert \nabla f\vert^2 + h\Delta_H^{(0)}f
,
\end{equation}
one has:
$$
\Delta^{D,(0)}_{f,h}=  -2\, h \, U\,  L^{D,(0)}_{f,h} \,\, U^{-1}
$$
where $U$ is the unitary operator
\begin{equation}\label{U}
U: \left\{
    \begin{array}{ll}
        \Lambda^pL^2_w\left (\Omega\right )\,  \rightarrow\ \Lambda^pL^2\left (\Omega\right)  \\
        \phi \mapsto e^{-\frac{1}{h}f} \phi.
    \end{array}
\right.
 \end{equation}

 In particular, the operator $L^{D,(0)}_{f,h}$
 has a natural extension to  $p$-forms defined by the relation\label{page.generatord}
 \begin{equation}\label{eq.unitary} 
L^{D,(p)}_{f,h}=  -\frac{1}{2 h} \,U^{-1}\,  \Delta^{D,(p)}_{f,h} \,\, U.
\end{equation}
For $p=1$, one recovers the operator $L^{(1)}_{f,h}$ with tangential Dirichlet boundary conditions defined by~\eqref{eq:L1_eig} and~\eqref{eq.L-1}. 
The operator $L^{D,(p)}_{f,h}$ with domain
$$D\left (L^{D,(p)}_{f,h}\right )= 
U^{-1}\,D\left (\Delta^{D,(p)}_{f,h}\right )= 
\left\{  w \in    \Lambda^pH^2_w\left (\Omega\right )  | \ \mathbf  tw=0, \ \mathbf  t d^*_{\frac{2f}h,1}w=0   \right \},$$
 is then self-adjoint on $\Lambda^pL^2_w\left (\Omega\right )$, non positive and its associated quadratic form is
$$
\Lambda^{p}H^1_{T}(\Omega)\ \ni\ \phi\ \mapsto\  -\frac{h}{2} \Big [\big\Vert d^{(p)}\phi \big\Vert_{  L^2_w }^2 +\big\Vert d_{\frac{2f}h,1}^{(p)*}\phi \big\Vert_{L^2_w }^2\Big].
$$
Let us also recall that $-L^{D,(p)}_{f,h}$
(and equivalently $\Delta^{D,(p)}_{f,h}$) has a compact resolvent. 
From general results on elliptic operators
when $p=0$,~$-L^{D,(0)}_{f,h}$ (and $\Delta^{D,(0)}_{f,h}$) admits a non degenerate smallest eigenvalue with an associated eigenfunction which has a sign on $\Omega$. 
Denoting moreover by~$\pi_{E}(L^{D,(p)}_{f,h})$ the spectral projector
  associated with $L^{D,(p)}_{f,h}$ and  some Borel set 
$E\subset \mathbb R$\label{page.piea}, the following commutation relations hold
on   $\Lambda^{p}H^1_{T}(\Omega)$:
\begin{equation}\label{eq.comutation}
d^{(p)}\, \pi_{E}(L^{D,(p)}_{f,h})= \pi_{E}(L^{D,(p+1)}_{f,h}(\Omega))\, d^{(p)}\ \text{and}\   
d_{\frac{2f}h,1}^{(p)*}\,\pi_{E}(L^{D,(p)}_{f,h})= \pi_{E}(L^{D,(p-1)}_{f,h})\,d_{\frac{2f}h,1}^{(p)*}.
\end{equation}
Let us recall that from the elliptic regularity of $L^{D,(p)}_{f,h}$, for any bounded Borel set $E\subset \mathbb R$,~$\Ran\,\pi_{E}(L^{D,(d)}_{f,h})\subset \Lambda^{p}\,C^{\infty}_T(\overline{\Omega})$,  the relation~\eqref{eq.comutation}  then 
leads to 
the following complex structure: 
$$\{0\} \longrightarrow 
\Ran\,\pi_{E}(L^{D,(0)}_{f,h}) \xrightarrow{\ d_{f,h}\  } \Ran\,\pi_{E}(L^{D,(1)}_{f,h}) \xrightarrow{\ d_{f,h}\  } \,\cdots\, \xrightarrow{\ d_{f,h}\  } \Ran\,\pi_{E}(L^{D,(d)}_{f,h}) \xrightarrow{\ d_{f,h}\  }\{0\}$$
and
$$\{0\} \xleftarrow{d_{\frac{2f}{h},1}^* } \Ran\,\pi_{E}(L^{D,(0)}_{f,h}) \xleftarrow{d_{\frac{2f}{h},1}^* } \Ran\,\pi_{E}(L^{D,(1)}_{f,h}) \xleftarrow{d_{\frac{2f}{h},1} ^*} \cdots\xleftarrow{d_{\frac{2f}{h},1} ^*} \Ran\,\pi_{E}(L^{D,(d)}_{f,h}) \longleftarrow \ \{0\}.$$
For ease of notation,  
one defines:\label{page.pihp}
\begin{equation}
\label{eq.proj-p}
\forall p\in \{0,\dots,d\}\,,\ \  \pi^{(p)}_h:=\pi_{[0,\frac{\sqrt h}2)}(-L^{D,(p)}_{f,h}).
\end{equation}
The following result, instrumental in our investigation of the smallest eigenvalue $\lambda_h$ of $-L_{f,h}^{D,(0)}$,  is an immediate consequence of \cite[Theorem 3.2.3]{HeNi1} together with~\eqref{eq.unitary}.
\begin{lemma} \label{ran1}
Under assumption \eqref{H-M}, there exists $h_0>0$ such that for all $h\in (0,h_0)$,
$$
\dim \Ran \, \pi^{(0)}_h= {\ft m_{0}^{\Omega}} \, \text{ and }  \dim \Ran \,  \pi^{(1)}_h= \ft m_1^{\overline \Omega},
$$
where ${\ft m_{0}^{\Omega}}={\rm Card}(\ft U_0^\Omega)$ and $\ft m_1^{\overline \Omega}={\rm Card}(\ft U_1^{\overline \Omega})$
are defined in Section~\ref{se.def-zj}.
\end{lemma}


\noindent
In the sequel, with a slight abuse of notation, one denotes the exterior differential $d$ acting on functions by~$\nabla$. 
Note that it follows from the above considerations and Lemma~\ref{ran1} that
under \eqref{H-M}, it holds
\begin{equation}
\label{eq.nablain}
u_{h}\in \,\Ran \, \pi^{(0)}_h\ \ \text{and}\ \ \nabla u_{h}\in\Ran \, \pi^{(1)}_h.
\end{equation}
Moreover,  from~\eqref{eq.unitary}, it is equivalent to study the spectrum of $L_{f,h}^{D,(0)}$ or the spectrum of $\Delta_{f,h}^{D,(0)}$.  We end this section with the following lemma which will be frequently used throughout this work.
\begin{lemma} \label{quadra}
Let $(A,D\left (A\right ))$ be a non negative self adjoint operator on a Hilbert space $\left(\mc H, \Vert\cdot\Vert\right)$ with associated quadratic form $q_A(x)=(x,Ax)$ whose domain is $Q\left (A\right )$. It then holds,  for any $u\in Q\left (A\right )$ and $b>0$,
$$\left\Vert \pi_{[b,+\infty)} (A) \, u\right\Vert^2 \leq \frac{q_A(u)}{b},$$
where, for a Borel set $E\subset \mathbb R$,~$\pi_{E}(A)$ is the spectral projector associated with $A$
and $E$.
\end{lemma}

\subsection{Construction  of the quasi-modes for~$-L_{f,h}^{D,(p)}$,~$p\in\{0,1\}$}
\label{sec:QM-L}
Let us recall that from Lemma~\ref{ran1}, one has for any $h$ small enough 
$$
\dim \Ran \, \pi^{(0)}_h= {\ft m_{0}^{\Omega}} \, \text{ and }  \dim \Ran \,  \pi^{(1)}_h= \ft m_1^{\overline \Omega},
$$
 where we recall that~$\ft m_0^{ \Omega}$ is the number of local minima of~$f$ in~$\Omega$ and~$\ft m_1^{\overline \Omega}$ is the number generalized saddle points of~$f$ in~$\overline\Omega$, see Section~\ref{se.def-zj}.
To prove Theorem~\ref{thm.main}, the strategy  consists in constructing a family  of $\ft m_0^{  \Omega}$ quasi-modes  in order to approximate  $\Ran\,\pi_h^{(0)}$ and  a family  of $\ft m_1^{\overline \Omega}$ quasi-modes in order to approximate  $\Ran\,\pi_h^{(1)}$, see~\eqref{eq.proj-p}.

Since the construction of the quasi-modes rely  on the one made for  Witten Laplacians in~\cite{HKN, HeNi1,HeHiSj}, we first construct quasi-modes for the Witten Laplacians $\Delta_{f,h}^{D,(0)}$ (Section~\ref{sec.construction-qm0}) and   $\Delta_{f,h}^{D,(1)}$ (Section~\ref{sec:delta-1-qm}). The  quasi-modes for  $-L_{f,h}^{D,(0)}$ and   $-L_{f,h}^{D,(1)}$ are then obtained using~\eqref{eq.unitary} (Section~\ref{sec.LOAA}).

\subsubsection{Quasi-modes for the Witten Laplacian $\Delta_{f,h}^{D,(0)}$}\label{sec.construction-qm0}
 \begin{sloppypar}
 Let us assume that  the assumption \eqref{H-M} is satisfied.
 Let us recall that from  Lemma~\ref{ran1} and~\eqref{eq.unitary}, there exists $h_0>0$ such that for any $h\in (0,h_0)$: 

 $$\dim \Ran\, \pi_{[0,h^{\frac 32} )}\big (\Delta^{D,(0)}_{f,h}\big)=\ft m_0^{  \Omega},$$
where we recall that $\ft m_0^{  \Omega}$ is the number of local minima of~$f$ in~$\Omega$.
In this section, one constructs using the maps $\mbf j$ and $\widetilde{\mbf j}$ constructed in Section~\ref{sec:labeling},  a family  of $\ft m_0^{\Omega}$  functions  whose span approximates   $\Ran\, \pi_{[0,h^{\frac 32})}\big (\Delta^{D,(0)}_{f,h}\big)$.
The  properties of this family which are listed in this section will be useful to prove Proposition~\ref{ESTIME1-base} below and Propositions \ref{gamma1} and  \ref{ESTIME1} in the next section. Following \cite{HKN,HeNi1,HeHiSj}, we   associate each critical point $x\in\ft U_0^\Omega$
with a quasi-mode for~$\Delta_{f,h}^{D,(0)}$. The notation follows the one introduced   in Section~\ref{sec:labeling}.\\
\end{sloppypar}

Let us first introduce two parameters $\ve_1>0$ and $\ve>0$ which will be  used to define the quasi-modes for~$\Delta_{f,h}^{D,(0)}$. In the following,~$d$ is the geodesic distance on $\overline \Omega$  for the initial metric.  Let us consider~$\ve_{1}>0$   small enough such that 
\begin{equation}
\label{eq.epsilon1}
\forall z,z'\in \ft U_1^{\overline \Omega} ,\  z\neq z'\ \text{implies}\ d(z,z')\geq 6\ve_{1}
\end{equation}
 and for all $z\in \ft U_1^{\overline \Omega} $,
\begin{equation}
\label{eq.epsilon1'}
z\in  \ft U^{\Omega}_1\ \ \text{and}\ \{f<f(z)\}\cap B(z,2\ve_{1})\  \text{has two connected components
(see~\eqref{eq.2-conn}),}
\end{equation}
or
\begin{equation}
\label{eq.epsilon1''}
z\in  \ft U^{\pa\Omega}_1\ \ \text{and}\ \{f<f(z)\}\cap B(z,2\ve_{1})\  \text{is connected.}
\end{equation}
The parameter~$\ve_1>0$ will be successively reduced a finite number of times in this section and in Section~\ref{sec:delta-1-qm}, and it will be kept fixed from the end of Section~\ref{sec:delta-1-qm}.

Let $\ve>0$ be such that 
$$0<\ve < \frac 12 \ \min_{k\ge 1, \, \ell\in \{1,\ldots,\ft N_k\} } \Big ( \max_{\overline{\ft E_{k,\ell}} } f-\max_{\ft U_1^{\ft{ssp}} \cap \ft E_{k,\ell}} f\Big),$$
which ensures in particular that $\ft E_{k,\ell}(\ve)$ and $\ft E_{k,\ell}(2\ve)$ are connected for all $k\ge 1$ and $\ell \in\{1,\dots,\ft N_1\}$, see~\eqref{eq.ve-E}. The parameter~$\ve>0$ will   be  further reduced a finite number of times in the following sections so that $\pa \chi_{k,\ell}^{\ve,\ve_1}$ is as close as necessary to $\pa \ft E_{k,\ell}$ near $\mbf j(x_{k,\ell})$, where $\chi_{k,\ell}^{\ve,\ve_1}$ is used to define   the quasi-mode for~$\Delta_{f,h}^{D,(0)}$ associated to $x_{k,\ell}$.

\begin{definition}\label{de.v1} Let us assume that  the assumption \eqref{H-M} holds.  
For~$k\geq 1$ and $\ell\in\{1,\dots, \ft N_{k}\}$, the quasi-mode associated with  $x_{k,\ell}$ is  defined by:\label{page.tilevkl} 
\begin{equation}
\label{eq.v1}
\forall \ell\in\{1,\dots,\ft N_k\},\  \widetilde v_{k,\ell}  :=\,  \frac{\chi_{k,\ell}^{\ve,\ve_1}\,e^{- \frac{1}{h} f} }{\big \|\chi_{k,\ell}^{\ve,\ve_1}\,e^{- \frac{1}{h} f}\big\|_{L^2}}\,,
\end{equation}
where the functions $\chi_{k,\ell}^{\ve,\ve_1}\in C_c^{\infty}(\Omega,\mathbb R^+)$. There exists $\ve_1^0>0$ such that for all $\ve_1\in (0,\ve_1^0]$, there exists $\ve^0>0$ such that for all $\ve\in (0,\ve^0]$,  the functions $\chi_{k,\ell}^{\ve,\ve_1}$ satisfy the following properties:
\begin{enumerate}[leftmargin=1.3cm,rightmargin=1.3cm]
\item[a)]  It holds
\begin{equation}
\label{eq.v1-support1}
\left\{
    \begin{array}{ll}
        &\overline {\ft E_{k,\ell}(2\ve)}\subset \{\chi_{k,\ell}^{\ve,\ve_1}=1\} \ \text{ and}  \\
        &\supp \chi_{k,\ell}^{\ve,\ve_1}\subset \Omega\cap\big \{x\in\overline\Omega, \, d(x,\overline{\ft E_{k,\ell}})\leq 3\ve_{1}\big \}\setminus \mathbf j(x_{k,\ell}),
            \end{array}
\right.
\end{equation}
see~\eqref{eq.E1i} for the definition of $\ft E_{k,\ell}$ and~\eqref{eq.ve-E} for the definition of  $\ft E_{k,\ell}(2\ve)$. 
\item[b)]
For all $y\in\supp\chi^{\ve,\ve_1}_{k,\ell}$,  
$$
f(y)\ \leq\ f(\mathbf{j}(x_{k,\ell}))\ \text{implies}\  y \in  \overline{\ft E_{k,\ell}}
$$
and hence, according to \eqref{eq.v1-support1},
\begin{equation}
\label{eq.v1-support2}
\left\{
    \begin{array}{ll}
&\argmin_{\supp \chi_{k,\ell}^{\ve,\ve_1}} f=\argmin_{\ft E_{k,\ell}} f\ \,\text{and}\\
 &  \min_{ {\rm supp} \, \nabla \chi_{k,\ell}^{\ve,\ve_1}}f\geq f(\mathbf j(x_{k,\ell}))-2\ve.
  \end{array}
\right.
\end{equation}

\item[c)] For all $z\in\mathbf j(x_{k,\ell})\cap\Omega$, it holds
\begin{equation}
\label{eq.v1-support3}
\supp\chi^{\ve,\ve_1}_{k,\ell}\cap B(z,2\ve_{1}) \neq \emptyset \text{ and } \supp\chi^{\ve,\ve_1}_{k,\ell}\cap B(z,2\ve_{1}) \subset \ft E_{k,\ell}.
\end{equation}
\item[d)] For all $z\in \ft U_1^{\overline \Omega}\setminus \mathbf j(x_{k,\ell})$, it holds
\begin{equation}
\label{eq.v1-support4}
\left\{
    \begin{array}{ll}
 & z\in \overline{\ft E_{k,\ell}}\ \text{and}\ \overline{B(z,2\ve_{1})}\subset \{\chi_{k,\ell}^{\ve,\ve_1}=1\}\,  \text{ or} \\
& z\notin \overline{\ft E_{k,\ell}}\ \text{and}\ \overline{B(z,2\ve_{1})}\subset \{\chi_{k,\ell}^{\ve,\ve_1}=0\}.
\end{array}
\right.
\end{equation}
\item[e)] For all $q\in\{1,\dots,\ft N_k\}\setminus\{\ell\}$, it holds $\supp \chi_{k,q}^{\ve,\ve_1}\cap \supp \chi_{k,\ell}^{\ve,\ve_1} =\emptyset$.

\item[f)] For $k\ge 2$ and for any $(k', \ell')\in \{1,\dots,k-1\}\times \{1,\dots,\ft N_{k'}\}$   such that
$\ft E_{k,\ell}\subset \ft E_{k',\ell'} $, it holds
\begin{equation}\label{eq.suport-inclusion}
\supp \chi^{\ve,\ve_1}_{k,\ell}\subset \{ \chi^{\ve,\ve_1}_{k',\ell'}=1\}.
\end{equation}
\begin{sloppypar}
Notice   that  by a conexity argument,  it holds ${\ft E_{k,\ell}}\subset \ft E_{k',\ell'} $
or $\ft E_{k,\ell}\cap \ft E_{k',\ell'}  =\emptyset$.
\end{sloppypar}
%
\end{enumerate}
\end{definition}
\noindent
In Figures~\ref{fig:pa_ssp_cutoff},~\ref{fig:ssp_cutoff} and~\ref{fig:non_ssp_cutoff}, for $k\geq 1$,  one gives a schematic representation of   the cut-off function $\chi^{\ve,\ve_1}_{k,\ell}$  near $z\in \ft U_1^{\overline \Omega}$ respectively in the three situations:
\begin{itemize}
\item $k=1$, $\ell\in \{1,\ldots,\ft N_1\}$ and  $z\in \mathbf j(x_{1,\ell})\cap \pa \Omega$.
\item  $\ell\in\{1,\dots, \ft N_{k}\}$ and $z\in\mathbf j(x_{k,\ell})\cap \Omega$.
\item $\ell\in \{1,\ldots,\ft N_k\}$, $z\in \big(\ft U_1^{\overline \Omega}\setminus \mathbf j(x_{k,\ell})\big )\cap \pa\ft  E_{k,\ell}$.
\end{itemize} 
%
%

For the ease of notation, we do not indicate the dependance on the parameters $\ve$ and $\ve_1$ in the notation of the functions $\widetilde v_{k,\ell}$ for $k\geq 1$,~$\ell\in \{1,\ldots,\ft N_k\}$, introduced in Definition~\ref{de.v1}.

  The  following 
 lemma will be useful to estimate when $h\to 0$ the quantities $$\big \Vert (1-\pi_h^{(0)})\, \widetilde v_{k,\ell} \big  \Vert_{L^2},$$ 
 for ${k\geq 1}$ and ${\ell\in \{1,\ldots,\ft N_k\}}$ (where the $\ft m_0^{  \Omega}$ functions $(\widetilde v_{k,\ell})_{k\geq 1,\,\ell\in \{1,\ldots,\ft N_k\}}$ are introduced in Definition~\ref{de.v1}),  see indeed item 2a in Proposition~\ref{ESTIME1-base} below. 
 \begin{lemma}\label{le.dfh-v}
Let us assume that  the assumption \eqref{H-M} holds.  Then, 
 for $k\geq 1$ and $\ell \in\{1,\dots,\ft N_{k}\}$, there exist  $c>0$,~$C>0$ and $h_0>0$ such that for all $h \in (0,h_0)$,
\begin{equation}
\label{eq.dv'}
 \big \|d_{f,h}\widetilde v_{k,\ell} \big \|_{L^{2}}= \frac{ \big  \|h\,e^{-\frac1hf}\,d\chi^{\ve, \ve_1}_{k,\ell} \big \|_{L^{2}}  }{   \big  \| e^{-\frac1hf} \, \chi^{\ve, \ve_1}_{k,\ell} \big \|_{L^{2}}  }  \le C\, e^{-\frac{1}{h}(f(\mathbf{j}(x_{k,\ell}))-f(x_{k,\ell})-c\ve)},
\end{equation}
where the function~$\widetilde v_{k,\ell}$ is introduced in Definition~\ref{de.v1}.
\begin{proof}
This estimate follows from \eqref{eq.v1}--\eqref{eq.v1-support2}, 
and   Laplace's method applied
to $\|\chi^{\ve, \ve_1}_{k,\ell}\,e^{- \frac{1}{h} f}\|_{L^{2}}$.
\end{proof}

\end{lemma}

The following lemma ensures that the family  $(\widetilde v_{k,\ell})_{k\geq 1,\,\ell \in \{1,\ldots,\ft N_k\}}$  is uniformly  linearly independent for any $h$ small enough.

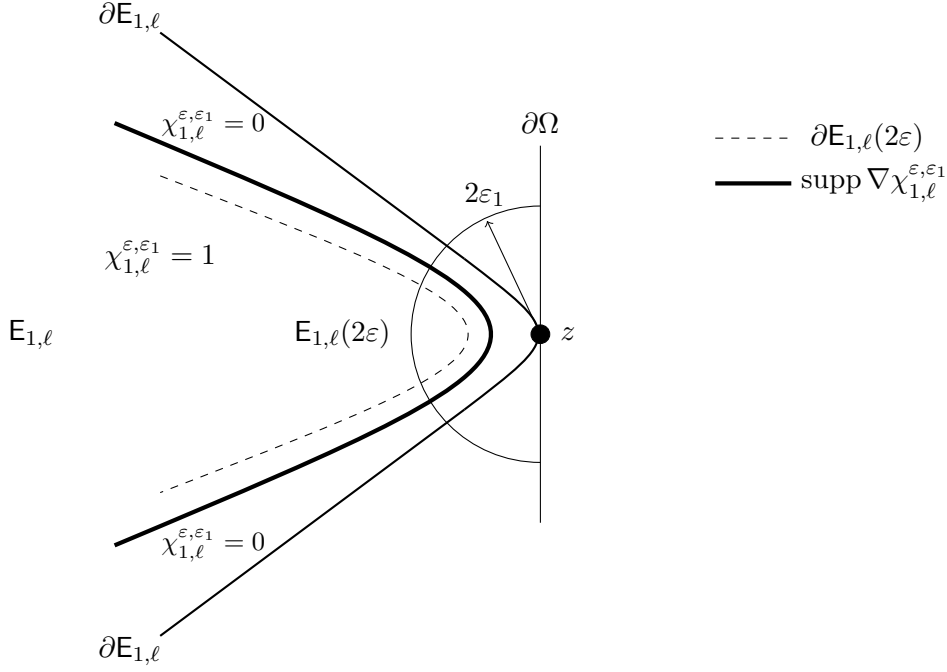
\begin{figure}[h!]
\begin{center}
\begin{tikzpicture}
\tikzstyle{vertex}=[draw,circle,fill=black,minimum size=7pt,inner sep=0pt]
\draw[thick] (-5,4)--(-1,1);
\draw[thick] (-5,-4)--(-1,-1);
\draw (0,-2.5)--(0,2.5);
 \draw (0,2.8) node[]{$\pa \Omega$};

  \draw (0,0)  node[vertex,label=east: {$z$}](v){};
   \draw (-6.7,0) node[]{$\ft E_{1,\ell}$};
      \draw (-5.4,4.2) node[]{$\pa \ft E_{1,\ell}$};
        \draw (-5.4,-4.2) node[]{$\pa \ft E_{1,\ell}$};
  \draw (-2.6,0) node[]{$\ft E_{1,\ell}(2\ve)$};
  \draw[thick]  (-1,1) ..controls (0.3,0)   .. (-1,-1) ;
   \draw[dashed] (-5,2.1) ..controls (0.4,0)   .. (-5,-2.1) ;
   \draw[ultra thick] (-5.6,2.8) ..controls (1,0)   .. (-5.6,-2.8) ;
   \draw (0,-1.7) arc (-90:-270:1.7cm);
   \draw[ultra thick] (2.3,2)--(3.3,2);
   \draw[->] (0,0)--(-0.7,1.5);
   \draw (4.4,2)node[]{${\rm supp} \, \nabla \chi^{\ve, \ve_1}_{1,\ell}$};
      \draw[dashed] (2.3,2.6)--(3.3,2.6);
   \draw (4.3,2.6) node[]{$\pa \ft E_{1,\ell}(2\ve)$};
   \draw  (-5,1) node[]{$\chi^{\ve, \ve_1}_{1,\ell}=1$};
    \draw (-0.73,1.85) node[]{$2\ve_1$};
    \draw  (-4.3,2.75) node[]{\small{$\chi^{\ve, \ve_1}_{1,\ell}=0$}};
    \draw  (-4.3,-2.8) node[]{\small{$\chi^{\ve, \ve_1}_{1,\ell}=0$}};
\end{tikzpicture}
\caption{Schematic representation of the cut-off function $\chi^{\ve, \ve_1}_{1,\ell}$ near  $z\in\mathbf j(x_{1,\ell})\cap \pa \Omega \subset \ft U_1^{\pa \Omega} \cap \ft U_1^{\ft{ssp}}  $ for $\ell\in \{1,\ldots,\ft N_1\}$. }
 \label{fig:pa_ssp_cutoff}
 \end{center}
\end{figure}

\begin{lemma}\label{le.indepp}
\begin{sloppypar}
Let us assume that  the assumption \eqref{H-M} holds. 
 The family $\mathcal B_{v}=(\widetilde v_{k,\ell})_{k\ge 1, \, \ell\in \{1,\ldots,\ft N_k\}} $ introduced in Definition~\ref{de.v1}, is linearly independent,
uniformly with respect to $h$ small enough. This is equivalent to: for some (and hence for any) orthonormal  (for the $L^2$-scalar product)  family  
$\mathcal B_{o}$
spanning $\sspan(\mathcal B_{v})$, for any matrix norm $\Vert \cdot\Vert $ on $\mathbb R^{\ft m_0^\Omega \times \ft m_0^\Omega}$, there exist  $C>0$ and $h_0>0$ such that for all $h \in (0,h_0)$,
\begin{equation}
\label{eq.unif-indpt}
\Vert \text{Mat}_{\mathcal B_{o}}\mathcal B_{v} \Vert \le C \quad
\text{and}
\quad
\Vert \text{Mat}_{\mathcal B_{v}}\mathcal B_{o}\Vert \le C.
\end{equation}
\end{sloppypar}
\end{lemma}
\begin{proof}
The proof of Lemma~\ref{le.indepp} is made in \cite[Section 4.2]{HeHiSj}. 
\end{proof}

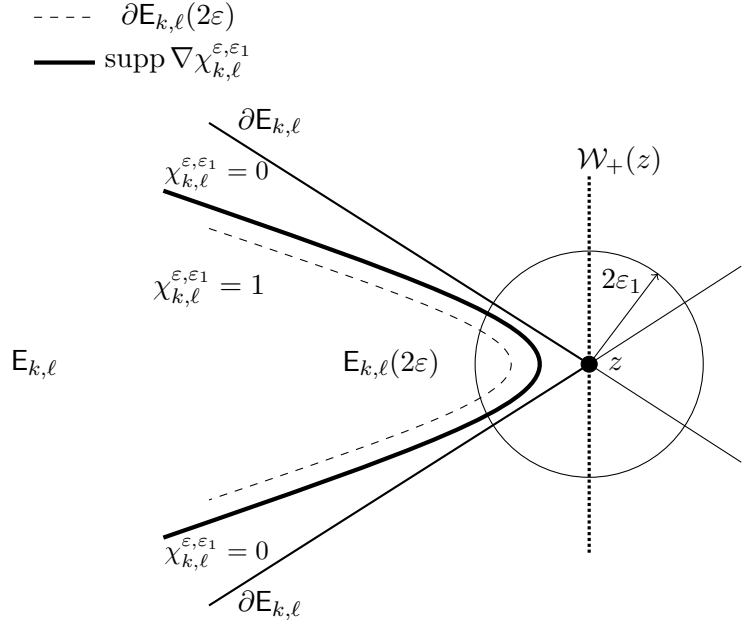
\begin{figure}[h!]
\begin{center}
\begin{tikzpicture}
\tikzstyle{vertex}=[draw,circle,fill=black,minimum size=6pt,inner sep=0pt]
\draw[thick] (-5,3.2)--(0,0);
\draw[thick] (-5,-3.2)--(0,0);
\draw (2,1.3)--(0,0);
\draw  (2,-1.3)--(0,0);
\draw[densely dotted, very thick] (0,-2.5)--(0,2.5);
 \draw (0.4,2.7) node[]{$\mathcal W_+(z)$};

  \draw (0,0)  node[vertex,label=east: {$z$}](v){};
   \draw (-7.3,0) node[]{$\ft E_{k,\ell}$};
      \draw (-4.2,3.2) node[]{$\pa \ft E_{k,\ell}$};
         \draw (-4.2,-3.2) node[]{$\pa \ft E_{k,\ell}$};
  \draw (-2.6,0) node[]{$\ft E_{k,\ell}(2\ve)$};
   \draw[dashed] (-5,1.8) ..controls (0.3,0)   .. (-5,-1.8) ;
   \draw[ultra thick] (-5.6,2.3) ..controls (1,0)   .. (-5.6,-2.3) ;
   \draw (0,0) circle(1.5);
   \draw[ultra thick] (-7.3,4)--(-6.5,4);
   \draw[->] (0,0)--(0.9,1.2);
   \draw (-5.4,4) node[]{${\rm supp} \, \nabla \chi^{\ve, \ve_1}_{k,\ell}$};
   \draw  (-5,1) node[]{$\chi^{\ve, \ve_1}_{k,\ell}=1$};
    \draw (0.43,1.1) node[]{$2\ve_1$};
      \draw[dashed] (-7.3,4.6)--(-6.5,4.6);
   \draw (-5.4,4.6) node[]{$\pa \ft E_{k,\ell}(2\ve)$};
\draw  (-4.9,2.52) node[]{\small{$\chi^{\ve, \ve_1}_{k,\ell}=0$}};
\draw  (-4.9,-2.52) node[]{\small{$\chi^{\ve, \ve_1}_{k,\ell}=0$}};
\end{tikzpicture}
\caption{Schematic representation of the cut-off function $\chi^{\ve, \ve_1}_{k,\ell}$ near  $z\in\mathbf j(x_{k,\ell})\cap \Omega \subset \ft U_1^{\Omega}$ for $k\ge 1$ and $\ell \in \{1,\ldots,\ft N_k\}$. The point $z$ is a {separating saddle point} as introduced in Definition~\ref{de.SSP}.
The set $\mathcal W_+$ is the stable manifold of the saddle point $z$. }
 \label{fig:ssp_cutoff}
 \end{center}
\end{figure}

\begin{figure}[h!]
\begin{center}
\begin{tikzpicture}[scale=0.8]
\tikzstyle{vertex}=[draw,circle,fill=black,minimum size=7pt,inner sep=0pt]
\draw[thick] (0,0) circle(5.8);
\draw[thick] (-2.3,2)--(2.3,-2);
\draw[thick] (-2.3,-2)--(2.3,2);
   \draw[thick ] (-2.3,2) ..controls (-3.3,3.3)  and   (3.3,3.3)  .. (2.3,2) ;
   \draw[thick ] (-2.3,-2) ..controls (-3.3,-3.3)  and   (3.3,-3.3)  .. (2.3,-2) ;
\draw[densely dotted, very thick] (0,-2.3)--(0,2.3);
 \draw (0,2.55) node[]{{\small $\mathcal W_+(z)$}};
\draw (0,0) node[vertex,label=east: {$z$}](v){};
         \draw (5,5) node[]{$\pa \ft E_{k,\ell}$};
          \draw[->] (4.8,4.8)--(4.2,4.2);
          \draw[->] (4.8,4.8)--(1.8,2.8);
  \draw (-3,0) node[]{$\ft E_{k,\ell}(2\ve)$};
  \draw (3,0) node[]{$\ft E_{k,\ell}(2\ve)$};
   \draw (0.2,4.5) node[]{$\ft E_{k,\ell}$};
   \draw[dashed] (-3.8,1.4) ..controls (-0.2,0)   .. (-3.8,-1.4) ;
   \draw[dashed] (3.8,1.4) ..controls (0.2,0)   .. (3.8,-1.4) ;
   \draw (0,0) circle(1.5);
   \draw[->] (0,0)--(0.9,1.2);
    \draw (0.43,1.1) node[]{$2\ve_1$};
          \draw (-6.8,3) node[]{${\rm supp} \, \nabla \chi^{\ve, \ve_1}_{k,\ell}$};
   \draw  (-4.4,1) node[]{$\chi^{\ve, \ve_1}_{k,\ell}=1$};
     \draw  (4.4,1) node[]{$\chi^{\ve, \ve_1}_{k,\ell}=1$};
      \draw[ultra thick] (-9,3)--(-8.2,3);      
      \draw[ultra thick] (-3.5,1.7) ..controls (-2.1,0.8) ..  (-1.8,1);
       \draw[ultra thick]  (-1.8,1) ..controls (0,2.1)   .. (1.8,1) ;
       \draw[ultra thick] (3.5,1.7) ..controls (2.1,0.8) ..  (1.8,1);
             \draw[ultra thick] (-3.5,-1.7) ..controls (-2.1,-0.8) .. (-1.8,-1);
       \draw[ultra thick]  (-1.8,-1) ..controls (0,-2.1)   .. (1.8,-1) ;
       \draw[ultra thick] (3.5,-1.7) ..controls (2.1,-0.8) ..  (1.8,-1);

         \draw[dashed] (-8.8,3.6)--(-8,3.6);
   \draw (-6.9,3.6) node[]{$\pa \ft E_{k,\ell}(2\ve)$};
\end{tikzpicture}
\caption{Schematic representation of the cut-off function $\chi^{\ve,\ve_1}_{k,\ell}$ near $z\in \big (\ft U_1^{\overline \Omega}\setminus \mathbf j(x_{k,\ell})\big )\cap \pa \ft E_{k,\ell}$. The point $z$ is a saddle point on $\pa \ft E_{k,\ell}$ but is not a {separating saddle point} as introduced in Definition~\ref{de.SSP}. }
 \label{fig:non_ssp_cutoff}
 \end{center}
\end{figure}
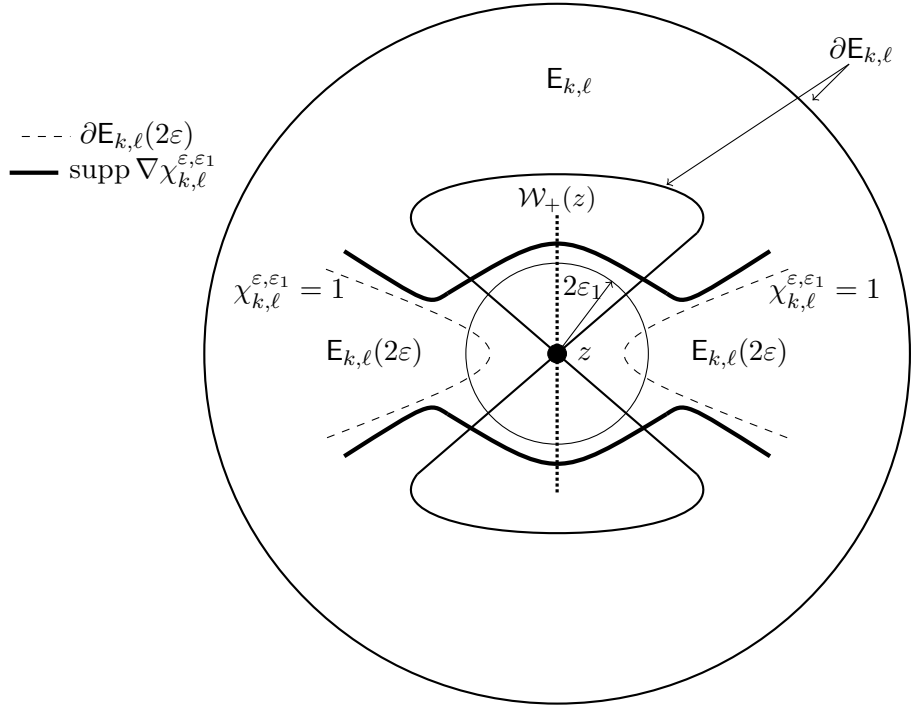
 
\subsubsection{Quasi-modes  for the Witten Laplacian $\Delta_{f,h}^{D,(1)}$} 
\label{sec:delta-1-qm}
 Let us assume that  the assumption \eqref{H-M} is satisfied. 
Let us recall that from  Lemma~\ref{ran1} and~\eqref{eq.unitary}, there exists $h_0>0$ such that for any $h\in (0,h_0)$: 
$$\dim \Ran\, \pi_{[0,h^{\frac 32} )}\big (\Delta^{D,(1)}_{f,h}\big)=\ft m_1^{\overline \Omega},$$
where we recall that~$\ft m_1^{\overline \Omega}$ is the number of generalized saddle points of~$f$ in~$\overline \Omega$, see Section~\ref{se.def-zj}.
 In this section, one constructs    
a family of $1$-forms $(\widetilde \phi_j)_{j\in \{1,\dots,\ft m_1^{\overline \Omega}\}}$ which aims at approximating    $\Ran\, \pi_{[0,h^{\frac 32} )}\big (\Delta^{D,(1)}_{f,h}\big)$. To this end, for each $z\in  \ft U^{\overline \Omega}_1$, one constructs a $1$-form locally supported in a neighborhood  of $z$ in~$\overline \Omega$.  More precisely, one proceeds as follows:
\begin{enumerate}
\item for  each $z\in \ft U^{\Omega}_1$, the  $1$-form associated with $z$ is  constructed following the procedure in \cite{HKN,HeHiSj} and,
\item  for  each $z\in \ft U^{\pa \Omega}_1$, the  $1$-form associated with $z$ is  constructed  as in \cite{HeNi1}.
\end{enumerate}
Let us recall these constructions and some estimates which will be used throughout this work. \\

\noindent
\textbf{Quasi-mode associated with $z\in  \ft U^{\Omega}_1$}.\\  
Let us recall that from \eqref{eq.U1Omega},
 $$ \ft U_1^\Omega=\{ z_{\ft m_1^{\pa \Omega}+1},\ldots,z_{\ft m_1^{\overline \Omega}}\}\subset \Omega,$$
is the set of saddle points of~$f$ in~$\Omega$. Let $j\in  \{\ft m_1^{\pa \Omega}+1,\dots,\ft m_1^{\overline \Omega} \big \}$ and $z_{j}\in\ft U_1^\Omega$. 
Let~$\mathcal V_{j}$ be some small smooth  neighborhood  of $z_{j}$
such that $\overline {\mathcal V_{j}}\cap \pa\Omega=\emptyset$ and for $x\in~\overline{\mathcal V_{j}}$,~$\vert \nabla f(x)\vert=0$ if and only if $x=z_j$.
Let us now consider the full Dirichlet realization $\Delta_{f,h}^{FD,(1)}(\mathcal V_{j})$
of the Witten Laplacian $\Delta_{f,h}^{(1)}$ in~$\mathcal V_{j}$ whose
domain is
$$D\Big(\,\Delta_{f,h}^{FD,(1)}\big(\mathcal V_{j})\,\Big)= \left\{  w \in    \Lambda^1H^2\left (\mathcal V_{j}\right ), \,   w|_{\pa\mathcal V_{j}}=0   \right \},$$
where the superscript $FD$ stands for full Dirichlet boundary conditions. Let us recall that according to~\cite[Section 2]{HeSj4}, 
   there exists, choosing if necessary $\mathcal V_{j}$ smaller,  a $C^{\infty}$ non negative solution $\Phi_j:\mathcal V_j\to \mathbb R^+$ to the eikonal equation\label{page.Phij}
 \begin{equation} \label{eq.eikoj0}
\vert  \nabla \Phi_j \vert =  \vert  \nabla f \vert   \ {\rm in \ }  \mathcal V_{j}
\quad\text{such that  
$\Phi_j(y)= 0$ iff $y=z_{j}$}.
\end{equation}
Moreover,~$\Phi_j$ is the unique non negative solution to~\eqref{eq.eikoj0}  in the sense that if $\widetilde \Phi_j:\widetilde{\mathcal V_j}\to \mathbb R^+$ is  another  non negative $C^{\infty}$ solution to~\eqref{eq.eikoj0} on a neighborhood $\widetilde{\mathcal V_j}$ of~$z_j$, then $\widetilde \Phi_j=\Phi_j$ on $\widetilde{\mathcal V_j}\cap \mathcal V_j$. 
\begin{remark}\label{re.Ag}
The function $\Phi_j$ is actually   the Agmon distance to $z_j$, i.e. $\Phi_j$ is the distance to $z_j$   in $\overline \Omega$ associated with the metric $\vert  \nabla f \vert^2dx^2$, where $dx^2$ is the Riemannian metric on~$\overline \Omega$ (see~\cite[Section~1]{HeSj4}). 
\end{remark}

The next proposition, which follows from \cite[Theorem~1.4 and Lemma~1.6]{HeSj4}, gathers all the estimates one needs in the following on the operator $\Delta_{f,h}^{FD,(1)}(\mathcal V_{j})$.

\begin{proposition}\label{pr.zj-omega}
 Let us assume that  the assumption \eqref{H-M} is satisfied. 
Then,  the operator $\Delta_{f,h}^{FD,(1)}(\mathcal V_{j})$ is  self-adjoint, has compact resolvent and is positive. Moreover:
\begin{itemize}[leftmargin=1.3cm,rightmargin=1.3cm]
\item  There exist $\ve_0>0$ and $h_0>0$ such that for all $h\in (0,h_0)$:
\begin{equation}\label{eq.dim11}
\dim \Ran\, \pi_{[0,\ve_0h )}\big (\Delta_{f,h}^{FD,(1)}(\mathcal V_{j})\big)=1.
\end{equation}
\item  The smallest eigenvalue $\lambda_{h}(\mathcal V_{j})$ of 
$\Delta_{f,h}^{FD,(1)}(\mathcal V_{j})$ is exponentially small: there exist $C>0$,~$c>0$ and $h_0>0$ such that for any  $h\in (0,h_0)$:

\begin{equation}
\label{eq.lambdah-0}
\lambda_{h}(\mathcal V_{j})\le Ce^{-\frac ch}. 
\end{equation}
\item Any  $L^2$-normalized eigenform $w_{j}$\label{page.wj}  associated with the smallest eigenvalue~$\lambda_{h}(\mathcal V_{j})$ of $\Delta_{f,h}^{FD,(1)}(\mathcal V_{j})$  satisfies  the following Agmon estimates (see Remark~\ref{re.Ag}): for all $ \ve>0$, there exist $C_{\ve}>0$ and $h_0>0$ such that for any  $h\in (0,h_0)$, it holds:
\begin{equation}
\label{eq.Agmon10}
\big \|e^{\frac1h \Phi_{j}}w_j\big \|_{H^1(\mathcal V_{j})}\leq C_{\ve}\,e^{\frac \ve h}.
\end{equation} 

\end{itemize}
\end{proposition}

\noindent
Choosing $\ve_{1}$ smaller if necessary, one assumes that there exists $\alpha>0$ such that 
$$ B(z_{j},2\ve_{1}+\alpha)  \subset \mathcal V_{j}.$$
Let us now define the quasi-mode
associated with $z_{j}\in \ft U_1^\Omega$.

\begin{definition}  \label{de.tildephi}
 Let us assume that  the assumption \eqref{H-M} is satisfied. 
Let $j\in  \{\ft m_1^{\pa \Omega}+1,\dots,\ft m_1^{\overline \Omega} \big \}$ and $z_{j}\in\ft U_1^\Omega$.  The quasi-mode
associated with $z_{j}$ is  defined by\label{page.tildephij}
\begin{equation}
\label{eq.tildephi}
\widetilde\phi_{j} := \frac{\theta_j\,w_{j}}{\left\|\theta_j\,w_{j}\right\|_{L^2}}\in \Lambda^1C^\infty_c(\Omega),
\end{equation}
where $w_{j}$ is a $L^2$-normalized eigenform associated with the smallest eigenvalue $\lambda_{h}(\mathcal V_{j})$ of 
$\Delta_{f,h}^{FD,(1)}(\mathcal V_{j})$ and $\theta_j$ is a smooth non negative cut-off function satisfying, 
$\supp \theta_j\subset B(z_{j},2\ve_{1}) \subset \mathcal V_j $
and $\theta_j=1$ on $B(z_{j},\ve_{1})$.
\end{definition}
Notice that both $w_j$ and $-w_j$ can be used to build a quasi-mode and the choice of the sign is determined in Proposition~\ref{compa_wkb-bis}. Moreover,    using~\eqref{eq.Agmon10} together with the fact that    for all $j\in  \{\ft m_1^{\pa \Omega}+1,\dots,\ft m_1^{\overline \Omega} \big \}$, $\inf_{\text{supp } (1-\theta_j) \cap \mathcal V_j} \Phi_j>0$   (see~\eqref{eq.eikoj0}), one has when $h\to 0$:
\begin{equation}\label{eq.norme-phij}
\big \|(1-\theta_j)\,w_{j}\big \|_{L^2(\mathcal V_j)}=O\big(e^{-\frac ch}\big) \text{ and therefore, }\,  \big\|\theta_j\,w_{j}\big\|_{L^2}=1+O\big(e^{-\frac ch}\big),
\end{equation}
for some $c>0$ independent of $h$. 

Using~Proposition~\ref{pr.zj-omega} and~\eqref{eq.norme-phij}, one deduces the following estimate on the quasi-mode~$\widetilde \phi_j$ introduced in Definition~\ref{de.tildephi}.
\begin{corollary}
\label{co.dfh-phi-j}
    Let us assume that  the assumption~\eqref{H-M} holds. Let  
$\widetilde\phi_{j}$ be the quasi-mode associated with $z_j\in U_1^\Omega$ ($j\in  \{\ft m_1^{\pa \Omega}+1,\dots,\ft m_1^{\overline \Omega} \big \}$), see Definition~\ref{de.tildephi}. 
Then, there exist $C>0$,~$c>0$ and $h_0>0$ such that for any  $h\in (0,h_0)$:
\begin{equation}
\label{eq.delta-tildephi}
 \big  \|d_{f,h}\widetilde\phi_{j} \big \|_{L^{2}}\,+\, \big \|d^{*}_{f,h}\widetilde\phi_{j} \big \|_{L^{2}}\le C e^{-\frac ch}.
\end{equation}
\end{corollary} 
Let us now  recall the construction of a WKB approximation of $w_j$ made in~\cite{HeSj4} and which will be needed in the following.
Let us denote by~$\mathcal  W_{+}(z_{j})$ and $\mathcal  W_{-}(z_{j})$
respectively the stable and unstable manifolds of $z_{j}$ associated with the flow of $-\nabla f$ which are defined as follows. Denoting by~$\varphi_{t}(y)$ the solution of $ \frac{d}{dt}\varphi_{t}(y)=-\nabla f(\varphi_{t}(y))$
with initial condition $\varphi_{0}(y)=y$,  
\begin{equation}
\label{eq.W+}
\mathcal  W_{\pm}(z_{j}):= \big \{y\in\Omega, \varphi_{t}(y)\underset{t\to\pm \infty}{\longrightarrow }z_{j}  \big \}\,.
\end{equation}
It then holds (see indeed \cite[Section~2]{HeSj4} and \cite[Section 4.2]{HKN}): $\dim \mathcal  W_{+}(z_{j})=d-1 $,~$\dim \mathcal  W_{-}(z_{j})=1 $, and 
for all   $y\in \mathcal V_{j}$ (assuming $\mathcal V_{j}$ small enough),
\begin{equation}
\label{eq.stable-manifold1}
|f(y)-f(z_{j})|\leq \Phi_{j}(y) \  \text{and}\  
|f(y)-f(z_{j})|= \Phi_{j}(y)\ \text{iff}\ y\in \mathcal  W_{+}(z_{j})\cup  \mathcal  W_{-}(z_{j})
\end{equation}
with moreover
\begin{equation}
\label{eq.stable-manifold2}
\Phi_{j}= \pm(f-f(z_{j}))\, \text{on}\ \mathcal  W_{\pm}(z_{j})
\ \text{and}
\
\det  \Hess \Phi_{j}(z_{j})= \big |\det  \Hess f(z_{j})\big |. 
\end{equation}
Additionally, there exists from \cite[Proposition~1.3 and Section 2]{HeSj4}
 a $C^{\infty}(\overline{\mathcal V_j})$ $1$-form $ a_j(x,h)=\widetilde a_j(x)+O(h)$ such that   $a_j(z_{j},h)=\widetilde a_j(z_{j})=n(z_{j}) $, where $n(z_{j})$ is a unit normal 
to $\mathcal  W_{+}(z_{j})$,  and such that  the $1$-form 
$
u^{(1)}_{j,wkb}=a_{j} e^{-\frac{1}{h}\Phi_j}$\label{page.u1jwkb} 
satisfies
$$\big (\Delta_{f,h}^{(1)}-\mu(h)\big)u^{(1)}_{j,wkb}=O\big (h^{\infty}\big )\, e^{-\frac{1}{h}\Phi_j}\
\text{in} \ 
\mathcal V_{j},$$
where $\mu(h)\sim h^2 \sum_{k=0}^\infty\mu_kh^k$. Moreover, one has in the limit $h\to 0$ (see \cite[Section~2]{HeSj4}):
\begin{equation}\label{eq.norme-wkb0}
\|\theta_j u^{(1)}_{j,wkb}\|_{L^2}= \frac{ (\pi h)^{\frac{d}{4}} }{ |{\rm det\, \Hess} f(z_j)  |^{\frac 14}  }\, \big(1+O(h)\big),
\end{equation}
where the remainder terms $O(h)$   admits a full asymptotic expansion in~$h$.
Using in addition the fact that $\Phi_j>0$ on ${\rm supp} \, \nabla \theta_j$, there exists $c>0$ such that for $h$ small enough:
$$\Big\Vert \big (\Delta_{f,h}^{(1)}-\mu(h)\big)(\theta_j\, u^{(1)}_{j,wkb})\Big\Vert _{L^2}= O\big (h^{\infty}\big )+ O\big (e^{-\frac ch}\big )=O\big (h^{\infty}\big ).$$
From~\eqref{eq.norme-wkb0},  one then obtains that $\Delta_{f,h}^{FD,(1)}(\mathcal V_{j})$ admits an eigenvalue which equals $\mu(h)+O(h^\infty)$. Since $\mu(h)=O(h^2)$, from~\eqref{eq.dim11} and~\eqref{eq.lambdah-0}, one deduces that $\mu_k=0$ for all $k\ge 0$ and thus $\mu(h)=O(h^\infty)$. Finally, one has:
\begin{equation} \label{eq.wWKB10}
\Delta_{f,h}^{(1)}u^{(1)}_{j,wkb} =O\big (h^{\infty}\big )\, e^{-\frac{1}{h}\Phi_j}\
\text{in} \ 
\mathcal V_{j}.
\end{equation}
In the following proposition,  $w_j$ and $ u^{(1)}_{j,wkb}$ are compared. 

\begin{proposition}\label{compa_wkb-bis}
         Let us assume that  the assumption~\eqref{H-M} is satisfied.  Let $w_j$ be a $L^2$-normalized eigenform  associated with the smallest eigenvalue~$\lambda_{h}(\mathcal V_{j})$ of $\Delta_{f,h}^{FD,(1)}(\mathcal V_{j})$ ($j\in  \{\ft m_1^{\pa \Omega}+1,\dots,\ft m_1^{\overline \Omega} \big \}$). Then, there exists $h_0>0$ such that for all $h\in (0,h_0)$ one has:
\begin{equation}
\label{eq.a-priori-comparison0}
\left\|\theta_j\big (w_j-c_j(h)u^{(1)}_{j,wkb}\big )\right\|_{H^{1}}=O\big (h^{\infty}\big )
\end{equation}
where
$$
 c_j(h)^{-1}=\big \langle  w_j, \theta_j u^{(1)}_{j,wkb}\big  \rangle_{L^2 }.
$$
In addition, up to replacing $w_j$ by~$-w_j$, one can assume that $c_j(h)\geq0$ for $h$ small enough and then,  in the limit $h\to 0$, one has\label{page.cj}: 
\begin{equation}
\label{eq.ch0}
 c_j(h)=\frac{ \big |{\rm det\, \Hess} f(z_j)  \big |^{\frac 14}}{(\pi h)^{\frac{d}{4}}  }\, \big(1+O(h)\big),
\end{equation}
where the remainder terms $O(h)$   admits a full asymptotic expansion in~$h$.
 \end{proposition}
\noindent
 
 \begin{proof}
Let us define $k_j(h):=\big \langle  w_j, \theta_j u^{(1)}_{j,wkb}\big \rangle_{L^2 }=c_j(h)^{-1}$. If $k_j(h)<0$, then one
changes $w_j$ to $-w_j$ so that one can suppose without
loss of generality that $k_j(h) \geq 0$. For $h$ small enough, one
has from~\eqref{eq.dim11}:
$$\pi_{[0,\ve_0h)}\big(\Delta_{f,h}^{FD,(1)}  (\mathcal V_{j}) \big)  \big ( \theta_j u^{(1)}_{j,wkb}\big )= k_j(h) w_j.$$
Let us define  the following $1$-form
$$\alpha_j:= \theta_j  \big ( u^{(1)}_{j,wkb} -k_j(h)\, w_j \big) .$$ 
Thus, the following identity holds for $h$ small enough
$$\alpha_j=  k_j(h) \left(1-\theta_j\right)w_j+\pi_{[\ve_0h,+\infty)}\big (\Delta_{f,h}^{FD,(1)}(\mathcal V_{j})\big )\big   ( \theta_j u^{(1)}_{j,wkb}\big ).$$
Notice that, from~\eqref{eq.norme-wkb0}, there exist $C>0$ and  $h_0>0$ such that for all $ h\in (0,h_0)$
$$\big \vert k_j(h)\big \vert \leq Ch^{\frac{d}{4}}.$$
Therefore, using Lemma~\ref{quadra},~\eqref{eq.norme-phij},  and~\eqref{eq.wWKB10},  there exist $c>0$ and $C>0$ such that for~$h$ small enough:
\begin{align*}
\big \Vert  \alpha_j\big  \Vert^2_{L^2 } &\leq 2  k_j(h)^2 \, \big \Vert
\left(1-\theta_j \right)w_j\big \Vert^2_{L^2 }  \, +\, 2 \left\Vert \pi_{[\ve_0h,+\infty)}\big (\Delta_{f,h}^{FD,(1)}(\mathcal V_{j})\big )  ( \theta_j u^{(1)}_{j,wkb}) \right\Vert^2_{L^2 }\\
&\leq Ch^{\frac{d}{2}}   \,  e^{-\frac{c}{h} }   +   C h^{-1} O\big (h^{\infty}\big ) = O\big (h^{\infty}\big ).
\end{align*}
  Moreover, since $d_{f,h} =hd+df\wedge$ and $d_{f,h}^* =hd^*+\mbf{i}_{\nabla
  f}$, one obtains  using  the Gaffney
inequality (see~\cite{GSchw}):
\begin{align*}
\Vert \alpha_j\Vert_{H^1 }^2
&\leq C\big ( \Vert d\alpha_j\Vert^2_{L^2 }+\Vert d^*\alpha_j\Vert^2_{L^2 }+\Vert \alpha_j\Vert^2_{L^2 } \big )\\
&\leq Ch^{-2} \left(\big  \|d_{f,h}\alpha_j \big \|_{L^{2}}^2\,+\, \big \|d^{*}_{f,h}\alpha_j\big \|_{L^{2}}^2+\Vert \alpha_j\Vert^2_{L^2 } \right).
\end{align*}
Furthermore, from~\eqref{eq.delta-tildephi}, it holds
$$ \big  \|d_{f,h}(\theta_jw_j )\big \|_{L^{2}}\,+\, \big \|d^{*}_{f,h}(\theta_jw_j)  \big \|_{L^{2}}\le C e^{-\frac ch}$$
   and from~\eqref{eq.wWKB10}
   $$ \big  \|d_{f,h}(\theta_j u^{(1)}_{j,wkb} )\big \|_{L^{2}}\,+\, \big \|d^{*}_{f,h}(\theta_j u^{(1)}_{j,wkb})\big \|_{L^{2}}=O(h^\infty).$$
Thus, there exists $C>0$ such that:
\begin{align*}
\Vert \alpha_j\Vert_{H^1 }^2
&= O\big (h^{\infty}\big ).
\end{align*}
This concludes the proof of~\eqref{eq.a-priori-comparison0}.   Finally, since $ \Vert \theta_j w_j
\Vert_{L^2 } =1+O(e^{-\frac ch})$ (see~\eqref{eq.norme-phij}), by considering $ \Vert \theta_j (  u^{(1)}_{j,wkb}-  k_j(h) w_j
)\Vert^2_{L^2 } =O\big (h^{\infty}\big )$, one gets using~\eqref{eq.norme-wkb0}:  
\begin{align*}
 k_j(h) ^2  & = \frac{\Vert \theta_j u^{(1)}_{j,wkb}
  \Vert^2_{L^2 } + O\big (h^{\infty}\big )}{2-\Vert \theta_j w_j\Vert^2_{L^2 }}= \frac{ (\pi h)^{\frac{d}{2}} }{ |{\rm det\, \Hess} f(z_j)  |^{\frac 12}  }\, \big(1+O(h)\big).
\end{align*}
Since $k_j(h)
\geq 0$, one has  
$k_j(h)=\frac{ (\pi h)^{\frac{d}{4}} }{ |{\rm det\, \Hess} f(z_j)  |^{\frac 14}  }\, \big(1+O(h)\big)$. This concludes the proof of~\eqref{eq.ch0} since $c_j(h)=k_j(h)^{-1}$.
 \end{proof}
 

\noindent
\textbf{Quasi-mode associated with  $z\in  \ft U^{\pa \Omega}_1$}.\\      
Let us recall that from \eqref{eq.U1paOmega},
 $$ \ft  U_1^{\pa \Omega}=\{ z_{ 1},\ldots,z_{\ft m_1^{\pa  \Omega}}\}\subset \pa \Omega.$$
Let $j\in  \{ 1,\dots,\ft m_1^{\pa \Omega} \big \}$ and $z_{j}\in\ft U_1^{\pa \Omega}$.  To construct a $1$-form locally supported in a neighborhood  of $z_j$ in~$\overline \Omega$,  one proceeds in the same way as in  \cite[Section 4.3]{HeNi1}.
Let~$\mathcal V_{j}$ be a small   neighborhood   of
$z_{j}$ in~$\overline\Omega$ such that~$\mathcal V_{j}$  satisfies:  $\vert \nabla f\vert>0$ on $\overline{ \mathcal V_{j}}$, for all $x\in \pa  \mathcal V_{j} \cap \pa \Omega$,~$\vert \nabla_T f (x)\vert= 0$ if and only if $x=z_j$, and  $\partial_n f >0$ on $\pa\Omega\cap \pa \mathcal V_{j}$. 
Let us now consider the mixed 
 full Dirichlet--tangential Dirichlet
realization  $\Delta_{f,h}^{MD,(1)}(\mathcal V_{j})$
of the Witten Laplacian $\Delta_{f,h}^{(1)}$ in~$\mathcal V_{j}$ whose
domain is
\begin{align*}
D\big (\Delta_{f,h}^{MD,(1)}(\mathcal V_{j})\big )= \big\{&  w \in    \Lambda^1H^2\left (\mathcal V_{j}\right ), \,   w|_{\pa\mathcal  V_j\cap \Omega  }=0 ,\ \mbf{t}w|_{\pa\mathcal  V_j\cap \pa \Omega  }=0  \\
&  \text{ and }  \mbf{t}d^*_{f,h}w|_{\pa \mathcal V_j\cap \pa \Omega   }=0  \big \},
\end{align*}
where the superscript $MD$ stands for  mixed 
 full Dirichlet--tangential Dirichlet  boundary conditions (see \cite[Remark 4.3.1]{HeNi1} for the characterization of its domain).  Since $\partial_n f >0$ on $\pa\Omega\cap \mathcal V_{j}$, from  \cite[Section 4.2]{HeNi1}, one has that,  choosing~$\mathcal V_{j}$ small enough,  there exists a $C^{\infty}( \mathcal V_{j},\mathbb R^+)$ non negative solution $\Phi_j$ to the eikonal equation \label{page.Phij2}
 \begin{equation} \label{eq.eikoj}
\left.
    \begin{array}{ll}
&\vert  \nabla \Phi_j \vert =  \vert  \nabla f \vert   \ {\rm in \ }  \Omega \cap \mathcal V_{j} \\ 
&\Phi_j = f -f(z_j)\ {\rm on \ } \partial \Omega \cap \mathcal V_{j} \\
&\partial_n \Phi_j=-\partial_n f \ {\rm on \ } \partial \Omega \cap \mathcal V_{j} \end{array}
\right \} \quad\text{ and such that  
$\Phi_j(y)= 0$ iff $y=z_{j}$.}
\end{equation}
Moreover,~$\Phi_j$ is the unique non negative solution to~\eqref{eq.eikoj}  in the sense that if $\widetilde \Phi_j:\widetilde{\mathcal V_j}\to \mathbb R^+$ is  another  non negative $C^{\infty}$ solution to~\eqref{eq.eikoj} on a neighborhood $\widetilde{\mathcal V_j}$ of~$z_j$, then $\widetilde \Phi_j=\Phi_j$ on $\widetilde{\mathcal V_j}\cap \mathcal V_j$. 
\begin{remark}
The function $\Phi_j$ is actually   the Agmon distance to $z_j$, see~\cite[Section~3]{di-gesu-le-peutrec-lelievre-nectoux-16} for a precise definition of the Agmon distance in a bounded domain.
\end{remark}
\noindent
Choosing $\ve_{1}$ smaller if necessary, one can assume that there exists $\alpha>0$ such that 
$$
B(z_{j},2\ve_{1}+\alpha)\cap \overline \Omega\subset \mathcal V_{j}.
$$
The next proposition, which follows from \cite[Proposition 4.3.2]{HeNi1}, gathers all the estimates one needs in the following on the operator $\Delta_{f,h}^{MD,(1)}(\mathcal V_{j})$.

\begin{proposition}\label{pr.zj-omega-2}
    Let us assume that  the assumption~\eqref{H-M}  is satisfied. Then,  the operator $\Delta_{f,h}^{MD,(1)}(\mathcal V_{j})$ is  self-adjoint, has compact resolvent and is positive. Moreover:
\begin{itemize}[leftmargin=1.3cm,rightmargin=1.3cm]
\item  There exists $h_0>0$ such that for all $h\in (0,h_0)$:
\begin{equation}\label{eq.dim11-bis}
\dim \Ran\, \pi_{[0,h^{\frac 32} )}\big (\Delta_{f,h}^{MD,(1)}(\mathcal V_{j})\big)=1.
\end{equation}
\item  The smallest eigenvalue $\lambda_{h}(\mathcal V_{j})$ of 
$\Delta_{f,h}^{MD,(1)}(\mathcal V_{j})$ is exponentially small: there exist $C>0$,~$c>0$ and $h_0>0$ such that for any  $h\in (0,h_0)$:

\begin{equation}
\label{eq.lambdah-0-bis}
\lambda_{h}(\mathcal V_{j})\le Ce^{-\frac ch}. 
\end{equation}
\item Any  $L^2$-normalized eigenform $w_{j}$\label{page.wj2} associated with the smallest eigenvalue~$\lambda_{h}(\mathcal V_{j})$ of $\Delta_{f,h}^{MD,(1)}(\mathcal V_{j})$  satisfies  the following Agmon estimates:   there exist $C >0$,~$n\in \mathbb N$ and $h_0>0$ such that for any  $h\in (0,h_0)$, it holds:
\begin{equation}
\label{eq.Agmon1}
\big \|e^{\frac1h \Phi_j}w_j \big  \|_{H^1(B(z_{j},2\ve_{1})\cap \Omega)} \leq Ch^{-n}.
\end{equation} 

\end{itemize}
\end{proposition}


Let us now define the quasi-mode
associated with $z_{j}\in U_1^{\pa \Omega}$.
\begin{definition}   \label{de.zj-pa-omega}
    Let us assume that  the assumption~\eqref{H-M} holds. 
Let $j\in  \{1,\dots,\ft m_1^{\pa \Omega} \big \}$ and $z_{j}\in\ft U_1^{\pa \Omega}$.  The quasi-mode
associated with $z_{j}$ is  defined by\label{page.tildephij2}
\begin{equation}
\label{eq.tildephi'}
\widetilde\phi_{j}\ :=\ \frac{\theta_j\,w_{j}}{\left\|\theta_j\,w_{j}\right\|_{L^{2}}}\in  \Lambda^1H^1_T(\Omega) \cap   \Lambda^1 C^{\infty}\left (\overline \Omega\right ),
\end{equation}
where
$w_{j}$ is a $L^2$-normalized eigenform associated with the first eigenvalue $\lambda_{h}(\mathcal V_{j})$ of 
$\Delta_{f,h}^{(1),MD}(\mathcal V_{j})$ and
 $\theta_j$ is a smooth non negative cut-off function satisfying 
$\supp \theta_j\subset B(z_{j},2\ve_{1})\cap\overline\Omega \subset \mathcal V_j$,
$\{z_{j}\}=\argmin_{\supp\theta_{j}\cap\pa\Omega}f$,
and $\theta_j=1$ on $B(z_{j},\ve_{1})\cap\overline\Omega$.
\end{definition}

\noindent

Notice again that both $w_j$ and $-w_j$ can be used to build a quasi-mode and the choice of the sign is determined in Proposition~\ref{compa_wkb-bis2}. Notice also that the fact that $\widetilde\phi_{j}\in  \Lambda^1 C^{\infty}\left (\overline \Omega\right )$ follows from standard    elliptic regularity results. In addition,   for all $j\in  \{1,\dots,\ft m_1^{\pa \Omega} \big \}$, using~\eqref{eq.Agmon1} together with the fact that $\inf_{\text{supp }(1-\theta_j)\cap \mathcal V_j} \Phi_j>0$   (see~\eqref{eq.eikoj}), there exists $c>0$ such that    when $h\to 0$:
\begin{equation}\label{eq.norme-phij2}
\big \|(1-\theta_j)\,w_{j}\big \|_{L^2(\mathcal V_j)}=O\big(e^{-\frac ch}\big) \text{ and thus, } \,  \big \|\theta_j\,w_{j}\big \|_{L^2}=1+O\big(e^{-\frac ch}\big).
\end{equation}

Using~Proposition~\ref{pr.zj-omega-2} and~\eqref{eq.norme-phij2}, one deduces the following estimate on the quasi-mode $\widetilde \phi_j$ introduced in Definition~\ref{de.tildephi}.
\begin{corollary}
     Let us assume that  the assumption~\eqref{H-M} holds. Let  
$\widetilde\phi_{j}$ be the quasi-mode associated with $z_j\in U_1^{\pa \Omega}$ ($j\in  \{1,\dots,\ft m_1^{\pa \Omega} \big \}$), see Definition~\ref{de.zj-pa-omega}. 
Then, there exist $C>0$,~$c>0$ and $h_0>0$ such that for any  $h\in (0,h_0)$:
\begin{equation}
\label{eq.delta-tildephi'}
 \big  \|d_{f,h}\widetilde\phi_{j} \big \|_{L^{2}}\,+\, \big \|d^{*}_{f,h}\widetilde\phi_{j} \big \|_{L^{2}}\le C e^{-\frac ch}.
\end{equation}
\end{corollary}
\noindent
Let us now  give  the corresponding versions
of the WKB approximation estimates~\eqref{eq.wWKB10}--\eqref{eq.ch0} for the quasi-mode $\widetilde \phi_j$ introduced in Definition~\ref{de.zj-pa-omega}.
From \cite[Section~4.2]{HeNi1},  there exists
 a $C^{\infty}(\overline{\mathcal V_j})$ function $a_j(x,h)=\widetilde a_{j}(x)+O(h)$ with $a_j\equiv \widetilde a_{j} \equiv 1$ on $\partial \Omega \cap \overline{\mathcal V_{j}}$ such that  the $1$-form \label{page.u1jwkb2}
\begin{equation} \label{eq.WKBj}
u^{(1)}_{j,wkb}=d_{f,h}\,\left ( a_j(x,h)e^{-\frac{1}{h}\Phi_j}\right)=\big(\widetilde a_{j}\,d(f-\Phi_{j})+O(h)\big)e^{-\frac{1}{h}\Phi_j},
\end{equation} 
 satisfies 
\begin{equation} \label{eq.wWKB1}
\left\{
\begin{aligned}
\Delta_{f,h}^{(1)}u^{(1)}_{j,wkb} &=O\big (h^{\infty}\big )e^{-\frac{1}{h}\Phi_j}\
\text{in} \ 
\mathcal V_{j}\\
\mathbf{t}u^{(1)}_{j,wkb}&=0\ \text{on}\ 
\partial \Omega \cap \overline{\mathcal V_{j}}\\
\mathbf{t}d_{f,h}^{*}u^{(1)}_{j,wkb}&= 
O\big  (h^{\infty}\big )\, e^{-\frac{1}{h}\Phi_j} \ \text{on}\ 
\partial \Omega \cap \overline{\mathcal V_{j}}.
\end{aligned}\right.
\end{equation}

\noindent
Moreover, one has in the limit $h\to 0$ (see \cite[Section~4.2]{HeNi1}):
$$
\big \|\theta_j u^{(1)}_{j,wkb}\big \|_{L^2}=\frac{   \pi^{\frac{d-1}{4}} \, \sqrt{2\partial_nf(z_j)}    }{\big({\rm det \ Hess} f|_{ \partial \Omega}(z_j)\big)^{\frac 14}    } \, h^{\frac{d+1}{4}}\big(1+O(h)\big),
$$
where the remainder terms $O(h)$   admits a full asymptotic expansion in~$h$.
In the following proposition, $w_j$ and $ u^{(1)}_{j,wkb}$ are compared. 

\begin{proposition}\label{compa_wkb-bis2}
     Let us assume that  the assumption~\eqref{H-M} holds.  Let $w_j$ be a $L^2$-normalized eigenform  associated with the smallest eigenvalue~$\lambda_{h}(\mathcal V_{j})$ of $\Delta_{f,h}^{FD,(1)}(\mathcal V_{j})$ ($j\in  \{1,\dots,\ft m_1^{\pa \Omega} \big \}$). Then, there exists $h_0>0$ such that for all $h\in (0,h_0)$ one has:
\begin{equation}
\label{eq.a-priori-comparison}
\left\|\theta_j\big (w_j-c_j(h)u^{(1)}_{j,wkb}\big )\right\|_{H^{1}}=O\big  (h^{\infty}\big )
\end{equation}
where
$$
 c_j(h)^{-1}=\big \langle  w_j, \theta_j u^{(1)}_{j,wkb}\big \rangle_{L^2 }.
$$
In addition, up to replacing $w_j$ by~$-w_j$, one can assume that $c_j(h)\geq0$ for $h$ small enough and then,  in the limit $h\to 0$, one has\label{page.cj2}: 
\begin{equation}
\label{eq.ch}
c_j(h)=\frac{\big({\rm det \ Hess} f|_{ \partial \Omega}(z_j)\big)^{\frac 14}}{\pi^{\frac{d-1}{4}} \, \sqrt{2\partial_nf(z_j)} } \, h^{-\frac{d+1}{4}}\big(1+O(h)\big),
\end{equation}
where the remainder terms $O(h)$   admits a full asymptotic expansion in~$h$.
 \end{proposition}
 Proposition~\ref{compa_wkb-bis2} is proved exactly as Proposition~\ref{compa_wkb-bis}.\\

In conclusion, a family of $1$-forms 
$(\widetilde \phi_j)_{j\in \{1,\dots,\ft m_1^{\overline \Omega}\}}$ has been constructed in this section. Since \eqref{eq.epsilon1} guarantees that $\overline{B(z,2\ve_{1})}\cap \overline{B(z',2\ve_{1})}=\emptyset$ 
for all $z\neq z'\in \ft U_1^{\overline \Omega}$,
the family $(\widetilde \phi_{j})_{j\in\{1,\dots,\ft m_1^{\overline \Omega}\}}$ is orthonormal in $L^2(\Omega)$.  From now on, the parameter $\ve_{1}$  is fixed and  $\ve>0$ will be successively reduced a finite number of times in the following. \\

\noindent
\textbf{WKB approximation of the quasi-modes $(\widetilde \phi_j)_{j\in \{1,\dots,\ft m_1^{\overline \Omega}\}}$}.\\  
 For upcoming computations, one needs the following definition. 
\begin{definition} Let us assume that  the assumption \eqref{H-M}  is satisfied. 
For all $j\in \{1,\ldots,\ft m_1^{\overline \Omega}\}$, one defines: \label{page.phiwkb}
\begin{equation}\label{tildephiwkb}
\widetilde \phi_{j,wkb}:=c_j(h) \, \theta_j \, u_{j,wkb}^{(1)},
\end{equation}
where for $j\in  \{ 1,\dots,\ft m_1^{\pa \Omega} \big \}$,~$u_{j,wkb}^{(1)}$ satisfies~\eqref{eq.wWKB1} and $\theta_j$ is introduced in Definition~\ref{de.zj-pa-omega} and, for $j\in  \{\ft m_1^{\pa \Omega}+1,\dots,\ft m_1^{\overline \Omega} \big \}$,~$u_{j,wkb}^{(1)}$ satisfies~\eqref{eq.wWKB10} and $\theta_j$ is introduced in Definition~\ref{de.tildephi}.
\end{definition}

   From~\eqref{eq.norme-phij}, Proposition~\ref{compa_wkb-bis},~\eqref{eq.norme-phij2},  and~Proposition~\ref{compa_wkb-bis2}  one has the following lemma. 
   \begin{lemma}\label{le.comp-wkb}
   Let us assume that  the assumption \eqref{H-M}  is satisfied. 
For $j\in  \{ 1,\dots,\ft m_1^{\pa \Omega} \big \}$, let $\widetilde \phi_j$ be as defined in~\eqref{eq.tildephi'}, and for $j\in  \{\ft m_1^{\pa \Omega}+1,\dots,\ft m_1^{\overline \Omega} \big \}$, let $\widetilde \phi_j$ be as defined in~\eqref{eq.tildephi}. Moreover, for $j\in  \{1,\dots,\ft m_1^{\overline \Omega} \big \}$, let $\widetilde \phi_{j,wkb}$ be as defined in~\eqref{tildephiwkb}. Then, one has:
$$
 \big \Vert \widetilde \phi_j-\widetilde \phi_{j,wkb}    \big  \Vert_{H^1}=O\big (h^{\infty}\big ).
$$
   \end{lemma}

\subsubsection{Quasi-modes for $-L_{f,h}^{D,(p)}$,~$p\in\{0,1\}$} 
\label{sec.LOAA}
Before defining the quasi-modes for   $L^{D,(0)}_{f,h}$ and $L^{D,(1)}_{f,h}$, let us label the quasi-modes for~$\Delta^{D,(0)}_{f,h}$ and the local minima of~$f$ using the lexicographic  order.
\label{page.lexico}
\begin{definition}\label{de.label-l}
Let us assume that  the  assumption \eqref{H-M}  is satisfied. 
Then, the family of   {critical connected components} $(\ft E_{k,\ell} )_{k\geq 1,\,\ell\in \{1,\ldots,N_k\}}$ introduced in Section~\ref{sec:labeling},  the local minima $(x_{k,\ell} )_{k\geq 1,\,\ell\in \{1,\ldots,N_k\}}$ of~$f$ labeled in Section~\ref{sec:labeling}, the family of cut-off functions  $( \chi_{k,\ell}^{\ve,\ve_1})_{k,\ell} )_{k\geq 1,\,\ell\in \{1,\ldots,N_k\}}$  introduced  in Definition~\ref{de.v1} and  the family of quasi-modes $(\widetilde v_{k,\ell} )_{k\geq 1,\,\ell\in \{1,\ldots,N_k\}}$   introduced  in Definition~\ref{de.v1} are labeled in the lexicographic order:
$$(\ft E_k)_{k\in  \{1,\ldots,\ft m_0^{ \Omega} \}},\ (  \chi_{k}^{\ve,\ve_1})_{k\in  \{1,\ldots,\ft m_0^{ \Omega} \}}, \ (\widetilde v_{k})_{k\in  \{1,\ldots,\ft m_0^{ \Omega} \}}\ \text{and } \ (x_{k})_{k\in  \{1,\ldots,\ft m_0^{  \Omega} \}}.$$  
\end{definition}
Let us recall that the  lexicographic order is defined by~$(k',\ell')< (k,l)$ if and only if $k'< k$ or if $k'=k$,~$\ell'< \ell$. From now on, one uses the labeling introduced in Definition~\ref{de.label-l}.

According to~\eqref{eq.unitary}, the quasi-modes for $L^{D,(0)}_{f,h}$ and $L^{D,(1)}_{f,h}$ are obtained from those constructed previously for $\Delta^{D,(0)}_{f,h}$ and $\Delta^{D,(1)}_{f,h}$ using the unitary transformation $  U$ defined in~\eqref{U}.

\begin{definition}\label{de.qm-L}
Let us assume that  the assumption \eqref{H-M}  is satisfied. 
Let  $(\widetilde v_{k})_{k\in  \{1,\ldots,\ft m_0^{ \Omega} \}}$ be the family of quasi-modes for $\Delta^{D,(0)}_{f,h}$  introduced  in Definition~\ref{de.v1} (and labeled  in the lexicographic order, see Definition~\ref{de.label-l}) and let   $(\widetilde \phi_j)_{j\in \{1,\dots,\ft m_1^{\overline \Omega}\}}$   be the family of quasi-modes for $\Delta^{D,(1)}_{f,h}$  introduced  in Definitions~\ref{de.tildephi} and~\ref{de.zj-pa-omega}. The family of quasi-modes  $(\widetilde u_{k})_{k\in  \{1,\ldots,\ft m_0^{ \Omega} \}}$ for $-L^{D,(0)}_{f,h}$ and the  family of quasi-modes $(\widetilde \psi_j)_{j\in \{1,\dots,\ft m_1^{\overline \Omega}\}}$  for $-L^{D,(1)}_{f,h}$   are defined  by:  for $k\in \{1,\ldots,\ft m_0^\Omega\}$, and for $j\in \{1,\dots,\ft m_1^{\overline \Omega}\}$:\label{page.qm}
\begin{equation}
\label{eq.weighted-v-psi}
\widetilde u_{k} := e^{\frac 1h f}\, \widetilde v_{k} \in \Lambda^0 H^1_{w,T}\left (\Omega\right )
\quad\text{and} \quad 
\widetilde \psi_{j} :=  e^{\frac 1 h f} \, \widetilde\phi_{j}\in \Lambda^1 H^1_{w,T}\left (\Omega\right ).
\end{equation}
\end{definition}
Notice that, according to~\eqref{eq.v1} and~\eqref{eq.weighted-v-psi}, for all $k\in \{1,\ldots,\ft m_0^\Omega\}$,  
$$\widetilde u_{k} \in   C^{\infty}_c\left (\Omega\right ),$$
and according to~\eqref{eq.tildephi} and~\eqref{eq.tildephi'}, for all $j\in \{1,\dots,\ft m_1^{\overline \Omega}\}$,
$$\widetilde \psi_{j} \in \Lambda^1 C^{\infty}\left (\overline \Omega\right ).$$

\subsection{Bases of $\range   \pi_h^{(0)}$ and $\range   \pi_h^{(1)}$}
Let us recall, that from~\eqref{eq.proj-p}, 
$$
  \pi^{(0)}_h =\pi_{[0,\frac{\sqrt h}2)}\big (-L^{D,(0)}_{f,h}\big )\ \text{ and } \ \pi^{(1)}_h =\pi_{[0,\frac{\sqrt h}2)}\big (-L^{D,(1)}_{f,h}\big ).
$$
In this section, one proves that the quasi-modes introduced in  Definition~\ref{de.qm-L} form two bases of $\range   \pi_h^{(0)}$ and $\range   \pi_h^{(1)}$. In the following, the finite dimensional spaces $\Ran \, \pi^{(0)}_h$ and $\Ran \, \pi^{(1)}_h$ are endowed with the scalar product $ \lp  \cdot , \cdot    \rp_{L^2_w}$. 

\begin{proposition} \label{ESTIME1-base}
Let us assume that the assumption \eqref{H-M} holds. Let $(\widetilde u_{k})_{k\in  \{1,\ldots,\ft m_0^{ \Omega} \}}$ be the family of quasi-modes for $-L^{D,(0)}_{f,h}$ and let   $(\widetilde \psi_j)_{j\in \{1,\dots,\ft m_1^{\overline \Omega}\}}$   be the family of quasi-modes for $-L^{D,(1)}_{f,h}$  introduced  in Definition~\ref{de.qm-L}. 
 Then,

\begin{enumerate}[leftmargin=1.2cm,rightmargin=1.2cm]

\item For all $k\in \big \{1,\dots,{\ft m_{0}^{\Omega}} \big \}$ and $ j\in  \big\{1,\dots,\ft m_1^{\overline \Omega} \big\}$,~$\widetilde u_k\in \Lambda^0 H^1_{w,T}\left (\Omega\right )$,~$\widetilde \psi_j\in \Lambda^1 H^1_{w,T}\left (\Omega\right )$ and
$$ \big  \|\widetilde u_k \big  \| _{L^2_w}= \big  \|\widetilde \psi_j \big  \| _{L^2_w} = 1\ \ \text{and}\ \  
\forall i \in  \big \{1,\dots,\ft m_1^{\overline \Omega} \big  \}\setminus\{j\}\,,\ \ 
 \big \lp \widetilde \psi_j, \widetilde \psi_i \big \rp_{L^2_w}= 0\,.$$
%
\item\begin{itemize} 
\item[a)]  For any $\delta>0$, one can
choose the parameter $\ve$ in~\eqref{eq.v1} (see also \eqref{eq.weighted-v-psi})  small enough such that for all $k\in \{1,\dots,{\ft m_{0}^{\Omega}}\}$, in the limit $h\to 0$:
 $$  \big\|     (1-\pi_h^{(0)} ) \widetilde u_k \big  \|_{L^2_w}^2\ \leq\  h^{\frac12}\,  \big \|   \nabla  \widetilde u_k \big \|_{L^2_w}^2      =    O \left(e^{-\frac{2}{h}(f(\mathbf{j}(x_k))-f(x_k) - \delta)}\right )\,.$$
 In particular,  
choosing the parameter $\ve>0$ small enough  in~\eqref{eq.v1}, there exists $c>0$ such that  in the limit $h\to 0$:
  $$\pi_h^{(0)}\widetilde u_k = \widetilde u_k + O \big (e^{-\frac{c}{h}} \big) \ \text{  in }\, L^2_w(\Omega).$$
\item[b)]   There exist $c>0$ such that for all $ j\in \big \{1,\dots,\ft m_1^{\overline \Omega}\big \}$, one has in the limit $h\to 0$:
\begin{equation*}
   \big\|   (1-\pi_h^{(1)} ) \widetilde \psi_j  \big \|_{H^1_w}^2=   O \big (e^{-\frac{c}{h}} \big).
     \end{equation*} 
\end{itemize}

\item \begin{itemize} 
\item[a)]  The family $(  \widetilde u_k)_{k=1,\dots,{\ft m_{0}^{\Omega}}}$ is uniformly linearly independent (for the $L^2_w$-scalar product) for all $h$ sufficiently small 
(as defined in Lemma~\ref{le.indepp}). 
\item[b)] 
For all $(i,j) \in  \big \{1,\dots,\ft m_1^{\overline \Omega} \big  \}^2$,
$$\big\lp \pi_h^{(1)}  \widetilde  \psi_i,  \pi_h^{(1)}  \widetilde \psi_j  \big \rp_{L^2_w}= \delta_{i,j}+ O (e^{-\frac{c}{h}}).$$
 \end{itemize}
In particular, there exists $h_0>0$ such that for all $h\in (0,h_0)$:
$$ \range   \pi_h^{(0)}= {\rm Span}\big( \pi_h^{(0)}\widetilde u_k,\ k=1,\dots,{\ft m_{0}^{\Omega}}\big )$$
and
$$ \range   \pi_h^{(1)}= {\rm Span}\big ( \pi_h^{(1)}\widetilde \psi_i,\ i=1,\dots,\ft m_1^{\overline \Omega}\big ).$$

  \end{enumerate}
 \end{proposition}


\begin{proof}
The proof of Proposition~\ref{ESTIME1-base} is divided into two steps. 
\medskip

\noindent 
\textbf{Step 1}. Proofs of items 1 and 2. 
\medskip

\noindent
The first item  is
immediate according to the definition of the families
$(  \widetilde u_k)_{k=1,\dots,{\ft m_{0}^{\Omega}}}$ and $(\widetilde \psi_i)_{i=1,\dots,\ft m_1^{\overline \Omega}}$ introduced in Definition~\ref{de.qm-L}. 
\medskip

\noindent 
The first inequality appearing in 2a is a direct consequence of Lemma~\ref{quadra}
applied to $A=-L_{f,h}^{D,(0)} $ whose associated quadratic form
is given by~$\frac h2\lp\nabla \cdot ,\nabla \cdot\rp_{L^2_w}$ on $H^1_{w,T}(\Omega)$.
The second inequality in 2a   follows from   Laplace's methods and from the properties of the cut-off functions used to define the quasi-modes $ \widetilde u_k$ (see   Definition~\ref{de.qm-L} and Lemma~\ref{le.dfh-v}). Indeed,
it is just a rewriting of \eqref{eq.dv'} using Definition~\ref{de.qm-L} and the labeling introduced in Definition~\ref{de.label-l}.\medskip

\noindent
Let us now deal with  2b.
First, Lemma~\ref{quadra} together with~\eqref{eq.delta-tildephi} and~\eqref{eq.delta-tildephi'} implies the existence of some $c>0$ such that for all $i\in \{1,\dots, \ft m_1^{\overline \Omega}\}$ and $h$ small enough, 
\begin{equation}
\label{eq.estim-L2}
\left\|  \Big ( \  1-\pi_{[0,h^{\frac32} )} ( \Delta^{D,(1)}_{f,h} ) \Big) \widetilde \phi_i \right\|_{L^2}= O (e^{-\frac{c}{h}}).
\end{equation}
Consequently,
using again \eqref{eq.delta-tildephi} and \eqref{eq.delta-tildephi'}, and owing to the following
 relations on $\Lambda^{1}H^{1}_{T}(\Omega)$,
$$d_{f,h}\Big (1-\pi_{[0,h^{\frac32} )}  \big( \Delta^{D,(1)}_{f,h} \big ) \Big) =\Big (1-\pi_{[0,h^{\frac32} )} \big ( \Delta^{D,(2)}_{f,h} \big )\Big)  d_{f,h},\,~$$
$$
d^{*}_{f,h}\Big (1-\pi_{[0,h^{\frac32} )}  \big ( \Delta^{D,(1)}_{f,h} \big )\Big) =\Big (1-\pi_{[0,h^{\frac32} )}  \big ( \Delta^{D,(0)}_{f,h} \big  ) \Big ) d^{*}_{f,h},$$
$d^*_{f,h}=hd^*+\mbf i_{\nabla f}$, and $d_{f,h}=hd+\nabla f \wedge$, 
one obtains the existence of $c>0$ such that in the limit $h\to 0$:
\begin{equation}
\label{eq.estim-d-L2}
\left\| d \Big ( \  1-\pi_{[0,h^{\frac32} )} ( \Delta^{D,(1)}_{f,h} ) \Big) \widetilde \phi_i \right\|_{L^2}+\left\| d^* \Big ( \  1-\pi_{[0,h^{\frac32} )} ( \Delta^{D,(1)}_{f,h} ) \Big) \widetilde \phi_i \right\|_{L^2}= O (e^{-\frac{c}{h}}).
\end{equation}
Since $\widetilde \phi_{i}\in \Lambda^{1}H^{1}_{T}(\Omega)$, the estimates \eqref{eq.estim-L2} and \eqref{eq.estim-d-L2}  then lead,
owing to Gaffney's inequality (see \cite[Corollary 2.1.6]{GSchw}),
 to
$$
\left\|  \Big ( \  1-\pi_{[0,h^{\frac32} )} ( \Delta^{D,(1)}_{f,h} ) \Big) \widetilde \phi_i \right\|_{H^1}= O(e^{-\frac{c}{h}} ).
$$
Therefore, we deduce from the relation
$\|u\|_{H^1_w} \le \frac{C}{h} \|u \, e^{-\frac{1}{h}f}\|_{H^1}$,
valid for all $u \in \Lambda^p H^1(\Omega)$ and $h>0$, and from
$$
\pi_{[0,h^{\frac32} )}  \big( \Delta^{D,(1)}_{f,h} \big)= e^{-\frac{1}{h}f} \pi_h^{(1)}e^{\frac{1}{h}f},
$$
resulting from \eqref{eq.unitary} and \eqref{eq.proj-p},
that
there exists $c>0$ such that for all $i\in \{1,\dots, \ft m_1^{\overline \Omega}\}$ and $h$ small enough, 
$$\left \|     (1-\pi_h^{(1)}) \widetilde \psi_i\right \|_{H^1_w} \le \frac{C}{h} \left\|  \Big( \  1-\pi_{[0,h^{\frac32} )} ( \Delta^{D,(1)}_{f,h} ) \Big) \widetilde \phi_i \right\|_{H^1}=O (e^{-\frac{c}{h}}).$$
This ends the  proof of 2b. 

\medskip

\noindent
\textbf{Step 2}. Proof of item 3.
\medskip

\noindent
The fact that the family $(  \widetilde u_k)_{k=1,\dots,{\ft m_{0}^{\Omega}}}$ is   uniformly linearly independent is a consequence of  Lemma~\ref{le.indepp} together with~\eqref{eq.weighted-v-psi}. 
Item 3b follows from items~1 and~2.b together with the relation
\begin{equation}\label{fg}
  \big \lp \pi_h^{(\ell)}  f , \pi_h^{(\ell)} g   \big \rp_{L^2_w} = -  \big \lp (\pi_h^{(\ell)}  -1)f ,   (\pi_h^{(\ell)}   -1) g   \big\rp_{L^2_w} +  \big \lp f,  g  \big \rp_{L^2_w}
 \end{equation}
holding for $f,g$ in~$\Lambda^{\ell}L^2_w(\Omega)$ and $\ell\in\{0,1\}$. Finally, the fact that for $h$ small enough, $ \range   \pi_h^{(0)}= {\rm Span}\big( \pi_h^{(0)}\widetilde u_k,\ k=1,\dots,{\ft m_{0}^{\Omega}}\big )$
and $ \range   \pi_h^{(1)}= {\rm Span}\big ( \pi_h^{(1)}\widetilde \psi_i,\ i=1,\dots,\ft m_1^{\overline \Omega}\big )$ are consequences of items 2a, 3a and 3b together with Lemma~\ref{ran1}. 
  \end{proof}

\section{On the smallest eigenvalue of $-L_{f,h}^{D,(0)}$}
\label{section-3}

This section is dedicated to the proof of the following theorem. 

\begin{theorem} \label{thm-big0}
Assume that the assumptions \eqref{H-M} and~\eqref{eq.hip1-j} are satisfied. Let~$\lambda_h$ be the principal eigenvalue of $-L_{f,h}^{D,(0)}$ (see~\eqref{eq.lh}). Then, denoting by~$\lambda_{2,h}$ the second smallest eigenvalue of  $-L_{f,h}^{D,(0)}$, there exists $c>0$ such that in the limit $h\to 0$:\label{page.lambda2h}
\begin{equation}\label{l2}
 \lambda_h= \lambda_{2,h}\ O(e^{-\frac ch}).
\end{equation}
Moreover, when \eqref{eq.hip2-j} is satisfied,  one has  in the limit $h\to 0$:
\begin{equation}\label{eq.lambda_h}
\lambda_h=  \frac{\displaystyle \sum \limits_{z\in \pa \ft C_1\cap \pa \Omega}   \partial_nf(z)  \big({\rm det \ Hess } f|_{\partial \Omega}   (z) \big)^{-\frac12} } {\sqrt{\pi\,h}\,\sum\limits_{x\in \argmin_{\ft C_1}f}  \big( {\rm det \ Hess } f   (x)   \big)^{-\frac12}  }\,e^{-\frac{2}{h}\big (f(\mbf j(x_{1}))-f(x_1)\big )}\big ( 1+ O(\sqrt{h}) \big )
\end{equation}
where we recall that   $x_1\in \argmin_{\ft C_1}f$. 
Finally, when \eqref{eq.hip4-j}  holds, the remainder term $O(\sqrt{h})$ in \eqref{eq.lambda_h}
is actually of order~$O(h)$ and admits a full asymptotic expansion in~$h$.
 \end{theorem}
 
\begin{remark}
Without the assumption \eqref{eq.hip4-j}, we  are not able to prove  an asymptotic expansion in~$\sqrt h$ of the remainder term  $O(\sqrt h)$    in \eqref{eq.lambda_h} except in   some specific cases, see Theorem~\ref{re-specific-case} below or  \cite[Proposition C.40]{BN2017}. 
\end{remark}
\noindent
Let us mention that sharp asymptotic estimates when $h\to 0$ of the principal eigenvalue of $-L_{f,h}^{D,(0)}$ have been obtained in \cite{HeNi1, LeNi,di-gesu-le-peutrec-lelievre-nectoux-16} in the Dirichlet case and in  \cite{Lep} in the Neumann case. However, these results do not apply under the assumptions considered  in Theorem~\ref{thm-big0}. Let us also mention that when  $\Omega=\mathbb R^d$ or when $\Omega$ is a compact Riemannian manifold, sharp asymptotic estimates of the second smallest eigenvalue of~$-L_{f,h}^{(0)}$ have been obtained in~\cite{BGK,HKN,michel2017small,landim2017dirichlet,HeHiSj,HKS,miclo-95}. 
\medskip
 
\noindent
The analysis led in this section will also allow us to give a   lower and an upper bound for all the $\ft m_0^\Omega$ small eigenvalues of   $-L^{D,(0)}_{f,h}$ (and not only $\lambda_h$). This is the purpose of  Theorem~\ref{pp} below. 

\begin{remark}\label{eq.expectation}
Combining Theorem~\ref{thm-big0} and Proposition~\ref{indep1}, under the assumptions \eqref{H-M}, \eqref{eq.hip1-j} and \eqref{eq.hip2-j}, one obtains   that in the limit $h\to 0$: 
\begin{align*}
 \mathbb E_{\nu_h}[\tau_{\Omega}]&=\frac{1}{\lambda_h} =\frac{  \sqrt{\pi\,h}\!\!\! \sum\limits_{x\in \argmin_{\ft C_1}f}  \big( {\rm det \ Hess } f   (x)   \big)^{-\frac12} }{\sum\limits_{z\in \pa \ft C_1\cap \pa \Omega} \partial_nf(z) \big({\rm det \ Hess } f|_{\partial \Omega}   (z) \big)^{-\frac12}}\,e^{\frac{2}{h}(f(\mbf j(x_{1}))-f(x_1))}\big ( 1+ O(\sqrt{h}) \big ).
\end{align*}
 
\end{remark}

\noindent
In some specific cases, one can  drop the assumption~\eqref{eq.hip2-j} in Theorem~\ref{thm-big0}   and still obtain a sharp asymptotic equivalent of $\lambda_h$ when $h\to 0$. Indeed, in view of the proof of Theorem~\ref{thm-big0}, one has the following result.

 \begin{theorem}\label{re-specific-case}
Assume that the assumptions \eqref{H-M} and \eqref{eq.hip1-j}  are satisfied. Assume moreover that  for all $j\in \{2,\ldots,\ft N_1\}$, $\pa \ft C_1\cap \pa \ft C_j=\emptyset$ (this last assumption is for instance satisfied when $\ft N_1=1$). Let us define,
$$\text{when } \pa \ft C_1\cap \pa \Omega\neq \emptyset, \, a_1:=\frac{\displaystyle \sum \limits_{z\in \pa \ft C_1\cap \pa \Omega}   \partial_nf(z)  \big({\rm det \ Hess } f|_{\partial \Omega}   (z) \big)^{-\frac12} } {\sqrt{\pi}\,\sum\limits_{x\in \argmin_{\ft C_1}f}  \big( {\rm det \ Hess } f   (x)   \big)^{-\frac12}  }, \text{ else, } a_1:=0,$$
and
$$\text{when } \pa \ft C_1\cap \ft U^{\ft{ssp}}_1\cap \Omega\neq \emptyset,\,  a_2:=\frac{1}{2\pi}\frac{\displaystyle \sum \limits_{z\in \pa \ft C_1\cap \ft U^{\ft{ssp}}_1\cap \Omega}   |\lambda_{-}(z)|\big({\rm det \ Hess } f    (z) \big)^{-\frac12} } {\sum\limits_{x\in \argmin_{\ft C_1}f}  \big( {\rm det \ Hess } f   (x)   \big)^{-\frac12}  }, \text{ else, } a_2:=0, $$
where $\lambda_{-}(z_{j})$ is the negative eigenvalue of $\Hess f(z)$ (notice that  $a_1$ and $a_2$ cannot be both equal to $0$ since from  Proposition~\ref{pr.p1}, $\pa  \ft C_{1} \cap \ft U_1^{\ft{ssp}}\neq \emptyset$).
Then, one has when $h\to 0$:
$$\lambda_h=\left[\frac{a_1}{\sqrt{h}} \, (1+O(h))\,+\, a_2\, (1+O(h))\right]\,e^{-\frac{2}{h}\big (f(\mbf j(x_{1}))-f(x_1)\big )},$$
where the two remainder terms  $O(h)$ admit a full asymptotic expansion in $h$. 
 \end{theorem}

This section is organized as follows. In Section~\ref{qm-propo}, one gives   the quasi-modal estimates  which are used to prove Theorem~\ref{thm-big0}. 
Section~\ref{sec:S} is then dedicated to the proof of Theorem~\ref{thm-big0}.


\subsection{Estimates of interactions between quasi-modes}
\label{qm-propo}
The main result of this section is Proposition~\ref{ESTIME1} which  gives the quasi-modal estimates in   $L^2_w(\Omega)$ needed  to prove Theorem~\ref{thm-big0}. This section is divided into two parts. In Section~\ref{sec.p1}, one gives the asymptotic estimates of the boundary terms $$\Big (\displaystyle \int_{\Sigma}   F \,  \widetilde \psi_i \cdot n  \,   e^{- \frac{2}{h} f}\Big)_{j\in\{1,\dots,\ft m_1^{\overline \Omega}  \}},$$ which are then used in the proof of Proposition~\ref{ESTIME1}. In Section~\ref{sec.p2}, one states and proves Proposition~\ref{ESTIME1}. 

For all  $j\in  \big  \{1,\dots,\ft m_1^{\overline \Omega} \big  \}$, let us define the constant
  \label{page.bj}
\begin{equation}
\label{eq.Bi}
B_{j} = \begin{cases}\pi ^{\frac{d-1}{4}} \sqrt{2\, \partial_n f(z_j) }  \,  \big( {\rm det \ Hess }f|_{ \partial \Omega}(z_j)   \big)^{-1/4}   & \text{if $z_{j}\in \pa\Omega$},\\
\pi ^{\frac{d-2}{4}} \sqrt{|\lambda_{-}(z_{j})| }  \,  | {\rm det \ Hess }f(z_j)   |^{-1/4}
& \text{if $z_{j}\in \Omega$},
\end{cases}
\end{equation}
where $\lambda_{-}(z_{j})$ is the negative eigenvalue of $\Hess f(z_{j})$. These constants will appear in the upcoming computations.

\subsubsection{Asymptotic estimates of boundary terms for   $(\widetilde \psi_i )_{i\in \{1,\ldots,\ft m_1^{\overline \Omega} \}}$} 
\label{sec.p1}
The following boundary estimates will be used several times in the sequel. 

\begin{proposition}\label{gamma1}  
Let us assume that the assumption \eqref{H-M} is satisfied. Let us consider  
$j\in\{1,\dots,\ft m_1^{\overline \Omega}\}$, 
an open set $\Sigma$ of $\pa\Omega$, 
 and  $F\in L^{\infty}(\partial \Omega,\mathbb R)$. Then, there exists $c>0$ such that one has  in the limit $h\to 0$:
\begin{equation*}
  \int_{\Sigma}   F \,  \widetilde \psi_j \cdot n  \,   e^{- \frac{2}{h} f}  \     =\begin{cases}  0   &   \text{ if } j\in  \big \{\ft m_1^{\pa \Omega}+1,\dots,\ft m_1^{\overline \Omega} \big \}, \\
O \big(e^{-\frac{1}{h} (f(z_j)+c)} \big)   &  \text{ if } j\in  \big \{1,\dots,\ft m_1^{\pa \Omega} \big \} \text{ and } z_{j}\notin\overline\Sigma, \\
O \big(h^{\frac{d-3}{4}}     e^{-\frac{1}{h} f(z_j)} \big)  &  \text{ if } j\in  \big \{1,\dots,\ft m_1^{\pa \Omega} \big \} \text{ and } z_{j}\in \Sigma ,
  \end{cases} 
  \end{equation*}
 where $\widetilde \psi_j$ is introduced in~\eqref{eq.weighted-v-psi}   and $m_1^{\pa \Omega}$ is defined  in~\eqref{eq.m1-pa}.
 Moreover, when $j\in  \big \{1,\dots,\ft m_1^{\pa\Omega} \big \}$,~$z_{j}\in \Sigma $, and  $F$ is  $C^{\infty}$ in a neighborhood  of $z_j$, it holds
  \begin{equation*}
  \int_{\Sigma}   F \,  \widetilde \psi_j \cdot n  \   e^{- \frac{2}{h} f}  \     =  h^{\frac{d-3}{4}}   \,     e^{-\frac{1}{h} f(z_j)}  \,    \big(   B_{j} \, F(z_j)   +     O(h )    \big) ,
  \end{equation*}
  where the remainder terms $O(h)$   admits a full asymptotic expansion in~$h$  (as defined in Remark~\ref{eq.asymptoO(h)}), and $B_j$ is defined by~\eqref{eq.Bi}
  \end{proposition}

 \begin{proof} 
 Let $F\in L^{\infty}(\partial \Omega,\mathbb R)$.
From~\eqref{eq.weighted-v-psi} and~\eqref{eq.tildephi}, the quasi-mode $\widetilde \psi_{j}$ is supported in~$\Omega$
for $j\in \{\ft m_1^{\pa \Omega}+1\dots \ft m_1^{\overline \Omega}\}$ and thus:
\begin{equation}
\label{eq.int-nulle}
\forall j\in \big \{\ft m_1^{\pa \Omega}+1,\dots, \ft m_1^{\overline \Omega} \big \}, \int_{\Sigma}   F \,  \widetilde \psi_j \cdot n  \   e^{- \frac{2}{h} f}  \    =0.
\end{equation}
Let us now consider  the case $j\in \{1,\dots, \ft m_1^{\pa \Omega}\}$. Notice that one has for all $h$ small enough, from the trace theorem, \eqref{eq.weighted-v-psi}, \eqref{eq.tildephi'},  and \eqref{eq.Agmon1},
\begin{align}
\nonumber
\int _{\Sigma} F \,   \widetilde \psi_j \cdot n  \   e^{- \frac{2}{h}f}   &=\int _{\supp \theta_{j}\cap\Sigma}  F  \, \widetilde \phi_j \cdot n  \   e^{- \frac{1}{h}f}   \\
\nonumber
&= O(\|\widetilde \phi_j\|_{H^{1}})
\Big(\int _{\supp \theta_{j}\cap\Sigma}   e^{- \frac{2}{h}f}   \Big)^{\frac12}\\
\label{eq.int-nulle2}
&=O \big(h^{-p} \big )
\Big(\int _{\supp \theta_{j}\cap\Sigma}   e^{- \frac{2}{h}f }   \Big)^{\frac12},
\end{align}
where $p$ is independent of $h$.
Therefore, since $z_j$ is the only minimum of~$f$ on ${\rm supp} \,\theta_j \cap \partial \Omega$, if $z_{j}\notin\overline\Sigma$, 
one has in the limit $h\to 0$:
\begin{equation}
\label{eq.int-nulle3}
\int _{\Sigma} F \,   \widetilde \psi_j \cdot n  \   e^{- \frac{2}{h}f}     = 
O \big (e^{-\frac{1}{h} (f(z_j)+c)} \big )   
\end{equation}
for some $c>0$ independent of $h$.\medskip

\noindent Let us now now consider  the case $j\in \{1,\dots, \ft m_1^{\pa \Omega}\}$ and $z_j\in \Sigma$. One has:
\begin{align}
\nonumber
\int _{\Sigma} F \,   \widetilde \psi_j \cdot n  \   e^{- \frac{2}{h}f }   &=\int _{\Sigma}  F \, \widetilde \phi_j \cdot n  \   e^{- \frac{1}{h}f }   
\\
\label{eq.wkb-dec}
&= \int _{\Sigma}   F\,  \widetilde \phi_{j,wkb} \cdot n  \ e^{- \frac{1}{h}f}   +\int _{\Sigma} F \,  \big  (\widetilde \phi_i-\widetilde \phi_{j,wkb}  \big ) \cdot n \  e^{- \frac{1}{h}f}  , 
\end{align}
where $\widetilde \phi_{j,wkb}=c_j(h)\theta_{j} u_{j,wkb}^{(1)}$ is defined in \eqref{tildephiwkb}.
From \eqref{eq.WKBj}, let us recall that  in the limit $h\to 0$:
$$u_{j,wkb}^{(1)}=e^{-\frac{1}{h}\Phi_j}\, \big(\widetilde a_j\,d(f-\Phi_j)+O(h)\big) \ \text{on }  \supp\theta_{j}$$
with  $\widetilde a_{j}= 1$ on $\pa\Omega\cap \supp\theta_{j}$. Thus on $\pa\Omega\cap\supp\theta_{j}$, using also \eqref{eq.eikoj},
$$n\cdot u_{j,wkb}^{(1)}=e^{-\frac{1}{h}\Phi_j} \, \pa_{n}(f-\Phi_j) \, \big(1+O(h)\big)=2\partial_n f\,e^{-\frac{1}{h}(f-f(z_j))} \, \big(1+O(h)\big).$$
Thus, the term $\displaystyle \int _{\Sigma}  F\,   \widetilde \phi_{j,wkb} \cdot n  \  e^{- \frac{1}{h}f}   $ appearing in the right-hand side of \eqref{eq.wkb-dec}  satisfies in the limit $h\to 0$:
\begin{align}
\nonumber
\int _{\Sigma}  F\,   \widetilde \phi_{j,wkb} \cdot n  \  e^{- \frac{1}{h}f}   
&=c_j(h)\int _{\Sigma\cap \supp\theta_{j}}  F\,\theta_j   \, u_{j,wkb}^{(1)} \cdot n \  e^{- \frac{1}{h}f}  
\\
\label{eq.int-laplace0}
&=c_j(h)\int _{\Sigma\cap \supp\theta_{j}}  2\partial_nf\,F\,\theta_j \, e^{-\frac{1}{h}(2f-f(z_j))}   \  \big(1+O(h)\big)
\\
\nonumber
&=O\Big(c_j(h)\int _{\pa \Omega \cap \supp\theta_{j}}  2\partial_nf\,\theta_j \, e^{-\frac{1}{h}(2f-f(z_j))}   \Big)\\
\label{eq.int-laplace1}
&=O\Big(c_j(h)\,  h ^{\frac{d-1}{2} }\, e^{- \frac{1}{h}f(z_j)}\Big)=O\big(  h ^{\frac{d-3}{4} }\, e^{- \frac{1}{h}f(z_j)}\big),
\end{align} 
where the last line follows from $\{z_{j}\}=\argmin_{\pa \Omega \cap \supp\theta_{j} }f$,  $\theta_{j}(z_{j})=1$, Laplace's method, and $c_j(h)=O(h ^{-\frac{d+1}{4}})$ according to \eqref{eq.ch}. When 
 $z_{j}\in \Sigma$ and  $F$ is  $C^{\infty}$ in a neighborhood  of $z_j$, the same arguments applied to  
 \eqref{eq.int-laplace0} yield,  in the limit $h\to 0$:
\begin{align}
\label{eq.int-laplace2}
\int _{\Sigma}  F\,   \widetilde \phi_{j,wkb} \cdot n  \  e^{- \frac{1}{h}f}   
= c_j(h)\,F(z_j)\, \frac{2\,\partial_nf(z_j)\, \pi ^{\frac{d-1}{2} }} {  \sqrt{ {\rm  det \ Hess }f|_{ \partial \Omega}(z_j)  }}  h ^{\frac{d-1}{2} }\ e^{- \frac{1}{h}f(z_j)}  \big(1+O(h)\big),
\end{align}
 where the remainder terms $O(h)$   admits a full asymptotic expansion in~$h$ (which follows from  Laplace's method).
 \medskip

\noindent
Besides, from the trace theorem and Lemma~\ref{le.comp-wkb},  
the second term in the right-hand side  of \eqref{eq.wkb-dec} satisfies in the limit $h\to 0$:
\begin{equation}
\label{eq.int-O2}
 \int_{\Sigma} F\,\big (\widetilde \phi_j-\widetilde \phi_{j,wkb} \big  ) \cdot n \ e^{- \frac{1}{h}f}  =O\big (h^{\infty}\big ) \ e^{- \frac{1}{h}f(z_j)}.
 \end{equation}
The first part of Proposition~\ref{gamma1} then results from 
\eqref{eq.int-nulle}--\eqref{eq.wkb-dec}, \eqref{eq.int-laplace1}, and 
\eqref{eq.int-O2}, and its second part from
 \eqref{eq.wkb-dec}, \eqref{eq.int-laplace2}--\eqref{eq.int-O2}, and from
the  asymptotic estimate 
of $c_j(h)$ given in~\eqref{eq.ch} 
 which yields, when $h\to0$
$$\int _{\Sigma}  F\,   \widetilde \psi_{j} \cdot n  \, e^{- \frac{2}{h}f}   =  \frac{ \pi ^{\frac{d-1}{4}} \sqrt{2\,\partial_n f(z_j) }   }{ \big( {\rm det \ Hess }f|_{ \partial \Omega}(z_j)   \big)^{1/4}  } \  h^{\frac{d-3}{4}} \ e^{- \frac{1}{h}f(z_j)} \big( F(z_j)+O(h)\big). $$
This ends the proof of Proposition~\ref{gamma1}.
\end{proof}

\subsubsection{Quasi-modal estimates in  $L^2_w(\Omega)$}
\label{sec.p2}

We are now in position to prove  Proposition~\ref{ESTIME1} which will be crucial to prove Theorem~\ref{thm-big0}. This proposition  allows indeed to study accurately the small singular values of the restricted differential $\nabla :\, \Ran \, \pi^{(0)}_h\to \Ran \, \pi^{(1)}_h$. The square of the smallest singular values is indeed~
 $\frac 2h \lambda_h$, where $\lambda_h$ is the principal eigenvalue of $-L^{D,(0)}_{f,h}$ (see \eqref{eq.lh} and Proposition~\ref{fried}). 

\begin{proposition} \label{ESTIME1}
Let us assume that the assumption \eqref{H-M} holds. Let $(\widetilde u_{k})_{k\in  \{1,\ldots,\ft m_0^{ \Omega} \}}$ be the family of quasi-modes for $-L^{D,(0)}_{f,h}$ and let   $(\widetilde \psi_j)_{j\in \{1,\dots,\ft m_1^{\overline \Omega}\}}$   be the family of quasi-modes for $-L^{D,(1)}_{f,h}$  introduced  in Definition~\ref{de.qm-L}. Then, there exists $\ve_0>0$ such that   for  for all $\ve\in (0,\ve_0)$, for all $k\in   \big\{1,\dots,{\ft m_{0}^{\Omega}}   \big\}$ and $j\in  \big  \{1,\dots,\ft m_1^{\overline \Omega}  \big\}$, there exists  $\ve_{j,k}\in \{-1,1\}$  independent of $h$ such that in the limit $h\to 0$,\label{page.pjk}
\begin{equation*}
  \lp       \nabla \widetilde u_k ,    \widetilde \psi_j \rp_{L^2_w} =\begin{cases} \ve_{j,k} C_{j,k} \ h^{p_{j,k}}\,  e^{-\frac{1}{h}(f(\mathbf{j}(x_{k}))- f(x_k))}   \    \big(  1  +     O(h )   \big) \!\!\!  \!\!\! &  \text{ if } z_j \in  \mathbf{j}(x_k)  ,\\
 0   &   \text{ if } z_j\notin \mathbf{j}(x_k),
  \end{cases} 
  \end{equation*}
  where  all the remainder terms $O(h)$   admits a full asymptotic expansions in~$h$ (as defined in Remark~\ref{eq.asymptoO(h)}), 
  \begin{equation*}p_{j,k}=
  \begin{cases} -\frac12 &\text{ if } z_{j}\in \mathbf{j}(x_k)\cap \Omega,\\ 
  -\frac34 &\text{ if } z_{j}\in \mathbf{j}(x_k)\cap \pa\Omega.
  \end{cases}
  \end{equation*}
and  
  \begin{equation}
\label{eq.Cip}
C_{j,k}= \frac{B_{j}\,  \pi^{-\frac{d}{4}}}{\Big( \sum\limits_{x\in \argmin_{\ft E_{k}}f}  \big( {\rm det \ Hess } f   (x)   \big)^{-\frac12}  \Big)^{1/2}},
\end{equation}
where the constant $B_j$ is defined in~\eqref{eq.Bi}, and $(\ft E_k)_{k\in   \{1,\dots, \ft m_{0}^{\Omega}   \}  }$ is defined  in Section~\ref{sec:labeling} and labeled in Definition~\ref{de.label-l}. Finally, if $z_j\in \mbf j(x_k)\cap \pa \Omega$ (and thus, it holds  necessarily  $k\in \{1,\ldots,\ft N_1\}$, see~\eqref{eq.j-x} where $\ft N_1$ is defined by~\eqref{eq.N11}), one has $$\ve_{j,k}=-1.$$
 \end{proposition}
\begin{proof}
The proof of Proposition~\ref{ESTIME1} is divided into three steps. 
\medskip

\noindent
\textbf{Step 1.}
\medskip

\noindent
Let $k\in \big \{   1,\ldots,\ft m_0^\Omega \big \}$ and  $j\in  \big  \{1,\dots,\ft m_1^{\overline \Omega} \big  \}$. 
Let us consider the case $z_j \notin \mathbf{j}(x_{k})$. According to Definitions~\ref{de.tildephi}, \ref{de.zj-pa-omega} and \ref{de.qm-L}, one has that for all $j\in \{1,\dots,\ft m_1^{\overline \Omega}\}$, the quasi-mode $\widetilde \psi_j$ is supported in  $B(z_j,2\ve_{1})\cap \overline \Omega$. Therefore, from~\eqref{eq.v1},~\eqref{eq.v1-support4}, and Definition~\ref{de.qm-L}, one has 
$\supp \nabla \chi_{k}^{\ve,\ve_1}\cap \overline{B(z_{j},2\ve_{1})}=\emptyset$ and thus
$$ \big \lp       \nabla \widetilde u_k ,    \widetilde \psi_j  \big \rp_{L^2_w}=0.$$
\medskip  

\noindent
\textbf{Step 2.}
\medskip

\noindent
 Let us now deal with the computation of the terms $ \lp       \nabla \widetilde u_k ,    \widetilde \psi_j\rp_{L^2_w}$ for $k\in \big \{   1,\ldots,\ft m_0^\Omega \big \}$ and  $j\in  \big  \{1,\dots,\ft m_1^{\overline \Omega} \big  \}$ such that $z_j\in \mbf j(x_k)\cap \Omega$. In this case, these computations  
 follow from the analysis led in the proof of \cite[Proposition 6.4]{HKN}, the
only difference arising from the fact~$\mathbf{j}(x_{k})$ and $\argmin_{\supp \chi^{\ve,\ve_1}_{{k}}} f$ were both reduced to one single point there. Let us give a proof
 for the sake of 
 completeness.
 One has:
\begin{equation}
\label{eq.interaction'}
 \big \lp       \nabla \widetilde u_k ,    \widetilde \psi_j  \big \rp_{L^2_w}= \int_{\Omega} \nabla  \widetilde u_k  \cdot \widetilde \phi_j \  e^{- \frac{1}{h}f}   = \frac{ \displaystyle{\int_{\Omega} \nabla  \chi_{k}^{\ve,\ve_1}\cdot \widetilde \phi_j \  e^{- \frac{1}{h}f}}  } { \sqrt{\displaystyle{\int_{\Omega} (\chi_{k}^{\ve,\ve_1})^2 e^{-\frac{2}{h}  f} } }    }=\frac{\displaystyle{ \int_{B(z_j,2\ve_{1})} \nabla  \chi_{k}^{\ve,\ve_1} \cdot \widetilde \phi_j \  e^{- \frac{1}{h}f}}  } { \sqrt{\displaystyle{\int_{\Omega} (\chi_{k}^{\ve,\ve_1})^2 e^{-\frac{2}{h}  f} } }    }.
\end{equation}
From~\eqref{eq.v1-support1}--\eqref{eq.v1-support2}, it holds $\argmin_{\supp \chi^{\ve,\ve_1}_{ {k}}} f=\argmin_{\ft E_k} f$. Thus, using Laplace's method together with the fact that $\min_{\overline{\ft E_k}} f=f(x_k)$,  one has in the limit $h\to 0$:
\begin{equation}
\label{eq.Lapl-u-k}
\int_{\Omega} (\chi_{k}^{\ve,\ve_1})^2 e^{-\frac{2}{h}  f} =(\pi \, h)^{\frac{d}{2} } e^{-\frac{2}{h}  f(x_{k})}\!\!\!\!\!\!\!\!
 \sum_{x\in\argmin_{\ft E_k }f} 
\big( \det \Hess  f   (x)   \big)^{-\frac12} \ \big(1+O(h) \big)\,.
\end{equation} 
Let us now give the estimate of the numerator of the right-hand side of \eqref{eq.interaction'}.\medskip

\noindent
To prepare this computation, let
us first recall that the set $B(z_j,2\ve_{1})\cap\{f<f(z_j)\}$ has, according to \eqref{eq.epsilon1'}, two connected components. Since $z_j\in \ft U_1^{\ft{ssp}}$ (see  Definition~\ref{de.SSP}),
exactly one of these two connected components intersects -- and is then included in --  the {critical connected component} $\widetilde{\mathbf j}(x_{k})=\ft E_{k}$ associated with $x_k$ (see Definition~\ref{de.SSP} and~\eqref{de.j}). Moreover, the set $B(z_j,2\ve_{1})\setminus \mathcal W_{+}(z_j)$, where the stable manifold $\mathcal W_{+}(z_j)$ has been defined in~\eqref{eq.W+}, has 
also
two connected components and one of them contains the connected 
component of $B(z_j,2\ve_{1})\cap\{f<f(z_j)\}$ which intersects  $\ft E_k$, namely  $B(z_j,2\ve_{1})\cap \ft E_k$.
Let us denote by 
$$
\ft B_{j,k}  \text{ the connected component of } B(z_j,2\ve_{1})\setminus \mathcal W_{+}(z_j) \text{ which contains } \ft E_k.
$$ 
Since $\supp \widetilde \phi_j \subset B(z_j,2\ve_{1})$, one has using in addition~\eqref{eq.v1-support3}:
$$\supp \big (\nabla  \chi_{k}^{\ve,\ve_1} \cdot\widetilde \phi_j\big )\subset B(z_j,2\ve_{1})\cap \ft E_k \subset \ft B_{j,k}.$$
Therefore, using   an integration by parts, it holds:
\begin{align}
\nonumber
 \int_{\Omega} \nabla  \chi_{k}^{\ve,\ve_1} \cdot \widetilde \phi_j \  e^{- \frac{1}{h}f} &= - \int_{\ft B_{j,k}  } \! \nabla(1- \chi_{k}^{\ve,\ve_1}) \cdot \widetilde \phi_j \  e^{-\frac 1h f}   \\
\nonumber
&= -  \int_{\ft B_{j,k}  } \! \big (1-\chi_{k}^{\ve,\ve_1}\big )  \, d^*\big( e^{-\frac 1h f}\,  \widetilde \phi_j \big)
- \int_{\partial \ft B_{j,k}   }   \!\!\!\big (1-\chi_{k}^{\ve,\ve_1}\big )\widetilde \phi_j\cdot n \  e^{-\frac 1h f}  \\
\label{eq.interaction'''}
&= - \frac1h\int_{\ft B_{j,k}  }  \big (1-\chi_{k}^{\ve,\ve_1}\big )  \, e^{-\frac 1h f}\, d_{f,h}^*  \widetilde \phi_j - \int_{\partial \ft B_{j,k}  \cap\mathcal W_{+}(z_j)}  \widetilde \phi_j\cdot n \  e^{-\frac 1h f},
\end{align}
since $\widetilde \phi_j=0$ on $\pa B(z_j,2\ve_1)$. 
From \eqref{eq.delta-tildephi}, it holds for $h$ small enough,  
$$d_{f,h}^* \widetilde \phi_j =O( e^{-\frac{c}{h}})\text{ in }  L^2(\Omega),$$
where $c>0$ is independent of $h$. 
 Since moreover $f\geq f(z_j)-2\ve$ on $\ft B_{j,k} \cap \supp(1-  \chi_{k}^{\ve,\ve_1})$
by \eqref{eq.v1-support1} and~\eqref{eq.v1-support2}, there exist $c'>0$  and $\ve_0>0$ such that   for   $\ve\in (0,\ve_0)$, in the limit $h\to 0$:
\begin{equation}
\label{eq.interaction''''}
\frac 1h\int_{\ft B_{j,k}  } \left (1- \chi_{k}^{\ve,\ve_1}\right )\, e^{-\frac 1h f}\, d_{f,h}^*\widetilde \phi_j =O\big( e^{-\frac{1}{h}(f(z_j)+c')}\big).
\end{equation} 
Lastly, using Lemma~\ref{le.comp-wkb} and  the trace theorem, one obtains in the limit $h\to 0$: 
\begin{align*}
\int_{\partial \ft B_{j,k}  \cap\mathcal W_{+}(z_j)}\widetilde \phi_j\cdot n \  e^{-\frac 1h f}  &=\int_{\partial \ft B_{j,k}  \cap\mathcal W_{+}(z_j)} \!\! c_j(h)\,\theta_j\,u^{(1)}_{j,wkb}\cdot n \  e^{-\frac 1h f} +O(h^{\infty}e^{- \frac{1}{h}f(z_j)})\\
&= \pm\frac{c_j(h)(\pi h)^{\frac {d-1}2}}{\big(\det \Hess f|_{\mathcal W_{+}(z_j)} (z_j)\big)^{\frac12}}e^{- \frac{1}{h}f(z_j)}\big(1+O(h)\big)\\
&= \pm\frac{(\pi h)^{\frac {d-2}4}|\lambda_{-}(z_j)|^{\frac12}}{|\det \Hess f(z_j)|^{\frac14}}e^{- \frac{1}{h}f(z_j)}\big(1+O(h)\big) \,,
\end{align*}
where   $\lambda_{-}(z_j)$ denotes the negative eigenvalue of $\Hess f(z_j)$.
The second equality follows from   Laplace's method  and from the fact that
$$u^{(1)}_{j,wkb}\cdot n=
e^{-\frac{1}{h}\Phi_j}|_{\mathcal W_{+}(z_j)}\  \big(\widetilde a_j\cdot n + O(h)\big)=
e^{-\frac{1}{h}(f-f(z_j))}|_{\mathcal W_{+}(z_j)}\  \big(\widetilde a_j\cdot n + O(h)\big)
$$
and  $ \widetilde a_j(z_j)\cdot n=\pm1$,  see indeed the lines between \eqref{eq.stable-manifold2} and \eqref{eq.norme-wkb0}. 
The last line  follows  from \eqref{eq.ch0}.
The asymptotic estimate of the term $\lp       \nabla \widetilde u_k ,    \widetilde \psi_j \rp_{L^2_w}$is a consequence of the latter estimate together with \eqref{eq.interaction'}--\eqref{eq.interaction''''} which gives in the limit $h\to 0$:
$$
 \lp       \nabla \widetilde u_k ,    \widetilde \psi_j \rp_{L^2_w}=\pm\frac{ \pi ^{-\frac{1}{2}} |\lambda_{-}(z_j)|^{\frac12}  \,h^{-\frac{1}{2}}    \, e^{-\frac{1}{h}(f(z_j)- f(x_k))} }{ \big | {\rm det \ Hess }f(z_j)  \big  |^{\frac 14} \Big( \sum_{x\in\argmin_{\ft E_k}f}  \big( {\rm det \ Hess } f   (x)   \big)^{-\frac12}   \Big )^{\frac 12} }\,\big ( 1+O(h)\big),
$$
where   the remainder terms $O(h)$   admits a full asymptotic expansion in~$h$ (which follows from  Laplace's method).
This proves  Proposition~\ref{ESTIME1} for all  $k\in \big \{\ft  1,\ldots,\ft m_0^\Omega \big \}$ and  $j\in  \big  \{1,\dots,\ft m_1^{\overline \Omega} \big  \}$ such that $z_j\in \mbf j(x_k)\cap \Omega$.
\medskip

\noindent
\textbf{Step 3.} 
\medskip

\noindent
  Let us now deal with the computation of the terms $ \lp       \nabla \widetilde u_k ,    \widetilde \psi_j\rp_{L^2_w}$ for $k\in \big \{\ft  1,\ldots,\ft m_0^\Omega \big \}$ and   $j\in  \big  \{1,\dots,\ft m_1^{\overline \Omega} \big  \}$ when   $z_j \in \mbf j (x_k)\cap \pa \Omega$. Notice that according to Definition~\ref{de.qm-L} and \eqref{eq.v1} and by definition of lexicographic labeling introduced in Definition~\ref{de.label-l}, this situation can only occur when
$$k\in \big \{1,\dots,\ft N_1  \big \}.$$  
  One has
\begin{equation}
\label{eq.interaction}
 \lp       \nabla \widetilde u_k ,    \widetilde \psi_j\rp_{L^2_w}= \int_{\Omega} \nabla  \widetilde u_k  \cdot \widetilde \phi_j \  e^{- \frac{1}{h}f}   = \frac{ \displaystyle{\int_{\Omega} \nabla   \chi_{k}^{\ve,\ve_1} \cdot \widetilde \phi_j \  e^{- \frac{1}{h}f} } } { \sqrt{\displaystyle{\int_{\Omega} ( \chi_{k}^{\ve,\ve_1})^2 e^{-\frac{2}{h}  f} } }    }.
\end{equation}
Notice that~\eqref{eq.Lapl-u-k} also holds here for $\int_{\Omega} ( \chi_{k}^{\ve,\ve_1})^2 e^{-\frac{2}{h}  f}$. 
Since $\supp \widetilde \phi_j \subset B(z_j,2\ve_{1})\cap\overline \Omega $, the numerator of the right-hand side of~\eqref{eq.interaction} can be rewritten as
\begin{align}
\nonumber
\int_{\Omega} \nabla  \chi_{k}^{\ve,\ve_1} \cdot \widetilde \phi_j \  e^{-\frac 1h f}  
&= - \int_{B(z_j,2\ve_{1})\cap \overline \Omega}  \nabla(1-  \chi_{k}^{\ve,\ve_1}) \cdot \widetilde \phi_j \  e^{-\frac 1h f}    \\
\nonumber
&=  -\int_{B(z_j,2\ve_{1})\cap \overline \Omega}   \big (1- \chi_{k}^{\ve,\ve_1}\big) \ d^{*}\big ( \widetilde \phi_j \  e^{-\frac 1h f}\big) \\
\nonumber
&\quad   - \int_{\partial (B(z_j,2\ve_{1})\cap \overline \Omega)}   (1- \chi_{k}^{\ve,\ve_1}  ) \widetilde \phi_j\cdot n \  e^{-\frac 1h f}    \\
\label{eq.numerator}
&= - \frac1h\int_{B(z_j,2\ve_{1})\cap \overline \Omega}   (1- \chi_{k}^{\ve,\ve_1}  )  \, e^{-\frac 1h f}\, d_{f,h}^*  \widetilde \phi_j  - \int_{\partial B(z_j,2\ve_{1})\cap\pa\Omega } \widetilde \phi_j \cdot n \  e^{-\frac 1h f}  .
\end{align}
From~\eqref{eq.delta-tildephi'}, there exists $c>0$ such that for $h$ small enough,  $d_{f,h}^* \widetilde \phi_j =O ( e^{-\frac{c}{h}} )$ in~$L^2(\Omega)$. Since  $f\geq f(z_j)-2\ve$ on $ B(z_j,2\ve_{1})\cap \supp(1-  \chi_{k}^{\ve,\ve_1})$ by 
\eqref{eq.epsilon1''}  and \eqref{eq.v1-support1},
there exist $c'>0$ and $\ve_0>0$ such that for  $\ve\in (0,\ve_0)$,   in the limit $h\to 0$:
\begin{equation}
\label{eq.expo-term}
\frac 1h\int_{B(z_j,2\ve_1)\cap \overline \Omega} \big (1- \chi_{k}^{\ve,\ve_1}\big )\, e^{-\frac 1h f}\, d_{f,h}^*\widetilde \phi_j \,=\, O\big ( e^{-\frac{1}{h}(f(z_j)+c')}\big).
\end{equation}
Furthermore, applying Proposition~\ref{gamma1}  with $\Sigma=\pa\Omega$ and $F=1_{\pa B(z_j,2\ve_1)\cap\partial \Omega}$, one has in  the limit $h\to 0$:
\begin{equation}
\label{eq.integral-term}
\int_{\partial B(z_j,2\ve_1) \cap \pa\Omega } \widetilde \phi_j \cdot n \  e^{-\frac 1h f} =\frac{ \pi ^{\frac{d-1}{4}} \sqrt{2\, \partial_n f(z_j) }   }{ \big( {\rm det \ Hess }f|_{ \partial \Omega}(z_j)   \big)^{1/4}  }\, \,h^{\frac{d-3}{4}} e^{-\frac{1}{h}f(z_{1})} \big ( 1+O(h)\big).
\end{equation}
Therefore, injecting the estimates~\eqref{eq.Lapl-u-k}, \eqref{eq.expo-term} and \eqref{eq.integral-term} into \eqref{eq.interaction}, one obtains in the limit $h\to 0$:
$$  \lp       \nabla \widetilde u_k ,    \widetilde \psi_j \rp_{L^2_w}=-\frac{ \pi ^{-\frac 14} \sqrt{2\, \partial_n f(z_j) }  \,h^{-\frac 34}     e^{-\frac{1}{h}(f(z_{j})- f(x_k))} }{ \big( {\rm det \ Hess }f|_{ \partial \Omega}(z_j)   \big)^{\frac 14} \left( \sum_{x\in\argmin_{\ft E_k}f}  \big( {\rm det \ Hess } f   (x)   \big)^{-\frac12}   \right )^{\frac 12} }\,\big ( 1+O(h)\big)
\,,$$
where the remainder terms $O(h)$   admits a full asymptotic expansion in~$h$ (which follows from  Laplace's method).
This  is the desired estimate according to \eqref{eq.Cip}--\eqref{eq.Bi}.  
This ends  the proof of Proposition~\ref{ESTIME1}. 
 
\end{proof}

\subsection{Restricted differential $\nabla:\,  \Ran \, \pi^{(0)}_h\to \Ran \, \pi^{(1)}_h$ }\label{sec:S}
This section is devoted to the proof of Theorem~\ref{thm-big0}. In this section, one also gives   lower and upper bounds on the~$\ft m_0^\Omega$ first eigenvalues of $-L^{D,(0)}_{f,h}$, see  Theorem~\ref{pp} below.  According to Lemma~\ref{ran1} and Proposition~\ref{fried}, the square of the~$\ft m_0^\Omega$ singular values of the  restricted  differential~$\nabla:\,  \Ran \, \pi^{(0)}_h\to \Ran \, \pi^{(1)}_h$ ($\Ran \, \pi^{(0)}_h$ and $\Ran \, \pi^{(1)}_h$ being endowed with the $L^2_w$ inner product) are the first~$\ft m_0^\Omega$  first eigenvalues of $-L^{D,(0)}_{f,h}$ times $\frac 2h$. Therefore, the strategy consists in estimating in the limit $h\to 0$ the~$\ft m_0^\Omega$ singular values of the  restricted  differential $\nabla:\,  \Ran \, \pi^{(0)}_h\to \Ran \, \pi^{(1)}_h$. 

 This section is organized as follows. Section~\ref{sec.e1} is dedicated to the definition of the matrix of the restricted  differential $\nabla:\,  \Ran \, \pi^{(0)}_h\to \Ran \, \pi^{(1)}_h$ and   preliminary asymptotic estimates on its coefficients. In Section~\ref{sec.e2},   lower and upper bounds for the  $\ft m_0^\Omega$ smallest eigenvalues of $-L^{D,(0)}_{f,h}$ are obtained. 
Finally,  one proves Theorem~\ref{thm-big0} in   Section~\ref{sec.e3}.

\subsubsection{Matrix of the restricted  differential $\nabla:\,  \Ran \, \pi^{(0)}_h\to \Ran \, \pi^{(1)}_h$}
\label{sec.e1}
Let us  introduce  the matrix of the restricted  differential $\nabla:\,  \Ran \, \pi^{(0)}_h\to \Ran \, \pi^{(1)}_h$ in a basis of projected quasi-modes.

\begin{definition}\label{de.S}
Let us assume that the assumption \eqref{H-M} is satisfied. Let $(\widetilde u_{k})_{k\in  \{1,\ldots,\ft m_0^{ \Omega} \}}$ be the family of quasi-modes for $-L^{D,(0)}_{f,h}$ and let   $(\widetilde \psi_j)_{j\in \{1,\dots,\ft m_1^{\overline \Omega}\}}$   be the family of quasi-modes for $-L^{D,(1)}_{f,h}$, both  introduced  in Definition~\ref{de.qm-L}. 
Let us denote by~$S=(S_{j,k})_{j,k}$ the $ \ft m_1^{\overline \Omega}\times {\ft m_{0}^{\Omega}}$ matrix defined by: for all $k\in \{1,\ldots,\ft m_0^\Omega\}$ and for all $j\in \{1,\dots,\ft m_1^{\overline \Omega}\}$\label{page.sjk}
\begin{equation}
\label{eq.S}
S_{j,k} := \big   \lp       \nabla  \pi_h^{(0)} \widetilde u_k ,  \pi_h^{(1)}  \widetilde \psi_j \big  \rp_{L^2_w}.
\end{equation}
\end{definition}
\noindent
Notice that  from \eqref{eq.comutation}, it holds for all $k\in \{1,\ldots,\ft m_0^\Omega\}$ and for all $j\in \{1,\ldots,\ft m_1^{\overline \Omega}\}$:
$$S_{j,k} = \big   \lp       \nabla  \pi_h^{(0)} \widetilde u_k ,  \pi_h^{(1)}  \widetilde \psi_j \big  \rp_{L^2_w}=\big   \lp       \nabla  \widetilde u_k ,  \pi_h^{(1)}  \widetilde \psi_j \big  \rp_{L^2_w}.$$
Then, using the identity  
$$ \big   \lp       \nabla \widetilde u_k ,  \pi_h^{(1)}  \widetilde \psi_j \big  \rp_{L^2_w}=  \big \lp       \nabla  \widetilde u_k ,   \widetilde \psi_j \big \rp_{L^2_w}+  \big \lp       \nabla  \widetilde u_k ,  \big ( \pi_h^{(1)} -1 \big)  \widetilde \psi_j \big  \rp_{L^2_w}$$
together with item 2 in Proposition~\ref{ESTIME1-base} and Proposition~\ref{ESTIME1}, one gets the following estimates of the coefficients of $S$:

\begin{proposition}\label{pr.S} Let us assume that the assumption~\eqref{H-M}  is satisfied. Let $S$ be the matrix introduced in Definition~\ref{de.S}. Let   $k\in \{1,\ldots,\ft m_0^\Omega\}$ and $j\in \{1,\dots,\ft m_1^{\overline \Omega}\}$. 
Then, there exists $\ve_0>0$ such that   for  $\ve\in (0,\ve_0)$ (where $\ve$ is introduced in~\eqref{eq.v1}),
there exists $c>0$ 
such that in the limit $h\to 0$:
$$
   \text{ if } z_{j}\notin \mathbf j(x_k), \ \  S_{j,k} =O \left(e^{-\frac{1}{h}(f(\mathbf{j}(x_k))- f(x_k)+c)}\right )
  $$
  and
\begin{align*}
 \text{ if } z_{j}\in  \mathbf j(x_k), \ \ S_{j,k} &= \big \lp       \nabla \widetilde u_k ,    \widetilde \psi_j \big \rp_{L^2_w}\big(1+O(e^{-\frac ch})\big)\\
 &=  \ve_{j,k}\, C_{j,k} \, h^{p_{j,k}}\,  e^{-\frac{1}{h}(f(\mathbf{j}(x_k))- f(x_k))}   \big   (  1  +     O(h )    \big   )   
 \end{align*}
  where we recall that, from  Proposition~\ref{ESTIME1},  $\ve_{j,k}\in\{-1,1\}$,  
    \begin{equation}
  \label{eq.pjk}
  p_{j,k}=
  \begin{cases} -\frac12 &\text{ if } z_{j}\in \mathbf{j}(x_k)\cap \Omega,\\ 
  -\frac34 &\text{ if } z_{j}\in \mathbf{j}(x_k)\cap \pa\Omega,
  \end{cases}
  \end{equation} 
 and  $C_{j,k}$ is defined in~\eqref{eq.Cip}. Moreover, when $z_j\in \mbf j(x_k)\cap \pa \Omega$ (and thus $k\in \{1,\ldots,\ft N_1\}$), one has~$\ve_{j,k}=-1$.
\end{proposition}
\noindent
To study the singular values of the  matrix $S$, one defines the following matrices:

\begin{itemize} [leftmargin=1.3cm,rightmargin=1.3cm]
\item let $\widetilde S= \big(\widetilde S_{j,k} \big )_{j,k}$ be the real value $  \ft m_1^{\overline \Omega} \times {\ft m_{0}^{\Omega}}$   matrix defined by\label{page.tildesjk}
\begin{equation}
\label{eq.S-tilde}
\widetilde S_{j,k} := \begin{cases} S_{j,k}    &  \text{ if } z_{j}\in  \mathbf{j}(x_k)  ,\\
0   &   \text{ if } z_j\notin  \mathbf{j}(x_k).
  \end{cases}
  \end{equation}

\item let $D$ be the diagonal ${\ft m_{0}^{\Omega}}\times {\ft m_{0}^{\Omega}}$ matrix 
defined by \label{page.dkk}
\begin{equation}
\label{eq.D}
 \left\{
    \begin{array}{ll} 
    &\forall k\in\{1,\dots,{\ft m_{0}^{\Omega}}\},\ 
D_{k,k}:=h^{q_{k}}  e^{-\frac{1}{h}(f(\mathbf{j}(x_k))- f(x_k))}\ \ \text{where }\\ 
& q_{k}:=\min_{j}\{p_{j,k}\},
  \end{array}
\right.
\end{equation}
and where  $p_{j,k}$ is defined in~\eqref{eq.pjk}. \\Notice that when  the assumptions~\eqref{eq.hip1-j} and \eqref{eq.hip2-j} are  satisfied,  one has in the limit $h\to 0$ (see~\eqref{eq.sj} together with the fact $q_1=-\frac 34$ since  $\mbf j(x_1)\cap \pa \Omega=\pa \ft C_1\cap \pa \Omega\neq\emptyset$ which follows from~\eqref{jx1} and~\eqref{eq.hip2-j}):
\begin{equation}
\label{eq.D''} 
D_{1,1}=h^{-\frac34}  e^{-\frac{1}{h}(f(\mbf j(x_{1}) )- f(x_1))},
\end{equation}
and for some $c>0$ independent of $h$, it holds:
\begin{equation}
\label{eq.D'''}
\text{for all } k\in \{2,\ldots,\ft m_0^\Omega\}, \ \frac{ D_{1,1} }{D_{k,k}}=O \big (e^{-\frac c h}   \big ).
\end{equation}

\item let $\widetilde C$\label{page.tildec} be the real value $ \ft m_1^{\overline \Omega} \times  \ft m_{0}^\Omega$ matrix defined by 
\begin{equation}
\label{eq.C-tilde}
\widetilde C:=\widetilde S  \, D^{-1}. 
\end{equation}
The matrix $\widetilde C$ is the  $ \ft m_1^{\overline \Omega} \times {\ft m_{0}^{\Omega}}$   matrix  whose coefficients are given by:
$$  \widetilde C_{j,k}=\frac{\widetilde S_{j,k}}{D_{k,k}},$$
for all $(j,k)\in  \big\{1,\dots,\ft m_1^{\overline \Omega} \big\}\times \big \{1,\dots,{\ft m_{0}^{\Omega}} \big \}$. It satisfies, according to Proposition~\ref{pr.S} and~\eqref{eq.D}, in the limit $h\to 0$:
\begin{equation}
\label{eq.C-tilde-estimate}
  \widetilde C_{j,k} =\begin{cases} \ve_{j,k}\,C_{j,k} \,  h^{p_{j,k}-q_{k}}\,  \big(  1  +     O(h )   \big)    &  \text{ if } z_{j}\in  \mathbf{j}(x_k)  ,\\
0   &   \text{ if } z_j\notin  \mathbf{j}(x_k),
  \end{cases}
  \end{equation}
where  $p_{j,k}$ is defined by~\eqref{eq.pjk},~$\ve_{j,k}\in \{-1,1\}$ and $C_{j,k}$ is defined in~\eqref{eq.Cip}. 
From~\eqref{eq.C-tilde-estimate}, one has when $h\to 0$,~$\widetilde C=O(1)$  which means that there exist $K>0$ and $h_0>0$ such that for all $h\in (0,h_0)$:
 \begin{equation}
\label{eq.C-tilde-borne}
\sup_{(j,k)\in   \{1,\dots,\ft m_1^{\overline \Omega}  \}\times   \{1,\dots,{\ft m_{0}^{\Omega}}\}} \big \vert \widetilde C_{j,k}\big \vert\le K.
\end{equation}
\end{itemize}

\noindent
Under~\eqref{H-M}, by definition of the matrices $S$,~$\widetilde S$,~$D$ and $\widetilde C$ (see Definition~\ref{de.S} and Equations~\eqref{eq.S-tilde},~\eqref{eq.D},~\eqref{eq.C-tilde}), there exists $c>0$ such that the  matrix $\big (S-\widetilde S\big) D^{-1}$ satisfies in the limit $h\to 0$:
\begin{equation}
\label{eq.expo-R}
\big (S-\widetilde S\big) D^{-1}=O(e^{-\frac ch}).
\end{equation}
The following Lemma will be needed in the sequel. 

\begin{lemma}\label{le.C-tilde-coercive}
Let us assume that the assumption \eqref{H-M} holds. Let $\widetilde{C}$ be the matrix defined in~\eqref{eq.C-tilde}. Then,    there exist    $c>0$ and $h_0>0$ such that for all 
$h \in (0,h_0)$:  
\begin{equation}
\label{eq.C}
\forall x\in \mathbb R^{{\ft m_{0}^{\Omega}}},\  \ \big \|\widetilde C x \big \|_{2}\geq c \big \| x \big\|_{2},
\end{equation}
where $\Vert.\Vert_2$ denotes the usual Euclidean norm on $\mathbb R^K$,~$K\in \mathbb N^*$.
\end{lemma}
\begin{proof}
The proof of Lemma~\ref{le.C-tilde-coercive} is divided into two steps.

\medskip

\noindent
\textbf{Step 1.} Block-diagonal decomposition of $\widetilde C$. 
\medskip

\noindent 
According to~\eqref{eq.C-tilde},~\eqref{eq.D} and \eqref{eq.S-tilde},   $\widetilde C$ has the form, up to  reordering   the $z_{i}$ for $i\in\{1,\dots,\ft m_1^{\overline\Omega} \}$:  
 \begin{equation*}
\widetilde C  = 
\begin{pmatrix}
0&0\\
[\widetilde C]_a &  0 \\
0 & \widetilde  [\widetilde C]_c 
\end{pmatrix},
\end{equation*}
where:
\begin{itemize} [leftmargin=1.3cm,rightmargin=1.3cm]
\item the block matrix $(0,0)$ on the first line corresponds to the rows of $\widetilde C$ associated with $j\in \{1,\ldots,\ft m_1^{\overline \Omega}\}$ such that $z_j\notin \cup_{k=1}^{\ft m_0^\Omega}\mbf j(x_k)$.   

\item $[\widetilde C]_a $ is a matrix of size $   {\rm Card}\big(\cup_{k=1}^{\ft N_1}\mathbf j(x_{k})\big)   \times \ft N_1$ (where $\ft N_1$ is defined in~\eqref{eq.N11}). The coefficients $([\widetilde C]_a )_{j,k}=\widetilde C_{j,k}$  are associated with  0-forms $\widetilde u_{k} $,  for $k\in\{1,\dots,\ft N_1\}$ (see Definition~\ref{de.qm-L} and~\eqref{eq.v1}) and  with 1-forms $\widetilde \psi_{j} $ for $j\in \{1,\ldots,\ft m_1^{\overline \Omega}\}$ such that $z_j\in  \cup_{k=1}^{\ft N_1}\mathbf j(x_{k}) $. 

\item  $[\widetilde C]_c$ is a matrix of size $   {\rm Card}\big(\cup_{k=\ft N_1+1}^{\ft m_0^\Omega}\mathbf j(x_{k})\big)   \times (\ft m_0^{ \Omega}-\ft N_1)$, 
which  has the following block diagonal form:
$$[\widetilde C]_c=\begin{pmatrix}
    [\widetilde C]_{c,1}& 0 &\dots  & 0 \\
    0 & [\widetilde C]_{c,2} & \dots  & 0 \\
    \vdots  & \vdots & \ddots & \vdots \\
    0 & 0 &  \dots  & [\widetilde C]_{c,N_1}
\end{pmatrix},
$$
\begin{sloppypar}
where for $\ell\in  \big \{1,\ldots,\ft  N_1 \big \}$,~$ [\widetilde C]_{c,\ell}$ is a matrix of size 
$${\rm Card}\left ( \bigcup_{k, \, \mbf j(x_k)\subset  \ft C_l} \mbf j(x_k) \right)\times\Big ( {\rm Card}\big (\argmin_{\overline{\ft  C_\ell}} f\big )~-1\Big),$$
 with the convention that $ [\widetilde C]_{c,\ell}$ does not exist if $\argmin_{\overline{\ft  C_\ell}} f =\{x_\ell\}$. Let us recall that  for $\ell \in  \big \{1,\ldots,\ft N_1 \big \}$,~$\ft C_\ell$ is introduced  in Definition~\ref{de.1}. 
For all $\ell \in \{ 1,\ldots,\ft N_1\}$,~$[\widetilde C]_{c,\ell}$ contains the non zero terms of~$\widetilde{C}$ associated with 0-forms
$\widetilde{u}_k$ and 1-forms $\widetilde \psi_j$ with:
\begin{enumerate}
\item  $\widetilde{u}_k$ such that $\supp (\widetilde{u}_k) \subset
\{ \widetilde \chi_\ell=1\}$  (according to~\eqref{eq.suport-inclusion}),
\item for those $\widetilde{u}_k$,~$j$ is such that $z_j \in \mbf j(x_k)\subset \ft C_\ell$.
\end{enumerate}
This explains in particular the block structure of $[\widetilde C]_{c}$ since by
construction $\widetilde S_{j,k}=0$ if $z_j \not\in \mbf j(x_k)$ (see~\eqref{eq.C-tilde} and~\eqref{eq.S-tilde}).
  \end{sloppypar}
  
\end{itemize}
From \cite[ Section 7.3 and Equation (7.4) in Section  7.2]{HeHiSj}, for all  $\ell\in  \big \{1,\ldots,\ft N_1 \big \}$  there exist $c_\ell>0$ and $h_0>0$ such that  for all $h\in (0,h_0)$ and for all $z\in \mathbb R^{{\rm Card}\big (\argmin_{\overline{\ft C_\ell}} f\big )~-1}$,
$$
\big\Vert \widetilde [\widetilde C]_{c,\ell} \,z \big\Vert_2\geq c_\ell\Vert z\Vert_2.
$$
Thus, there exist $\alpha >0$ and $h_0>0$ such that  for all $h\in (0,h_0)$ and for all $z\in \mathbb R^{\ft m_0^\Omega-\ft N_1}$,
\begin{equation}
\label{eq.TiTi}
\big\Vert \widetilde [\widetilde C]_{c} \,z \big\Vert_2\geq \alpha\Vert z\Vert_2.
\end{equation}
For any $x=\, (y,z)^T\in \mathbb R^{\ft m_0^\Omega}$ ($y\in \mathbb R^{\ft N_1}$ and $z\in \mathbb R^{\ft m_0^\Omega-\ft N_1}$), it holds 
$$\big\Vert \widetilde C \,x \big\Vert_2^2=  \big\Vert  [\widetilde C]_{a} \, y \big\Vert_2^2+ \big\Vert [\widetilde C]_c\,z \big \Vert_2^2\geq   \big\Vert  [\widetilde C]_{a} \,y \big\Vert_2^2+ \alpha^2\Vert z\Vert_2^2.$$ 
Therefore, to prove~\eqref{eq.C}, let show that  there exist $\beta >0$ and $h_0>0$ such that  for all $h\in (0,h_0)$ and for all $y\in \mathbb R^{\ft N_1}$,
\begin{equation}
\label{eq.TiTi-b}
\big\Vert  [\widetilde C]_{a} \, y \big\Vert_2\geq \beta\Vert y\Vert_2.
\end{equation}
\medskip

\noindent
\textbf{Step 2.} Proof of~\eqref{eq.TiTi-b}. 
\medskip

\noindent
Let us divide the family $(\ft C_k)_{k\in \{1,\ldots,\ft N_1\}}$ into $\ft K$  groups  ($\ft K\le \ft N_1$):
$$\{\ft C_1,\ldots,\ft C_{\ft N_1}\}=\bigcup_{\ell=1}^{\ft K}  \{\ft C_1^\ell,\ldots,\ft C^\ell_{\ft k_\ell}\} $$
which are such that for all $\ell\in \{1,\ldots,\ft K\}$, 
$$\text{the set } \, \bigcup_{j=1}^{k_\ell} \overline{\ft C_j^\ell} \ \text{ is connected},$$
and for all $q\in \{1,\ldots,\ft N_1\}$ such that $\ft C_q \notin \{\ft C_1^\ell,\ldots,\ft C^\ell_{\ft k_\ell}\}$, 
$$\overline{\ft C_q}  \cap \bigcup_{j=1}^{k_\ell} \overline{\ft C_j^\ell}  =\emptyset.$$
Then, by definition of the matrix $[\widetilde C]_a $ (see \textbf{Step 1} above), up to  a reordering,~$[\widetilde C]_a $ has the block-diagonal form
$$[\widetilde C]_a=\begin{pmatrix}
    [\widetilde C]_{a,1}& 0 &\dots  & 0 \\
    0 & [\widetilde C]_{a,2} & \dots  & 0 \\
    \vdots  & \vdots & \ddots & \vdots \\
    0 & 0 &  \dots  & [\widetilde C]_{a,\ft K}
\end{pmatrix},
$$
\begin{sloppypar}
\noindent
where for $\ell\in  \big \{1,\ldots,\ft K \big \}$,~$ [\widetilde C]_{a,\ell}$ is a matrix of size ${\rm Card}\Big (    \bigcup_{k=1, \,x_k\in \cup_{j=1}^{k_\ell} \ft C_j^\ell }^{\ft N_1} \mbf j(x_{k}) \Big)\times k_\ell$. For $\ell\in  \big \{1,\ldots,\ft K \big \}$, the coefficients $([\widetilde C]_{a,\ell} )_{j,k}=\widetilde C_{j,k}$  are associated with  0-forms $\widetilde u_{k} $,  for $k\in\{1,\dots,\ft N_1\}$  such that $x_k\in \cup_{j=1}^{k_\ell} \ft C_j^\ell    $  and  with 1-forms $\widetilde \psi_{j} $ for $j\in \{1,\ldots,\ft m_1^{\overline \Omega}\}$ such that $z_j\in  \bigcup_{k=1,\,  x_k\in \cup_{j=1}^{k_\ell} \ft C_j^\ell }^{\ft N_1} \mbf j(x_{k}) $. 
\end{sloppypar}
\noindent
Therefore, to prove~\eqref{eq.TiTi-b}, let us  show that for $\ell\in  \big \{1,\ldots,\ft K \big \}$,  there exist $\beta_\ell >0$ and $h_0>0$ such that  for all $h\in (0,h_0)$ and for all $y\in \mathbb R^{k_\ell}$,
$$
\big\Vert  [\widetilde C]_{a, \ell } \,y \big\Vert_2\geq \beta_\ell \Vert y\Vert_2.
$$
In view of the block structure of $ [\widetilde C]_a$, to prove it, one can assume,  without loss of generality, that $\ft K=1$ which is equivalent to the fact that the set  $\bigcup_{j=1}^{\ft N_1} \overline{\ft C_j}$ is connected. Let us thus assume that $\bigcup_{j=1}^{\ft N_1} \overline{\ft C_j}$ is connected and let us then write 
\begin{equation}
\label{eq.AA}  
[\widetilde C]_{a}= A +O(h),
\end{equation}
where $A $ is a matrix which has the same size as $[\widetilde C]_{a}$, and which satisfies,  from~\eqref{eq.C-tilde-estimate}, for all $k\in \{1,\ldots,\ft N_1\}$ and all $j$ such that $z_j\in  \bigcup_{k=1 }^{\ft N_1} \mbf j(x_{k})$,
\begin{equation}
\label{eq.1-1}
\text{if }  \mbf j(x_k)\cap \pa \Omega\neq \emptyset \text{ one has }  A _{j,k} =\begin{cases} \ve_{j,k}\,C_{j,k}      &  \text{ if } z_{j} \in  \mathbf{j}(x_k)  \cap \pa \Omega,\\
O\big (h^{\frac 14}\big )    &  \text{ if } z_{j}\in  \mathbf{j}(x_k)\cap  \Omega,  \\
0   &   \text{ if } z_j\notin  \mathbf{j}(x_k),
  \end{cases}
\end{equation}
 and  
\begin{equation}
\label{eq.1-2}
 \text{if }  \mbf j(x_k)\cap \pa \Omega=  \emptyset \text{ one has }  A _{j,k} =\begin{cases}  
 \ve_{j,k}\,C_{j,k}   &  \text{ if } z_{j} \in  \mathbf{j}(x_k) ,  \\
0   &   \text{ if } z_j\notin  \mathbf{j}(x_k),
  \end{cases}
\end{equation}
where $\ve_{j,k}\in \{-1,1\}$ and $\,C_{j,k} > 0$ for all $k\in \{1,\ldots,\ft N_1\}$ and $j$ such that $z_j\in \mathbf j(x_k)$. To prove~\eqref{eq.TiTi-b}, it sufficient to show that~\eqref{eq.TiTi-b} holds for $A $ instead of $[\widetilde C]_{a}$, i.e. that there exist $\beta >0$ and $h_0>0$ such that  for all $h\in (0,h_0)$ and for all $y\in \mathbb R^{\ft N_1}$,
\begin{equation}
\label{eq.T}
\big\Vert A \, y \big\Vert_2\geq \beta\Vert y\Vert_2.
\end{equation}
Before proving~\eqref{eq.T}, let us label the family $(\ft C_k)_{k\in \{1,\ldots,\ft N_1\}}$ as follows. According to Lemma~\ref{le.oubli}, one can assume without loss of generality that  $\ft C_1$ is such that there exists    $z\in \ft U_1^{\ft{ssp}}$ such that  \begin{equation}\label{eq.C1-ssp0}
z\in \pa \ft C_{1}  \setminus \Big (    \cup_{\ell=2}^{\ft N_1}\pa {\ft C_{\ell}}\Big).
\end{equation}
Let us now label $(\ft C_j)_{j\in \{2,\ldots,\ft N_1\}}$ such that for all $k\in \{1,\ldots,\ft N_1\}$, $\bigcup_{j=1}^{k} \overline{\ft C_j}$ is connected.
Let   us prove~\eqref{eq.T} by induction on $k\in \{1,\ldots,\ft N_1\}$ (the proof is similar to the proof made in  \cite[Section 7.3]{HeHiSj} in a different context). For $k\in \{1,\ldots,\ft N_1\}$, one denotes by~$\mathcal P_{k}$ the following property: there exists $\alpha _k>0$ and $h_0>0$ such that  for all $h\in (0,h_0)$ and for all $y\in \mathbb R^{\ft N_1}$, 
$$\big\Vert A \, y \big\Vert_2\geq \alpha_k\sum_{\ell=1}^k y_\ell^2.$$ 
Let us prove~$\mathcal P_{1}$. 
Using~\eqref{eq.C1-ssp0}  together with~\eqref{eq.1-1} and~\eqref{eq.1-2}, the   $j$-th row of $A$ equals  
$$
\big (\ve_{j,1}\,C_{j,1}, 0,\ldots,0 \big ),
$$
where $j$ is such that $z_j\in \pa \ft C_{1}  \setminus \Big (    \cup_{\ell=2}^{\ft N_1}\pa {\ft C_{\ell}}\Big)$.
Thus, one has  for $y\in \mathbb R^{\ft N_1}$, $\big\Vert A \, y \big\Vert_2\ge \vert\ve_{j,1}\,C_{j,1} y_1\vert  $. Therefore,~$\mathcal P_{1}$ is satisfied. Let us now assume that~$\mathcal P_{k}$ is satisfied for some $k\in \{1,\ldots,\ft N_1-1\}$ and let us prove $\mathcal P_{k+1}$.  If $\pa \ft C_{k+1}\cap \pa \Omega\neq \emptyset$,  there exists $j\in \{1,\ldots,\ft N_1\}$, such that $z_j\in \pa \ft C_{k+1}\cap \pa \Omega$. Thus, using~\eqref{eq.1-1}, one has
$$(Ay)_j=  \ve_{j,k+1}C_{j,k+1}\, y_{k+1} .$$
Therefore, one obtains $ y_{k+1}^2C_{j,k+1}^2 \le   \big \vert (Ay)_j  \big \vert^2 \le  \big \Vert Ay  \big \Vert^2_2$. 
Applying the property~$\mathcal P_{k}$, one gets 
$$ \min(C_{j,k+1}^2, \alpha_k) \sum_{\ell=1}^{k+1} y_\ell^2 \le  2\big \Vert Ay  \big \Vert^2_2.$$
This implies that the property~$\mathcal P_{k+1}$ is satisfied. Let us now consider the case  $\pa \ft C_{k+1}\cap \pa \Omega= \emptyset$. 
Using~\eqref{eq.1-2} together with the fact that  the set $\bigcup_{j=1}^{k+1} \overline{\ft C_j}$ is connected, there exist $\ell \in \{1,\ldots,k\}$ and  $j$ such that $z_j\in\pa \ft C_\ell \cap \pa \ft C_{k+1}$. Thus,    
$$(Ay)_j= A_{j\ell} \, y_\ell +\ve_{j,k+1}C_{j,k+1}\, y_{k+1} .$$
Thus, since the exists $M>0$ such that $\vert A_{j\ell}\vert\le M$ for all $h$ small enough, one obtains using the triangular inequality and the property~$\mathcal P_{k}$,
$$ y_{k+1}^2  \le  \frac{2}{C_{j,k+1}^2} \Big(\big \vert (Ay)_j  \big \vert^2+A_{j\ell} ^2\, y_\ell^2 \Big)\le   \frac{2( 1+M^2) }{C_{j,k+1}^2} \big \Vert Ay  \big \Vert^2_2.$$
Using the property~$\mathcal P_{k}$, one gets that $ \min\Big(\frac{ C_{j,k+1}^2 }{2( 1+M^2) }, \alpha_k\Big) \sum_{\ell=1}^{k+1} y_\ell^2 \le  2\big \Vert Ay  \big \Vert^2_2$. Therefore, the property $\mathcal P_{k+1}$ is satisfied. This ends the proof of~\eqref{eq.T} by induction. Together with~\eqref{eq.TiTi} and~\eqref{eq.TiTi-b}, one then obtains~\eqref{eq.C}. This  concludes the proof of Lemma~\ref{le.C-tilde-coercive}.
\end{proof}

As a consequence of~\eqref{eq.C}, the rectangular matrix $\widetilde C$ admits a left inverse $\widetilde C^{-1}$ which satisfies 
  \begin{equation}
\label{eq.C-1-tilde-borne}
\sup_{(j,k)\in     \{1,\dots,{\ft m_{0}^{\Omega}}\}\times  \{1,\dots,\ft m_1^{\overline \Omega}  \}} \vert (\widetilde C^{-1})_{j,k}\vert\le M,
\end{equation}
for some $M>0$ independent of $h$. 
This implies that, using~\eqref{eq.expo-R} and~\eqref{eq.C-tilde}:
\begin{equation}
\label{eq.decomp-S}
S= \big( I+R \big)\,\widetilde S 
\quad\text{where}\quad R= \big (S-\widetilde S\big) D^{-1}\widetilde {C}^{-1}=  O(e^{-\frac ch}).
\end{equation}

\subsubsection{On the $\ft m_0^\Omega$ smallest eigenvalues of $-L^{D,(0)}_{f,h}$ }
\label{sec.e2}
This section is dedicated to the proof of the following proposition which aims at giving    lower  bound  and an upper bounds for   the  $\ft m_0^\Omega$ smallest eigenvalues of $-L^{D,(0)}_{f,h}$.


\begin{theorem}\label{pp}
Let us assume that the assumption~\eqref{H-M} holds. Let  $\mathbf j$ be the map constructed in Section~\ref{sec:labeling}.   Let us  reorder the  set $\{x_{ 1},\ldots,x_{\ft m_0^\Omega}\}$    such that the sequence \label{page.ftsk}
\begin{equation}
\label{eq.Sk-def}
 \big (\ft S_{k} \big )_{k\in\{1,\dots,{\ft m_{0}^{\Omega}}\}}:= \big (f(\mathbf j(x_{k}))-f(x_{k}) \big )_{k\in\{1,\dots,{\ft m_{0}^{\Omega}}\}}
 \end{equation}
is decreasing,
and, on any $\mathcal I\subset \{1,\dots,{\ft m_{0}^{\Omega}}\} $ such that $ \big (\ft S_{k} \big )_{k\in\mathcal I}$ is constant, the sequence 
$(q_{k})_{k\in\mathcal I}$ is decreasing (where the $q_{k}$'s are introduced in~\eqref{eq.D}). Finally let us denote by~$\lambda_{k,h}$, for $k\in \mathbb N^{*}$,  the $k$-th eigenvalue of $-L_{f,h}^{D,(0)}$ counted with multiplicity (with these notations,~$\lambda_{1,h}=\lambda_h$, see~\eqref{eq.lh}). Then, there  exist $C>0$ and $h_0>0$ such that   for all $k\in \big \{1,\dots,{\ft m_{0}^{\Omega}} \big \}$ and  for all $h\in (0,h_0)$,
$$
C^{-1} \,h^{1+2q_{k}}\, e^{-\frac{2}{h}\ft S_{k}} \leq \lambda_{k,h} \leq C \,h^{1+2q_{k}}\,e^{-\frac{2}{h}\ft S_{k}}\,.
$$ 

\end{theorem}
\noindent
 The reordering of $\{x_{ 1},\ldots,x_{\ft m_0^\Omega}\}$  introduced in~\eqref{eq.Sk-def} is only used in  Theorem~\ref{pp}. One recalls that, except in Theorem~\ref{pp},  the labelling  of $\{x_{ 1},\ldots,x_{\ft m_0^\Omega}\}$ is the one introduced in Definition~\ref{de.label-l}.
  
 A direct consequence of Theorem~\ref{pp} is the following. 
\begin{corollary}\label{eq.co.l2}
Let us assume that the assumptions \eqref{H-M} and \eqref{eq.hip1-j} are satisfied. Then, the estimate~\eqref{l2} is satisfied. 
\end{corollary}

Before starting the proof of  Theorem~\ref{pp}, let us recall the Fan inequalities, which is the purpose of Lemma~\ref{le.fan} below (see for instance \cite[Theorem~1.6]{simon1979trace} or \cite{le2009small}) and its consequences, see Lemma~\ref{le.fan1} below. \label{page.mui}
\begin{lemma}
\label{le.fan}
Let  $A\in M_{m,m}(\mathbb C)$,~$B\in M_{m,n}(\mathbb C)$ and $C\in M_{n,n}(\mathbb C)$. Then, it holds 
$$
\forall i\in\{1,\ldots,n\},\  \eta_{i}(A\,B\,C)\leq\big \|A\big\|\,\big\|C\big\|\,\eta_{i}(B),
$$
where, for any matrix $T\in  M_{m,n}(\mathbb C)$, $\|T\|=\eta_{1}(T)\geq\cdots\geq \eta_{n}(T) $ denote the singular values of the matrix $T$ and where $\big\|T\big\|:= \sqrt{ \max \sigma( \, ^{t}TT)}$ is the spectral norm of  $T$.
\end{lemma}

Let us recall that from item 3 in Proposition~\ref{ESTIME1-base}, there exists $h_0>0$ such that for all $h\in (0,h_0)$,
$$ 
\range   \pi_h^{(0)}= {\rm Span}\big( \pi_h^{(0)}\widetilde u_k,\ k=1,\dots,{\ft m_{0}^{\Omega}}\big )$$
and
$$ \range   \pi_h^{(1)}= {\rm Span}\big ( \pi_h^{(1)}\widetilde \psi_i,\ i=1,\dots,\ft m_1^{\overline \Omega}\big )
$$
\begin{sloppypar}
\noindent
where the projectors $\pi^{(0)}_h$ and $\pi^{(1)}_h$ are defined in~\eqref{eq.proj-p}. Let us define~$\widetilde\Upsilon:=\big( \pi^{(0)}_h\widetilde u_k\big)_{1\leq k \leq {\ft m_{0}^{\Omega}}}$ and 
$\widetilde\Psi:=\big(\pi^{(1)}_h\widetilde \psi_j\big)_{1\leq j \leq \ft m_1^{\overline \Omega}}$.  For $i\in\{0,1\}$, let  
$\mathcal B_{i}$ be an orthonormal basis of $\Ran \, \pi^{(i)}_h$ and let us define  the matrices\label{page.c0c1}
\begin{equation}
\label{eq.C0C1}
C_{0}:=\text{Mat}_{{\widetilde\Upsilon}}\, \mathcal B_{0} \, \text{ and } \,  C_{1}:=\text{Mat}_{ \widetilde\Psi}\, \mathcal B_{1}.
\end{equation}
Notice that from item 3 in Proposition~\ref{ESTIME1-base}, there exist $K>0$ and $h_0>0$ such that for all $h\in (0,h_0)$:
 \begin{equation}
\label{eq.C0-tilde-borne}
\sup_{(l,k)\in     \big \{1,\dots,{\ft m_{0}^{\Omega}}\big\}^2 }\big \vert (C_{0})_{l,k}\big\vert + \sup_{(i,j)\in    \big \{1,\dots,{\ft m_{1}^{\overline \Omega}}\big\}^2 } \big\vert (C_{1})_{i,j}\big\vert  \le K
\end{equation}
and 
 \begin{equation}
\label{eq.C0-1-tilde-borne}
\sup_{(l,k)\in    \big \{1,\dots,{\ft m_{0}^{\Omega}}\big\}^2 }\big \vert (C_{0}^{-1})_{l,k}\big\vert+ \sup_{(i,j)\in   \big  \{1,\dots,{\ft m_{1}^{\overline \Omega}}\big\}^2 } \big\vert (C_{1}^{-1})_{i,j}\big \vert  \le K.
\end{equation}
A  consequence of the Fan inequalities  is the following. 
\end{sloppypar}

\begin{lemma}\label{le.fan1}
Let us assume that the assumption \eqref{H-M} holds. Let us denote by~$\lambda_{k,h}$, for $k\in \mathbb N^{*}$,  the $k$-th eigenvalue of $-L_{f,h}^{D,(0)}$ counted with multiplicity and let  $\tilde S$ be the matrix defined in~\eqref{eq.S-tilde}.  Then, there exists $c>0$ such that for all $k\in \{1,\ldots,\ft m_0^\Omega\}$, one has  in the limit $h\to 0$:\label{page.lambdakh}
\begin{equation}
\label{eq.fan1}
\lambda_{k,h}=\frac h2\,  \eta_{\ft m_{0}^{\Omega}+1-k}^2\big( \widetilde S C_0\big) \big(1+O(e^{-\frac c h})\big),
\end{equation}
where   $\|\widetilde S C_0\|=\eta_{1}(\widetilde S C_0)\geq\cdots\geq \eta_{ \ft m_0^\Omega}(\widetilde S C_0) $ denote the singular values of $\widetilde S C_0$.
\end{lemma}

\begin{proof}
The ${\ft m_{0}^{\Omega}}$ smallest eigenvalues of $-L^{D,(0)}_{f,h}$ are the  eigenvalues of 
$$-L^{D,(0)}_{f,h}|_{\Ran \, \pi^{(0)}_h}=\frac h2\,d^{*}_{\frac{2f}{h},1}|_{\Ran \, \pi^{(1)}_h}\,d|_{\Ran \, \pi^{(0)}_h}. $$ 
Moreover, since the  $L^{2}_{w}$-adjoint of $d|_{\Ran \, \pi^{(0)}_h}:\Ran \, \pi^{(0)}_h\to \Ran \, \pi^{(1)}_h$  is  $d^{*}_{\frac{2f}{h},1}|_{\Ran \, \pi^{(1)}_h}:\Ran \, \pi^{(1)}_h\to \Ran \, \pi^{(0)}_h$, one has that the ${\ft m_{0}^{\Omega}}$ smallest eigenvalues of $-L^{D,(0)}_{f,h}$ are given by~$\frac h2$ times the squares of the singular values of $d|_{\Ran \, \pi^{(0)}_h}:\Ran \, \pi^{(0)}_h\to \Ran \, \pi^{(1)}_h$.   Thus, 
the eigenvalues of 
$-L^{D,(0)}_{f,h}|_{\Ran \, \pi^{(0)}_h}$ are  given by~$\frac h2$ times the squares of the singular values
of the matrix $Q$ defined by 

$$Q:=\text{Mat}_{   \mathcal B_{0}, \mathcal B_{1}     }\big( d|_{\Ran \, \pi^{(0)}_h}\big).$$           
 In addition, by definition of the matrices $C_0$ and $C_1$ (see~\eqref{eq.C0C1}), one has
$$
Q= \, ^{t}C_{1} \,S\,C_{0}.
$$
By~\eqref{eq.decomp-S}, it holds
$$Q=\, ^{t}C_{1}\,\big(\,I+R\,\big)\,\widetilde S\,C_{0}.$$
Furthermore, from~\eqref{eq.decomp-S}, there exists $c>0$ such that   in the limit $h\to 0$
$$\|I+R\|=1+O(e^{-\frac c h}), \ \ \|(I+R)^{-1}\|=1+O(e^{-\frac c h}),$$
and from item 3b in Proposition~\ref{ESTIME1-base}, 
$$\|\, ^{t}C_{1}\|=1+O(e^{-\frac c h}), \ \ \|(\, ^{t}C_{1} )^{-1}\|=1+O(e^{-\frac c h}),$$
where we recall that $\big\|T\big\|:= \sqrt{ \max \sigma( \, ^{t}TT)}$ is the spectral norm of  a matrix $T$. 
Therefore, it follows from the Fan inequalities (see Lemma~\ref{le.fan}) that
the singular values of $Q$ are, up to multiplication by~$1+O(e^{-\frac c h})$, the singular values
of $\widetilde S\,C_{0}$. This concludes the proof of Lemma~\ref{le.fan1}. 
\end{proof}
\begin{remark}
Notice that in general, the spectral norm of the matrix  $C_0$ defined in~\eqref{eq.C0C1} does not equal $1+O(e^{-\frac c h})$ when $h\to 0$.  
For instance,  in the case when $f$ is a one-dimensional symmetric double-well potential with the saddle point  lower than $\min_{\pa \Omega}f$,  it can be checked that the Gramian matrix of the functions $\widetilde u_1$ and $\widetilde u_2$ introduced in Definition~\ref{de.qm-L} converges when $h\to 0$ towards the matrix 
$$\begin{pmatrix}
1 &  c_1 \\
  c_1 &  1 
\end{pmatrix},$$
where $0<c_1<1$. 
\end{remark}

Let us now prove Theorem~\ref{pp}.
 
\begin{proof}
Theorem~\ref{pp} is equivalent, according  Lemma~\ref{le.fan1}, to the existence of $C>0$ and $h_0>0$ such that for all $k\in \big \{1,\dots,{\ft m_{0}^{\Omega}} \big \}$ and for all $h\in (0,h_0)$
\begin{equation}
\label{eq.ChS}
C^{-1} h^{ q_{k}}e^{-\frac{1}{h}\ft S_{k}} \leq  \eta_{{\ft m_{0}^{\Omega}}+1-k}(\widetilde S C_0)\leq C h^{ q_{k}}e^{-\frac{1}{h}\ft S_{k}}\,.
\end{equation}
Let us prove~\eqref{eq.ChS}.
 According to~\eqref{eq.D}  and to the ordering  of $k\in\{1,\dots,{\ft m_{0}^{\Omega}}\}$ introduced in the statement of   
 Theorem~\ref{pp},   the singular values of~$D$ satisfy for $h$ small enough (see~\eqref{eq.D} and~\eqref{eq.Sk-def}):
\begin{equation}
\label{eq.sing-D}
\forall k\in\{1,\dots,{\ft m_{0}^{\Omega}}\},\  \eta_{{\ft m_{0}^{\Omega}}+1-k}(D)= D_{k,k}=h^{q_{k}}e^{-\frac{\ft S_{k}}{h}},
\end{equation}
Using the fact that $\widetilde S C_0=\widetilde C D C_0$ (see~\eqref{eq.C-tilde}) together with~\eqref{eq.C-tilde-borne},~\eqref{eq.C0-tilde-borne} and Lemma \ref{le.fan}, one obtains that for all  $k\in\{1,\dots,{\ft m_{0}^{\Omega}}\}$
\begin{equation}
\label{eq.sens1}
\eta_{{\ft m_{0}^{\Omega}}+1-k}(\widetilde S C_0)\leq \big\|\widetilde C\big\|\,\big\|C_{0}\big\|\,\eta_{{\ft m_{0}^{\Omega}}+1-k}(D) =O\big(D_{k,k}\big),
\end{equation}
which provide the required upper bound in~\eqref{eq.ChS}. To obtain a lower bound on the singular values of $\widetilde SC_0$, we write
$$
D=\widetilde C  ^{-1}\,\widetilde S\,C_{0}\, C_{0} ^{-1}.
$$
Using~\eqref{eq.C-1-tilde-borne},~\eqref{eq.C0-1-tilde-borne} and Lemma \ref{le.fan}, one has for all  $k\in\{1,\dots,{\ft m_{0}^{\Omega}}\}$

\begin{equation}
\label{eq.sens2}
 \eta_{{\ft m_{0}^{\Omega}}+1-k}(D)\leq\big \| \widetilde C ^{-1}\big\| \,\big\| C_{0}^{-1}\big\|\,\eta_{{\ft m_{0}^{\Omega}}+1-k}(\widetilde S\,C_{0})=O\big (\eta_{{\ft m_{0}^{\Omega}}+1-k}\big ).
\end{equation}
Then,~\eqref{eq.ChS} follows from \eqref{eq.sing-D},~\eqref{eq.sens2} and \eqref{eq.sens1}. This concludes the proof of Theorem~\ref{pp}.\end{proof}

To prove Theorem~\ref{thm-big0} and to  ease  upcoming computations, we replace in the Fan inequalities~\eqref{eq.fan1} the matrix $C_0$ by  another matrix which has a simpler form than $C_0$: 
this  is the purpose of Lemma~\ref{le.fan2}. Before stating Lemma~\ref{le.fan2}, let us choose a particular  orthonormal basis $\mathcal B_0$ of $\Ran \, \pi_h^{(0)}$ to define $C_0$ in~\eqref{eq.C0C1}. Recall that the  when the assumption \eqref{eq.hip1-j} is  satisfied  the well~$\ft  C_{1}$ (see Definition~\ref{de.1}) satisfies: for all $x\in \ft U_0^\Omega\setminus\{ x_1\}$,
$$
 f(\mathbf j(x ))-f(x )<f(\mathbf j(x_{1}))-f(x_{1}).
$$
Let us define 
$$e_1:=\frac{ \pi_h^{(0)} \widetilde u_1}{\Vert  \pi_h^{(0)} \widetilde u_1\Vert_{L^2}}.$$
According to item 2a in Proposition~\ref{ESTIME1-base}, there exists $c>0$ such that   in the limit $h\to 0$: 
$$
e_{1}=\big(1+O(e^{-\frac ch})\big)\,\pi^{(0)}_h\widetilde u_1.$$
Then, let us choose $\{e_{2},\dots,e_{{\ft m_{0}^{\Omega}}}\}$ such that 
$$
 \mathcal B_{0}:=\{e_1,e_2,\dots,e_{{\ft m_{0}^{\Omega}}}\}
$$
 is an orthonormal basis of $\Ran \,\pi^{(0)}_h$.  In that case, the matrix   
$   C_0$ defined in~\eqref{eq.C0C1}, satisfies in the limit $h\to 0$:
\begin{equation}
\label{eq.k-l-C0}
\forall k \in\{1,\dots ,{\ft m_{0}^{\Omega}}\},\ \ (  C_{0})_{k,1}=\begin{cases} 1 +O(e^{-\frac ch}) & \text{ if } k=1,\\
0 & \text{ if } k>1.
\end{cases}
\end{equation}
Let us now define the $\ft m_{0}^{\Omega}\times \ft m_{0}^{\Omega}$ matrix $\widetilde C_0$ by:
\begin{equation}
\label{eq.C0-tilde1}
\forall k \in\{1,\dots ,{\ft m_{0}^{\Omega}}\}, \ \ (  \widetilde C_{0})_{k,1}:=\begin{cases} 1 & \text{ if } k=1,\\
0 & \text{ if } k>1,
\end{cases}
\end{equation}
and 
\begin{equation}
\label{eq.C0-tilde2}
\forall (k,\ell)\in\{1,\dots ,{\ft m_{0}^{\Omega}}\}\times\{ 2,\ldots,\ft m_{0}^{\Omega}  \}, \ \ (  \widetilde C_{0})_{k,\ell}:=(   C_{0})_{k,\ell}.
\end{equation}
\begin{lemma}
\label{le.fan2}
Let us assume that the assumptions \eqref{H-M} and~\eqref{eq.hip1-j} are satisfied. Let us denote by~$\lambda_{k,h}$, for $k\in \mathbb N^{*}$,  the $k$-th eigenvalue of $-L_{f,h}^{D,(0)}$ counted with multiplicity and let~$\tilde S$ be the matrix defined in~\eqref{eq.S-tilde}.  Then, there exists $c>0$ such that in the limit $h\to 0$:
$$
\lambda_{k,h}=\frac h2\,  \eta_{\ft m_{0}^{\Omega}+1-k}^2\big( \widetilde S\, \widetilde C_0\big) \big(1+O(e^{-\frac c h})\big),
$$
where   $\|\widetilde S\widetilde C_0\|=\eta_{1}(\widetilde S C_0)\geq\cdots\geq \eta_{ \ft m_0^\Omega}(\widetilde S \widetilde C_0) $ denote the singular values of $\widetilde S \,\widetilde C_0$ and $\lambda_{1,h}=\lambda_{h}$ is the principal eigenvalue of $-L^{D,(0)}_{f,h}$ (see~\eqref{eq.lh}).
\end{lemma}
\begin{proof}
Let us prove that there exists $c>0$ such that in the limit $h\to 0$, 
\begin{equation}\label{eq.1++}
\big \|( \widetilde C_0)^{-1} C_0\big \|=1+O(e^{-\frac c h}) \text{  and  }\big  \|  C_0^{-1}\widetilde C_0 \big \|=1+O(e^{-\frac c h}) .
\end{equation}
From~\eqref{eq.k-l-C0}, the $\ft m_0^\Omega\times \ft m_0^\Omega$ matrix $C_0$ has the form
$$
C_0= \begin{pmatrix}
 1 +O(e^{-\frac ch}) &  [C_0]_4 \\
0 &  [ C_0]_2
\end{pmatrix}
$$
for some $c>0$. 
 Moreover, according to~\eqref{eq.C0-tilde1} and~\eqref{eq.C0-tilde2}, the $\ft m_0^\Omega\times \ft m_0^\Omega$ matrix $\widetilde C_0$ has the form
$$
\widetilde C_0= \begin{pmatrix}
1 &  [ C_0]_4 \\
0 &  [ C_0]_2
\end{pmatrix}.
$$
Let us recall that by definition of $C_0$ (see~\eqref{eq.C0C1}) and from  item 3a in    Proposition~\ref{ESTIME1-base},~$C_0$ is invertible and thus~$[ C_0]_2$ is invertible.  Therefore, one has
 $$
\widetilde C_0^{-1}= \begin{pmatrix}
1  & - [ C_0]_4 \, [ C_0]_2^{-1} \\
0 &  [ C_0]_2^{-1}
\end{pmatrix}
$$
and thus,
$$
\widetilde C_0^{-1}C_0= \begin{pmatrix}
1 +O(e^{-\frac ch})  & 0 \\
0 &  I_{\ft m_0^\Omega-1}
\end{pmatrix}.
$$
This proves~\eqref{eq.1++}.  Lemma~\ref{le.fan2} is then a consequence of~\eqref{eq.1++} together with Lemma~\ref{le.fan1} and Lemma~\ref{le.fan}. \end{proof}

\subsubsection{Proof of Theorem~\ref{thm-big0}}
\label{sec.e3}
This section is dedicated to the proof of Theorem~\ref{thm-big0} which gives the asymptotic estimate of the principal eigenvalue of $-L^{D,(0)}_{f,h}$ under the assumptions \eqref{eq.hip1-j} and \eqref{eq.hip2-j}. 


\begin{proof}[Proof of Theorem~\ref{thm-big0}]
Let us assume that the assumptions \eqref{H-M} and~\eqref{eq.hip1-j}  hold. 
The spectral gap~\eqref{l2} has  already been proved, see Corollary~\ref{eq.co.l2}. It thus remains to prove~\eqref{eq.lambda_h}.  
The proof of~\eqref{eq.lambda_h} is partly inspired by the analysis led in \cite[Section~7.4]{HeHiSj}. According to Lemma~\ref{le.fan2}, there exists $c>0$ such that in the limit $h\to 0$:
\begin{equation}
\label{eq.eta_lambda}
\lambda_{1,h}=\frac h2 \eta^{2}_{{\ft m_{0}^{\Omega}}}(\widetilde S\widetilde C_0) \big(1+O(e^{-\frac ch})\big),
\end{equation}
where $\widetilde C_0$ is defined in~\eqref{eq.C0-tilde1} and~\eqref{eq.C0-tilde2}. 
Therefore, the analysis of the estimate of $\lambda_h$ is then reduced to precisely computing $\eta_{{\ft m_{0}^{\Omega}}}(\widetilde S\widetilde C_0) $. One has:  
\begin{equation}
\label{eq.sing-val}
\eta_{{\ft m_{0}^{\Omega}}}(\widetilde S\widetilde C_0) = \min_{y\in \mathbb R^{{\ft m_{0}^{\Omega}}} \, ;\, \|y\|_2=1}\big  \|\widetilde S\, \widetilde C_0\, y\big \|_2.
\end{equation}
Let us assume in addition to~\eqref{H-M} and \eqref{eq.hip1-j} that \eqref{eq.hip2-j} holds. 
  Recall that \eqref{eq.hip1-j}  and \eqref{eq.hip2-j}   consists in assuming  that    
 for all $x\in \ft U_0^\Omega\setminus\{ x_1\}$,
$$
 f(\mathbf j(x ))-f(x )<f(\mathbf j(x_{1}))-f(x_{1})
$$
and 
$$
\pa\ft  C_{1}\cap\pa\Omega\neq \emptyset.$$
Then,  it holds
 $$\mbf j(x_1)\cap \pa \Omega=\pa\ft  C_{1}\cap\pa\Omega\neq \emptyset \ \text{(see~\eqref{eq.k1-paC1-0})}.$$
Thus, using in addition~Proposition~\ref{pr.S} and~\eqref{eq.S-tilde}, one has in the limit $h\to 0$:
\begin{equation}
\label{eq.fin21}
\sum_{j=1}^{\ft m_1^{\overline \Omega}}  \widetilde S^{2}_{j,1}=\left ( \ \sum_{j:  z_j\in \mbf j(x_1) \cap \pa \Omega}  \widetilde S^{2}_{j,1} + \sum_{j:  z_j\in \mbf j(x_1) \cap   \Omega}  \widetilde S^{2}_{j,1}\ \right)\big (  1+O(e^{-\frac ch}) \big )
\end{equation}
for some $c>0$ independent of $h$ and where
\begin{equation}
\label{eq.fin2}
\begin{cases}
&\sum_{j:  z_j\in \mbf j(x_1) \cap \pa \Omega}  \widetilde S^{2}_{j,1}= h^{-\frac 32}   \left(\sum_{j:  z_j\in \mbf j(x_1) \cap \pa \Omega} C_{j,1}^2  \right)\, e^{-\frac 2h (f(\mathbf j(x_{1}))-f(x_{1}))} (1+O(h)) ,\\
 &\sum_{j: z_j\in \mbf j(x_1)  \cap \Omega} \widetilde S^{2}_{j,1}=h^{-1 }  \left ( \sum_{j: z_j\in \mbf j(x_1)  \cap \Omega} C_{j,1}^2 \right )\, e^{-\frac 2h  (f(\mathbf j(x_{1}))-f(x_{1}))} (1+O(h)),
 \end{cases}
\end{equation}
 where the constants $C_{j,1}$'s are defined in~\eqref{eq.Cip} and where all the remainder terms~$O(h)$  admits a full expansion in~$h$. 
 
Let us first obtain an upper bound on $\eta_{{\ft m_{0}^{\Omega}}}(\widetilde S\widetilde C_0) $. Let us denote by~$y_0$ the vector $\,^t(1,0,\dots,0)$. Then, it holds from~\eqref{eq.sing-val}, 
$$\eta_{{\ft m_{0}^{\Omega}}}(\widetilde S\widetilde C_0)^{2}\ \leq \big\|\widetilde S\,C_{0}\, y_0\big\|_2^2.$$
Using in addition the fact from~\eqref{eq.C0-tilde1}, one has $\widetilde C_0 y_0=y_0$, one obtains 
\begin{equation}\label{eq.est1_etah}
\eta_{{\ft m_{0}^{\Omega}}}(\widetilde S\widetilde C_0)^{2}\ \leq \big\|\widetilde S y_0\big\|_2^2=   \sum_{j=1}^{\ft m_1^{\overline \Omega}} \widetilde S^{2}_{j,1}.
\end{equation}
 This provides the required upper bound. 
 Notice that~\eqref{eq.est1_etah},~\eqref{eq.fin21},  and~\eqref{eq.fin2} imply  that  in the limit $h\to 0$
\begin{equation}\label{eq.est1_etah-bis}
\eta_{{\ft m_{0}^{\Omega}}}(\widetilde S\widetilde C_0)= O\big (h^{-\frac34}  e^{-\frac{1}{h}(f(\mathbf j(x_{1}))-f(x_{1}))}\big).
\end{equation}
Let us now give a lower bound on $\eta_{{\ft m_{0}^{\Omega}}}(\widetilde S\widetilde C_0)$. To this end let us consider $y^* \in \mathbb R^{  \ft m_{0}^{\Omega} } $ with~$\|y^*\|_2=1$, realizing the minimum in~\eqref{eq.sing-val}. Let us write $y^*=\, ^t( y^*_{\alpha},y^*_{\beta})$, where $y^*_{\alpha}\in \mathbb R $ and~$y^*_{\beta}$ is a row vector of size ${\ft m_{0}^{\Omega}}-1$. We claim that  
 there exists $\mu>0$ such that for $h$ small enough,
\begin{equation}
\label{eq.ybeta}
\|y^*_{\beta}\|_2= O \big(e^{-\frac{\mu}{h}} \big).
\end{equation}
 Let us prove~\eqref{eq.ybeta}. By definition of $y^*$ and according to~\eqref{eq.C-tilde}, one has
 $$\eta_{{\ft m_{0}^{\Omega}}}(\widetilde S\widetilde C_0)\, = \big\|\widetilde S \widetilde C_0 y^*\big\|_2= \big\| \widetilde C D  \widetilde C_0 y^*\big\|_2.$$
To prove~\eqref{eq.ybeta}, we use the block structure of the matrices $\widetilde C$,~$\widetilde C_0$ and $D$. Let us recall that from~\eqref{eq.k1-paC1-0} and~\eqref{eq.k1-paC1}, since~\eqref{eq.hip2-j} holds, 
$$
 \ft k_1^{\pa \ft C_1}={\rm Card}\big( \mathbf j(x_{1})\cap \pa \Omega \big)\ge 1.
 $$
Then, according to~\eqref{eq.C-tilde},~\eqref{eq.D} and \eqref{eq.S-tilde}, the $  \ft m_1^{\overline \Omega} \times \ft m_{0}^{\Omega}$ matrix $\widetilde C$ has the form, up to  reordering   the~$z_{i}$,~$i\in\{1,\dots,\ft m_1^{\overline\Omega} \}$,
\begin{equation}
\label{eq.tilde-C}
\widetilde C= 
\begin{pmatrix}
 [\widetilde C] _1 &  0 \\
 [\widetilde C] _3   &  [\widetilde C]_2
\end{pmatrix},
\end{equation}
where:
\begin{itemize} [leftmargin=1.3cm,rightmargin=1.3cm]
\item  $[\widetilde C]_1$  is a matrix of size $ \ft k_1^{\pa \ft C_1}  \times 1$ where we recall that $\ft k_1^{\pa \ft C_1}$ is defined in~\eqref{eq.k1-paC1}. The coefficients $([\widetilde C]_1)_{j,1}=\widetilde C_{j,1}$  are associated with the function~$\widetilde u_{1} $ (see Definition~\ref{de.qm-L} and~\eqref{eq.v1}) and  with  1-forms $\widetilde \psi_{j} $ for $j\in \{1,\ldots,\ft k_1^{\pa \ft C_1}\}$ (or equivalently, $j$ such that $z_j\in \mbf j(x_1)\cap \pa \Omega$).

\item $[\widetilde C]_3$ is  a matrix of size $\big(\ft m_1^{\overline \Omega}- \ft k_1^{\pa \ft C_1} \big) \times 1$ . The coefficients $([\widetilde C]_3)_{j,k}=\widetilde C_{j,k}$  are associated with the function~$\widetilde u_{1} $ and  with  1-forms $\widetilde \psi_{j} $ for $j\in \{\ft k_1^{\pa \ft C_1}+1,\ldots,\ft m_1^{\overline \Omega}\}$ (or equivalently, $j$ such that $z_j\notin \mbf j(x_1)\cap \pa \Omega$).

\item $[\widetilde C]_2$ is  a matrix of size $\big(\ft m_1^{\overline \Omega}-   \ft k_1^{\pa \ft C_1} \big) \times \big( \ft m_0^{\Omega}-1\big) $. The coefficients $([\widetilde C]_2)_{j,k}=\widetilde C_{j,k}$  are associated with  0-forms~$\widetilde u_{k} $,  for $k\in\{2,\ldots,\ft m_0^\Omega\}$  and  with  1-forms~$\widetilde \psi_{j} $ for $j\in \{\ft k_1^{\pa \ft C_1}+1,\ldots,\ft m_1^{\overline \Omega}\}$ (or equivalently, $j$ such that $z_j\notin \mbf j(x_1)\cap \pa \Omega$).
\end{itemize}
From \eqref{eq.C} and \eqref{eq.tilde-C}, 
 $[\widetilde C]_2$ is injective 
and satisfies, for some constant $c>0$ and  for all~$h>0$ small enough,
\begin{equation}
\label{eq.C'}
\forall x\in \mathbb R^{{\ft m_{0}^{\Omega}}-1},\ \ \big\|[\widetilde C]_2 x\big\|_{2}\geq c\big\|x\big\|_{2}.
\end{equation}
This is indeed obvious by applying~\eqref{eq.C}  to the vector $\,^t(0,x)$.  
Let us now decompose the square matrices $D$ and $\widetilde C_0$ in  blocks which are compatible with the decomposition of $\widetilde C$ made in~\eqref{eq.tilde-C}. According to~\eqref{eq.D},~\eqref{eq.C0-tilde1},  and~\eqref{eq.C0-tilde2}, one has
\begin{equation}\label{eq.C0-dec}
  D= 
\begin{pmatrix}
D_{1,1}&  0 \\
0  &  [ D]_\beta
\end{pmatrix} \ \text{ and } \widetilde C_0=\begin{pmatrix}
 1 &  [\widetilde C_0]_\gamma  \\
0  &  [\widetilde C_0]_\beta
\end{pmatrix},
\end{equation}
where for a  square matrix $U$ of size $ \ft m_0^\Omega$:
$$   [U]_{\beta}=(U_{i,j})_{2\leq i,j\leq {\ft m_{0}^{\Omega}}},\ 
\text{and}\ [U]_{\gamma}=(U_{1,j})_{2\leq j\leq {\ft m_{0}^{\Omega}}}.$$
Notice that from~\eqref{eq.D''}, it holds
\begin{equation}\label{eq.D-alpha}
D_{1,1}=h^{-\frac34}  e^{-\frac{1}{h}(f(\mathbf j(x_{1}))-f(x_{1}))},
\end{equation}
and from~\eqref{eq.C0-1-tilde-borne}, there exists $M>0$ such that when $h>0$ 
\begin{equation}\label{eq.C0-alpha}
 \big \Vert [\widetilde  C_{0}]_{\beta} ^{-1} \big \Vert \le M.
\end{equation}
We are now in position to prove~\eqref{eq.ybeta}. 
Le us recall that by definition of $y^*$ , one has
$$\eta_{{\ft m_{0}^{\Omega}}}(\widetilde S\,\widetilde C_{0})=  \big  \|\widetilde S\,\widetilde C_{0}\,^t (y^*_{\alpha},y^*_{\beta} )\big \|_2\geq \big\| \widetilde S\,\widetilde C_{0}\, ^t(0,y^*_{\beta})\big\|_2-\big\|\widetilde S\,\widetilde  C_{0}\, ^t(y^*_{\alpha},0)\big\|_2.$$
Therefore, since $\widetilde C_{0}\,^t(y^*_{\alpha},0)=\, ^t(y^*_{\alpha},0)$ (see~\eqref{eq.C0-tilde1}) and $\widetilde S=\widetilde CD$ (see~\eqref{eq.C-tilde}), one has using~\eqref{eq.est1_etah-bis},~\eqref{eq.tilde-C},  and~\eqref{eq.D-alpha} together with the fact that $|y^*_\alpha |\le 1$ and~$[\widetilde C]_1=O(1)$
(see~\eqref{eq.tilde-C} and~\eqref{eq.C-tilde-borne}), 
\begin{align}
\nonumber
\| \widetilde S\,\widetilde C_{0}\, ^t(0,y^*_{\beta})\|_2 \le \eta_{{\ft m_{0}^{\Omega}}}(\widetilde S\,\widetilde C_{0}) +\big\|\widetilde CD \, ^t(y^*_{\alpha},0)\big\|_2 &\le  \eta_{{\ft m_{0}^{\Omega}}}(\widetilde S\,\widetilde C_{0})+\big\|[\widetilde C]_1D_{1,1} y^*_{\alpha}\big\|_2 \\
\label{eq.S-1}
&=O\big (h^{-\frac34}  e^{-\frac{1}{h}(f(\mathbf j(x_{1}))-f(x_{1}))}\big).
\end{align}
Moreover, using~\eqref{eq.D-alpha} and since $[\widetilde C]_3=O(1)$ (see~\eqref{eq.tilde-C} and~\eqref{eq.C-tilde-borne}) and $[\widetilde C_{0}]_{\gamma}=O(1)$ (since $\widetilde C_0=C_0+O(e^{-\frac ch})$ and $C_0=O(1)$ see~\eqref{eq.C0-tilde1},~\eqref{eq.C0-tilde2},  and~\eqref{eq.C0-tilde-borne}), one has 
\begin{align*}
\big  \|  \widetilde S\,\widetilde C_{0}\, ^t(0,y^*_{\beta})\big \|_2&=
\left(  \big  \|[\widetilde C]_1\,D_{1,1} \, [\widetilde C_{0}]_{\gamma}\,^ty^*_{\beta} \big \|^{2}_{2} + \big \|  [\widetilde C]_3\,D_{1,1} \, [\widetilde C_{0}]_{\gamma}\,^ty^*_{\beta} +  [\widetilde C]_2\, [D]_{\beta}\, [\widetilde C_{0}]_{\beta}\,^ty^*_{\beta} \big\|^{2}_{2}\right)^{\frac12}\\
&\geq
\big \|[\widetilde C]_2\,  [D]_{\beta}\,[\widetilde C_{0}]_{\beta}\, ^ty^*_{\beta}\big \|_2-\big \| [\widetilde C]_3\, D_{1,1} \, [\widetilde C_{0}]_{\gamma}\,^ty^*_{\beta}\big\|_2\\
&=\big  \|[\widetilde C]_2\,  [D]_{\beta}\,[\widetilde C_{0}]_{\beta}\, ^ty^*_{\beta} \big \|_2+ O\big(h^{-\frac34}  e^{-\frac{1}{h}(f(\mathbf j(x_{1}))-f(x_{1}))}\big ).
\end{align*}
Therefore, one deduces  from the latter inequality and from~\eqref{eq.S-1} and~\eqref{eq.C'} that
\begin{align}\label{eq.D-1}
\big \| [D]_{\beta}\,[\widetilde C_{0}]_{\beta}\, y^*_{\beta}\big\|_2= O\left (\|[\widetilde C]_2[D]_{\beta}\,[C_{0}]_{\beta}\, ^ty^*_{\beta}\|_2\right) =O\big(h^{-\frac34}  e^{-\frac{1}{h}(f(\mathbf j(x_{1}))-f(x_{1}))}\big),
\end{align}
In addition, since $  [\widetilde C_{0}]_{\beta}^{-1}= O(1)$ (which follows from~$\widetilde C_0=C_0+O(e^{-\frac ch})$, see indeed~\eqref{eq.C0-tilde1},~\eqref{eq.C0-tilde2},~\eqref{eq.C0-1-tilde-borne},  and~\eqref{eq.C0-dec}) and since there exists $c>0$ such that it holds: 
$$
[D]_{\beta} ^{-1}
= O\big (e^{\frac{1}{h}(f(\mathbf j(x_{1}))-f(x_{1})-c)}\big),$$
which follows from~\eqref{eq.D} and \eqref{eq.D'''}, one obtains from~\eqref{eq.D-1} that  
there exists $\mu>0$ such that for $h$ small enough,
$$
\| ^ty^*_{\beta}\|_2= O(e^{-\frac{\mu}{h}}).
$$
This ends the proof of~\eqref{eq.ybeta}.  We are now in position to give a lower bound on $\eta_{{\ft m_{0}^{\Omega}}}(\widetilde S\widetilde C_0)$. 
Notice that from~\eqref{eq.ybeta} together with the fact that $\|y^* \|_2=1$,  one has
\begin{equation}\label{eq.yalpha}
\vert y^*_{\alpha}\vert = 1+ O(e^{-\frac{\mu}{h}}).
\end{equation}
Using~\eqref{eq.tilde-C} and~\eqref{eq.C0-dec}, there exists $c>0$ such that
$$\eta^{2}_{{\ft m_{0}^{\Omega}}}(\widetilde S\widetilde C_0) \geq \sum_{j=1}^{\ft k_1^{\pa \ft C_1}}(\widetilde C\,D\,\widetilde C_{0}\, y^*)^{2}_{j} =  \sum_{j=1}^{\ft k_1^{\pa \ft C_1}}
 D_{1,1}^{2} \widetilde C_{j,1}^{2}\left(y^*_{\alpha}  +\sum_{\ell=2}^{{\ft m_{0}^{\Omega}}}(\widetilde C_{0})_{1,\ell}\, y^*_{\ell}\right)^{2},$$
where we recall that $\ft k_1^{\pa \ft C_1}$ is defined by~\eqref{eq.k1-paC1}. 
Using in addition~\eqref{eq.ybeta} and~\eqref{eq.yalpha} together with the fact that $\widetilde C_{0}=O(1)$, there exists $c>0$ such that in the limit $h\to 0$: 
$$
\eta^{2}_{{\ft m_{0}^{\Omega}}} (\widetilde S\widetilde C_0) \geq   D_{1,1}^{2}\sum_{j=1}^{\ft k_1^{\pa \ft C_1}}\widetilde  C_{j,1}^2  \big(1+ O (e^{-\frac{c}{h}})\big)^{2}.$$
By definition of~$\ft k_1^{\pa \ft C_1}$ (see~\eqref{eq.k1-paC1}) it holds
\begin{align*}
D_{1,1}^{2}\sum_{j=1}^{\ft k_1^{\pa \ft C_1}}\widetilde  C_{j,1}^2  \big(1+ O (e^{-\frac{c}{h}})\big)^{2}&=  \sum_{j:\, j\in \mbf j(x_1)\cap \pa  \Omega} \widetilde   C_{j,1}^2D_{1,1}^2 \big(1+ O (e^{-\frac{c}{h}})\big)^{2}\\
&\quad= \sum_{j: \,j\in \mbf j(x_1)\cap \pa  \Omega} \widetilde  S_{j,1}^2 \big(1+ O (e^{-\frac{c}{h}})\big)^{2},
\end{align*}
where the last equality follows from~\eqref{eq.C-tilde}.
Thus, one obtains the following lower bound:
\begin{equation}
\label{eq.est2_etah}
\eta^{2}_{{\ft m_{0}^{\Omega}}} (\widetilde S\widetilde C_0) \geq \sum_{j: \,j\in \mbf j(x_1)\cap \pa  \Omega} \widetilde  S_{j,1}^2 \big(1+ O (e^{-\frac{c}{h}})\big)^{2}.
\end{equation}
In conclusion, from~\eqref{eq.est1_etah} and \eqref{eq.est2_etah}, one has for some $c>0$, in the limit $h\to 0$:
\begin{equation}
\label{eq.fin1}
\sum_{j: \, j\in \mbf j(x_1)\cap \pa  \Omega}\!\!\widetilde  S_{j,1}^2 \big(1+ O (e^{-\frac{c}{h}})\big)^{2}\leq \eta_{{\ft m_{0}^{\Omega}}}^{2}(\widetilde S\widetilde C_0) \leq
\sum_{j=1}^{\ft m_1^{\overline \Omega}} \widetilde S^{2}_{j,1}.
\end{equation}
Using~\eqref{eq.fin21}  and~\eqref{eq.fin2}, one gets 
\begin{equation*}
\sum\limits_{j=1}^{\ft m_1^{\overline \Omega}} \widetilde S^{2}_{j,1}=
\begin{cases}
\sum\limits_{j: j\in \mbf j(x_1)\cap \pa  \Omega} \widetilde  S_{j,1}^2  & \text{ if } \mathbf j(x_{1})\cap \Omega=\emptyset,\\
\sum\limits_{j: j\in \mbf j(x_1)\cap \pa  \Omega} \widetilde  S_{j,1}^2 \big(1+O(\sqrt h)\big) & \text{ if } \mathbf j(x_{1})\cap \Omega\neq\emptyset.
\end{cases}
\end{equation*}
Thus, since $\lambda_{1,h}=\lambda_h$, using in addition~Proposition~\ref{pr.S},~\eqref{eq.eta_lambda},  and~\eqref{eq.fin1}, it holds int the limit $h\to 0$:
\begin{equation} \label{eq.besoin}
\lambda_h= \begin{cases}\frac h2\displaystyle  \sum\limits_{j: j\in \mbf j(x_1)\cap \pa  \Omega} \big \lp       \nabla \widetilde u_1,    \widetilde \psi_j \big \rp_{L^2_w}^2 \big(1+ O (e^{-\frac{c}{h}})\big) &\!\!\!\! \!\!\text{ if }  \mathbf j(x_{1})\cap \Omega =\emptyset,\\
\frac h2 \displaystyle  \sum\limits_{j: j\in \mbf j(x_1)\cap \pa  \Omega} \big \lp       \nabla \widetilde u_1 ,    \widetilde \psi_j \big \rp_{L^2_w}^2\big(1+O(\sqrt h)\big) &\!\!\!\!\!\!\text{ if }  \mathbf j(x_{1})\cap \Omega \neq\emptyset.
\end{cases}
\end{equation}

Then,~\eqref{eq.besoin} together with   Proposition~\ref{ESTIME1} and the fact that
$$\{z_1,\ldots,z_{\ft k_1^{\pa \ft C_1}}\}=\mbf j(x_1)\cap \pa \Omega, \ (\text{see~\eqref{eq.k1-paC1},  and~\eqref{eq.k1-paC1-0}}),$$
 imply when $h\to 0$:
 \begin{equation*} 
\lambda_h= \begin{cases}\frac h2 \displaystyle \sum\limits_{j: j\in \mbf j(x_1)\cap \pa  \Omega}C_{j,1} ^2h^{-\frac 32}\,  e^{-\frac{2}{h}(f(\mathbf{j}(x_{1}))- f(x_1))} \big(1+ O (h)\big) &\!\!\!\! \!\!\text{ if }  \mathbf j(x_{1})\cap \Omega =\emptyset,\\
\frac h2 \displaystyle \sum\limits_{j: j\in \mbf j(x_1)\cap \pa  \Omega}C_{j,1} ^2h^{-\frac 32}\,  e^{-\frac{2}{h}(f(\mathbf{j}(x_{1}))- f(x_1))}  \big(1+O(\sqrt h)\big) &\!\!\!\!\!\!\text{ if }  \mathbf j(x_{1})\cap \Omega \neq\emptyset.
\end{cases}
\end{equation*}
Recall that \eqref{eq.hip4-j} consists in assuming that $\mathbf j(x_{1})\cap \Omega =\emptyset$.  
This concludes   the proof of Theorem~\ref{thm-big0}.
\end{proof}  

\section{On the principal eigenfunction of $-L_{f,h}^{D,(0)}$}
\label{section-4}
This section is dedicated to the proof of Proposition~\ref{pr.masse} and Theorem~\ref{thm-big-pauh} stated below which gives respectively   the asymptotic behaviour in the limit $h\to 0$ of $\displaystyle \int  _{\Omega} u_h \ e^{- \frac{2}{h} f } $ and $\pa_nu_h$ on~$\pa \Omega$. 
%

Proposition~\ref{pr.masse} gives a sufficient condition to obtain that $u_h\,e^{-\frac 2h f}$  (and thus the quasi-stationnary distribution $\nu_h$, see Proposition~\ref{uniqueQSD})  concentrates in only one of the wells $(\ft C_k)_{k\in \{1,\ldots,\ft N_1\}}$ when $h\to 0$ in the   $L^1(\Omega)$-norm. 
 
\begin{proposition}
\label{pr.masse}
Assume that the assumptions \eqref{H-M}  and \eqref{eq.hip1-j}  are satisfied.  Let us moreover assume that 
$$\min_{\overline{\ft C_1}}f= \min_{\overline \Omega}f.$$
 Let~$u_h$ be  the eigenfunction associated with the principal eigenvalue $\lambda_h$ of $-L_{f,h}^{D,(0)}$ (see~\eqref{eq.lh})  which satisfies~\eqref{eq.norma}. Let $\ft O$ be an open subset of $\Omega$. 
On the one hand, if 
 $$\ft O\cap \argmin_{\ft C_1}f\neq \emptyset,$$
 one has in the limit $h\to 0$:
 \begin{equation}\label{concentration01}
\int  _{\ft O} u_h \ e^{- \frac{2}{h} f } =h^{\frac{d}{4} }\, \pi^{\frac{d}{4} }\frac{   \sum_{x\in\ft O\cap \argmin_{\ft C_1}f}  \big( {\rm det \ Hess } f   (x)   \big)^{-\frac12}  }{ \Big(\sum_{x\in \argmin_{\ft C_1}f}  \big( {\rm det \ Hess } f   (x)   \big)^{-\frac12}\Big)^{\frac12} }\ e^{-\frac{1}{h} \min_{\overline \Omega}f} \  \big(1+O(h) \big).
 \end{equation}
On the other hand, if
$$\overline{\ft O}\cap \argmin_{\ft C_1}f= \emptyset,$$
 then, there exists $c>0$ such that  when $h \to 0$:
\begin{equation}\label{eq.concentration02}
\int  _{\ft O} u_h \ e^{- \frac{2}{h} f } =O\big ( e^{-\frac{1}{h}( \min_{\overline \Omega}f+c)} \big).
  \end{equation}
\end{proposition}
\noindent
When  \eqref{H-M}  and \eqref{eq.hip1-j}  are satisfied and  when  $\min_{\overline{\ft C_1}}f= \min_{\overline \Omega}f$ holds,  Proposition~\ref{pr.masse} implies that     when $h\to 0$, $u_h\,e^{-\frac 2h f}$ concentrates in the   $L^1$-norm      on the global minima of~$f$ in~$\ft C_{\ft{max}}$.   
Proposition~\ref{pr.masse} together with~\eqref{eq.expQSD} and the fact that $\ft C_1=\ft C_{ \ft{max}  }$ when  \eqref{eq.hip1-j} holds (or equivalently \eqref{eq.hip1}, see Lemma~\ref{equiv-hipo}), imply Proposition~\ref{pr.con}.
Notice that when $\ft O=\Omega$ in Proposition~\ref{pr.masse}, one has from~\eqref{concentration01},  when $h\to 0$:
\begin{equation}\label{eq.concentration1}
\int  _{\Omega} u_h \ e^{- \frac{2}{h} f } = 
h^{\frac{d}{4} }\, \pi^{\frac{d}{4} }e^{-\frac{1}{h} \min_{\overline \Omega}f} 
  \Big(\sum_{x\in\argmin_{ \ft C_1}f}  \big( {\rm det \ Hess } f   (x)   \big)^{-\frac12}\Big)^{\frac12} \ \big(1+O(h) \big).
  \end{equation}

The following theorem shows that, under  the  hypotheses~\eqref{eq.hip1-j}, \eqref{eq.hip2-j},  and~\eqref{eq.hip3-j},     the $L^1_w(\pa \Omega)$-norm of the normal derivative of the principal eigenfunction of~$-L_{f,h}^{D,(0)}$ concentrates when $h\to 0$ on $\pa \Omega \cap \pa \ft C_1$.
\begin{theorem} \label{thm-big-pauh}
Let us assume that the assumptions~\eqref{H-M},~\eqref{eq.hip1-j}, \eqref{eq.hip2-j} and \eqref{eq.hip3-j} are satisfied. Let $u_h$ be  the eigenfunction associated with the principal eigenvalue $\lambda_h$ of $-L_{f,h}^{D,(0)}$  which satisfies~\eqref{eq.norma}.
Let $F\in L^{\infty}(\partial \Omega,\mathbb R)$ and $\Sigma$ be  an open subset of $\pa\Omega$.
\begin{itemize} [leftmargin=1.3cm,rightmargin=1.3cm]
\item[(i)] When $\overline \Sigma \cap \{z_{1},\dots,z_{\ft k_{1}^{\pa \Omega}}\}=\emptyset$, one has  in the limit $h\to 0$:
$$
\int_{\Sigma}   F   \, \partial_{n}u_h \,   e^{-\frac{2}{h}f }   = 
O\left(e^{-\frac{1}{h}\big (2\min_{\pa \Omega}f-\min_{\overline{\Omega}}f+c\big)} \right ),
$$
where $c>0$ is independent of $h$.
\item[(ii)] When $\overline \Sigma \cap \{z_{1},\dots,z_{\ft k_{1}^{\pa\ft  C_1}}\}=\emptyset$, one has  in the limit $h\to 0$:
$$
\int_{\Sigma}   F   \, \partial_{n}u_h\,   e^{-\frac{2}{h}f }   = 
O\left ( h^{\frac{d-6}{4}}   e^{-\frac{1}{h}\big (2\min_{\pa \Omega}f-\min_{\overline{\Omega}}f\big)}\,\sqrt{\ve_{h}}  \right ),
$$
where, for some $c>0$ independent of $h$,\label{page.veh}
 \begin{equation}
 \label{eq.epsilon-h}
 \ve_h=\begin{cases} \sqrt h &\text{ or },\\
 e^{-\frac ch} &\text{ if  (\ref{eq.hip4-j}) is satisfied.}
\end{cases}
\end{equation} 
 \item[(iii)] When,  for some $i\in \{1,\dots,\ft k_1^{\pa \ft C_1}\}$,~$\overline \Sigma \cap \{z_{1},\dots,z_{\ft k_{1}^{\pa \ft C_1}}\}=\{z_{i}\}$,
$z_{i}\in {\Sigma}$,
and
$F$ is~$C^{\infty}$ in a neighborhood  of $z_i$,  one has 
 in the limit $h\to 0$:
\begin{equation*}
 \int_{\Sigma}   F   \, \partial_{n}u_h\,   e^{-\frac{2}{h}f}   =  -\big( F(z_i)+ O(\sqrt{\ve_{h}}) + O(h) \big)C_{i,1} B_{i}  \ h^{\frac{d-6}{4}}  \,  e^{-\frac{1}{h}\big (2\min_{\pa \Omega}f-\min_{\overline{\Omega}}f\big)},  
  \end{equation*} 
  where $\ve_{h}$ satisfies \eqref{eq.epsilon-h}  and the constants $B_{i}$ and $C_{i,1}$ are defined in \eqref{eq.Cip}--\eqref{eq.Bi}.
\end{itemize}
 \end{theorem}
The following  rewriting  of Theorem~\ref{thm-big-pauh} will be useful to prove Theorem~\ref{thm.main}.
Assume that the assumptions \eqref{H-M}, \eqref{eq.hip1-j}, \eqref{eq.hip2-j},  and~\eqref{eq.hip3-j} are satisfied. 
Let~$F\in L^{\infty}(\partial \Omega,\mathbb R)$ and $(\Sigma_{i})_{i\in\{1,\dots,\ft k_{1}^{\pa \Omega} \}}$ be a family  of disjoint open subsets of $\pa \Omega$ such that 
$$ \text{for all } i\in\{1,\dots,\ft k_{1}^{\pa\Omega}\},\  z_{i}\in \Sigma_{i},$$  where we recall that  $\{z_1,\dots,z_{\ft k_{1}^{\pa\Omega}}\}= \ft U_1^{\pa \Omega}\cap \argmin_{\pa \Omega} f$ (see~\eqref{eq.z11}). Then:
\begin{enumerate} [leftmargin=0.8cm,rightmargin=0.5cm]

\item There exists $c>0$ such that in the limit~$h\to 0$,
\begin{equation}\label{eq.dnuh0}
 \int_{\partial \Omega} \!F   \, \partial_n u_h \,   e^{- \frac{2}{h}  f} =
\sum_{i=1}^{\ft k_{1}^{\pa \Omega}}\int_{\Sigma_{i}}\! F  \, \partial_n u_h \,   e^{- \frac{2}{h}  f} +
O\left( e^{-\frac{1}{h}\big(2\min_{\pa \Omega}f-\min_{\overline{\Omega}}f+c\big)}\right ),
\end{equation}
and
\begin{equation}\label{eq.dnuh1}
\sum_{i=\ft k_{1}^{\pa \ft C_1}+1}^{\ft k_{1}^{\pa \Omega}}\int_{\Sigma_{i}} \!F\,  \partial_n u_h \,  e^{- \frac{2}{h}  f} =
O\left(  h^{\frac{d-5}{4}}   e^{-\frac{1}{h}\big(2\min_{\pa \Omega}f-\min_{\overline{\Omega}}f\big)}\right ),
\end{equation}
with the convention~$\sum \limits_{i=n}^{m}=0$ if $n>m$ and where we recall that (see~\eqref{eq.k1-paC1-0}, \eqref{eq.k1-paC1} and \eqref{eq.hip3-j}),
$$ \{z_1,\dots,z_{\ft k_{1}^{\pa \ft C_1}}\}=\pa \ft C_{1}\cap \pa \Omega \subset \argmin_{\pa \Omega}f\cap \ft U_1^{\pa \Omega}.$$
The asymptotic  estimate~\eqref{eq.dnuh0} follows from item $(i)$ in Theorem~\ref{thm-big-pauh}  taking $\Sigma=\pa \Omega\setminus \cup_{i=1}^{\ft k_1^{\pa \Omega}} \Sigma_i$, while~\eqref{eq.dnuh1} follows from item $(ii)$ in Theorem~\ref{thm-big-pauh}  taking $\Sigma=  \cup_{i=\ft k_1^{\pa \ft C_1}+1}^{\ft k_1^{\pa \Omega}} \Sigma_i$. 
\item 
Moreover, when, for some $i\in\{1,\dots,\ft k_{1}^{\pa \ft C_1}\}$,~$F$ is $C^{\infty}$ in a neighborhood  of $z_{i}$, one has  in the limit $h\to 0$:
\begin{equation}\label{eq.dnuh}
\int_{\Sigma_{i}} F \, \partial_n u_h\,  e^{- \frac{2}{h}  f} = A_i\,
\left ( F(z_i) + O\big (h^{\frac14}\big ) \right )\,  h^{\frac{d-6}{4}}\,e^{-\frac{1}{h}\big(2\min_{\pa \Omega}f-\min_{\overline{\Omega}}f\big)},
\end{equation}
 where
\begin{equation}\label{eq.Ai-pa-uh}
A_i=-2\,\partial_nf(z_i)\,\pi^{\frac{d-2}{4}}\,\left(  {\rm det \ Hess } f|_{\partial \Omega}   (z_i)\!\!\!\!\!\! \!\!\!\!\!\sum_{x\in \argmin_{\ft C_1}f}  \!\!\!\!\! \big( {\rm det \ Hess } f   (x)   \big)^{-\frac12}  \right)^{-\frac12} .
\end{equation}
This asymptotic equivalent   follows from item~$(iii)$ in Theorem~\ref{thm-big-pauh}  taking $\Sigma=   \Sigma_i$ for some $i\in \{1,\dots,\ft k_1^{\pa \ft C_1}\}$. 
\item
\begin{sloppypar}
Lastly, when \eqref{eq.hip4-j}  (i.e when~$\mbf j(x_1)\subset \pa  \Omega$), the remainder term~$O\big(  h^{\frac{d-5}{4}}   e^{-\frac{1}{h}(2f(z_{1})- f(x_1))}\big)$ in \eqref{eq.dnuh1}
is  of the order $O\big (e^{-\frac{1}{h}\big(2\min_{\pa \Omega}f-\min_{\overline{\Omega}}f+c\big)}\big )$ for some $c>0$ 
and
the remainder term~$O\big (h^{\frac14}\big )$ in~\eqref{eq.dnuh} is of the order $O(h)$ and   admits a full asymptotic expansion in~$h$. 
\end{sloppypar}
\end{enumerate}
%
%

 According to  Theorem~\ref{thm-big-pauh}, when the function $F$ belongs to $C^{\infty}(\pa \Omega,\mathbb R)$, one has the following equivalent of~\eqref{eq.dnuh0}  in the limit $h\to 0$:
$$
 \int_{\partial \Omega} \!F   \, \partial_n u_h \,   e^{- \frac{2}{h}  f} =
\sum_{i=1}^{\ft k_{1}^{\pa \ft C_1} } A_i\,
\big ( F(z_i) + O(h^{\frac14}) \big )\,  h^{\frac{d-6}{4}}\,e^{-\frac{1}{h}\big(2\min_{\pa \Omega}f-\min_{\overline{\Omega}}f\big)}.
$$
\begin{remark}\begin{sloppypar} 
When the assumption \eqref{eq.hip4-j} is not satisfied,  the remainder terms  in~\eqref{eq.dnuh1}  and    \eqref{eq.dnuh} may not be optimal. 
In~\cite[Section C.4.2.2]{BN2017}, it is proved with a one-dimensional example, that when the assumption \eqref{eq.hip4-j} is not satisfied,    the optimal  remainder term  in~\eqref{eq.dnuh1} is $O\big (  h^{\frac{d-4}{4}}   e^{-\frac{1}{h}\big(2\min_{\pa \Omega}f-\min_{\overline{\Omega}}f\big)}\big )$ and  the optimal  remainder term   in \eqref{eq.dnuh} is~$O(\sqrt h)$.  In higher-dimension, these optimal remainder terms can be obtained in some specific cases, see~\cite[Proposition C.40]{BN2017}.   
 \end{sloppypar} 
\end{remark}
This section is organized as follows. In Section~\ref{sec.masse}, one proves Proposition~\ref{pr.masse}. Section~\ref{sec.thm-big} is then dedicated to the proof of Theorem~\ref{thm-big-pauh}.

\subsection{Proof of Proposition~\ref{pr.masse}}
\label{sec.masse}

This section is dedicated to the proof Proposition~\ref{pr.masse}.  Let us first give a corollary of Theorem~\ref{pp} which is used in the proof of Proposition~\ref{pr.masse}. 

\begin{corollary}\label{co.thm4}
Let us assume that the assumptions~\eqref{H-M} and \eqref{eq.hip1-j} are satisfied. Then, there exists $\beta_0>0$ such that for all $\beta\in (0,\beta_0)$, there exists $h_0>0$ such that  for all $h\in (0,h_0)$, the orthogonal projector \label{page.tildephio}
$$\widetilde  \pi^{(0)}_h:=\pi_{\big [0,e^{-\frac{2}{h}( f(\mbf j(x_1))-f(x_{1})-\beta)}\big )}(-L^{D,(0)}_{f,h}) \text{ has rank } 1.$$
Moreover, choosing the parameter $\ve>0$ appearing  in~\eqref{eq.v1}  small enough, there exists $h_0>0$ such that  for all $h\in (0,h_0)$, one has:
\begin{equation}\label{eq.spa-pi-tilde}
\Ran\, \widetilde  \pi^{(0)}_h=\sspan \big (  \widetilde \pi^{(0)}_h \widetilde u_1\big ),
\end{equation}
where the function $ \widetilde u_1$ is  introduced in Definition~\ref{de.qm-L}.
\end{corollary}

\begin{proof}
The fact that $\dim \Ran\, \widetilde  \pi^{(0)}_h=1$ is a direct consequence of Corollary~\ref{eq.co.l2}. Let us now prove~\eqref{eq.spa-pi-tilde}.  Using Lemma~\ref{quadra}, Proposition~\ref{fried} and using  item 2a of  Proposition~\ref{ESTIME1-base}, for any   $\delta>0$, there exist   $\ve>0$ (see~\eqref{eq.v1}), $C>0$  and $h_0>0$ such that one has  for all 
$h\in (0,h_0)$, 
\begin{align}
\nonumber
  \big \|    (1-\widetilde  \pi^{(0)}_h ) \widetilde u_1\big\|_{L^2_w}^2  &\leq  e^{\frac{2}{h}(f(\mbf j(x_1))-f(x_{1}) -\beta )} \, \frac h2\big\|   \nabla  \widetilde u_1\big\|_{L^2_w}^2 \\
  \nonumber
 &\leq C\,  e^{\frac{2}{h}(f(\mbf j(x_1))-f(x_{1})-\beta )}\,e^{-\frac{2}{h}(f(\mbf j(x_1))-f(x_{1})-\delta )} \\
 \label{eq.1-pi}
 &\leq   C\,  e^{-\frac{2}{h}(\beta- \delta)}.
 \end{align}
Therefore, choosing $\ve>0$ small enough such that $\delta \in (0,\beta)$,
there
exists $c>0$ and $h_0>0$ such that one has  for all 
$h\in (0,h_0)$, 
\begin{equation}\label{normeu}
 \big \|    \widetilde  \pi^{(0)}_h \widetilde u_1\big\|_{L^2_w}= 1+O ( e^{-\frac{c}{h}}).
\end{equation}
This concludes the proof of~\eqref{eq.spa-pi-tilde} and thus the proof of Corollary~\ref{co.thm4}. 
\end{proof}

Let us now prove Proposition~\ref{pr.masse}.
\begin{proof}[Proof of  Proposition~\ref{pr.masse}]
Let us first assume that only the assumptions~\eqref{H-M} and \eqref{eq.hip1-j} are satisfied. Let us recall that $u_h$ is the eigenfunction associated with the principal eigenvalue $\lambda_h$ of $-L_{f,h}^{D,(0)}$ (see~\eqref{eq.lh}) which satisfies~\eqref{eq.norma}.
As a direct consequence of Corollary~\ref{co.thm4} and~\eqref{normeu}, one has since
 the functions $ u_{h}$ and $\widetilde u_{1}$ are non negative,
\begin{equation}\label{eq.uh=}
u_h= \frac{\widetilde  \pi^{(0)}_h \widetilde u_{1}}{\| \widetilde  \pi^{(0)}_h \widetilde u_{1} \|_{L^2_w}}=\widetilde u_1+O ( e^{-\frac{c}{h}})\ \text{  in  } L^2_w(\Omega).
\end{equation} 
Let $\ft O$ be an open subset of $\Omega$. Using \eqref{eq.uh=} and thanks to the Cauchy-Schwarz inequality, one obtains  in the limit $h\to 0$:
\begin{align}
\nonumber
\int  _{\ft O}u_h \ e^{-\frac{2}{h} f} &= \int  _{\ft O}  \widetilde u_{1} \,e^{-\frac{2}{h} f} + O ( e^{-\frac{c}{h}})  \sqrt{  \int  _{\ft O}  e^{-\frac{2}{h} f} }\\
\label{eq.equa0}
&= \int  _{\ft O}  \widetilde u_{1} \,e^{-\frac{2}{h} f} + O \Big ( e^{-\frac{1}{h}\, \big (\min_{\overline \Omega} f + c\big )}\Big).
\end{align}
Let us  recall that by construction (see Definition~\ref{de.qm-L} and~\eqref{eq.v1}),
$$\widetilde u_{1} =\frac{\chi_{1}^{\ve,\ve_1}}{\|\chi_{1}^{\ve,\ve_1}\|_{L^{2}_{w}}}.$$
 Then, 
from the definition of $\chi_{1}^{\ve,\ve_1}$ (see~\eqref{eq.v1} and the lines below)
and  using Laplace's method, one has
  in the limit $h\to 0$,
\begin{equation}
\label{eq.er7}
  \int  _{\Omega} (\chi_{1}^{\ve,\ve_1})^2\, e^{-\frac{2}{h} f }  = (h \pi) ^{\frac{d}{2} }e^{-\frac{2}{h} f(x_1)}  \!\! \!\!\!\!\sum_{x\in\argmin_{\ft C_{1}}f }\!\!\!\!\! \!\! \!\!\big( {\rm det \ Hess } f   (x)   \big)^{-\frac12}    (1+O(h) )
\end{equation}
Let us assume that  $$\ft O\cap \argmin_{\ft C_1}f\neq \emptyset.$$
Then, using Laplace's method, one has when $h\to 0$,
\begin{equation}
\label{eq.er7'}
  \int  _{\ft O} \chi_{1}^{\ve,\ve_1} \, e^{- \frac{2}{h}f }  =(h \pi) ^{\frac{d}{2} }e^{-\frac{2}{h} f(x_1)} \!\!\!\!\!\!\! \sum_{x\in\ft O\cap \argmin_{\ft C_{1}}f}  \!\!\!\!\!\big( {\rm det \ Hess } f   (x)   \big)^{-\frac12}    (1+O(h) ),
\end{equation}
where  we recall that $x_1\in \argmin_{\ft C_{1}}f$. 
Thus, from \eqref{eq.equa0},~\eqref{eq.er7},  and~\eqref{eq.er7'}, one has when $h\to 0$:
\begin{align}
\nonumber
\int  _{\ft O} u_h \ e^{- \frac{2}{h} f } &= h^{\frac{d}{4} }\, \pi^{\frac{d}{4} }\frac{   \sum_{x\in\ft O\cap \argmin_{\ft C_1}f}  \big( {\rm det \ Hess } f   (x)   \big)^{-\frac12}  }{ \Big(\sum_{x\in \argmin_{\ft C_1}f}  \big( {\rm det \ Hess } f   (x)   \big)^{-\frac12}\Big)^{\frac12} }\ e^{-\frac{1}{h} f(x_1)} \  \big(1+O(h) \big)\\
\label{last-e}
&\quad+  O \left ( e^{-\frac{1}{h}\big (\min_{\overline \Omega} f + c\big )}\right).
\end{align}
Let us  assume moreover  that 
$$\min_{\overline{\ft C_1}}f=\min_{\overline \Omega}f.$$
Then,~\eqref{concentration01} in Proposition~\ref{pr.masse} is a consequence of~\eqref{last-e}. Let us now consider the case where
  $$\overline{\ft O}\cap \argmin_{\ft C_1}f=\emptyset.$$
Then, it holds 
\begin{equation}
\label{eq.plus haut}
\min_{\overline {\ft O}\cap  \overline {\ft C_1} } f>\min_{\overline{\ft C_1}} f=\min_{\overline \Omega} f.
\end{equation}
Since in the limit $h\to 0$:
$$
 \int  _{\ft O} \chi_{1}^{\ve,\ve_1} \, e^{- \frac{2}{h}f }  =O\big ( e^{- \frac{2}{h}\min_{  \overline {\ft O}\cap  \overline {\ft C_1} }f}\big),
$$
one obtains using~\eqref{eq.plus haut},~\eqref{eq.equa0},  and~\eqref{eq.er7}, that there exist $c>0$ and $\tilde c>0$ such that when $h\to 0$:
\begin{align}
\nonumber
\int  _{\ft O} u_h \ e^{- \frac{2}{h} f } &= O\Big ( h^{-\frac d4}\, e^{- \frac{2}{h}\min_{\overline {\ft O\cap  \ft C_1}} f }\, e^{\frac{1}{h}\min_{\overline \Omega} f  }\Big) +  O \Big ( e^{-\frac{1}{h}\big (\min_{\overline \Omega} f + c\big )} \Big )\\
&=O \Big ( e^{-\frac{1}{h}\big (\min_{\overline \Omega} f + \tilde c\big )} \Big ).
\end{align} 
This proves~\eqref{eq.concentration02} and concludes the proof of proposition~\ref{pr.masse}.
\end{proof}

\subsection{Proof of Theorem~\ref{thm-big-pauh}}
\label{sec.thm-big}

Let us briefly explain
the strategy for the proof of Theorem~\ref{thm-big-pauh}.
The basic idea is to notice that, since $\nabla u_{h}$ belongs to $\Ran\, \pi_{h}^{(1)}$ (according to \eqref{eq.nablain}),
one has for any open set $\Sigma$ of $\pa\Omega$ and for any $L_{w}^{2}$-orthonormal basis $(\psi_{1},\dots,\psi_{\ft m_1^{\overline \Omega}})$ of $\Ran\, \pi_{h}^{(1)}$,
\begin{equation}
\label{eq.rere}
 \int_{\Sigma}   F   \, \partial_{n}u_h\,  \   e^{-\frac{2}{h}f}  =\sum_{i=1}^{\ft m_1^{\overline \Omega} }
 \langle \nabla u_{h},\psi_{i}\rangle_{L^{2}_{w}}
 \int_{\Sigma}   F    \, \psi_{i}\cdot n \   e^{-\frac{2}{h}f}  .
\end{equation}
Notice that this decomposition of $\nabla u_h$ is valid on~$\pa \Omega$. Indeed, for all $i\in\{1,\ldots,\ft m_1^{\overline \Omega}\}$, $\psi_{i}$ has a smooth trace on $\pa \Omega$ since~$\psi_{i}\in \Lambda^1C^{\infty}(\overline\Omega)$ (due to the fact that  the eigenforms of~$L^{D,(1)}_{f,h}$ belongs to~$C^{\infty}(\overline\Omega)$ and~$\pi_{h}^{(1)}$ is a projector onto a finite number of eigenforms of $-L^{D,(1)}_{f,h}$).
In the rest of this section, one first 
 introduces such a family $\{\psi_{1},\dots,\psi_{\ft m_1^{\overline \Omega}}\}$
 using a Gram-Schmidt orthonormalization of  the family~$\big\{\pi_{h}^{(1)}\widetilde\psi_{1},\dots,\pi_{h}^{(1)}\widetilde\psi_{\ft m_1^{\overline \Omega}}\big\}$.
Then, one gives estimates of the terms $\langle \nabla u_{h},\psi_{i}\rangle_{L^{2}_{w}}$
 appearing in~\eqref{eq.rere}.  Finally, one concludes  
the proof of Theorem~\ref{thm-big-pauh} in Section~\ref{sec:estim_bound}, with estimations
of the boundary terms $\int_{\Sigma}   F    \, \psi_{i}\cdot n \   e^{-\frac{2}{h}f}  $.

\subsubsection{Gram-Schmidt orthonormalization}\label{sec:gram1}
Let us assume that the hypothesis~\eqref{H-M} holds, and assume $h>0$
small enough such that the family $\left\{\pi_h^{(1)}\widetilde \psi_i,i=1,\dots,\ft m_1^{\overline \Omega}\right \}$ is independent
(which is guaranteed for small $h$ by item 3b in  Proposition~\ref{ESTIME1-base}). Using a Gram-Schmidt procedure, there exists, 
for all $j\in \big \{1,\dots,\ft m_1^{\overline \Omega}\big \}$, a family
 $(\kappa_{ji})_{i=1,\dots,j-1}\subset \mathbb R^{j-1}$ such that the $1$-forms\label{page.kji}
\begin{equation}\label{eq.gramfj}
 f_j:=\pi_h^{(1)}  \Big [ \widetilde \psi_j + \sum \limits_{i=1}^{j-1} \kappa_{ji} \widetilde \psi_i \  \Big]
\end{equation}
satisfy:
\begin{itemize} [leftmargin=1.3cm,rightmargin=1.3cm]
\item[(i)] for all $k\in  \{1,\dots, \ft m_1^{\overline \Omega}\}$,~${\rm Span}\big(\{  f_i,i=1,\dots,k\}\big)={\rm Span}\big(\{\pi_h^{(1)}\widetilde \psi_i,i=1,\dots,k \}\big )$,
\item[(ii)] for all $i\neq j$,~$\lp f_i , f_j \rp_{L^2_w}=0$.
\end{itemize}
One defines moreover, for $j\in \left \{1,\dots,\ft m_1^{\overline \Omega} \right \}$, \label{page.zjpsij}
\begin{equation}\label{eq.psij}
Z_j:= \Vert f_j \Vert_{L^2_w}\ \text{and}\  \psi_j:=\frac{1}{Z_j}f_{j},
 \end{equation}
so that $(\psi_j)_{j\in \big \{1,\dots,\ft m_1^{\overline \Omega}\big \}}$ is a $L^{2}_{w}$-orthonormal basis of
$\Ran\, \pi_{h}^{(1)}$.
By reasoning by induction (see~\cite[Section 2]{di-gesu-le-peutrec-lelievre-nectoux-16} for a similar proof), Proposition~\ref{ESTIME1-base} easily leads 
to the following estimates showing in particular that the family $(\pi_h^{(1)}\widetilde \psi_i)_{i\in\{1,\dots,\ft m_1^{\overline \Omega}\}}$ is close to the family  $( \psi_i)_{i\in\{1,\dots,\ft m_1^{\overline \Omega} \}}$.
\begin{lemma}\label{le.e1}   
Let us assume that the assumption~\eqref{H-M} is satisfied.
Then, there exists $c>0$ such that for all  $j\in \big \{1,\dots,\ft m_1^{\overline \Omega} \big \}$,~$i\in \{1,\dots,j-1\}$ and $h>0$ small enough,
\begin{align*}
Z_j=1+O ( e^{-\frac{c}{h}} ) \ {\rm and}\ \kappa_{ji}&= O (e^{-\frac{c}{h}}).
\end{align*}
\end{lemma}

\subsubsection{Estimates of the interaction terms $\left(\lp    \nabla
    u_h ,   \psi_j \rp_{L^2_w}\right)_{j \in \{1,\ldots,\ft m_1^{\overline \Omega}\}}$       }\label{sec:estim_inter1}
Let us begin with the estimates of the terms  $\lp    \nabla  \pi_h^{(0)} \widetilde u_k,   \psi_j \rp_{L^2_w}$, where $j\in \big \{1,\dots,\ft m_1^{\overline \Omega}\big \}$ and $k\in \{1,\ldots,\ft m_0^{ \Omega}\}$.
\begin{lemma} \label{interaction1}  
Let us assume that the assumption \eqref{H-M} holds. Then, there exists $c>0$ such that  for all for all $k\in \{1,\ldots,\ft m_0^{ \Omega}\}$, $j\in\big\{1,\dots,\ft m_1^{\overline \Omega}\big\}$ and $h>0$ small enough, it holds: 
\begin{align*} 
  \lp       \nabla \pi_h^{(0)}\widetilde u_k ,      \psi_j   \rp_{L^2_w} &=\begin{cases}  \lp       \nabla \widetilde u_k ,    \widetilde \psi_j \rp_{L^2_w}\big(1+O(e^{-\frac ch}))   &  \text{ if } z_j\in\mathbf j(x_{k}) ,\\
O \left(e^{-\frac{1}{h}\big (f(\mbf j(x_k))- f(x_k)+c\big )}\right )   &   \text{ if } z_j\notin\mathbf j(x_{k}),
  \end{cases} 
\end{align*}
  where we recall that  the asymptotic expansion of the term $ \lp       \nabla \widetilde u_k ,    \widetilde \psi_j \rp_{L^2_w}$ is given in Proposition~\ref{pr.S}.

\end{lemma}

\begin{proof} 
Using \eqref{eq.gramfj}, \eqref{eq.psij}, and  Lemma~\ref{le.e1}, one has for some $c>0$ and for all $j\in\big \{1,\dots,\ft m_1^{\overline \Omega}\big \}$ and $h>0$ small enough,
      \begin{align*} 
 \lp    \nabla \pi_h^{(0)} \widetilde u_k ,   \psi_j \rp_{L^2_w}    &=Z^{-1}_j \left [    \lp       \nabla \pi_h^{(0)}\widetilde u_k ,     \pi_h^{(1)} \widetilde \psi_j   \rp_{L^2_w}   +\sum \limits_{i=1}^{j-1}\kappa_{ji} \,   \lp       \nabla \pi_h^{(0)}\widetilde u_k ,    \pi_h^{(1)} \widetilde \psi_i   \rp_{L^2_w} \  \right]\\
 &= \big (1+O(e^{-\frac ch} )\big) 
 \left [    \lp       \nabla \pi_h^{(0)}\widetilde u_k ,     \pi_h^{(1)} \widetilde \psi_j   \rp_{L^2_w}   +\sum \limits_{i=1}^{j-1}O(e^{-\frac ch} ) \,   \lp       \nabla \pi_h^{(0)}\widetilde u_k ,    \pi_h^{(1)} \widetilde \psi_i   \rp_{L^2_w} \  \right].
  \end{align*} 
Using Proposition~\ref{pr.S}, the statement of Lemma~\ref{interaction1} follows immediately.
\end{proof}

\noindent 
In view of~\eqref{eq.rere},  we need to give an asymptotic estimate of the terms
$$ \lp       \nabla u_h ,      \psi_j \rp_{L^2_w} \text{ for } j\in\big \{1,\dots,\ft m_1^{\overline \Omega}\big \}.$$
 To do that, we would like to prove that $ \lp       \nabla u_h ,      \psi_j \rp_{L^2_w}$ is well approximated by~$ \lp       \nabla \pi_h^{(0)} \widetilde u_1 ,      \psi_j \rp_{L^2_w}$. 
 To this end, let us show  that~$\pi_h^{(0)} \widetilde u_1$ is an accurate
approximation of~$u_{h}$ in~$H^{1}_{w}(\Omega)$.\medskip

\noindent
Before, let us recall that when \eqref{H-M} and~\eqref{eq.hip1-j} hold,   Corollary~\ref{co.thm4} implies that there exists $\beta_0>0$ such that for all $\beta\in (0,\beta_0)$, there exists $h_0>0$ such that  for all $h\in (0,h_0)$, the orthogonal projector 
$$ \widetilde  \pi^{(0)}_h=\pi_{[0,e^{-\frac{2}{h}( f(\mbf j(x_{1}))-f(x_{1})-\beta)})}(-L^{D,(0)}_{f,h}) \text{ has rank } 1.$$
 Therefore,~$\widetilde  \pi^{(0)}_h$ is the orthogonal projector onto $\sspan \, u_h$.  
Moreover, from the second equality in~\eqref{eq.uh=} and item 3a in Proposition~\ref{ESTIME1-base}, one has
$$u_h= \pi^{(0)}_h \widetilde u_1  +O ( e^{-\frac{c}{h}}) \ \text{  in  } L^2_w(\Omega).$$
Therefore,  
 $\pi^{(0)}_h \widetilde u_1$ is an accurate approximation of
$u_{h}$ in~$L^{2}_{w}(\Omega)$.
The following result extends this result in~$H^{1}_{w}(\Omega)$  when assuming   \eqref{eq.hip2-j} in addition to \eqref{H-M} and~\eqref{eq.hip1-j}. 

\begin{lemma} \label{nabla-pi-u1}
Assume  that \eqref{H-M},~\eqref{eq.hip1-j} and \eqref{eq.hip2-j}   hold. Then, it holds in the limit $h\to 0$:
\begin{equation*}
 \big\Vert \nabla  \pi_h^{(0)} \widetilde u_1  \big\Vert^{2}_{L^2_w}=\frac{2}{h}\lambda_h\,\big(1+O(\ve_h)\big)
\end{equation*} 
and
  \begin{equation*}
  \big\Vert \nabla  (\pi_h^{(0)}-\widetilde \pi_h^{(0)} ) \widetilde u_1  \big\Vert^{2}_{L^2_w}  =\frac{2}{h}\lambda_h\,O(\ve_h)
 =O\big(h^{-\frac32}\,e^{-\frac2h(f(\mbf j(x_{1}))-f(x_{1}))}\,\ve_h\big),
 \end{equation*} 
 where   in the limit $h\to 0$,~$\ve_h$  satisfies~\eqref{eq.epsilon-h}
  \end{lemma}
\begin{proof}
Applying the Parseval identity to 
$\nabla  \pi_h^{(0)} \widetilde u_1\in \Ran\, \pi_h^{(1)}$ (see~\eqref{eq.comutation}), one gets
$$ \big\Vert \nabla  \pi_h^{(0)} \widetilde u_1  \big\Vert^2_{L^2_w}=\sum \limits_{j=1}^{\ft m_1^{\overline \Omega}}\langle \nabla\pi^{(0)}_h\widetilde u_1, \psi_j  \rangle_{L^2_w}^2.$$
Using Lemma~\ref{interaction1}, there exists consequently $c>0$ such that for all $h>0$ small enough, 
 \begin{align*}
 \big \Vert \nabla  \pi_h^{(0)} \widetilde u_1  \big\Vert^2_{L^2_w}  = \sum \limits_{j:\, z_{j}\in\mathbf j(x_{1})} \big\langle \nabla \widetilde u_1, \widetilde \psi_j   \big\rangle_{L^2_w}^2 \, (1+O (e^{-\frac ch}  )) \end{align*} 
Using in addition~\eqref{eq.besoin}, one then obtains the first part of 
Lemma \ref{nabla-pi-u1}: 
\begin{equation}
\label{eq.first-estim}
 \big\Vert \nabla  \pi_h^{(0)} \widetilde u_1 \big\Vert^2_{L^2_w}=\frac{2}{h}\lambda_h\big(1+O(\ve_{h}) \big),
\end{equation}
 where, in the limit $h\to 0$,  $\ve_h$ satisfies~\eqref{eq.epsilon-h}.\medskip
 
\noindent 
Let us  recall  that from~\eqref{eq.uh=} and~\eqref{normeu}, one has for any $h$ small enough  
 \begin{equation}
  \label{eq.uh-tildeu1}
 u_h= \frac{\widetilde  \pi^{(0)}_h \widetilde u_{1}}{\| \widetilde  \pi^{(0)}_h \widetilde u_{1} \|_{L^2_w}} \, \text{ where }\, \big  \| \widetilde  \pi^{(0)}_h \widetilde u_{1} \big \|_{L^2_w}= 1+O ( e^{-\frac{c}{h}}) .
 \end{equation}
Now, since the projectors $ \pi_h^{(0)}$ and $\widetilde \pi_h^{(0)} $ commute with $L^{D,(0)}_{f,h}$, and $\widetilde \pi_h^{(0)}\pi_h^{(0)}=\widetilde \pi_h^{(0)}$, one has
\begin{align*}
 \frac h2  \big\Vert \nabla  (\pi_h^{(0)}-\widetilde \pi_h^{(0)}  ) \widetilde u_1  \big\Vert^2_{L^2_w}  &=  \big  \langle   (\pi_h^{(0)}-\widetilde \pi_h^{(0)} ) \widetilde u_1, -L^{D,(0)}_{f,h}\,(\pi_h^{(0)}-\widetilde \pi_h^{(0)} ) \widetilde u_1   \big \rangle_{L^2_w} \\
 &=  \big\langle   \pi_h^{(0)} \widetilde u_1, -L^{D,(0)}_{f,h}\,  (\pi_h^{(0)}-\widetilde \pi_h^{(0)}) \widetilde u_1  \big \rangle_{L^2_w} \\
 &=   \big \langle   \pi_h^{(0)} \widetilde u_1, -L^{D,(0)}_{f,h} \pi_h^{(0)} \widetilde u_1  \big \rangle_{L^2_w}-  \big \langle   \widetilde \pi_h^{(0)} \widetilde u_1, -L^{D,(0)}_{f,h} \widetilde \pi_h^{(0)} \widetilde u_1 \big   \rangle_{L^2_w}\\
 &= \frac h2  \big\Vert \nabla \pi_h^{(0)} \widetilde u_1  \big \Vert^2_{L^2_w}-\lambda_h\big (1+O ( e^{-\frac{c}{h}})\big ),
  \end{align*} 
  where the last line follows from \eqref{eq.uh-tildeu1}.
  Using  in addition \eqref{eq.first-estim}, one   obtains in the limit $h\to 0$: 
 \begin{align*}
\frac h 2  \big\Vert \nabla  (\pi_h^{(0)}-\widetilde \pi_h^{(0)}  ) \widetilde u_1 \big  \Vert^2_{L^2_w}   &=\lambda_h \big(1+O(\ve_h)\big )
-\lambda_h\big (1+O ( e^{-\frac{c}{h}})\big )=\lambda_h\,O(\ve_h),
 \end{align*} 
 which proves Lemma~\ref{nabla-pi-u1}, 
 using also the asymptotic estimate of $\lambda_{h}$ given in Theorem~\ref{thm-big0}, see~\eqref{eq.lambda_h}. \end{proof}

\noindent
 We are now in position to estimate  the interaction terms $\left(\lp    \nabla
    u_h ,   \psi_j \rp_{L^2_w}\right)_{j \in \{1,\ldots,\ft m_1^{\overline \Omega}\}}$.
\begin{corollary} \label{lemm2}
Let us assume that the assumptions~\eqref{H-M},~\eqref{eq.hip1-j} and \eqref{eq.hip2-j}   hold.  Let $u_h$ be the eigenfunction associated with the principal eigenvalue $\lambda_h$ of $-L_{f,h}^{D,(0)}$ (see~\eqref{eq.lh})  which satisfies~\eqref{eq.norma}.
Then,  
in the limit $h\to0$:
\begin{enumerate} [leftmargin=1.3cm,rightmargin=1.3cm]
\item[(i)] for all $j\in \big\{1,\dots,\ft m_1^{\overline \Omega}\big \}$ such that $z_{j}\in\mathbf j(x_{1})\cap\pa\Omega$
(i.e. for all $j\in\big \{1,\dots,\ft k_{1}^{\pa \ft C_1}\big\}$, see~\eqref{eq.k1-paC1-0} and~\eqref{eq.k1-paC1}),
  \begin{align*}
   \big \lp    \nabla   u_h,   \psi_j  \big \rp_{L^2_w}&=   \big \langle \nabla  \widetilde u_1, \widetilde \psi_j  \big \rangle_{L^2_w} \, \big (1+O(\sqrt{\ve_h}) \big) \\ 
&=-C_{j,1} \ h^{-\frac34}  e^{-\frac{1}{h}(f(\mbf j(x_{1}))-f(x_{1}))}   \, \big (1+O(\sqrt{\ve_h})  +     O(h )   \big),
\end{align*} 
\item[(ii)] for all $j\in \big \{1,\dots,\ft m_1^{\overline \Omega}\big \}$ such that $z_{j}\in\mathbf j(x_{1})\cap\Omega$,
 \begin{align*}
   \big \lp    \nabla   u_h,   \psi_j  \big \rp_{L^2_w}&=    \big \langle \nabla  \widetilde u_1, \widetilde \psi_j  \big \rangle_{L^2_w} \, \big (1+O(h^{-\frac14}\sqrt{\ve_h}) \big) \\ 
&=O\big(h^{-\frac12}  e^{-\frac{1}{h}(f(\mbf j(x_{1}))-f(x_{1}))}\big),
\end{align*} 
\item[(iii)]  and for all $j\in \big \{1,\dots,\ft m_1^{\overline \Omega}\big \}$ such that $z_{j}\notin\mathbf j(x_{1})$,
  \begin{equation*}
   \big \lp    \nabla    u_h,   \psi_j  \big \rp_{L^2_w}   =O\big( h^{-\frac34}  e^{-\frac{1}{h}(f(\mbf j(x_{1}))-f(x_{1}))}\,\sqrt{\ve_{h}}  \big ),
  \end{equation*} 
  \end{enumerate}
   where  in the limit $h\to 0$,~$\ve_h$ satisfies~\eqref{eq.epsilon-h}. 
\end{corollary}
\begin{proof}
Using \eqref{eq.uh-tildeu1}, there exists $c>0$ such that for all $j\in \big \{1,\dots,\ft m_1^{\overline \Omega}\big \}$,  in the limit $h\to 0$: 
\begin{equation}\label{eq.lpuh}
 \big \lp    \nabla u_h,   \psi_j \big  \rp_{L^2_w}=  \big\lp    \nabla  \widetilde \pi_h^{(0)} \widetilde u_1,   \psi_j  \big \rp_{L^2_w}\,\big(1+O(e^{-\frac{c}{h}})\big).
\end{equation}
\begin{sloppypar}
\noindent
In addition, using the Cauchy-Schwarz inequality and the second statement in Lemma~\ref{nabla-pi-u1}, it holds for all $j\in \big \{1,\dots,\ft m_1^{\overline \Omega}\big \}$, in the limit $h\to 0$:
\end{sloppypar}
\begin{align}
\nonumber
  \big \lp    \nabla  \widetilde \pi_h^{(0)} \widetilde u_1,   \psi_j  \big \rp_{L^2_w} &= \big \lp    \nabla   \pi_h^{(0)} \widetilde u_1,   \psi_j  \big \rp_{L^2_w}+ \big \lp    \nabla   (\widetilde \pi_h^{(0)}- \pi_h^{(0)} )\widetilde u_1,   \psi_j  \big \rp_{L^2_w}\\
 \label{eq.ici}
 &= \big \lp    \nabla   \pi_h^{(0)} \widetilde u_1,   \psi_j  \big \rp_{L^2_w}+O\big(h^{-\frac34}\,e^{-\frac1h(f(\mbf j(x_{1}))-f(x_{1}) )}\,\sqrt{\ve_h}\big),
\end{align}
 where $\ve_h$ is of the order given by~\eqref{eq.epsilon-h}.
Then, the statement of Corollary~\ref{lemm2} follows by injecting \eqref{eq.ici}
into \eqref{eq.lpuh} and by using the estimates of the terms $ \big \lp    \nabla   \pi_h^{(0)} \widetilde u_1,   \psi_j  \big \rp_{L^2_w}$ ($j\in\{1,\dots,\ft m_1^{\overline \Omega}\}$) given in 
Lemma~\ref{interaction1}. \end{proof}

\subsubsection{Estimates of the boundary terms $\Big(\displaystyle{\int_{\Sigma}   \,  F  \, \psi_j \cdot n  \,   e^{- \frac{2}{h} f} } \Big)_{j\in  \{1,\ldots,\ft m_1^{\overline \Omega} \}}$}
\label{sec:estim_bound}

\begin{proposition} \label{gamma}  
Let us assume that the assumption \eqref{H-M} is satisfied. Let us consider  $i\in\{1,\dots,\ft m_1^{\overline \Omega}\}$, 
an open set $\Sigma$ of $\pa\Omega$, 
 and  $F\in L^{\infty}(\partial \Omega,\mathbb R)$. Then, there exists $c>0$ such that   in the limit $h\to 0$:
\begin{equation*}
  \int_{\Sigma}   F \, \psi_i \cdot n  \   e^{- \frac{2}{h} f}   =\begin{cases}  O \big (e^{-\frac{1}{h} (\min_{\pa\Omega}f+c)} \big)   &   \text{ if } i\in  \big  \{\ft k_1^{\pa \Omega}+1,\dots,\ft m_1^{\overline \Omega} \big \} , \\
O \big (e^{-\frac{1}{h} (\min_{\pa\Omega}f+c)} \big )   &  \text{ if } i\in  \big \{1,\dots,\ft k_1^{\pa \Omega} \big \} \text{ and } z_{i}\notin\overline\Sigma, \\
 O \big (h^{\frac{d-3}{4}}     e^{-\frac{1}{h} \min_{\pa\Omega}f} \big )  &  \text{ if } i\in  \big \{1,\dots,\ft k_1^{\pa \Omega}  \big \}  \text{ and } z_{i}\in\overline\Sigma,
  \end{cases} 
  \end{equation*}
  where we recall that $\{z_1,\ldots,z_{\ft k_1^{\pa \Omega}}\}= \ft U_1^{\pa \Omega}\cap \argmin_{\pa \Omega} f$ (see~\eqref{eq.z11}). Moreover, when $i\in  \big\{1,\dots,\ft k_1^{\pa \Omega} \big \}$,~$z_{i}\in {\Sigma}$, and  $F$ is  $C^{\infty}$ in a neighborhood  of $z_i$, it holds in the limit $h\to 0$:
  \begin{equation*}
  \int_{\Sigma}   F \, \psi_i \cdot n  \   e^{- \frac{2}{h} f}   = \, h^{\frac{d-3}{4}}   \     e^{-\frac{1}{h} \min_{\pa\Omega}f}  \    \big( B_{i}\,  F(z_i)  +     O(h )    \big) ,
  \end{equation*}
  where the constant $B_{i}$ is defined in \eqref{eq.Bi}.
  \end{proposition}

\begin{proof}
Let $F\in L^{\infty}(\partial \Omega,\mathbb R)$.
Using~\eqref{eq.gramfj}, \eqref{eq.psij}, the trace theorem, and the Cauchy-Schwarz inequality,  one has for all $j\in \{1,\dots,\ft m_1^{\overline \Omega}\}$,
 \begin{align*}
 Z_{j}\int_{\Sigma} F \,  \psi_j \cdot n  \, e^{-\frac{2}{h}f}  &=\left [ \int_{\Sigma}  F \, \widetilde \psi_j \cdot n \, e^{- \frac{2}{h}f}   +\int_{\Sigma}  F \, \big((\pi_{h}^{(1)}-1)\widetilde \psi_j \big)\cdot n \, e^{- \frac{2}{h}f}\right ] \\
&\quad +\sum \limits_{i=1}^{j-1}\kappa_{ji} \ \left[ \int_{\Sigma}F \,  \widetilde  \psi_i \cdot n  \, e^{- \frac{2}{h} f}    + \int_{\Sigma}  F \, \big((\pi_{h}^{(1)}-1)\widetilde \psi_i \big)\cdot n \, e^{- \frac{2}{h}f} \right] \\
&=\left [ \int_{\Sigma}  F \, \widetilde \psi_j \cdot n \, e^{- \frac{2}{h}f}  +\Vert(1-\pi_h^{(1)})\widetilde \psi_j\Vert_{H^1_w}O \big(h^{-1}\,e^{-\frac{\min_{\pa\Omega}f}{h}}  \big)\right ] \\
&\quad +\sum \limits_{i=1}^{j-1}\kappa_{ji} \ \left[ \int_{\Sigma} F \, \widetilde  \psi_i \cdot n  \, e^{- \frac{2}{h} f}       + \Vert(1-\pi_h^{(1)})\widetilde \psi_i\Vert_{H^1_w}O \big(h^{-1}\,e^{-\frac{\min_{\pa\Omega}f}{h}}  \big)\right].
  \end{align*} 
 From Lemma~\ref{le.e1} and item 2b in Proposition~\ref{ESTIME1-base}, there exists  $c>0$ such that for all  $j\in \big \{1,\dots,\ft m_1^{\overline \Omega}\big \}$,~$i\in \{1,\dots,j-1\}$,   in the limit $h\to 0$:
  $$Z_j=1+O ( e^{-\frac{c}{h}} ),\ \kappa_{ji}=O(  e^{-\frac{c}{h}} ),  \ {\rm and }\  \big \Vert(1-\pi_h^{(1)})\widetilde \psi_j  \big\Vert_{H^1_w} =O(  e^{-\frac{c}{h}} ).$$
Therefore, using   Proposition~\ref{gamma1}, there exists  $c>0$ such that for all  $j\in \{1,\dots,\ft m_1^{\overline \Omega}\}$, in the limit $h\to 0$:
  $$\int_{\Sigma} F \,  \psi_j \cdot n  \, e^{-\frac{2}{h}f}= \int_{\Sigma}  F \, \widetilde \psi_j \cdot n \, e^{- \frac{2}{h}f} + O \big (e^{-\frac{1}{h} (\min_{\pa\Omega}f+c)} \big ).$$
  The statement of Proposition~\ref{gamma} is then a straightforward consequence of 
 Proposition~\ref{gamma1}.
\end{proof}

\noindent We are now in position to prove Theorem~\ref{thm-big-pauh}.
 
\begin{proof}[Proof of Theorem~\ref{thm-big-pauh}]
\begin{sloppypar}
 Let us assume that the assumptions \eqref{H-M},~\eqref{eq.hip1-j},~\eqref{eq.hip2-j}  and \eqref{eq.hip3-j}   hold. Recall  that in this case,   for all $x\in \ft U_0^\Omega\setminus\{ x_1\}$, one has
$$
 f(\mathbf j(x ))-f(x )<f(\mathbf j(x_{1}))-f(x_{1})
$$
and 
$$ \mbf j(x_1) \cap\pa \Omega=  \pa\ft  C_{1}\cap\pa\Omega =\{z_1,\ldots,z_{\ft k_1^{\pa \ft C_1}}\} \subset \argmin_{\pa \Omega}f\cap \ft U_1^{\pa \Omega}.$$ 
Moreover,   from~\eqref{hs2-c1}, it holds
$$
x_1\in \argmin_{\ft C_1}f=\argmin_{\Omega}f=\argmin_{\overline \Omega}f.$$
Thus, one  has 
\begin{equation}\label{just1}
 f(\mathbf j(x_{1}))=\min_{\pa\Omega}f \text{  and  } f(x_{1})= \min_{ \overline \Omega}f.
\end{equation} 
\end{sloppypar}
\noindent
Let us now consider $F\in L^{\infty}(\partial \Omega,\mathbb R)$ and $\Sigma$  an open  subset of $\pa\Omega$.
 First, since $ \big \{\psi_j, \, j=1,\dots,\ft m_1^{\overline \Omega} \big\}$ is an orthonormal basis of $\Ran \, \pi^{(1)}_h$ and $\nabla u_h\in \Ran \, \pi^{(1)}_h$, one has the following decomposition:
$$
\int_{\Sigma}  F \,  \partial_{n}u_h\,  e^{-\frac{2}{h} f}  = \sum \limits_{j=1}^{\ft m_1^{\overline \Omega}} \big \langle \nabla u_h , \psi_j \big \rp_{L^2_w}  \int_{\Sigma}  F \,  \psi_j \cdot n \  e^{- \frac{2}{h}  f}.
$$
Using in addition Corollary~\ref{lemm2}, Proposition~\ref{gamma},  and~\eqref{just1}, there exists $c>0$ such that for all $h>0$ small enough, 
\begin{align}
\nonumber
\int_{\Sigma} F   \partial_{n}u_h \,  e^{-\frac{2}{h} f} &= \sum \limits_{j=1}^{\ft k_{1}^{\pa \Omega}}  \big\langle \nabla u_h , \psi_j \big \rp_{L^2_w}  \int_{\Sigma}  F \,  \psi_j \cdot n \  e^{- \frac{2}{h}  f} \\
\label{eq.decomp-minima}
&\quad+\sum \limits_{j=\ft k_{1}^{\pa \Omega}+1}^{\ft m_1^{\overline \Omega}} O\big( h^{-\frac12}  e^{-\frac{1}{h}(\min_{\pa\Omega}f - \min_{\overline \Omega}f )}  \big ) \, O \big(e^{-\frac{1}{h} (\min_{\pa\Omega}f +c)} \big).
\end{align}  
Hence, when $\overline{\Sigma}$ does not contain any of the $z_{i}$,~$i\in\{1,\dots,\ft k_{1}^{\pa \Omega}\}$,
 from \eqref{eq.decomp-minima}, Corollary~\ref{lemm2},  Proposition~\ref{gamma},  and~\eqref{just1}, one deduces the following relation
for some $c>0$ independent of $h$ and every $h>0$ small enough:
\begin{align*}
\int_{\Sigma} F\,\partial_{n}u_h   e^{-\frac{2}{h} f} &= \sum \limits_{j=1}^{\ft k_1^{\pa \Omega}}O\big(h^{-\frac 34}  e^{-\frac{1}{h}(\min_{\pa\Omega}f- \min_{ \overline \Omega}f)}\big) O \big(e^{-\frac{1}{h} (\min_{\pa\Omega}f+c)} \big)\\
&\quad +O\big(  e^{-\frac{1}{h}(2\min_{\pa\Omega}f- \min_{\overline \Omega}f+c)}  \big )\\
&= O\big( e^{-\frac{1}{h}(2\min_{\pa\Omega}f- \min_{\overline \Omega}f+\frac c2)} \big ).
\end{align*}
This proves item $(i)$ in Theorem~\ref{thm-big-pauh}.\medskip

\noindent Assume now that  $\overline{\Sigma}$ does not contain any of the $z_{i}$,~$i\in\{1,\dots,\ft k^{\pa \ft C_1}_{1}\}$.
Then ,  from \eqref{eq.decomp-minima}, Corollary~\ref{lemm2}, Proposition~\ref{gamma},  and~\eqref{just1}, one deduces that
in the limit $h\to 0$:
\begin{align*}
\int_{\Sigma} F\, \partial_{n}u_h\,   e^{-\frac{2}{h} f} &= \sum \limits_{j=1}^{\ft k_1^{\pa\ft C_1}}O\big(h^{-\frac 34}  e^{-\frac{1}{h}(\min_{\pa\Omega}f- \min_{ \overline \Omega}f)}\big) O \big(e^{-\frac{1}{h} (\min_{\pa\Omega}f+c)} \big) \\
 &\quad+\sum_{j=\ft k_{1}^{\pa \ft C_1}+1}^{\ft k_{1}^{\pa \Omega}}O\big( h^{-\frac{3}{4}}   e^{-\frac{1}{h}(\min_{\pa\Omega}f- \min_{ \overline \Omega}f)}\,\sqrt{\ve_{h}}  \big )
\ O\big( h^{\frac{d-3}4} e^{-\frac{1}{h}\min_{\pa\Omega}f} \big)\\
&\quad +O\big(  e^{-\frac{1}{h}(2\min_{\pa\Omega}f- \min_{ \overline \Omega}f+c)}  \big )\\
&=O\big( e^{-\frac{1}{h}(2\min_{\pa\Omega}f- \min_{ \overline  \Omega}f+\frac c2)} \big )+ O\big( h^{\frac{d-6}{4}}   e^{-\frac{1}{h}(2\min_{\pa\Omega}f- \min_{\overline  \Omega}f)}\,\sqrt{\ve_{h}}  \big ),
\end{align*}
where the constant $c>0$ is  independent of $h$ and  $\ve_h$ satisfies~\eqref{eq.epsilon-h}. This proves item $(ii)$ in Theorem~\ref{thm-big-pauh}.\medskip

\noindent
Assume lastly that $\overline{\Sigma}\cap\{z_{1},\dots,z_{\ft k^{\pa \ft C_1}_{1}}\}=\{z_{i}\}$,  $F$ is $C^{\infty}$ in a neighborhood  of $z_{i}$ and  $z_{i}\in{\Sigma}$. From 
\eqref{eq.decomp-minima}, Corollary~\ref{lemm2}, Proposition~\ref{gamma},  and~\eqref{just1}, one then deduces that in the limit $h\to0$,
it holds for some $c>0$ and $\ve_{h}$ which satisfies \eqref{eq.epsilon-h}, 
\begin{align*}
\int_{\Sigma} F\, \partial_{n}u_h\,   e^{-\frac{2}{h} f} &= \big\langle \nabla u_h , \psi_i \big\rp_{L^2_w} \ \int_{\Sigma}  F \,  \psi_i \cdot n \  e^{- \frac{2}{h}  f} +
O\big( h^{\frac{d-6}{4}}   e^{-\frac{1}{h}(2\min_{\pa\Omega}f- \min_{ \overline \Omega}f)}\,\sqrt{\ve_{h}}  \big )\\
&=
 -B_{i} \,C_{i,1} \  h^{\frac{d-6}{4}}\,  e^{-\frac{1}{h}(2\min_{\pa\Omega}f- \min_{ \overline \Omega}f)} \   \big (F(z_{i})+O(\sqrt{\ve_h})  +     O(h )   \big) ,
\end{align*}
where the constants $B_{i}$ and $C_{i,1}$ are defined in~\eqref{eq.Cip}--\eqref{eq.Bi}.
This concludes the proof of item $(iii)$ in Theorem~\ref{thm-big-pauh}.
\end{proof}

\section{On the law of $X_{\tau_\Omega}$}
\label{section-5}
The main goal of this section is to prove Theorem~\ref{thm.main}. In Section~\ref{sec-xo-nuh}, one proves Theorem~\ref{thm.main} when $X_0\sim \nu_h$ (where $\nu_h$ is the quasi-stationary distribution of the process~\eqref{eq.langevin} in~$\Omega$, see Definition~\ref{defqsd}). In Section~\ref{sec-xo-nuh2}, one proves Theorem~\ref{thm.main} when $X_0=x\in   \mathcal A(\ft C_{\ft{max}})$.

\subsection{Proof of Theorem~\ref{thm.main}  when $X_0\sim \nu_h$}
\label{sec-xo-nuh}
The proof of Theorem~\ref{thm.main} when $X_0\sim \nu_h$  is a straightforward consequence of Theorem~\ref{thm-big0},  Proposition~\ref{pr.masse}  and Theorem~\ref{thm-big-pauh}. Indeed, let us recall that  from~\eqref{eq.dens}, one has:
\begin{equation}\label{eq.dens2}
\mathbb E_{\nu_h}\left [ F\left (X_{\tau_{\Omega}} \right )\right]=- \frac{h}{2\lambda_h} \frac{ \displaystyle \int_{\pa \Omega} F\,  \partial_n
  u_h  e^{-\frac{2}{h} f }}{\displaystyle \int_\Omega u_h  e^{-\frac{2}{h} f }  }.
  \end{equation}
Moreover, recall that~\eqref{eq.hip1},~\eqref{eq.hip2},  and~\eqref{eq.hip3} (see Section~\ref{sec:hip} and more precisely Lemma~\ref{equiv-hipo}) are equivalent to   the assumptions~\eqref{eq.hip1-j},~\eqref{eq.hip2-j},  and~\eqref{eq.hip3-j} . In addition,  under \eqref{eq.hip1-j}, one has $\ft C_1=\ft C_{\ft{max}}$ (see  Lemma~\ref{equiv-hipo}),~$\ft k_1^{\pa \ft  C_{1}}=\ft k_1^{\pa \ft  C_{\ft{max}}}$ (see~\eqref{eq.k1-paC-m}) and 
$$f(\mbf j(x_1))=\min_{\pa \Omega}f \ \ \text{(see~\eqref{eq.hip3-j} together with the fact that } \mbf j(x_1)\subset \pa \ft C_1)$$
and
$$f( x_1)=\min_{\overline  \Omega}f \ \ \text{(see~\eqref{hs2-c1})}.$$ 
Thus,  injecting the results of  Theorem~\ref{thm-big0} (and more precisely~\eqref{eq.lambda_h}),  Proposition~\ref{pr.masse} (applied to $\ft O=\Omega$, see~\eqref{eq.concentration1}) and Theorem~\ref{thm-big-pauh} in~\eqref{eq.dens2}, one obtains the statement of Theorem~\ref{thm.main} when $X_0\sim \nu_h$.

\subsection{Proof of Theorem~\ref{thm.main}  when $X_0=x\in  \mathcal A(\ft C_{\ft{max}})$} 
\label{sec-xo-nuh2}

Recall that, under~\eqref{eq.hip1},~$\ft C_{\ft{max}}=\ft C_1$ (see Lemma~\ref{equiv-hipo}). 
To prove Theorem~\ref{thm.main} when $X_0=x\in  \mathcal A(\ft C_1)$,  one first proves that a sufficiently accurate leveling property (as introduced in~\cite{Day4}) holds  in~$\ft C_1$ for $x\mapsto \mathbb E_x[F(X_{\tau_\Omega})]$,  see Proposition~\ref{level} in Section~\ref{sec.level}. Then, combining Theorem~\ref{thm.main} when $X_0\sim \nu_h$ with Proposition~\ref{level}, one proves Theorem~\ref{thm.main}  when $X_0=x\in \mathcal A(\ft C_{1})$  in Section~\ref{se3}. 

\subsubsection{Leveling results}\label{sec.level}
The leveling property  for $x\mapsto \mathbb E_x[F(X_{\tau_\Omega})]$ is defined as follows. 

\begin{definition}\label{level0}
Let $K$ be a compact subset of $\Omega$ and $F\in C^{\infty}(\partial \Omega,\mathbb R)$. We say that $x\mapsto \mathbb E_x[F(X_{\tau_\Omega})]$ satisfies a leveling property on $K$ if  
\begin{equation}  \label{eq.leveling}
\lim_{h\to 0}\big \vert \,\mathbb E_{x} \left [ F\left (X_{\tau_{\Omega}} \right )\right]-\mathbb E_{y} \left [ F\left (X_{\tau_{\Omega}} \right )\right] \, \big  \vert=0
\end{equation} 
and this limit holds uniformly with respect to $(x,y)\in K\times K$.
\end{definition}

The leveling property~\eqref{eq.leveling} has been widely studied in the literature in various geometrical settings, see for example  \cite{Per,Kam,DeFr,Day4,FrWe,Eiz}. 
We prove the following proposition which is a leveling property in our framework. 

\begin{proposition} \label{level}
Let us assume that the assumption \eqref{H-M} holds. Let $\lambda \in \mathbb R$ and $\ft C$ be a connected component of $\{f<\lambda\}$ such that $\ft C\subset \Omega$. Then, for any path-connected compact set $K\subset\ft C$ and for any $F\in  C^{\infty}(\partial \Omega,\mathbb R)$, there exist $c>0$ and $M>0$, such that for all $(x,y)\in K\times K$,
\begin{equation}\label{eq.LP}
\big \vert \,\mathbb E_{x} \left [ F\left (X_{\tau_{\Omega}} \right )\right]-\mathbb E_{y} \left [ F\left (X_{\tau_{\Omega}} \right )\right] \, \big \vert \leq M e^{-\frac{c}{h}}.
\end{equation}
\end{proposition}

\begin{proof}
 
\noindent The proof is inspired from techniques used in \cite{DeFr}.  The proof of Proposition~\ref{level} is divided into two steps. 
In the following $C>0$ is a constant which can change from one occurrence to another and which does not depend on $h$.

\medskip
 \noindent
 \textbf{Step 1.}
 Let $F\in  C^{\infty}(\partial \Omega, \mathbb R)$. Let us denote by~$v_h\in H^1(\Omega)$ the unique weak solution to the elliptic boundary value problem
\begin{equation} \label{vh}
\left\{
\begin{aligned}
 \frac{h}{2}  \ \Delta v_h -\nabla f \cdot \nabla v_h  &=  0    \ {\rm on \ }  \Omega  \\ 
v_h&= F \ {\rm on \ } \partial \Omega .\\
\end{aligned}
\right.
\end{equation}
Then,~$v_h$ belongs to $ C^{\infty}(\overline \Omega, \mathbb R)$ and    for all $k\in \mathbb N$, there exist $C>0$,~$n \in \mathbb N$ and $h_0>0$ such that for all $h\in (0,h_0)$, it holds
\begin{equation} \label{est1}
\Vert v_h\Vert_{H^{k+2}(\Omega)}\leq \frac{C}{h^n} \left ( \Vert \nabla v_h\Vert_{L^2(\Omega)}+1\right).
\end{equation} 
Moreover, the Dynkin's formula implies that\label{page.vh}
\begin{equation} \label{eq.vh-dynkin}
\forall x\in \overline \Omega, \ v_h(x)=\mathbb E_{x} \left [F\left ( X_{\tau_{\Omega}} \right)\right].
\end{equation}
Let us prove that $v_h$ belongs to $ C^{\infty}(\overline \Omega, \mathbb R)$ and~\eqref{est1}. Since $F$ is $C^\infty$, for all $k\geq 1$, the exists $\widetilde F\in H^k(\Omega)$ such that $\widetilde F=F$ on $\pa \Omega$ and $\Vert \widetilde F\Vert_{H^k}\le C\Vert F\Vert_{H^{k-\frac 12}(\pa \Omega)}$.
From~\eqref{vh}, the function $\widetilde v_h=v_h-\widetilde F\in H^1(\Omega)$ and is the weak solution to
\begin{equation} \label{est1-2}
\left\{
\begin{aligned}
  \Delta \widetilde v_h &=  \frac{2}{h}  \nabla f \cdot \nabla v_h-\Delta \widetilde F      \ {\rm on \ }  \Omega  \\ 
\widetilde v_h&= 0 \ {\rm on \ } \partial \Omega .
\end{aligned}
\right.
\end{equation}
Thus, using~\cite[Theorem 5, Section 6.3]{Eva},~$ \widetilde v_h\in H^2(\Omega)$ (and thus $v_h\in H^2(\Omega)$)  and there exist $C>0$ and $h_0>0$ such that for all $h\in (0,h_0)$
\begin{align}  
\nonumber
  \Vert \widetilde v_h\Vert_{H^2(\Omega)}&\leq C \left (  \frac 1h \Vert \nabla v_h \Vert_{L^2(\Omega)}  +\Vert  \widetilde F \Vert_{H^2(\Omega)} \right) \leq \frac Ch \Big ( \Vert \nabla v_h \Vert_{L^2(\Omega)}+\Vert  F \Vert_{H^{\frac 32}(\pa \Omega)}\Big),
\end{align} 
and thus
\begin{equation} \label{est-H2}
\Vert v_h\Vert_{H^2(\Omega)} \le \frac Ch \left (  \Vert \nabla v_h \Vert_{L^2(\Omega)}+ 1\right).
\end{equation} 
This proves~\eqref{est1} for $k=0$ The inequality~\eqref{est1} is then obtained by a bootstrap argument (by induction on $k$). This implies by Sobolev embeddings that  $v_h$ belongs to $ C^{\infty}(\overline \Omega, \mathbb R)$. 
 Let us  now prove that there exist $\alpha>0$ and $C>0$ such that 
\begin{equation} \label{est2}
\Vert v_h\Vert_{L^{\infty}(\Omega)}+ \Vert \nabla v_h\Vert_{L^{\infty}(\Omega)} \leq C h^{-\alpha}.
\end{equation}
  Notice that from~\eqref{eq.vh-dynkin}, one has that for all $h>0$,~$\Vert v_h\Vert_{L^{\infty}(\Omega)}\leq \Vert F\Vert_{L^\infty(\pa \Omega)}$. 
From~\eqref{vh} and~\eqref{est-H2} there exists $C>0$ such that for any $\ve>0$ and $\ve'>0$,

\begin{align*} 
h \int  _{\Omega} \vert  \nabla v_h \vert^2   &\leq C \left (h \int_{\partial \Omega} \vert F\, \partial_n v_h \vert \, d\sigma+ \Vert F\Vert_{L^\infty(\pa \Omega)}\int  _{\Omega} \vert \nabla f \cdot \nabla v_h  \vert   \right)\\
&\leq C \left (h \frac{\Vert F\Vert_{L^2(\partial \Omega)}^2}{4\ve} +  \, h  \, \ve\, \Vert v_h\Vert_{H^2(\Omega)}^2 + \frac{\Vert \nabla f\Vert_{L^2( \Omega)}^2}{4\ve'} +   \ve'\, \Vert \nabla v_h\Vert_{L^2(\Omega)}^2  \right)\\
&\leq C \left (h \frac{\Vert F\Vert_{L^2(\partial \Omega)}^2}{4\ve} +  \, h  \, \ve\, C_1\, h^{-2}\left ( \Vert \nabla v_h\Vert_{L^2(\Omega)}^2+ 1\right) + \frac{\Vert \nabla f\Vert_{L^2( \Omega)}^2}{4\ve'} +   \ve'\, \Vert \nabla v_h\Vert_{L^2(\Omega)}^2  \right).
\end{align*}
Choosing $\ve= \frac{h^2}{4(CC_1+1)}$ and $\ve'= \frac{h}{4(C+1)}$ we get 
$$\Vert \nabla v_h\Vert_{L^{2}(\Omega)} \leq  \frac{C}{h}.$$
Therefore, using~\eqref{est1}, one obtains that for all $k\geq 0$, there exist $C>0$,~$n\in \mathbb N$ and $h_0>0$ such that for all $h\in (0,h_0)$
  $$\Vert v_h\Vert_{H^{k}(\Omega)} \leq  \frac{C}{h^n}.$$
 Let $k\geq 0$ such that $k-\frac{d}{2}>1$. Then, one obtains~\eqref{est2} from the continuous Sobolev injection $H^k(\Omega)\subset W^{1,\infty}(\Omega)$. 

\medskip
 \noindent
 \textbf{Step 2.}  
Let us assume that the assumption~\eqref{H-M} holds. Let $\lambda \in \mathbb R$ and $\ft C$ be a connected component of $\{f<\lambda\}$ such that $\ft C\subset \Omega$. To prove Proposition~\ref{level}, we will  prove that for  any compact subset $K$ of $\ft C$, there exists $c_K>0$ such that 
\begin{equation} \label{eq1}\Vert \nabla   v_h\Vert_{L^{\infty}(K)}\leq C\,e^{-\frac{c_K}{h}}.\end{equation}
Indeed, since $K$ is  path-connected,  the following inequality 
$$\forall (x,y)\in K\times K, \  \vert v_h(x)-v_h(y)\vert \leq C_K\Vert \nabla   v_h\Vert_{L^{\infty}(K)},$$
where $C_K>0$ depends on $K$, will then conclude the proof of~\eqref{eq.LP}. \\Let us now define the set $\ft C_r$ by \label{page.cr}
\begin{equation}\label{pr.cr}
\ft C_r=\{f<\lambda-r\}\cap\ft C\subset \Omega
\end{equation}
which is not empty  and $C^{\infty}$  for all $r\in (0,r_1)$, for some $r_1>0$. Indeed, the boundary of $\ft C_r$ is   the set $\{f=\lambda-r\}\cap \ft C$ (since $\ft C\subset \Omega$)  which contains no critical points of~$f$  for $r\in (0,r_1)$, with  $r_1>0$ small enough (since there is  a finite number of critical points under the assumption \eqref{H-M}). 
We now prove that for all $r_0\in (0,r_1)$ there exists $\alpha_0>0$ such that 
\begin{equation}\label{est3}
\Vert \nabla v_h\Vert_{L^{\infty}(\ft C_{r_0})} \leq  e^{-\frac{\alpha_0}{h}}.
\end{equation}
Let $r$ be such that $2^nr=r_0$ where $n\in \mathbb N$ will be fixed later (since $r\le r_0$,~$r\in (0,r_1)$). Equation~\eqref{vh} rewrites 
$$
\left\{
\begin{aligned}
 {\rm div}\, \big (e^{-\frac{2}{h}f} \nabla v_h \big) &=  0    \ {\rm on \ }  \Omega  \\ 
v_h&= F \ {\rm on \ } \partial \Omega.\\
\end{aligned}
\right.
$$
 Using~\eqref{est2},  there exist $C>0$  and $\alpha>0$ such that,
$$
\left \vert \, \,  \int_{\ft C_{r/2}} \vert  \nabla v_h \vert^2 \, e^{-\frac{2}{h}f}\,  \right \vert = \left \vert \ \int_{\partial \ft C_{r/2}} \, e^{-\frac{2}{h}f} v_h\, \partial_n v_h  \, d\sigma \right\vert \leq \frac{C}{h^{\alpha}}\, e^{-\frac{2}{h}(\lambda-\frac r2)},
$$
where we used  the Green formula (valid since $\ft C_r$ is $C^{\infty}$  for all $r\in (0,r_1)$) and the inclusion $\partial \ft C_{r/2}\subset \{f=\lambda-\frac r2\}$. 
In addition, since $\ft C_r\subset \ft C_{r/2}$ it holds,
\begin{align*} e^{-\frac{2}{h}(\lambda-r)}\, \int_{\ft C_{r}} \vert  \nabla v_h \vert^2 \,  \leq \int_{\ft C_{r}} \vert  \nabla v_h \vert^2 \, e^{-\frac{2}{h}f}\,  
&\leq \int_{\ft C_{r/2}} \vert  \nabla v_h \vert^2 \, e^{-\frac{2}{h}f}\leq \frac{C}{h^{\alpha}}\ e^{-\frac{2}{h}(\lambda-\frac r2)}.
\end{align*}
Therefore, there exists $\beta>0$ such that for $h$ small enough,
$$\int_{\ft C_{r}} \vert  \nabla v_h \vert^2 \,  \leq \frac{C}{h^{\alpha}}\, e^{-\frac{r}{h}}\leq  C\,e^{-\frac{\beta}{h}},$$
and from~\eqref{vh}, we then have $\Vert \Delta  v_h\Vert_{L^{2}(\ft C_{r})}\leq C\, e^{-\frac{\beta}{h}}$ for some constant  $\beta>0$ which  has been reduced.  In the following,~$\beta>0$ is a  constant which may change from one occurrence to another and does not depend on $h$. Let $\chi_1 \in C_c^{\infty}(\ft C_r)$ be such that $\chi_1 \equiv 1$ on $\ft C_{2r}$.  Since $\Delta (\chi_1 v_h)=\chi_1 \, \Delta v_h+v_h\, \Delta \chi_1 +2\nabla \chi_1 \cdot \nabla v_h$, there exists $C$, such that $\Vert \Delta (\chi_1 v_h)\Vert_{L^{2}(\ft C_{r})}\leq C$ for $h$ small enough. By  elliptic regularity (see~\cite[Theorem 5, Section 6.3]{Eva}) it comes 
$$\Vert  v_h\Vert_{H^{2}(\ft C_{2r})}\leq C.$$
Let $\alpha\in (0,1)$ be such that  $p_1=\frac{2d}{d-2\alpha}>0$. From the Gagliardo-Nirenberg interpolation inequality (see \cite[Lecture II]{Nir}), the following inequality holds
$$
\Vert \nabla v_h\Vert_{L^{p_1}(\ft C_{2r})} \leq C \Vert v_h\Vert_{H^{2}(\ft C_{2r})}^\alpha\Vert \nabla v_h\Vert_{L^{2}(\ft C_{2r})}^{1-\alpha}+C\Vert \nabla v_h\Vert_{L^{2}(\ft C_{2r})}\leq C\, e^{-\frac{\beta}{h}}. 
$$
From~\eqref{vh},~$\Vert \Delta  v_h\Vert_{L^{p_1}(\ft C_{2r})}\leq C\, e^{-\frac{\beta}{h}}$. Using a cutoff  function $\chi_2 \in C_c^{\infty}(\ft C_{2r})$ such that $\chi_2 \equiv 1$ on $\ft C_{4r}$, we get, as previously, from the elliptic regularity $\Vert  v_h\Vert_{W^{2,p_1}(\ft C_{4r})}\leq C$. Let $p_2=\frac{2d}{d-4\alpha}$ (i.e. $1/p_2=1/p_1-\alpha/d$). If $p_2<0$, then~\cite[Lecture II]{Nir} implies
$$
\Vert \nabla v_h\Vert_{L^{\infty}(\ft C_{4r})} \leq C \Vert v_h\Vert_{W^{2,p_1}(\ft C_{4r})}^\alpha\Vert \nabla v_h\Vert_{L^{p_1}(\ft C_{4r})}^{1-\alpha}+C\Vert \nabla v_h\Vert_{L^{p_1}(\ft C_{4r})}\leq C\, e^{-\frac{\beta}{h}}. 
$$
Thus,~\eqref{est3} is proved  (and one then  chooses $n=2$, i.e. $2^2r=r_0$). Otherwise, we prove~\eqref{est3} by induction as follows. 
From the Gagliardo-Nirenberg interpolation inequality (see \cite[Lecture II]{Nir}), we get
$$
\Vert \nabla v_h\Vert_{L^{p_2}(\ft C_{4r})} \leq C \Vert v_h\Vert_{W^{2,p_1}(\ft C_{4r})}^\alpha\Vert \nabla v_h\Vert_{L^{p_1}(\ft C_{4r})}^{1-\alpha}+C\Vert \nabla v_h\Vert_{L^{p_1}(\ft C_{4r})}\leq C\, e^{-\frac{\beta}{h}}. 
$$ 
We repeat this procedure $n$ times where $n$ is the first integer such that $d-2n\alpha< 0$ and  the Gagliardo-Nirenberg interpolation inequality  implies that $\Vert \nabla v_h\Vert_{L^{\infty}(\ft C_{2^nr})} \leq  C\, e^{-\frac{\beta}{h}}$ which ends the proof of~\eqref{est3}. Since for  any compact $K\subset \ft C$, there exists $r_0\in (0,r_1)$ such that  $K\subset \ft C_{r_0}$,  the inequality~\eqref{eq1} is proved. This concludes the proof of Proposition~\ref{level}. 
\end{proof}

\subsubsection{End of the proof of Theorem~\ref{thm.main}} \label{se3}

To prove Theorem~\ref{thm.main} when $X_0=x\in  \mathcal A(\ft C_{1})$, the following lemma will be needed.
 \begin{lemma}\label{le.exp-qsd-x}
Assume that the assumptions \eqref{H-M}  and \eqref{eq.hip1-j}  are satisfied.  Let $\ft C_1$ be as in~\eqref{eq.hip1-j}. Let us moreover assume that 
$$\min_{\overline{\ft C_1}}f= \min_{\overline \Omega}f.$$
 Let $K$ be  a compact subset of $\Omega$ such that $K\subset \mathcal A(\ft C_{1})$ and let $F\in C^{\infty}(\partial \Omega,\mathbb R)$. Then, there exists $c>0$ such that for all $y\in K$, one has  in the limit $h \to 0$:
$$\mathbb E_{\nu_h}  \left [ F\left (X_{\tau_{\Omega}}\right)\right ]=\mathbb E_{y}  \left [ F\left (X_{\tau_{\Omega}}\right)\right ]  +O\big (e^{-\frac{c}{h} }\big )$$
 uniformly in~$y \in K$.
\end{lemma}
Lemma~\ref{le.exp-qsd-x} implies Proposition~\ref{pr.exp-qsd-x} since we recall that when~\eqref{H-M} and  \eqref{eq.hip1-j}  are satisfied, $\ft C_1=\ft C_{\ft{max}}$ (see~Lemma~\ref{equiv-hipo}).

\begin{proof}
Assume that the assumptions \eqref{H-M}  and \eqref{eq.hip1-j}  are satisfied.  Let us  assume that 
$$\min_{\overline{ \ft C_1} }f= \min_{\overline \Omega}f.$$
 \textbf{Step 1.}    For $\alpha>0$ small enough, let $\ft C_1(\alpha)$ be  as introduced in Definition~\ref{de.ve-E} (see~\eqref{eq.ve-E}): 
\begin{equation}\label{eq.c1a}
\ft C_1(\alpha)= \ft C_1\cap \big\{f<\max_{\overline{\ft  C_1} } f-\alpha\big\}.
\end{equation} 
Let $F\in C^{\infty}(\partial \Omega,\mathbb R)$. In this first step, we will prove that ~$\exists \alpha_0>0$,  $\forall \alpha\in (0,\alpha_0)$,~$\exists c>0$,~$\forall y\in \overline{\ft C_1(\alpha)}$:
\begin{equation}\label{eq.step1-e}
\mathbb E_{\nu_h}  \left [ F\left (X_{\tau_{\Omega}}\right)\right ]=\mathbb E_{y}  \left [ F\left (X_{\tau_{\Omega}}\right)\right ]  +O\big (e^{-\frac{c}{h} }\big )
\end{equation}
in the limit $h \to 0$ and uniformly in~$y \in \overline{\ft C_1(\alpha)}$.

Let us recall that from the notation of Proposition~\ref{level}, for all $x\in \overline \Omega$:
$$v_h(x)=\mathbb E_{x}  \left [ F\left (X_{\tau_{\Omega}}\right)\right ].$$
From~\eqref{eq.expQSD}, one has:
\begin{align}
\nonumber
\mathbb E_{\nu_h}  \left [ F\left (X_{\tau_{\Omega}}\right)\right ]&= \left (\int_{\Omega} u_h\, e^{-\frac{2}{h}f}\right )^{-1}\ \int_{\Omega}v_h\, u_h\, e^{-\frac{2}{h}f} \\
\label{eq.equation1} 
&= \frac{1}{Z_h(\Omega)}\, \int_{\overline{\ft C_1(\alpha)}   } \ v_h\, u_h\, e^{-\frac{2}{h}f}
 + \frac{1}{Z_h(\Omega)}\int_{ \Omega \setminus \overline{\ft C_1(\alpha)}  } \ v_h\, u_h\, e^{-\frac{2}{h}f} ,
\end{align}
where $Z_h(\Omega):=\displaystyle \int  _{\Omega} u_h \ e^{- \frac{2}{h} f } $ and $u_h$ is the eigenfunction associated with the principal eigenvalue $\lambda_h$  of $-L^{D,(0)}_{f,h}$ (see~\eqref{eq.lh}) which satisfies~\eqref{eq.norma}. Let us first deal with the second term in~\eqref{eq.equation1}. Since~\eqref{H-M} and \eqref{eq.hip1-j} hold, and because it is assumed that $\min_{\overline{\ft C_1}}f= \min_{\overline \Omega}f$, one obtains from~Proposition~\ref{pr.masse} (applied to  $\ft O=\Omega$, see~\eqref{eq.concentration1}) that   there exists $C>0$ such that for   $h$ small enough:
$$\frac{1}{Z_h(\Omega) }  \le Ch^{-\frac d4} e^{\frac 1h \min  \limits_{\overline \Omega} f}. $$
For all  $\alpha>0$ small enough, one has $\big(\overline{\Omega} \setminus  \ft C_1(\alpha) \big)\cap \argmin_{\ft C_1} f=\emptyset$. 
Therefore, using~\eqref{eq.concentration02}  (applied for $\alpha$ small enough with $\ft O=\Omega \setminus \overline{\ft C_1(\alpha)}$), one has from~\eqref{eq.c1a}, that for all $\alpha>0$ small enough,  there exists $c>0$ such that  when $h\to 0$:
$$ \int_{ \Omega \setminus \overline{\ft C_1(\alpha)}  }  u_h\, e^{-\frac{2}{h}f}=O\Big(e^{-\frac 1h \big(\min  \limits_{\overline \Omega} f+c\big)}\Big).$$
 Thus,  there exists $\alpha_0>0$ such that for all $\alpha\in (0,\alpha_0)$ there exists $c>0$ such that  when $h\to 0$:
\begin{equation}
\label{eq.pp2}
\frac{1}{Z_h(\Omega) } \int_{ \Omega \setminus \overline{\ft C_1(\alpha)}  }  u_h\, e^{-\frac{2}{h}f} =O\big (e^{-\frac ch}\big).
\end{equation}
Then, since $\Vert v_h\Vert_{L^{\infty}(\overline \Omega)}\leq \Vert F\Vert_{L^{\infty}(\partial \Omega)}$,  one obtains that
 \begin{equation}
\label{eq.pp2-bis}
\frac{1}{Z_h(\Omega) }  \int_{ \Omega \setminus \overline{\ft C_1(\alpha)}  }  \ v_h\, u_h\, e^{-\frac{2}{h}f}=O\big (e^{-\frac ch}\big).
\end{equation}
Let us now deal with the first term in~\eqref{eq.equation1}. Let us recall that 
 $\ft C_1\subset \Omega$ is a connected component of $ \ft \{f<\max_{\overline{\ft C_1}}f\}$.  
Moreover, for  $\alpha\in (0,\alpha_0)$ ($\alpha_0>0$ small enough),  
the compact set $\overline{\ft C_1(\alpha)}$ is   connected and  $\overline{\ft C_1(\alpha)}\subset \ft C_1$. Therefore, from Proposition~\ref{level} applied to  $K=\overline{\ft C_1(\alpha)}$ for $\alpha\in (0,\alpha_0)$, one obtains that there exists $\delta_{\alpha}>0$ such that for all $y\in \overline{\ft C_1(\alpha)}$,
\begin{equation}\label{eq.equality}
\frac{1}{Z_h(\Omega) } \int_{\overline{\ft C_1(\alpha)}} \  v_h\, u_h\, e^{-\frac{2}{h}f}=\, \frac{v_h(y)}{Z_h(\Omega) } \int_{\overline{\ft C_1(\alpha)}}  \ u_h\, e^{-\frac{2}{h}f}+ \frac{O\big( e^{-\frac{\delta_{\alpha}}{h}} \big ) }{Z_h(\Omega) }  \displaystyle{\int_{\overline{\ft C_1(\alpha)}}  \  u_h\, e^{-\frac{2}{h}f} } 
\end{equation}
 in the limit $h\to 0$ and uniformly with respect to $y \in \overline{\ft C_1(\alpha)}$.
Moreover,  for all $\alpha\in (0,\alpha_0)$ there exists $c>0$ such that  in the limit $h\to 0$:
\begin{equation}\label{eq.1-ka}
 \frac{1}{Z_h(\Omega) }  \int_{\overline{\ft C_1(\alpha)}}  u_h\, e^{-\frac{2}{h}f} =1+O\left (e^{-\frac{c}{h}}\right ).
\end{equation}
which follows from the fact that 
$$\frac{1}{Z_h(\Omega) }  \int_{\overline{\ft C_1(\alpha)}}  u_h\, e^{-\frac{2}{h}f}    =1- \frac{1}{Z_h(\Omega) } \int_{\Omega\setminus \overline{\ft C_1(\alpha)}}  u_h\, e^{-\frac{2}{h}f},$$
together with~\eqref{eq.pp2}. 
Let us now fix $\alpha\in (0,\alpha_0)$. Then, using~\eqref{eq.equality} and~\eqref{eq.1-ka},  $\exists c>0$,  $\exists \delta_\alpha>0$,  $\forall y\in  \overline{\ft C_1(\alpha)}$:
\begin{equation}\label{eq.equ1}
\frac{1}{Z_h(\Omega) }  \, \int_{\overline{\ft C_1(\alpha)}} \ v_h\, u_h\, e^{-\frac{2}{h}f}=v_h(y)\left (1+O\left (e^{-\frac{c}{h}}\right )\right ) +O\left ( e^{-\frac{\delta_{\alpha}}{h}}  \right)
\end{equation}
in the limit $h \to 0$ and uniformly with respect to~$y \in \overline{\ft C_1(\alpha)}$. Therefore, using~\eqref{eq.equation1}, \eqref{eq.pp2-bis} and \eqref{eq.equ1},~$\exists \alpha_0>0$,  $\forall \alpha\in (0,\alpha_0)$,~$\exists c>0$,~$\forall y\in \overline{\ft C_1(\alpha)}$:
$$
\mathbb E_{\nu_h}  \left [ F\left (X_{\tau_{\Omega}}\right)\right ]= \mathbb E_{y}  \left [ F\left (X_{\tau_{\Omega}}\right)\right ]+O\big (e^{-\frac{c}{h} }\big ),
$$
in the limit $h \to 0$ and uniformly with respect to  $y \in \overline{\ft C_1(\alpha)}$. 
This concludes the proof of~\eqref{eq.step1-e}.
\medskip

\noindent
 \textbf{Step 2.}  Let us now conclude the  proof of Lemma~\ref{le.exp-qsd-x} by considering a compact subset $K$ of $\Omega$ such that $ K\subset \mathcal A(\ft C_{1})$. Let us recall that (see~\eqref{eq.ad}):
$$\mathcal A(\ft C_{1})=\{ x\in \Omega, \,t_x=+\infty \text{ and } \omega(x)\subset \ft C_{1}\}.$$ 
Since $\ft C_{1}$ is open and stable by the flow $\varphi_t(\cdot)$ (see~\eqref{hbb}), the continuity of  $\varphi_t(\cdot)$   implies that  there exists  $T_K\ge 0$ such that for all  $x\in K$, 
$$\varphi_{T_K}(x)\in \ft C_1.$$ 
Moreover, since $K$ is compact and for all $x\in K$, $t_x=+\infty$ (i.e. $\varphi_t(x)\in \Omega$ for all $t\ge 0$), there exists $\delta>0$ such that  all continuous curves $\gamma: [0,T_K]\to \overline \Omega$ such that  
$$\exists x\in K, \  \sup_{t\in [0,T_K]} \big \vert \gamma(t)-\varphi_t(x) \big \vert \le \delta,  $$
satisfy: 
\begin{equation}\label{eq.incluo}
\forall t\in [0,T_K], \ \gamma(t)\in  \Omega.
\end{equation}
Furthermore, up to choosing $\delta>0$ smaller, there exists $\alpha_K>0$ such that 
\begin{equation}\label{eq.incluo2}
\big \{\varphi_{T_K}(x)+z,  \ x\in K \text{ and } \vert z\vert \le \delta \big  \}\subset \ft C_1( {\alpha_K})\,\,  \text{ (see \eqref{eq.c1a})}.
\end{equation}
Let us now recall the following estimate of  Freidlin and Wentzell (see~\cite[Theorems 2.2 and 2.3 in Chapter 3, and Theorem 1.1 in Chapter 4]{FrWe},~\cite{Day2},~\cite[Theorem 3.5]{DZ} and~\cite[Theorem 5.6.3]{friedman2012stochastic}). For all $x\in K$, it holds:
\begin{equation}\label{eq.WFr-est}
\limsup_{h\to 0} h\, \ln \mathbb P_x\Big [ \sup_{t\in [0,T_K]} \big \vert X_t-\varphi_t(x) \big \vert \ge \delta   \Big ]\le - I_{x,T_K},
\end{equation}
where 
$$I_{x,T_K}=  \frac 12\,  \inf_{ \gamma \in H^1_{x,T_K} (\delta) } \,  \int_0^{T_K}\Big \vert \frac{d}{dt}  \gamma(t)+ \nabla f( \gamma(t))\Big \vert^2dt \ \ \in \mathbb R_+^*\cup \{+\infty\},$$
and $H^1_{x,T_K}(\delta)$ is the set of curves $ \gamma: [0,T_K]\to \Omega$ of regularity $H^1$ such that $ \gamma(0)=x$ and $\sup_{t\in [0,T_K]} \big \vert  \gamma(t)-\varphi_t(x) \big \vert \ge \delta$. 
Since  $K$ is compact, there exists $\eta_K>0$ such that for    $h$ small enough, it holds:
\begin{equation}\label{eq.WFr}
\sup_{x\in K}\mathbb P_x\Big [ \sup_{t\in [0,T_K]} \big \vert X_t-\varphi_t(x) \big \vert \ge \delta    \Big ]\le e^{-\frac{\eta_K}{h}}.
\end{equation}
Notice that  when $X_0=x\in K$ and  $\sup_{t\in [0,T_K]} \big \vert X_t-\varphi_t(x) \big \vert \le \delta$, it holds from~\eqref{eq.incluo} and~\eqref{eq.incluo2}:
\begin{equation}\label{eq.TtC}
\tau_\Omega>T_K \text{ and } X_{T_K}\in \ft C_1( {\alpha_K}).
\end{equation}
Let us now consider $F\in C^\infty(\pa \Omega,\mathbb R)$. Let  $x\in K$. Then, 
$$ \mathbb E_{x}  \left [ F\left (X_{\tau_{\Omega}}\right)\right ]=  \mathbb E_{x}  \left [ F\left (X_{\tau_{\Omega}}\right) \mathbf{1}_{\sup_{t\in [0,T_K]} \big \vert X_t-\varphi_t(x) \big \vert \le \delta} \right ]+ \mathbb E_{x}  \left [ F\left (X_{\tau_{\Omega}}\right) \mathbf{1}_{\sup_{t\in [0,T]} \big \vert X_t-\varphi_t(x) \big \vert \ge \delta} \right ].$$
Using~\eqref{eq.WFr}, it holds for $h$ small enough:
$$\Big \vert \mathbb E_{x}  \left [ F\left (X_{\tau_{\Omega}}\right) \mathbf{1}_{\sup_{t\in [0,{T_K}]} \big \vert X_t-\varphi_t(x) \big \vert \ge \delta} \right ]\Big \vert  \le  \Vert F\Vert_{L^{\infty}}    \,  e^{-\frac{\eta_K}{h}}.$$
Using~\eqref{eq.TtC},~\eqref{eq.step1-e} (with $\alpha= \alpha_K$),~\eqref{eq.WFr}, and  the Markov property of the process~\eqref{eq.langevin},  there exists $c>0$ such that for all $x\in K$, one has when $h\to 0$:
\begin{align*}
  \mathbb E_{x}  \left [ F\left (X_{\tau_{\Omega}}\right) \mathbf{1}_{\sup_{t\in [0,{T_K}]} \big \vert X_t-\varphi_t(x) \big \vert \le \delta} \right ]&=  \mathbb E_{x}  \left [   \mathbb E_{X_{T_K}}\big[F\left (X_{\tau_{\Omega}}\right)\big] \mathbf{1}_{\sup_{t\in [0,{T_K}]} \big \vert X_t-\varphi_t(x) \big \vert \le \delta} \right ] 
  \\
  &=\Big(\mathbb E_{\nu_h}  \left [ F\left (X_{\tau_{\Omega}}\right)\right ]  +O\big (e^{-\frac{c}{h} }\big )\Big)\\
  &\quad \times \mathbb P_x\Big  [ \sup_{t\in [0,{T_K}]} \big \vert X_t-\varphi_t(x) \big \vert \le \delta    \Big  ] \\
  &=\mathbb E_{\nu_h}  \left [ F\left (X_{\tau_{\Omega}}\right)\right ]  +O\big (e^{-\frac{c}{h} }\big ),
  \end{align*}
  uniformly in $x\in K$. This concludes the proof of Lemma~\ref{le.exp-qsd-x}.
\end{proof}

 Let us now end the proof of Theorem~\ref{thm.main} when $X_0=x\in K$, where  $K$ is  a compact subset of $\Omega$ such that $K\subset \mathcal A(\ft C_{1})$.  Recall that, when~\eqref{eq.hip1} holds, which is equivalent to~\eqref{eq.hip1-j}, one has   $\ft C_1=\ft C_{\ft{ max}}$, see Lemma~\ref{equiv-hipo}. 
 
 \begin{proof} Let $K$ be  a compact subset of $\Omega$ such that $$K\subset \mathcal A(\ft C_{1})$$
 and let us consider that the process starts from $X_0=x\in K$.
Let $F\in C^{\infty}(\partial \Omega,\mathbb R)$.
 Let us notice that the proof is not a direct consequence of Lemma~\ref{le.exp-qsd-x} since in Theorem~\ref{thm.main}, less regular test functions $F$ are considered. 
 The proof of Theorem~\ref{thm.main} is  divided into three steps. In the following we assume that  \eqref{H-M}, \eqref{eq.hip1-j}, \eqref{eq.hip2-j} and \eqref{eq.hip3-j} are satisfied.

\medskip

\noindent
\textbf{Step 1.} Proof of~\eqref{eq.t1}  and \eqref{eq.t2} when  $X_0=x\in K$. 
\medskip

\noindent
 Let    us first show that if $\Sigma\subset \partial \Omega$ is open and   there exists $\beta>0$ such that $\Sigma\cap \bigcup_{i=1}^{k_1^{\pa \Omega}}B_{\partial \Omega}(z_i,\beta)=\emptyset$ (where $B_{\partial \Omega}(z_i,\beta)$ is the open ball in~$\partial \Omega$ of radius $\beta$ centered at $z_i$), then, for all $x\in K$, 

\begin{equation}\label{eq.sigma-x}
\mathbb P_{x}  \left [ X_{\tau_{\Omega}}\in \Sigma \right ]=O\left(e^{-\frac{c}{h}}\right )
\end{equation}
in the limit $h \to 0$ and uniformly in~$x \in K$. 
 To this end, let us 
 consider $\tilde F\in C^{\infty}(\partial \Omega,[0,1])$ be such that 
 $$\tilde F=1 \text{ on } \Sigma \ \text{ and } \ \tilde F=0 \text{ on }\, \bigcup_{i=1}^{\ft k_1^{\pa \Omega}}B_{\partial \Omega}(z_i,\frac{\beta}{2}).$$
  Using Lemma~\ref{le.exp-qsd-x}, 
there exists $  c>0$ such that for all $x\in K$:
 $$\mathbb P_{x}  \left [ X_{\tau_{\Omega}}\in \Sigma \right ]\le \mathbb E_{x}  \left [ \tilde F(X_{\tau_{\Omega}})\right]\le \mathbb E_{\nu_h}  \left [ \tilde F(X_{\tau_{\Omega}})\right]+O\left(e^{-\frac ch}\right )$$
 in the limit $h \to 0$ and uniformly in~$x \in K$. 
Then,~\eqref{eq.sigma-x} follows from~\eqref{eq.t1} applied to  $\tilde F$ and the family of sets $\Sigma_i=B_{\partial \Omega}(z_i,\frac{\beta}{2})$ for $i\in\{1,\ldots,\ft k_1^{\pa \Omega}\}$ when $X_0\sim  \nu_h$. 

Let us now prove~\eqref{eq.t1} and~\eqref{eq.t2}. 
Let $F\in L^{\infty}(\partial \Omega,\mathbb R)$ and for all $i\in\{1,\dots,\ft k_{1}^{\pa \Omega}\}$, let $\Sigma_{i}\subset \pa \Omega$ be an open set which contains  $z_{i}$.  Let us assume in addition that $\Sigma_i\cap \Sigma_j=\emptyset$ if $i\neq j$. One has for any 
 $x\in K$
$$\mathbb E_{x}  \left [ F(X_{\tau_{\Omega}})\right]=\sum \limits_{i=1}^{\ft k_1^{\pa \Omega}}
\mathbb E_{x} \left [ \mathbf{1}_{\Sigma_{i}}F\left (X_{\tau_{\Omega}} \right )\right]  +  \mathbb E_{x} \Big  [ \mathbf{1}_{\pa \Omega\setminus \bigcup_{i=1}^{\ft k_1^{\pa \Omega}}\Sigma_{i} }F\left (X_{\tau_{\Omega}} \right )\Big ].$$
Moreover, one has:
$$\Big \vert  \mathbb E_{x} \Big  [ \mathbf{1}_{\pa \Omega\setminus\bigcup_{i=1}^{\ft k_1^{\pa \Omega}}\Sigma_{i}}F\left (X_{\tau_{\Omega}} \right )\Big  ]\Big \vert \le \Vert F\Vert_{L^{\infty}}\, \mathbb P_{x} \Big [ X_{\tau_{\Omega}} \in \pa \Omega\setminus \bigcup_{i=1}^{\ft k_1^{\pa \Omega}}\Sigma_{i}   \Big].$$
Using~\eqref{eq.sigma-x} with $\Sigma=\pa \Omega\setminus \bigcup_{i=1}^{\ft k_1^{\pa \Omega}}\Sigma_{i} $, one gets~\eqref{eq.t1}. 

Let us now prove~\eqref{eq.t2}. Let $j\in \{\ft k_1^{\pa \ft C_1}+1,\ldots,\ft k_1^{\pa \Omega}\}$. Let us introduce $\ft k_1^{\pa \Omega}$  disjoint open  sets $\tilde \Sigma_{i}\subset \pa \Omega$ ($i\in \{1,\ldots,\ft k_1^{\pa \Omega}\}$)  and  a  smooth function $G$ supported  in $\tilde \Sigma_j$ such that $\mathbf{1}_{\Sigma_{j}}F \le \Vert F\Vert_{L^{\infty}} \, G$ (in order to apply Lemma~\ref{le.exp-qsd-x} with $G$). 
To this end, let $\delta>0$ be such that for any $k\in \{1,\ldots,\ft k_1^{\pa \Omega}\}$ with $k\neq j$, the sets $\tilde \Sigma_k  := B_{\partial \Omega}(z_k, \delta) $  and $\tilde \Sigma_j:=\cup_{z\in \Sigma_j} B_{\partial \Omega}(z, \delta)$ are disjoint. Let us consider 
$$G\in C^{\infty}_c(\tilde \Sigma_j,[0,1]) \text{ such that } G=1 \text{ on } \Sigma_j.$$
 Using Lemma~\ref{le.exp-qsd-x}, there exists $  c>0$ such that for all $x\in K$
\begin{align*}
\big\vert \mathbb E_{x}  \left [ \mathbf{1}_{\Sigma_{j}}F(X_{\tau_{\Omega}})\right] \big \vert \le \Vert F\Vert_{L^{\infty}}\, \mathbb P_{x} [ X_{\tau_{\Omega}} \in  \Sigma_j  ]
&\le \Vert F \Vert_{L^{\infty}}\,\mathbb E_{x}  \left [ G(X_{\tau_{\Omega}})\right]\\
&= O(\mathbb E_{\nu_h}  \left [ G(X_{\tau_{\Omega}})\right]  )+ O(e^{-\frac ch})
\end{align*}
 in the limit $h \to 0$ and uniformly in~$x \in K$.  
Then, using~\eqref{eq.t2} and item 3 in Theorem~\ref{thm.main}  applied with  $X_0\sim  \nu_h$, $(\tilde \Sigma_{i})_{i\in \{1,\ldots,\ft k_1^{\pa \Omega}\}}$, and~$G$, it holds when $h\to 0$: $  \mathbb E_{x}  \left [ \mathbf{1}_{\Sigma_{j}}F(X_{\tau_{\Omega}})\right]=O\big (h^{\frac 14}\big )$ and when \eqref{eq.hip4-j} holds, one has  when $h\to 0$: $$ \mathbb E_{x}  \left [ \mathbf{1}_{\Sigma_{j}}F(X_{\tau_{\Omega}})\right]=O(e^{-\frac ch}),$$
 for some $c>0$. 
 This concludes the proof of~\eqref{eq.t2}. 

\medskip
\noindent
\textbf{Step 2.} Proof of~\eqref{eq.t3} when  $X_0=x\in K$. 
\medskip

\noindent
For all  $j\in\{1,\dots,\ft k_{1}^{\pa \Omega}\}$, let $\Sigma_{j}$ be open subset of $\pa \Omega$ such that $z_{j}  \in \Sigma_{j}$.  Let us assume that $\Sigma_k\cap \Sigma_j=\emptyset$ if $k\neq j$. Let $F\in L^{\infty}(\partial \Omega,\mathbb R)$ be $C^{\infty}$ in a neighborhood  of $z_i$ for some $i\in\{1,\dots,\ft k_{1}^{\pa \ft C_1}\}$.  Let $\beta>0$ be such that $F$ is $C^{\infty}$ on $B_{\partial \Omega}(z_i,2\beta )\subset \Sigma_i$ and let 
$\chi_i\in C^{\infty}(\pa \Omega,[0,1])$ be such that ${\rm supp}\, \chi_i\subset B_{\partial \Omega}(z_i,\beta )$,  $\chi_i=1$ on $B_{\partial \Omega}(z_i,\frac{\beta}{2})$. One has:
$$\mathbb E_{x}  \left [ \mathbf{1}_{\Sigma_i}F(X_{\tau_{\Omega}})\right]=\mathbb E_{x}  \left [ \chi_iF(X_{\tau_{\Omega}})\right]+\mathbb E_{x}  \left [ (\mathbf{1}_{\Sigma_i}-\chi_i)F(X_{\tau_{\Omega}})\right].$$
Using Lemma~\ref{le.exp-qsd-x} with $\chi_iF\in C^\infty$ and~\eqref{eq.t1}-\eqref{eq.t2} with $X_0\sim \nu_h$,~$F\chi_i$ and the family of disjoint open sets $\{ \Sigma_j, j=1,\ldots,\ft k_1^{\pa \Omega}, j\neq i \} \cup \{B_{\partial \Omega}(z_i,\frac{\beta}{2})\}$,  
there exists $  c>0$ such that for all $x\in K$:
\begin{align*}
\mathbb E_{x}  \left [ \chi_iF(X_{\tau_{\Omega}})\right] &=\mathbb E_{\nu_h}  \left [ \chi_iF(X_{\tau_{\Omega}})\right] +O\left (e^{-\frac{c}{h} }\right )\\
&=\mathbb E_{\nu_h}  \left [ \mathbf{1}_{B_{\partial \Omega}(z_i,\frac{\beta}{2})}F(X_{\tau_{\Omega}})\right] +O\left (e^{-\frac{c}{h} }\right )=F(z_i)\, a_i+O\big (h^{\frac 14}\big )
\end{align*}
 in the limit $h \to 0$ and uniformly in~$x \in K$,  
and where $a_i$ is defined in~\eqref{ai}. In addition, using Theorem~\ref{thm.main} when $X_0\sim  \nu_h$, when~\eqref{eq.hip4-j} holds, one can replace $O\big (h^{\frac 14}\big )$ in the last computation by~$O(h)$.
Moreover, using~\eqref{eq.sigma-x} with $\Sigma=\Sigma_i\setminus  B_{\partial \Omega}(z_i,\frac{\beta}{2})$:  there exists $  c>0$ such that for all $x\in K$:
$$\big \vert  \mathbb E_{x}  \left [ (\mathbf{1}_{\Sigma_i}-\chi_i)F(X_{\tau_{\Omega}})\right] \big \vert \le \Vert F\Vert_{L^{\infty}}\, \mathbb P_{x} \left [ X_{\tau_{\Omega}} \in \Sigma_i\setminus  B_{\partial \Omega}\Big (z_i,\frac{\beta}{2}\Big ) \right]= O\left (e^{-\frac ch}\right)$$
in the limit $h \to 0$ and uniformly in~$x \in K$. 
Thus, one has when $h\to 0$ and uniformly with respect to $x\in K$:
$$\mathbb E_{x}  \left [ \mathbf{1}_{\Sigma_i}F(X_{\tau_{\Omega}})\right]=F(z_i)\, a_i+O\big (h^{\frac 14}\big ),$$ and when~\eqref{eq.hip4-j} holds, one has: 
$$\mathbb E_{x}  \left [ \mathbf{1}_{\Sigma_i}F(X_{\tau_{\Omega}})\right]=F(z_i)\, a_i+O(h).$$  This concludes the proof of~\eqref{eq.t3} when $X_0=x \in  K$ 
 and the proof of   Theorem~\ref{thm.main}.
\end{proof}


\subsection{Proof of Theorem~\ref{thm.2}}
In this section, one proves Theorem~\ref{thm.2}.

\begin{proof}[Proof of Theorem~\ref{thm.2}.]
Let us assume that  \eqref{H-M} holds. Let $\ft C\in \mathcal C$. Assume that  (see~\eqref{eq.cc1})
$$
\pa \ft C\cap \pa \Omega\neq \emptyset \  \text{ and } \  \vert \nabla f\vert \neq 0 \text{ on } \pa \ft C.
$$
To prove Theorem~\ref{thm.2}, the strategy consists in using Theorem~\ref{thm.main}  with a    subdomain   $\Omega_{\ft C}$ of $\Omega$ containing $\ft C$ such that in the limit  $h\to 0$,  the most probable places of exit of the process~\eqref{eq.langevin} from $\Omega_{\ft C}$ when $X_0=x\in \ft C$  are the elements of $\pa \ft C \cap \pa \Omega$. This will   imply (since the trajectories of the process \eqref{eq.langevin} are continuous) that the most probable places of exit of the process~\eqref{eq.langevin} from $\Omega$ when $X_0=x\in \ft C$  are the elements of $\pa \ft C \cap \pa \Omega$, which is the statement of Theorem~\ref{thm.2}. \\
The proof of Theorem~\ref{thm.2} is divided into two steps. 
\medskip

\noindent
\textbf{Step 1}: Construction of a domain $\Omega_{\ft C}$ containing $\ft C$.  
\medskip

\noindent
In this step, one  constructs a subset $\Omega_{\ft C}$ of $\Omega$ such that 
\begin{equation}\label{eq.OmegaC}
\left\{
\begin{aligned}
&\text{$\Omega_{\ft C}$ is a $C^\infty$ connected open subset of $\Omega$ containing $\ft C$},\\
 &\text{$\pa \Omega_{\ft C}\cap \pa \Omega $ is a neighborhood of  $\pa \ft  C \cap \pa \Omega$ in $\pa \Omega$},  \\ 
  &\text{argmin}_{\pa \Omega_{\ft C}}f=\pa \ft C\cap \pa \Omega, \\
  &\big \{ x\in  \overline{\Omega_{\ft C}},\,  f(x)<\min_{\pa \Omega_{\ft C}}f\big \}=\ft C,\\
  & \text{the critical points of $f$ in  } \overline{\Omega_{\ft C}} \text{ are included in } \ft C,
\end{aligned}
\right.
\end{equation}
and
\begin{equation}\label{eq.OmegaC2}
f: \pa \Omega_{\ft C} \to \mathbb R \text{ is a Morse function}.
\end{equation}
To construct a domain $\Omega_{\ft C}\subset \Omega$ which satisfies~\eqref{eq.OmegaC} and~\eqref{eq.OmegaC2}, we first briefly recall the local structure of $f$ near $\pa \ft C$ to then    build a neighborhood $\ft V_{\ft C}$ of $ \overline{ \ft C}$ in $\overline \Omega$ (this construction is similar to the the construction of  $\ft V_{k}$ made in Step 5 in the proof of Proposition~\ref{pr.p1}). The set $\ft V_{\ft C}$ is then used to justify the existence of a domain  $\Omega_{\ft C}$ satisfying \eqref{eq.OmegaC} and~\eqref{eq.OmegaC2}.  
Let $\lambda\in \mathbb R$ be such that $\ft C$ is a connected component of $\{f<\lambda\}$ (see~\eqref{eq.Cdef2}).   Then, for $z\in \pa \ft C$, we introduce a ball of radius $\ve_z>0$ centred at $z$ in $\overline \Omega$ as follows.  
\begin{enumerate}
\item If $z\in \pa \ft C\cap  \Omega$: Since $z\in \Omega$ and $\vert \nabla f(z)\vert \neq 0$, there exists  $\ve_z >0$ such that $\overline {B(z,\ve_z)}\subset \Omega$, 
 $\vert \nabla f(z)\vert \neq 0$ on $\overline {B(z,\ve_z)}$, and,  according to~\cite[Section 5.2]{HeNi1}, $ {B(z,\ve_z)}\cap \{f<\lambda \}$ is connected and   $ {B(z,\ve_z)}\cap \pa \{f<\lambda \}= {B(z,\ve_z)}\cap \{f=\lambda \}$ (where we recall that $B(z,\ve_z)=\{x\in \overline \Omega \ \text{s.t.} \ |x-z|<\ve_z\}$).

\item If $z\in \pa \ft C\cap \pa \Omega$: Recall  that $z\in 
\ft U_1^{\pa \Omega}$  (see~\eqref{eq.U1paOmega})  and thus, $\pa_nf(z)>0$ and $z$ is a non degenerate local minimum of $f|_{\pa \Omega}$. Thus, there exists $\ve_z >0$,  such that $\vert \nabla f(z)\vert \neq 0$ on $\overline {B(z,\ve_z)}$  and such that, according to~\cite[Section 5.2]{HeNi1},   
 $ {B(z,\ve_z)}\cap \{f<\lambda \}$ is connected and included in $\Omega$. In addition,   it holds $ {B(z,\ve_z)}\cap \pa \{f<\lambda \}= {B(z,\ve_z)}\cap \{f=\lambda \}$ is included in $\overline \Omega$.
Finally,   up to choosing   $\ve_z>0$ smaller,  one has:
\begin{equation}\label{eq.GAMMAC0}
\argmin_{  \overline{B_{\pa \Omega}(z,\ve_z )} }   f=\{z\},
\end{equation}
where we recall that $B_{\partial \Omega}(z, \ve_z )$ is the open ball of radius $\ve_z$ centred in~$z$ in~$\pa \Omega$, and, 
\begin{equation}\label{eq.panf-b}
\vert \nabla_Tf \vert \neq 0 \text{ on } \overline {B_{\pa \Omega}(z,\ve_z)}\setminus\{z\}   \ \text{ and } \ \pa_nf>0 \text{ on }   \overline{B(z,\ve_z)}\cap \pa \Omega.
\end{equation}
\end{enumerate}
Items 1 an 2 above imply that for all $z\in \pa \ft C$, by definition of $\ft C$ (see Theorem~\ref{thm.2} and Definition~\ref{de.1}), 
\begin{equation}\label{eq.GAMMAC01}
{B(z,\ve_z)}\cap \ft C={B(z,\ve_z)}\cap \{f<\lambda \}  \text{ and thus, }  {B(z,\ve_z)}\cap \pa \ft C={B(z,\ve_z)}\cap \{f=\lambda \}.
\end{equation}  
One then defines:
$$\ft V_{\ft C}:=  \left (\,  \bigcup_{z\in \pa \ft C}  B(z,\ve_z)\,  \right)\bigcup \ft C  .$$
The set $\ft V_{\ft C}$ is an   open neighborhood of   $\overline{\ft C}$ in $\overline \Omega$. Moreover, according to items~1 and~2 above,
\begin{equation}\label{eq.vC-s}
\vert\nabla f\vert \neq 0 \text{ on } \overline{\ft V_{\ft C}}\setminus \ft C,
\end{equation}
and  using in addition~\eqref{eq.GAMMAC01}, 
\begin{equation}\label{eq.vC-s2}
\{f< \lambda\} \cap \ft V_{\ft C}=\ft C\,  \text{ and } \{f\le \lambda\} \cap \ft V_{\ft C}= \overline{\ft C}.
\end{equation} 
The second statement in~\eqref{eq.vC-s2} implies that  $\overline{\ft C}$ is a connected component of $\{f\le \lambda\}$. Thus, for $r>0$ small enough $\overline{ {\ft C}(\lambda+r) }\subset \ft V_{\ft C}$, where ${\ft C}(\lambda+r)$ is the connected component of  $\{f< \lambda+r\}$ which contains $\ft C$. 
 This suggests that a natural candidate to satisfy~\eqref{eq.OmegaC} and~\eqref{eq.OmegaC2}  is  the domain ${\ft C}(\lambda+r)$. However, for $r>0$ small enough, the boundary of  ${\ft C}(\lambda+r)$ is not   $C^\infty$, it is composed  of two smooth pieces     $\overline{\pa\ft C(\lambda+r)\cap \Omega}=\{x\in \pa \ft C(\lambda+r), \, f(x) = \lambda + r\, \} $   and  $  \pa\ft C(\lambda+r) \cap \partial \Omega$. The union of this two sets  gives  rise  to "corners". 
  Moreover, the function $f|_{\pa\ft C(\lambda+r)\cap \Omega}$ is not a Morse function since $f\equiv \lambda+r$ on $\overline{\pa\ft C(\lambda+r)\cap \Omega}$.  
To justify the existence of  a domain~$\Omega_{\ft C}$ which satisfies \eqref{eq.OmegaC} and~\eqref{eq.OmegaC2}, we proceed in two steps, as follows.
\begin{itemize}
\item \underline{Domain  $D_{\ft C}$   containing $\ft C$ which satisfies \eqref{eq.OmegaC} and $\pa_nf>0$ on $\pa D_{\ft C}$}. 
The subdomain $D_{\ft C}$ of $\Omega$ is constructed as a smooth regularization of the set $\ft C(\lambda+r)$ with $r>0$ such that $\overline{ {\ft C}(\lambda+r) }\subset \ft V_{\ft C}$ by modifying ${\ft C}(\lambda+r)$ in a neighborhood of ${ \{x\in \pa \ft C(\lambda+r), \, f(x) = \lambda + r\, \} } \cap \partial \Omega$ (where the two smooth pieces of $\pa  \ft C(\lambda+r)$ intersect each other).      Moreover, $\pa_n f>0$ on   $\overline{\pa\ft C(\lambda+r)\cap \Omega}$  (since there is no critical point of $f$ on $\overline{\pa\ft C(\lambda+r)\cap \Omega}= \{x\in \pa \ft C(\lambda+r), \, f(x) = \lambda + r\, \} $)   and on  $  \pa\ft C(\lambda+r) \cap \partial \Omega$ (since $\overline{ {\ft C}(\lambda+r) }\subset \ft V_{\ft C}$ and $\pa_nf>0$ on $\ft V_{\ft C}\cap \pa \Omega$,  see the second inequality in~\eqref{eq.panf-b}).
 Thus, using in addition~\eqref{eq.vC-s2} together with the fact that $\ft V_{\ft C}$ is an   open neighborhood of   $\overline{\ft C}$ in $\overline \Omega$,  there exists    a $C^\infty$ connected   open   subset  $D_{\ft C}$ of $\Omega$ such that 
\begin{equation}\label{G-DC0}
\ft C\subset D_{\ft C}, \ \overline{D_{\ft C}}\subset \ft V_{\ft C},
\end{equation}
and  
\begin{equation}\label{eq.pan-dc}
\pa_n f>0\text{ on } \pa D_{\ft C},
\end{equation}
which satisfies,  for some $\beta>0$ and $\Sigma_{\ft C}\subset \Omega$, 
\begin{equation}\label{G-DC}
\pa D_{\ft C}=  \left (\,  \bigcup_{z\in \pa \ft C\cap \pa \Omega}  B_{\pa \Omega}(z,\ve_z/2)    \right)\bigcup \,\overline{ \Sigma_{\ft C}}, \,  \text{ where, }  \, f\ge  \lambda+\beta \, \text{ on } \, \overline{\Sigma_{\ft C}}.  
\end{equation}
Finally, according to the first statement in~\eqref{eq.panf-b}, there exists $\delta_0>0$ such that for any  open  $\delta$-neighborhood $U_{\pa \Omega}^\delta$ of $\pa \Omega$ in $\overline \Omega$, with $\delta\in (0,\delta_0)$, one has 
\begin{equation}\label{eq.PC-dc}
\vert  \nabla_T f\vert \neq 0 \text{ on } \overline{ \pa D_{\ft C} \cap  {U_{\pa \Omega}^\delta}} \, \setminus (\pa \ft C\cap \pa \Omega),
\end{equation}
where $\nabla_Tf$ is the tangential gradient of $f$ on $\pa D_{\ft C}$. 

\item \underline{Domain  $\Omega_{\ft C}$   containing $\ft C$ which satisfies \eqref{eq.OmegaC} and~\eqref{eq.OmegaC2}}.  
To construct such a domain, we use Proposition~\ref{Lau1} (see Appendix~C below) with    $D=D_{\ft C}$, $\mathcal V_-=\ft C$,  $\mathcal V_+=\ft V_{\ft C}$,
\begin{itemize}
\item[(i)] $S_1=\pa D_{\ft C} \cap  U^{\delta/2}_{\pa \Omega}$ on which $f$ is a Morse function with no critical point on $\pa S_1$ (see~\eqref{eq.PC-dc} together with the fact that $\pa \ft C \cap \pa \Omega$ is composed of non degenerate critical points of $f|_{\pa \Omega}$),
\item[(ii)]  $S_1'=\pa D_{\ft C} \cap  U^{\delta/4}_{\pa \Omega}$ which satisfies, according to~\eqref{eq.PC-dc}, $\vert \nabla_T f\vert \neq 0$ on $\overline{S_1\setminus S_1'}$. 
\end{itemize}
Therefore,   using in addition the fact that $D_{\ft C}$ satisfies~\eqref{G-DC0}--\eqref{G-DC},    there exists  a $C^\infty$ connected   open   subset  $\Omega_{\ft C}$ of $\Omega$ such that $\ft C\subset \Omega_{\ft C}$, $\overline{\Omega_{\ft C}}\subset \ft V_{\ft C}$, 
  $$f: \pa \Omega_{\ft C} \to \mathbb R \text{ is a Morse function},$$ 
  and for some $r>0$ and $\Gamma_{\ft C}\subset \Omega$, 
 \begin{equation}\label{eq.GAMMAC}
\pa \Omega_{\ft C}=  \left (\,  \bigcup_{z\in \pa \ft C\cap \pa \Omega}  B_{\pa \Omega}(z,\ve_z/2)    \right)\bigcup \,\overline{ \Gamma_{\ft C}}, \,  \text{ where, }  \, f\ge  \lambda+r \, \text{ on } \, \overline{\Gamma_{\ft C}}.  
 \end{equation}
%
%
\end{itemize}
It then remains to check that  $\Omega_{\ft C}$ satisfies~\eqref{eq.OmegaC}. 
From~\eqref{eq.GAMMAC} and~\eqref{eq.GAMMAC0}, $\Omega_{\ft C}$ satisfies the two first statements in~\eqref{eq.OmegaC} and $\min_{\pa \Omega_{\ft C}}f=\lambda$. 
Since  $\ft C\subset \Omega_{\ft C}$ and $\overline{\Omega_{\ft C}}\subset \ft V_{\ft C}$, one deduces from the first statement in~\eqref{eq.vC-s2}, that 
$$\big \{ x\in  \overline{\Omega_{\ft C}},\,  f(x)<\lambda \big \}=\ft C,$$
and from~\eqref{eq.vC-s},   
$$\vert\nabla f\vert \neq 0 \,  \text{ on } \,  \overline{\Omega_{\ft C}}\setminus \ft C.$$
This proves that   $\Omega_{\ft C}$ satisfies the two last statements in~\eqref{eq.OmegaC}. 
This concludes the construction of a domain $\Omega_{\ft C}$ which satisfies~\eqref{eq.OmegaC} and~\eqref{eq.OmegaC2}. 
 A schematic representation of such a domain $ {\Omega_{\ft C}}$ is given in Figure~\ref{fig:oo}.

\medskip

\noindent
\textbf{Step 2}: End of the proof of Theorem~\ref{thm.2}.
\medskip

\noindent
 For all $z \in \pa \ft C\cap \pa \Omega$, let   $\Sigma_{z}$  be an open subset   of~$\pa \Omega$ such that~$ z\in \Sigma_{z}$. Let $K$ be a compact subset of~$\Omega$ such that $K\subset \mathcal A(\ft C)$. 
Let us first consider the case when~$K\subset \ft C$. 
\medskip

\noindent
Let $\Omega_{\ft C}$ be the $C^\infty$ subdomain of $\Omega$ constructed in the previous step and which, we recall, contains $\ft C$ and satisfies~\eqref{eq.OmegaC} and~\eqref{eq.OmegaC2}. Then, one easily deduces  that when~$\Omega$ is replaced by~$\Omega_{\ft C}$, the function $f: \overline{\Omega_{\ft C}}\to \mathbb R$ satisfies~\eqref{H-M} and $\mathcal C=\{\ft C\}$ (see~\eqref{mathcalC-def} for the definition of $\mathcal C$). Thus,  in this case   $\ft C_{\ft{max}}=\ft C$. Moreover, using in addition  the second and third statements in~\eqref{eq.OmegaC},   one obtains that the assumptions~\eqref{eq.hip1},~\eqref{eq.hip2},~\eqref{eq.hip3} and~\eqref{eq.hip4}  are satisfied for the function $f: \overline{\Omega_{\ft C}}\to \mathbb R$. Thus, according Theorem~\ref{thm.main} applied to the function $f: \overline{\Omega_{\ft C}}\to \mathbb R$,  the most probable places of exit of the process~\eqref{eq.langevin} from $\Omega_{\ft C}$ when $X_0=x\in \ft C$,  are $\pa \ft C \cap \pa \Omega$ (and the relative asymptotic probabilities are given by  item 2 in Theorem~\ref{thm.main}). In particular, from items~1 and~3 in Theorem~\ref{thm.main}, for any open subset $\Sigma$ of $\pa \Omega_{\ft C}$ such that 
$$\min_{\overline \Sigma} f>\min_{\pa \Omega_{\ft C}} f \ \, \text{(where we recall $\argmin_{\pa \Omega_{\ft C}} f =\pa \ft C\cap \pa \Omega$, see~\eqref{eq.OmegaC})},$$
there exists $c>0$ such that for $h$ small enough:
\begin{equation}\label{eq.item13}
 \sup_{x\in K} \mathbb P_x \big[  X_{\tau_{\Omega_{\ft C}}}  \in  \Sigma \big]\le e^{-\frac ch},
\end{equation}
where $\tau_{\Omega_{\ft C}}$ is  the first exit time from $\Omega_{\ft C}$ of the process~\eqref{eq.langevin}. 
\medskip

\noindent
\textbf{Step 2a}: Proof of the   first asymptotic estimate in Theorem~\ref{thm.2} when $K\subset \ft C$. 
\medskip

\noindent
Writing $ \pa \Omega =( \pa \Omega \cap \pa \Omega_{\ft C}  )\cup  ( \pa \Omega \setminus \pa \Omega_{\ft C} )$, it holds:
\begin{equation}\label{eq.utdec}
  \Big (\pa \Omega \setminus \bigcup_{z\in \pa\ft C\cap \pa \Omega } \Sigma_z\Big)  \   \subset \  \left (\pa \Omega_{\ft C} \cap \pa \Omega \setminus  \bigcup_{z\in \pa\ft C\cap \pa \Omega } \Sigma_z  \right) \   \bigcup   \   \big (  \pa \Omega \setminus  \pa \Omega_{\ft C}   \big ) . 
  \end{equation}
To prove the first asymptotic estimate in Theorem~\ref{thm.2}, let us prove that when $X_0=x\in K$, the probability that   $X_{\tau_\Omega}$ belongs to each of the two sets in the right-hand side of~\eqref{eq.utdec}, is exponentially small when $h\to 0$. 
Let us recall that~$\tau_{\Omega_{\ft C}}$ is  the first exit time from $\Omega_{\ft C}$ of the process~\eqref{eq.langevin} and thus,  when  $X_0=x\in \Omega_{\ft C}$, $\tau_{\Omega_{\ft C}}\le \tau_{\Omega}$, and  
\begin{equation}\label{eq.tauC}
\text{$\tau_{\Omega_{\ft C}}=\tau_{\Omega}$ if and only if  $X_{\tau_{\Omega_{\ft C}}}\in \pa \Omega_{\ft C}\cap  \pa \Omega$.}
\end{equation} 
Thus, from~\eqref{eq.tauC}, when $X_0=x\in  \Omega_{\ft C}$,  it holds:
 $$\left \{ X_{\tau_\Omega}\in \pa \Omega_{\ft C} \cap \pa \Omega \setminus  \bigcup_{z\in \pa\ft C\cap \pa \Omega } \Sigma_z   \right\}  = \left \{ X_{\tau_{\Omega_{\ft C} }}  \in  \pa \Omega_{\ft C} \cap \pa \Omega \setminus  \bigcup_{z\in \pa\ft C\cap \pa \Omega } \Sigma_z    \right\} .$$
Using~\eqref{eq.item13}, there exists $c>0$ such that for $h$ small enough: 
$$\sup_{x\in K} \mathbb P_x\Big [  X_{\tau_{\Omega_{\ft C} }}  \in  \pa \Omega_{\ft C} \cap \pa \Omega \setminus  \bigcup_{z\in \pa\ft C\cap \pa \Omega } \Sigma_z    \Big]\le e^{-\frac ch}.$$
Thus, there exists $c>0$ such that for $h$ small enough: 
\begin{equation}\label{eq.ut}
\sup_{x\in K} \mathbb P_x\Big [  X_{\tau_{\Omega }}  \in  \pa \Omega_{\ft C} \cap \pa \Omega \setminus  \bigcup_{z\in \pa\ft C\cap \pa \Omega } \Sigma_z    \Big]\le e^{-\frac ch}.
\end{equation}
Let us now consider the case when $X_{\tau_\Omega}\in   \pa \Omega \setminus  \pa \Omega_{\ft C} $ and  $X_0=x\in K$. When $X_0=x\in K$, it holds from~\eqref{eq.tauC}:
$$\big\{X_{\tau_\Omega}\in   \pa \Omega \setminus  \pa \Omega_{\ft C} \big\}\subset \big\{ X_{\tau_{\Omega_{\ft C}}}  \in   \pa \Omega_{\ft C}\setminus (\pa \Omega \cap  \pa \Omega_{\ft C}) \big\}$$
From~\eqref{eq.item13},  there exists $c>0$ such that for $h$ small enough: 
$$
\sup_{x\in K} \mathbb P_x\Big [  X_{\tau_{\Omega_{\ft C}}}  \in    \pa \Omega_{\ft C}\setminus (\pa \Omega \cap \pa \Omega_{\ft C})  \Big]\le e^{-\frac ch}.
$$
Therefore, there exists $c>0$ such that for $h$ small enough: 
  \begin{equation}\label{eq.ut2}
\sup_{x\in K} \mathbb P_x\Big [  X_{\tau_{\Omega }}  \in    \pa \Omega \setminus  \pa \Omega_{\ft C}  \Big]\le e^{-\frac ch}.
\end{equation}
In conclusion, from~\eqref{eq.utdec},~\eqref{eq.ut} and~\eqref{eq.ut2}, one obtains that there exists $c>0$ such that for $h$ small enough: 
\begin{equation}\label{utili1}
\sup_{x\in K} \mathbb P_x\Big [X_{\tau_{\Omega}}\in\bigcup_{z\in \pa\ft C\cap \pa \Omega } \Sigma_z  \Big]\le e^{-\frac ch}.
\end{equation}
This proves the first asymptotic estimate in Theorem~\ref{thm.2} when $K\subset\ft  C$. \medskip

\noindent
\textbf{Step 2b}: Proof of the second  asymptotic estimate in Theorem~\ref{thm.2} when $K\subset\ft  C$.
\medskip

\noindent 
Let us assume   that   the open sets $(\Sigma_{z})_{z\in \pa \ft C\cap \pa \Omega}$  are two by two disjoint. Let us consider  $z\in \pa \ft C\cap \pa \Omega$ and $\beta>0$ such that (see indeed the first statement in~\eqref{eq.OmegaC}), 
\begin{equation}\label{eq.bomegac}
 B_{\partial \Omega}(z, \beta)  \subset   \Sigma_z  \cap \pa \Omega_{\ft C}.
 \end{equation}
Then, one writes:
\begin{equation}\label{eq.dec-o}
\mathbb P_x  [X_{\tau_{\Omega}}\in \Sigma_z   ]=\mathbb P_x\big [X_{\tau_{\Omega}}\in       B_{\partial \Omega}(z, \beta)  \big]+ \mathbb P_x\big [X_{\tau_{\Omega}}\in    \Sigma_z \setminus    B_{\partial \Omega}(z, \beta)  \big].
 \end{equation}
Let us first deal with the second term in the right-hand side~\eqref{eq.dec-o}. It holds (since the sets $(\Sigma_{y})_{y\in \pa \ft C\cap \pa \Omega}$  are two by two disjoint and $B_{\partial \Omega}(z, \beta)  \subset   \Sigma_z  $, see~\eqref{eq.bomegac}), when $X_0=x\in \Omega$:
$$
\mathbb P_x\big [X_{\tau_{\Omega}}\in    \Sigma_z \setminus    B_{\partial \Omega}(z, \beta)  \big]\le \mathbb P_x\left [X_{\tau_{\Omega}}\in    \pa \Omega  \setminus  \Big(    B_{\partial \Omega}(z, \beta) \cup \bigcup_{y\in \pa\ft C\cap \pa \Omega,y\neq z } \Sigma_y\Big)   \right].
$$
Thus, from~\eqref{utili1} (applied with    $B_{\partial \Omega}(z, \beta) $ instead  of   $\Sigma_z$), one obtains   that   there exists $c>0$ such that for $h$ small enough: 
\begin{equation}\label{eq.step2-ch}
\sup_{x\in K} \mathbb P_x\Big [X_{\tau_{\Omega}}\in  \Sigma_z \setminus    B_{\partial \Omega}(z, \beta)     \Big]\le e^{-\frac ch}.
 \end{equation}
Let us now deal with the first term in the right-hand side~\eqref{eq.dec-o}. It holds from~\eqref{eq.bomegac} and~\eqref{eq.tauC}, when $X_0=x\in K$:
\begin{equation}\label{eq.step2-ch2}
\mathbb P_x\big [X_{\tau_{\Omega}}\in       B_{\partial \Omega}(z, \beta)  \big]=\mathbb P_x\big [X_{\tau_{\Omega_{\ft C}}}\in       B_{\partial \Omega}(z, \beta)  \big].
 \end{equation}
Applying item 2 in Theorem~\ref{thm.main} with the function $f: \overline{\Omega_{\ft C}}\to \mathbb R$ and $F=\mbf{1}_{B_{\partial \Omega}(z, \beta)}$, one has:
$$\mathbb P_{x}[X_{\tau_{\Omega_{\ft C}}} \in  B_{\partial \Omega}(z, \beta) ]= \frac{  \partial_nf(z)      }{  \sqrt{ {\rm det \ Hess } f|_{\partial \Omega}   (z) }  } \left (\sum \limits_{y\in\pa \ft C \cap \pa \Omega  } \frac{  \partial_nf(y)      }{  \sqrt{ {\rm det \ Hess } f|_{\partial \Omega}   (y) }  }\right)^{-1}(1+O(h)),$$
in the limit $h \to 0$ and uniformly in~$x \in K$. 
 Together with~\eqref{eq.dec-o},~\eqref{eq.step2-ch}, and~\eqref{eq.step2-ch2}, this concludes the proof of the second asymptotic estimate in Theorem~\ref{thm.2} when $X_0=x\in   K\subset \ft C$. 
 \medskip
 
 \noindent
The case when $X_0=x\in K\subset \mathcal A( \ft C)$ is proved using the estimate of  Freidlin and Wentzell~\eqref{eq.WFr-est} (see the second step of the proof of Lemma~\ref{le.exp-qsd-x}). This concludes the proof of Theorem~\ref{thm.2}.
\end{proof}
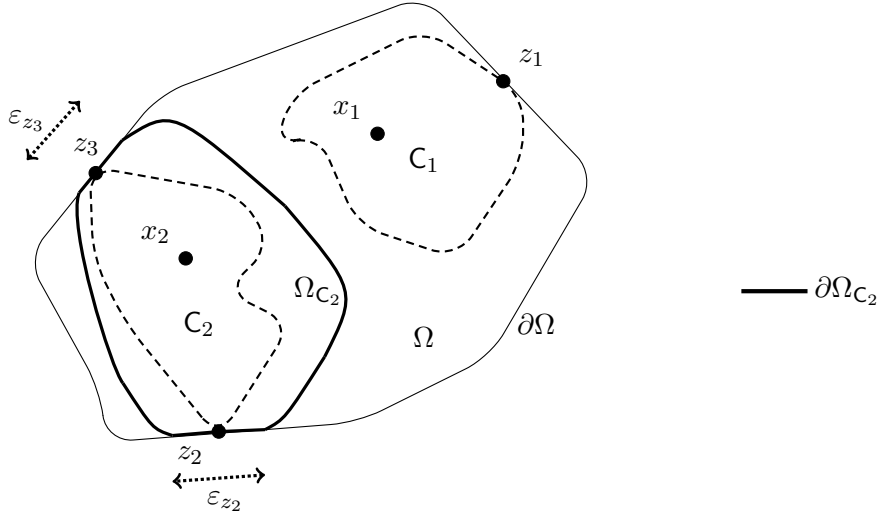
\begin{figure}[h!]
\begin{center}
\begin{tikzpicture}[scale=0.87]
\tikzstyle{vertex}=[draw,circle,fill=black,minimum size=5pt,inner sep=0pt]
\tikzstyle{ball}=[circle, dashed, minimum size=1cm, draw]
\tikzstyle{point}=[circle, fill, minimum size=.01cm, draw]
\draw [rounded corners=10pt] (0.1,0.5) -- (-1,2.5) -- (1,5) -- (5,6.5) -- (7.6,3.75) -- (6,1) -- (4,0) -- (0.2,-0.3) --cycle;
\draw [thick, densely dashed, rounded corners=10pt] (1.9,-0.24) -- (.5,1.5) -- (0,2.5) -- (-0.05,3.9) -- (2.2,3.5) -- (2.7,2.7) -- (2,2) -- (3,1.5) --cycle;
\draw [thick, densely dashed,rounded corners=10pt]    (2.75,4.2)  -- (3.4,4.4) -- (4,3) --(5.5,2.5)--(6.5,4) --(6.5,5)  -- (5,6) -- (3,5)  --cycle;
 \draw (1.6,1.5) node[]{$\ft C_2$};
  \draw (5,4) node[]{$\ft C_1$};
     \draw  (5,1.3) node[]{$\Omega$};
    \draw  (6.7,1.5) node[]{$\pa \Omega$};

\draw [very thick,black] (0.44,4.3)--(-0.21,3.5);
\draw [very thick,black] (1.2,-0.2)--(1.9,-0.14) ; 
\draw [very thick,black] (2.6,-0.11)--(1.9,-0.14) ;

\draw [very thick,black]  (-0.21,3.5) ..controls (-0.4,3)  and (0.28,0.8)  .. (0.45,0.6) ;
\draw [very thick,black]  (2.6,-0.11) ..controls (2.9,0.05)  ..  (3.5,1);
\draw [very thick,black]  (0.44,4.3) ..controls (1.2,4.8)  .. (2.9,3.3);
\draw [very thick,black] (3.5,1)  ..controls (4,2)  .. (2.9,3.3);
\draw [very thick,black] (0.45,0.6)  ..controls (0.9,-0.1)  .. (1.2,-0.2);

     \draw  (-0.1,4.2) node[]{$z_3$};
\draw (0.04,3.8) node[vertex,label=north west: {}](v){};
\draw (1.4,2.5) node[vertex,label=north west: {$x_2$}](v){};
\draw (4.3 ,4.4) node[vertex,label=north west: {$x_1$}](v){};
\draw (1.9,-0.14) node[vertex,label=south west: {$z_2$}](v){};
\draw (6.2,5.2) node[vertex,label=north east: {$z_1$}](v){};
     \draw  (3.4,2) node[]{$ \Omega_{\ft C_2}$};

       \draw [<->,very thick, densely dotted] (-0.2,4.9)--(-1,4) ;
  \draw  (-1,4.6) node[]{$\ve_{z_3}$};

\draw [<->,very thick, densely dotted] (2.6,-0.8)--(1.2,-0.9) ;
  \draw  (2,-1.2) node[]{$\ve_{z_2}$};
     
     \draw[ultra thick] (9.8,2)--(10.8,2);
   \draw (11.4,2)node[]{$\pa \Omega_{\ft C_2}$};

\end{tikzpicture}

\caption{Schematic representation  of $\Omega_{\ft C_2}$ satisfying~\eqref{eq.OmegaC} when $\ft C=\ft C_2$. On the figure, $\pa \ft C_2 \cap \pa \Omega=\{z_2,z_3\}$, $x_2$ is the global minimum of $f$ in $\ft C_2$ and $\ft C_1$ is another element of $\mathcal C$ (see~Definition~\ref{de.1}). }
 \label{fig:oo}
 \end{center}
\end{figure}
 \renewcommand\appendixpagename{\Large{Appendix}}
\begin{appendices}
\addtocontents{toc}{\protect\setcounter{tocdepth}{0}}
 \addappheadtotoc
 \label{appendix-c}

 \medskip

 \noindent
\textbf{A.  On the assumption~\eqref{H-M} and  Lemma~\ref{ran1} }
 \medskip

 \noindent
In this appendix and as stated in Section~\ref{sec.on-A0}, one explains why the conclusion of  Lemma~\ref{ran1}, proved in~\cite[Section 3.4]{HeNi1} (and made for $\Delta^{D,(p)}_{f,h}$, $p\in \{0,1\}$), is still valid when assuming in~\eqref{H-M} that $f: \{x\in \pa \Omega, \, \pa_nf(x)>0\}\to \mathbb R$ is a Morse function   instead of $f|_{\pa \Omega}$ is a Morse function.

For that purpose, let us assume that $f:\overline \Omega\to \mathbb R$ is $C^\infty$, $\vert \nabla f\vert \neq 0$ on $\pa \Omega$ and that, the functions $f:\overline \Omega\to \mathbb R$ and $f: \{x\in \pa \Omega, \, \pa_nf(x)>0\}\to \mathbb R$ are Morse functions.
 From~\eqref{eq.unitary}, we are going to prove that in this case,  for $h$ small enough, one still has:
\begin{equation}\label{eq.change1}
\dim \Ran \, \pi_{[0,h^{\frac 32} )}(\Delta^{D,(0)}_{f,h})= {\ft m_{0}^{\Omega}} \text{ and } \dim \Ran \, \pi_{[0,h^{\frac 32} )}(\Delta^{D,(1)}_{f,h})= {\ft m_{1}^{\overline \Omega}}.
\end{equation}
Let us notice that  there exists an open  neighborhood $\mathcal U\subset \pa\Omega$ of  $\{x\in \pa \Omega, \, \pa_nf(x)\ge  0\}$ such that $f|_{\mathcal U}$ is a Morse function. Therefore, in view of~\cite[Section 3.4]{HeNi1} (and more precisely of the IMS formula used there to prove an upper bound on the
number of small eigenvalues),  to prove~\eqref{eq.change1}, it is sufficient to show  that for all $z\in \{x\in \pa \Omega, \, \pa_nf(x)<0\}$, there exists a neighbourhood $\mathcal V_z$ of $z$ in $\overline \Omega$ such that for any  
$w\in \Lambda^{p}H^1_{T}(\Omega)$ (for $p\in \{0,1\}$) supported in $\mathcal V_z$, 
\begin{equation}\label{expla-1}
\big \Vert  d_{f,h}w\big \Vert_{L^2} ^2  +  \big \Vert  d_{f,h}^*w\big \Vert_{L^2 }^2\ge C\, h  \big \Vert   w\big \Vert_{L^2}^2,
    \end{equation}
    for some $C>0$ independent of $h$ and $w$. 
Let us recall the  two following Green formulas~\cite[Lemma 2.3.2]{HeNi1}.  For all $w\in \Lambda^{0}H^1_{T}(\Omega)$, one has:
\begin{equation}\label{green01}
\big \Vert  d_{f,h}w\big \Vert_{L^2} ^2  +  \big \Vert  d_{f,h}^*w\big \Vert_{L^2 }^2=h^2\big \Vert  d \,  w \big \Vert_{L^2 } ^2  +  h^2\big \Vert  d ^*\, w \big \Vert_{L^2 }^2 +\big  \lp  w, \big(\vert \nabla f\vert^2  +h\Delta_H^{(0)}f \big)\, w\big \rp_{L^2  }.
    \end{equation}
For all $w\in \Lambda^{1}H^1_{T}(\Omega)$, one has:
\begin{align}
\nonumber
\big \Vert  d_{f,h}w\big \Vert_{L^2 } ^2  +  \big \Vert  d_{f,h}^*w\big \Vert_{L^2}^2&=h^2\big \Vert  d \,  w \big \Vert_{L^2 } ^2  +  h^2\big \Vert  d ^*\, w \big \Vert_{L^2 }^2 +\big  \lp  w, \big(\vert \nabla f\vert^2 +h\big (\mathcal L_{\nabla f}+\mathcal L_{\nabla f}^*\big)\big)\, w\big \rp_{L^2 }\\
\label{green02}
    &\quad -h\int_{\pa \Omega}\vert w \vert^2\pa_{ n}f,
    \end{align}
    where $\mathcal L_{\nabla f}$ is the Lie derivative with respect to  the vector field $\nabla f$ and $\mathcal L_{\nabla f}^*$ its formal adjoint in $L^2(\Omega)$. Let us recall that the operator $\mathcal L_{\nabla f}+\mathcal L_{\nabla f}^*$ is a zeroth order operator (see for instance~\cite[Appendice 1]{HeSj4}). 

Since there is no boundary term in~\eqref{green01}, the first equality in~\eqref{eq.change1} is just a consequence of $\vert \nabla f\vert \neq 0$ on $\pa \Omega$. Indeed, 
 there exist    
a neighbourhood $\mathcal V_{\pa \Omega}$ of $\pa \Omega$ in $\overline \Omega$  and  a constant $c>0$  such that $\inf_{\overline{\mathcal V_{\pa \Omega}}} \vert \nabla f\vert \ge  c$.
Then, from~\eqref{green01},  for $h$ small enough, one has for all $w\in \Lambda^{0}H^1_{T}(\Omega)$ supported in $ \mathcal V_{\pa \Omega}$:
 \begin{equation}\label{eq.QF0}
\big \Vert  d_{f,h} w\big \Vert_{L^2} ^2  +  \big \Vert  d_{f,h}^*w\big \Vert_{L^2 }^2\ge \frac c2 \, \big \Vert  w\big \Vert_{L^2}^2.
\end{equation}
Thus,~\eqref{expla-1} is satisfied and for $h$ small enough, it holds:
 $$\dim \Ran \,\pi_{[0,h^{\frac 32} )}(\Delta^{D,(0)}_{f,h})= {\ft m_{0}^{\Omega}}.$$
 Let us now prove the second equality in~\eqref{eq.change1}.   
To this end, let $z\in \pa \Omega$  such that  $\pa_nf(z)<0$.
 Then, 
there exists a neighborhood $\mathcal V_z\subset \mathcal V_{\pa \Omega}$ of $z$ in $\overline \Omega$ such that $\pa_nf<0$ on $\pa \Omega \cap \mathcal V_z$.
 Therefore, using~\eqref{green02}, for $h$ small enough, one has  for all   $w\in \Lambda^{1}H^0_{T}(\Omega)$ supported in $\mathcal V_z$: 
\begin{equation}\label{eq.QF2}
\big \Vert  d_{f,h}w\big \Vert_{L^2} ^2  +  \big \Vert  d_{f,h}^*w\big \Vert_{L^2 }^2\ge \frac c2 \, \big \Vert  w\big \Vert_{L^2}^2-h\int_{\pa \Omega}\vert w\vert^2\pa_{ n}f\ge  \frac c2 \, \big \Vert   w\big \Vert_{L^2}^2.
\end{equation}
The estimate~\eqref{eq.QF2}  
implies, according again to~\eqref{expla-1}, that for $h$ small enough, it  holds
$$\dim \Ran \, \pi_{[0,h^{\frac 32} )}(\Delta^{D,(1)}_{f,h})= {\ft m_{1}^{\overline \Omega}}.$$
This ends the proof of~\eqref{eq.change1}.

  \medskip

 \noindent
 \textbf{B. Proofs of the results of Sections~\ref{sec.A2},~\ref{sec.A3},  and~\ref{sec.A4}}
  \medskip

 \noindent
 In this section, one proves the asymptotic estimates \eqref{discussion1}, \eqref{discussion2}, \eqref{discussion3} and \eqref{discussion4} stated in Section~\ref{discussion-hyp}. Let us start with the following result. Let $z_1<z_2$ and $f$: $[z_1,z_2]\to \mathbb R$ be a $C^{\infty}$ function. Then, for all $x\in [z_1,z_2]$, one has:
\begin{equation}\label{exp}
\mathbb P_{x}  [ X_{\tau_{(z_1,z_2)}}=z_1 ]=\frac{\displaystyle{\int_x^{z_2} e^{\frac 2h f}} }{\displaystyle{\int_{z_1}^{z_2} e^{\frac 2h f}} } \text{ and thus } \mathbb P_{x}  [ X_{\tau_{(z_1,z_2)}}=z_2 ]=\frac{\displaystyle{\int_{z_1}^x e^{\frac 2h f}} }{\displaystyle{\int_{z_1}^{z_2} e^{\frac 2h f}} }.
\end{equation} 
Indeed, let $v\in C^\infty([z_1,z_2],\mathbb R)$  be the unique solution to the elliptic boundary value problem on $(z_1,z_2)$:
$$  \frac{h}{2}  \, \frac{d^2}{dx^2} v(x) -\frac{d}{dx} f(x) \frac{d}{dx}   v(x)  =  0,    \ v(z_1)=1,\  v(z_2)=0.$$
Clearly, one has for all $x\in [z_1,z_2]$,
$$v(x)=\frac{\displaystyle{\int_x^{z_2} e^{\frac 2h f}} }{\displaystyle{\int_{z_1}^{z_2} e^{\frac 2h f}} } .$$
Finally, using the  Dynkin's formula \cite[Theorem 11.2]{karlin1981second}, one has for all $x\in [z_1,z_2]$,
 $$v(x)=\mathbb P_{x} \left [  X_{\tau_{\Omega}}=z_1\right].$$
 This proves~\eqref{exp}. 
 \medskip
 
 \noindent
 Let us now explain how to prove \eqref{discussion1}, \eqref{discussion2}, \eqref{discussion3} and \eqref{discussion4}. 
 The asymptotic estimates~\eqref{discussion1} and \eqref{discussion4}  follow directly from~\eqref{exp} together with  Laplace's method.  
 
 Let us now prove~\eqref{discussion2}. In the example depicted in Figure~\ref{fig:hip2}, the assumption~\eqref{eq.hip1} is satisfied.
Therefore, using~\eqref{eq.uh=}, there exits $\chi\in C^{\infty}_c\left ( (c,d), [0,1]\right )$ such that $\chi=1$ on a neighborhood of $x_1$ and:
\begin{equation}
\label{uh=chi}
u_h=\frac{\chi}{\Vert \chi\Vert_{L^2_w} } + r,
\end{equation}
where $r\in L^2_w(z_1,z_2)$ satisfies $\Vert r \Vert_{L^2_w}=O ( e^{- \frac{c}{h} })$  and  $c>0$ is  independent of $h$. Moreover, one has (see Figure~\ref{fig:hip2}) $$\argmin_{[z_1,z_2]}f =\argmin_{\ft C_{\ft{max}}}f=\{x_1\}.$$
 Thus, from Proposition~\ref{pr.masse} (applied to  $\ft O=(z_1,z_2)$, see~\eqref{eq.concentration1}), one has in the limit $h\to 0$:
$$
\int_{z_1}^{z_2} u_h \, e^{-\frac 2h f}= f''(x_1)^{-\frac 14}(\pi h)^{\frac 14}e^{- \frac{1}{h}f(x_1)} (1+O(h)).
$$
Moreover, if we denote by~$g(x)=\mathbb P_{x}  [ X_{\tau_{(z_1,z_2)}}=z_1 ]$ for $x\in [z_1,z_2]$, since $\chi\in C^{\infty}_c\left ( (c,d), [0,1]\right )$ and $\Vert g\Vert_{L^{\infty}}\le 1$, one has in the limit $h\to 0$ (using~~\eqref{uh=chi} and~\eqref{exp} in the third equality):
\begin{align}
\nonumber
\mathbb P_{\nu_h}  [ X_{\tau_{(z_1,z_2)}}=z_1 ]&=\frac{1}{\displaystyle{\int_{z_1}^{z_2} u_h \, e^{-\frac 2h f}} }\left [ \displaystyle{\int_{z_1}^c u_h ge^{-\frac 2hf} + \int_{c}^{d} u_h   g e^{-\frac 2hf}+ \int_{d}^{z_2} u_h ge^{-\frac 2hf} } \right]\\
\nonumber
&=\frac{1}{\displaystyle \int_{z_1}^{z_2} u_h \, e^{-\frac 2h f}}\left [ \displaystyle \int_{z_1}^c r   g e^{-\frac 2hf}+ \int_{c}^{d} \frac{\chi  ge^{-\frac 2hf}}{\Vert \chi\Vert_{L^2_w} } +  \int_{c}^{d} r   ge^{-\frac 2hf}+ \int_{d}^{z_2} r   g  e^{-\frac 2hf}\right]\\
\nonumber
&= \frac{1}{ \displaystyle \int_{z_1}^{z_2} u_h \, e^{-\frac 2h f}}\left [  \frac{  \displaystyle  \int_{c}^{d} \chi(x) \int_{x}^{z_2} e^{\frac 2h (f(y)-f(x))}dy  dx}{\Vert \chi \Vert_{L^2_w} \displaystyle \int_{z_1}^{z_2} e^{\frac 2h f} }  + O(e^{-\frac 1h (f(x_1)+c)}) \right]\\
\nonumber
&=\frac{1}{\displaystyle \int_{z_1}^{z_2} u_h \, e^{-\frac 2h f}}\left [    \frac{ O(e^{\frac 2h ( f(d)-f(x_1))} ) }{\Vert \chi  \Vert_{L^2_w}\displaystyle  \int_{z_1}^{z_2} e^{\frac 2h f} }   + O(e^{-\frac 1h (f(x_1)+c)}) \right]\\
\label{discussion2-b}
& =O(e^{-\frac ch}),
\end{align}
for some $c>0$ independent of $h$. This proves~\eqref{discussion2}.

Let us now prove~\eqref{discussion3}.   In the example depicted in Figure~\ref{fig:hip3}, the assumption~\eqref{eq.hip1} is satisfied.
Therefore, using~\eqref{eq.uh=}, there exits $\chi\in C^{\infty}_c\left ( (z,z_2), [0,1]\right )$ such that $\chi=1$ on a neighborhood of $x_2$ and:
\begin{equation}
\label{uh=chi2}
u_h=\frac{\chi}{\Vert \chi\Vert_{L^2_w} } + r,
\end{equation}
where $r\in L^2_w(z_1,z_2)$ satisfies $\Vert r \Vert_{L^2_w}=O ( e^{- \frac{c}{h} })$  and  $c>0$ is  independent of $h$. Moreover, one has (see Figure~\ref{fig:hip3}) $$\argmin_{[z_1,z_2]}f =\argmin_{\ft C_{\ft{max}}}f=\{x_2\}.$$
 Thus, from Proposition~\ref{pr.masse} (applied to $\ft O=(z_1,z_2)$, see~\eqref{eq.concentration1}), one has in the limit $h\to 0$:
$$
\int_{z_1}^{z_2} u_h \, e^{-\frac 2h f}= f''(x_2)^{-\frac 14}(\pi h)^{\frac 14}e^{- \frac{1}{h}f(x_2)} (1+O(h)).
$$
 Then, the same computations as those made  to obtain~\eqref{discussion2-b}, imply that when $h\to 0$:
 \begin{equation}\label{eq.discussion3-b}
 \mathbb P_{\nu_h}  [ X_{\tau_{(z_1,z_2)}}=z_1 ]=O(e^{-\frac ch}),
 \end{equation}
 for some $c>0$ independent of $h$. This proves~\eqref{discussion3}.  
 \medskip

   \noindent
 \textbf{C. On the proof of~\eqref{eq.OmegaC2}.  }
  \medskip
  
  \noindent
In this section, we prove the existence of a domain $\Omega_{\ft C}$  which satisfies~\eqref{eq.OmegaC2} in addition to~\eqref{eq.OmegaC}.  To this end, we first  give in Proposition~\ref{pr.Morsesigma} a simple perturbation result to present the main idea of the proof. Then, we extend this result to the setting we are interested in to prove   the existence of such a domain $\Omega_{\ft C}$ in  Proposition~\ref{Lau1}.

\begin{proposition}\label{pr.Morsesigma}
Let $f:\mathbb R^{d}\to \mathbb R$ be a $C^\infty$ function and 
$D$ be a $C^\infty$ open  bounded and connected  subset of $\mathbb R^d$. Let us assume that 
$$\text{for all } x\in \pa D, \ \nabla f(x)\oplus T_x\pa D=\mathbb R^d.$$
Then, for any open sets $ \mathcal V_-$ and $ \mathcal V_+$ such that $\overline{ \mathcal V_-} \subset D \text{ and } \overline{ D} \subset \mathcal V_+$,    
there exists a $C^\infty$ open bounded  and connected  subset $D'$ of~$\mathbb R^d$ such that 
$$\overline{ \mathcal V_-} \subset D', \ \overline{ D'} \subset \mathcal V_+, \ \text{ and } \ f|_{\pa D'} \text{  is a Morse function}.$$
 
\end{proposition}

\begin{remark}
We would like to thank Fran\c cois Laudenbach who gave us the main ingredient of the proof of Proposition~\ref{pr.Morsesigma}. The proof is inspired  by a method due to  Ren\'e Thom~\cite{thom} based on Sard's theorem~\cite{sard}, see~\cite[Section 5.6]{laudenbach1993topologie}. 
\end{remark}

\begin{proof}
Let   $ \mathcal V_-$ and $ \mathcal V_+$ be two open subsets of $\mathbb R^d$ such that $\overline{ \mathcal V_-} \subset D \text{ and } \overline{ D} \subset \mathcal V_+$. 
Let us denote by  $S$ the boundary of $D$ which is a  smooth compact hypersurface of $\mathbb R^d$. 
 For $r>0$, one denotes    by $B(0,r) $ the ball of radius $r$ centred at $0$ in $\mathbb R^d$.  Let   $\mathcal V$ be a neighborhood of $S$
in $\mathbb R^{d}$.  
 By assumption on $S$, there exist  $\varepsilon_{0}>0$ and $\varepsilon_{1}>0$ such that the map
$$   (x,\lambda)\in  S\times (-\varepsilon_{0},\varepsilon_{0})     \mapsto x+\lambda
\nabla f(x)\in \mathcal V
$$
is well defined and is a diffeomorphism onto its image, and, for all $ (x,v) \in S\times B(0,\varepsilon_{1}) $, there exists a unique $\lambda(x,v)\in  
(-\varepsilon_{0},\varepsilon_{0})$ such that 
$$
   f\big(x+\lambda(x,v)\nabla f(x)\big) = f(x)+  v\cdot x\,.
$$
Moreover, for every $v\in B(0,\varepsilon_{1}) $, according to the implicit function theorem, the map $x\in S\mapsto \lambda(x,v)\in (-\varepsilon_{0},\varepsilon_{0})$
is smooth and then also is $x\in S\mapsto x+ \lambda(x,v)\nabla f(x)\in \mathbb R^{d}$.
The latter application is then an injective immersion and hence, since $S$ is compact,
it follows that $S_{v}:=\{x+ \lambda(x,v)\nabla f(x)\}$ is a smooth compact   hypersurface. Up to choosing $\ve _1>0$ smaller, for any $v\in B(0,\ve_1)$, $S_{v}$ is the boundary of a $C^\infty$ open bounded  and connected  subset $D_v$ of~$\mathbb R^d$ such that 
$$\overline{ \mathcal V_-} \subset D_v \, \text{ and }  \, \overline{ D_v} \subset \mathcal V_+.$$
To prove Proposition~\ref{pr.Morsesigma},  it remains to show  that there exists   $v\in B(0,\varepsilon_{1})$
such that $f|_{S_{v}}$ is a Morse function. 
Let us   introduce
the function
$$
F:(x,v)\in S\times B(0,\varepsilon_{1})\mapsto  f|_{S_{v}}\big(x+\lambda(x,v)\nabla f(x)\big) =f(x)+  v\cdot x
\in \mathbb R\,.
$$
For all $x\in  S$ and for all $v\in B(0,\varepsilon_{1})$, let   $v^{T}_{x}\in T_xS$ and $v_{x}^{N}\in \mathbb R$ be such that    
\begin{equation}\label{vtt}
v=v^{T}_{x} + v_{x}^{N} n(x),
\end{equation}
where  we recall that $n(x)$ is a unit outward normal to $S$. 
At  $(x,v)\in S\times B(0,\varepsilon_{1})$,   it holds  $\partial _{x}F (x,v): z \in T_xS\mapsto d_xf(x)z+v_x^T\cdot z$, where $\partial_xF(x,v)$ is the $x$-derivative of $F$ at $(x,v)$. 
The function
$G: S\times B(0,\varepsilon_{1})\to T_x^*S $ defined
by
$$G: (x,v)\mapsto (x, \partial_{x}F (x,v))$$
 is a submersion onto a small tube 
around the zero section of $T^*S$. 
This is obvious by considering the $v$-derivative of $G$. 
Hence, $G$ is transverse to the zero section $0_{T^{*}S}$ of $T^{*}S$ (see~\cite[Chapitre~5.1]{laudenbach1993topologie} for the definition of transversality). Using the parametric transversality theorem   (which is a consequence of Sard's theorem, see for instance~\cite[Chapitre~5.3.1]{laudenbach1993topologie}), one obtains that for almost every   $v\in B(0,\varepsilon_{1})$,
$\partial_{x}(F|_{S\times \{v\}})=d (f|_{S_{v}})$ is transverse  to $0_{T^{*}S}$,  which is equivalent to~$f|_{S_{v}}$ is a Morse function. 
This concludes the proof of Proposition~\ref{pr.Morsesigma}. 
\end{proof}

  \noindent  
The next proposition gives sufficient conditions on $D$ and $f$ to modify the result of Proposition~\ref{pr.Morsesigma} so that the perturbed domain $D′$ has the same boundary as $D$ on a prescribed subset $S_1'$ of $\partial D$ on which $f$ is already a Morse function.

\begin{proposition}\label{Lau1}
Let $f:\mathbb R^{d}\to \mathbb R$ be a $C^\infty$ function and 
$D$ be a $C^\infty$ open  bounded and connected  subset of $\mathbb R^d$. Let us assume that 
$$\text{for all } x\in \pa D, \ \nabla f(x)\oplus T_x\pa D=\mathbb R^d.$$
Furthermore, let us assume that there exists  an open subset $S_1$ of $\pa D$ such that  $f: \overline {S_1}\to \mathbb R$ is a Morse function with no critical point on $\pa   S_1$.
Let us now consider an open set  $S_1'$ such that $\overline{S_1'}  \subset S_1$ and  $ f|_{\pa D}$ has no critical point on  $\overline{S_1\setminus S_1'}$.  
Then, for any open sets $ \mathcal V_-$ and $ \mathcal V_+$ such that $\overline{ \mathcal V_-} \subset D\cup S_1' \text{ and } \overline{ D} \setminus S_1' \subset \mathcal V_+$,    
there exists a $C^\infty$ open bounded  and connected  subset $D'$ of~$\mathbb R^d$ such that  $S_1'\subset \pa D'$, 
$$\overline{ \mathcal V_-} \subset D'\cup S_1', \ \overline{ D'}  \setminus S_1' \subset \mathcal V_+, \ \text{ and } \, f|_{\pa D'} \text{  is a Morse function}.$$
\end{proposition}
\noindent

\begin{proof}
Let   $ \mathcal V_-$ and $ \mathcal V_+$ be two open subsets of $\mathbb R^d$ such that $\overline{ \mathcal V_-} \subset D\cup S_1' \text{ and } \overline{ D} \setminus S_1' \subset \mathcal V_+$.  
Let us denote by  $S$ the boundary of $D$ which is a  smooth compact hypersurface of $\mathbb R^d$. 
 Let us introduce a function $\chi\in C^\infty(S)$ such that $\chi(x)=1$ for all   $x\in S\setminus S_1$ and  $\chi(x)=0$ for all   $x\in \mathcal V_{S_1'}$ where $\mathcal V_{S_1'}$ is an open neighborhood of $\overline{S_1'}$ in $S$ such that $\overline{\mathcal V_{S_1'}}\subset S_1$. To prove Proposition~\ref{Lau1}, one  uses the cutoff function $\chi$  in the definition of $\lambda(x,t)$   to ensure that $S_1'\subset S_v$ (see the proof of Proposition~\ref{pr.Morsesigma} for the notation $S_v$). This is made as follows. 
Let us first consider   $\varepsilon_{0}>0$ and $\varepsilon_{1}>0$ such that the map
$$   (x,\lambda)\in  S\times (-\varepsilon_{0},\varepsilon_{0})     \mapsto x+\lambda
\nabla f(x)\in \mathcal V
$$
is well defined and is a diffeomorphism onto its image, and, for all $ (x,v) \in S\times B(0,\varepsilon_{1}) $, there exists a unique $\lambda(x,v)\in  
(-\varepsilon_{0},\varepsilon_{0})$ such that 
$$
   f\big(x+\lambda(x,v)\nabla f(x)\big) = f(x)+  \chi(x)\, v\cdot x\,.
$$
Notice that $\lambda(x,v)=0$ for all $x\in \mathcal V_{S_1'}$ and $v\in B(0,\varepsilon_{1})$ (since $\chi=0$ on $\mathcal V_{S_1'}$). Thus, for all $v\in B(0,\varepsilon_{1})$,  $ \mathcal V_{S_1'}\subset S_v$ which implies that $S_1'\subset  S_v$. 
Again, $S_{v}:=\{x+ \lambda(x,v)\nabla f(x)\}$ is a smooth compact  hypersurface. 
A schematic representation of the function $\chi$ and the hypersurface $S_v$ are given in Figure~\ref{fig:chii}. 
Up to choosing $\ve _1>0$ smaller, for any $v\in B(0,\ve_1)$, $S_{v}$ is the boundary of a $C^\infty$ open bounded  and connected  subset $D_v$ of~$\mathbb R^d$ such that, since  $\mathcal V_{S_1'}\subset S_v$,
$$ \overline{ \mathcal V_-} \subset D_v\cup S_1', \text{ and }  \overline{ D_v}  \setminus S_1' \subset \mathcal V_+.$$
 Let us now show that there exists   $v\in B(0,\varepsilon_{1})$
such that $f|_{S_{v}}$ is a Morse function. 
For that purpose, we    consider 
the function
$$
F:(x,v)\in S\times B(0,\varepsilon_{1})\mapsto  f|_{S_{v}}\big(x+\lambda(x,v)\nabla f(x)\big) =f(x)+ \chi(x)\,  v\cdot x
\in \mathbb R\,,
$$
and the function  
$G: S \times B(0,\varepsilon_{1})\to T_x^*S $ defined
by $G: (x,v)\mapsto (x, \partial_{x}F (x,v))$.
  Notice that   for all  $v\in    B(0,\varepsilon_{1})$, $x\in \overline{S_1'}\mapsto F(x,v)=f(x)$ is already, by assumption,  a Morse function (with no critical point on $\pa S_1'$). 
   This implies that $G$ is  transverse to the zero section  $0_{T^{*}S}$ of $T^{*}S$ along $S_1'\times B(0,\varepsilon_{1})$.
  Thus,  to prove  Proposition~\ref{Lau1}, it remains to study the function   $x\in  S\setminus S'_1\mapsto F(x,v)$, for $v\in    B(0,\varepsilon_{1})$. 
For  $(x,v) \in \overline{S_1\setminus S'_1} \times B(0,\varepsilon_{1})$ and  for all $z\in T_xS$, it holds:
$$\partial_{x}F (x,v)z= d_xf(x)z+ O(\Vert v\Vert)\, z.$$ 
 Since by assumption $d_xf(x)\neq 0_{T^{*}_xS}$ for all $x$ belonging to the compact set $\overline{ S_1\setminus S_1'}$, one has up to choosing $\ve_1>0$  smaller, for all $x\in \overline{S_1\setminus S_1'}$ and $v\in B(0,\varepsilon_{1})$, $\partial_{x}F (x,v)\neq 0_{T^{*}_xS}$.  
Finally, for  $(x,v) \in S\setminus S_1 \times B(0,\varepsilon_{1})$ and  for all $z\in T_xS$, it holds:
$$G(x,v)= (x,d_xf(x)z+  v_x^T\cdot z), \text{ where $v_x^T$ is defined by~\eqref{vtt}}.$$  
Thus, the function $G: S\setminus S_1\times B(0,\varepsilon_{1})\to T_x^*S $ is a submersion onto a small tube 
around the zero section $0_{T^{*}S}$ of $T^*S$. This implies that $G$ is  transverse to the zero section   of $T^{*}S$ along $S\setminus S_1\times B(0,\varepsilon_{1})$. 
 In conclusion, the function  $G: S \times B(0,\varepsilon_{1})\to T_x^*S $ is  transverse to the zero section of  $T^{*}S$.  The parametric transversality theorem  implies that for almost every   $v\in B(0,\varepsilon_{1})$,
$\partial_{x}(F|_{S\times \{v\}})=d (f|_{S_{v}})$ is transverse  to $0_{T^{*}S}$,  which is equivalent to~$f|_{S_{v}}$ is a Morse function. 
 This concludes the proof of       Proposition \ref{Lau1}. 
 
 \end{proof}

 \begin{figure}[h!]
  \begin{center}
 \begin{tikzpicture}
\draw[->] (-6,0)--(6,0) node[right] {$S$} ;
 \draw[->] (0,0)--(0,2);
\draw (6.99 ,1.4) node [anchor=north east] {1};
\draw[dashed] (-2,-1)--(-2,0);
\draw[dashed] (2,-1)--(2,0);
\draw[dashed] (-4.1,-1.7)--(-4.1,0);
\draw[dashed] (4.1,-1.7)--(4.1,0);
\draw[<->](-2,-1)--(2,-1)  node[midway,below] {$S_1'$} ;
\draw[very thick] (-4.7,1.5) ..controls (-4,0.2) and  (-2.9,0) .. (-2.7,0);
\draw[very thick] (5,0.3) ..controls (4,0.6) and  (3.1,0) .. (2.7,0);
   \draw  (-5,1.5) node[]{$S_v$};
\draw[very thick](-2.7,0)--(2.7,0);
\draw[<->](-4.1,-1.7)--(4.1,-1.7)  node[midway,below] {$S_1$} ;
\draw (-4.1,1) ..controls (-3.5,1) and  (-3.5,0) .. (-2.5,0);
\draw (4.1,1) ..controls (3.5,1) and  (3.5,0) .. (2.5,0);
\draw (-6.3,1) -- (-4.1,1);
\draw (6.5,1) -- (4.1,1);
\draw[dashed] (-4.1,-2.5)--(-4.1,0);
\draw[dashed] (-2,-2.5)--(-2,0);
\draw[<->] (-4.1,-2.7)--(-2,-2.7) node[midway,below] {$\nabla_Tf\neq 0$} ;
\draw[dashed] (4.1,-2.5)--(4.1,0);
\draw[dashed] (2,-2.5)--(2,0);
\draw[<->] (4.1,-2.7)--(2,-2.7) node[midway,below] {$\nabla_Tf\neq 0$} ;

   \draw  (-6.6,1) node[]{$\chi$};
\end{tikzpicture}
\caption{The support of $\chi$ on $S$, the compact sets $S_1$ and  $S_1'$, and the hypersurface $S_v$.}
 \label{fig:chii}
  \end{center}
\end{figure}
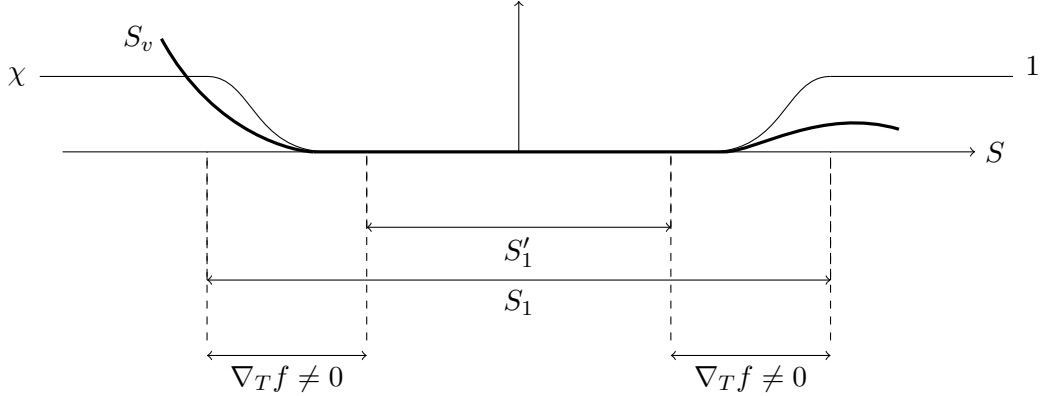

\end{appendices}

\section*{Acknowledgements}
This work is supported by the
European Research Council under the European Union's Seventh Framework
Programme (FP/2007-2013) / ERC Grant Agreement number 614492.
 The authors thank Laurent Michel and Fran\c cois Laudenbach for fruitful discussions. 

\break

\section*{Main notation used in this work}

\begin{multicols}{2}
\begin{itemize}
\item[$\bullet$]  $(\tau_\Omega,X_{\tau_\Omega})$, p.~\pageref{page.tau}

\vspace{-0.1cm}

 \item[$\bullet$]  $L^{(0)}_{f,h}$, p.~\pageref{page.lofh}

\vspace{-0.1cm}

\item[$\bullet$]  $L^{D,(0)}_{f,h}$, p.~\pageref{page.ldofh} (see also p.~\pageref{page.generatord})

\vspace{-0.1cm}

 \item[$\bullet$]  $\lambda_h$, p.~\pageref{page.lambdah}

\vspace{-0.1cm}

\item[$\bullet$]  $u_h$, p.~\pageref{page.uh}

\vspace{-0.1cm}

 \item[$\bullet$]  $\nu_h$, p.~\pageref{page.qsd}

\vspace{-0.1cm}

 \item[$\bullet$]  $\mathbb P_x$, p.~\pageref{page.qsd}

\vspace{-0.1cm}

\item[$\bullet$] Assumption~\eqref{H-M}, p.~\pageref{page.HM}.

\vspace{-0.1cm}

 \item[$\bullet$]   $\mathcal A(D)$,  p.~\pageref{page.AD}
 
 \vspace{-0.1cm}
 
\item[$\bullet$]  $\{f<a\}$, $\{f\le a\}$,  and  $\{f=a\}$,  p.~\pageref{page.fa}

\vspace{-0.1cm}

 \item[$\bullet$]  $\ft H_f(x)$,  p.~\pageref{page.hfx}, 

\vspace{-0.1cm}

\item[$\bullet$]  $\ft C_{\ft{max}}$,  p.~\pageref{page.fa}

\vspace{-0.1cm}

 \item[$\bullet$]  $\mathcal C$,  p.~\pageref{page.c} (see also p.~\pageref{page.c2})

\vspace{-0.1cm}

\item[$\bullet$]  $\ft C(x)$,  p.~\pageref{page.cx}

\vspace{-0.1cm}

\item[$\bullet$] Assumptions~\eqref{eq.hip1},~\eqref{eq.hip2},~\eqref{eq.hip3},  and~\eqref{eq.hip4}, p.~\pageref{page.hypo}

\vspace{-0.1cm}

 \item[$\bullet$]  $ \ft U_0^{  \Omega}=\{x_1,\dots,x_{\ft m_{0}^{\Omega}}\}$ and  $\ft m_0^{\Omega}$, p.~\pageref{page.u0omega}

 \vspace{-0.1cm}

 \item[$\bullet$]  $ \ft U_1^{  \Omega}=\{ z_{\ft m_1^{\pa \Omega}+1},\ldots,z_{\ft m_1^{\overline \Omega}}\}$ and   $\ft m_1^{  \Omega}$, p.~\pageref{page.u1omega}

 \vspace{-0.1cm}

 \item[$\bullet$] $ \ft U_1^{\pa \Omega}=\{z_1,\ldots,z_{\ft m_1^{\pa \Omega}}\}$ and   $\ft m_1^{\pa \Omega}$, p.~\pageref{page.u1paomega}

\vspace{-0.1cm}

 \item[$\bullet$]   $ \ft U_1^{\overline \Omega}=\{z_1,\ldots,z_{\ft m_1^{\pa \Omega}}, z_{\ft m_1^{\pa \Omega}+1},\ldots,z_{\ft m_1^{\overline \Omega}}\}$ and   $\ft m_1^{\overline \Omega}$, p.~\pageref{page.u1overlineomega}
 
 \vspace{-0.1cm}

 \item[$\bullet$]  $\{z_1,\ldots,z_{\ft k_1^{\pa \Omega}}\}$  and   $\ft k_1^{\pa \Omega}$, p.~\pageref{page.k1paomega}

 \vspace{-0.1cm}

 \item[$\bullet$] $\{z_1,\ldots,z_{\ft k_1^{\pa \ft C_{\ft{max}} }}\}$,  and  $\ft k_1^{\pa \ft C_{\ft{max}} }$, p.~\pageref{page.k1pacmax}

 \vspace{-0.1cm}

 \item[$\bullet$] $ a_i$, p.~\pageref{page.ai}

 \vspace{-0.1cm}

 \item[$\bullet$] $\ft C(\lambda,x)$,   p.~\pageref{page.clambdax}

 \vspace{-0.1cm}

 \item[$\bullet$] $\ft C^+(\lambda,x)$, p.~\pageref{page.c+lambdax}

 \vspace{-0.1cm}

 \item[$\bullet$] $ \lambda(x)$, p.~\pageref{page.lambdax}

 \vspace{-0.1cm}

 \item[$\bullet$] $\ft N_{1}$ and  $ \{ \ft C_{1},\dots,   \ft  C_{  N_{1}}\}$, p.~\pageref{page.N1}
 
  \vspace{-0.1cm}

 \item[$\bullet$]  $\ft E_{1,\ell}$ ($\ell\in \{1,\ldots,\ft N_1\}$), p.~\pageref{page.N1}  (labeled with the lexicographic order in p.~\pageref{page.lexico})

 \vspace{-0.1cm}

 \item[$\bullet$] $\ft U_1^{\ft{ssp}}$, page \pageref{page.u1ssp}

 \vspace{-0.1cm}

 \item[$\bullet$] $\mathcal C_{crit}$, page \pageref{page.ccrit}

 \vspace{-0.1cm}

 \item[$\bullet$]    $\mbf{j}$ and  $\mbf{\widetilde j}$, page \pageref{page.j}
 
  \vspace{-0.1cm}

 \item[$\bullet$]  $x_{k,l}$ ($k\ge 1$, $\ell\in \{1,\ldots,\ft N_k\}$), page \pageref{page.xkl} (labeled with the lexicographic order in p.~\pageref{page.lexico})

 \vspace{-0.1cm}

 \item[$\bullet$]  $\ft E_{k,l}$ ($k\ge 2$, $\ell\in \{1,\ldots,\ft N_k\}$), page \pageref{page.ekl2}  (labeled with the lexicographic order in p.~\pageref{page.lexico})
 
  \vspace{-0.1cm}
 
 \item[$\bullet$]  $\ft N_2, \ft N_3, \ft N_4,\ldots$,  page \pageref{page.n2}
 
\vspace{-0.1cm}

\item[$\bullet$] $\Lambda^pC^\infty(\overline \Omega)$, p.~\pageref{page.cinfty}

\vspace{-0.1cm}

\item[$\bullet$] $\Lambda^pC^\infty_T(\overline \Omega)$, p.~\pageref{page.cinftyt}

\vspace{-0.1cm}

\item[$\bullet$] $\Lambda^pL^2_w( \Omega)$ and  $\Lambda^pH^q_w(  \Omega)$,   p.~\pageref{page.wsobolevq}

\vspace{-0.1cm}

\item[$\bullet$]    $\Lambda^pH^q_{w,T}( \Omega)$,   p.~\pageref{page.wsobolevqt}

\vspace{-0.1cm}

\item[$\bullet$] $\Lambda^pL^2(  \Omega)$ and $\Lambda^pH^q(  \Omega)$,   p.~\pageref{page.psl2}

\vspace{-0.1cm}

\item[$\bullet$] $\Lambda^pH^q_{T}(  \Omega)$ and $\Lambda^pH^q_{N}(  \Omega)$,   p.~\pageref{page.psl2}

\vspace{-0.1cm}

\item[$\bullet$] $\Vert .\Vert_{H^q_w}$ and $\langle\, ,\, \rangle_{L^2_w}$,  p.~\pageref{page.psl2w}

\vspace{-0.1cm} 

\item[$\bullet$]$ \Vert .\Vert_{H^q}$ and $\langle\, ,\, \rangle_{L^2}$,  p.~\pageref{page.psl2}

\vspace{-0.1cm}

\item[$\bullet$] $\Delta^{(p)}_{f,h}$,  p.~\pageref{page.wlaplacien}

\vspace{-0.1cm}

\item[$\bullet$] $\Delta^{D,(p)}_{f,h}(\Omega)$,  p.~\pageref{page.wlaplaciend}

\vspace{-0.1cm}

\item[$\bullet$] $L^{D,(p)}_{f,h}(\Omega)$,  p.~\pageref{page.generatord}

\vspace{-0.1cm}

\item[$\bullet$] $\pi_E$,  p.~\pageref{page.piea}

\vspace{-0.1cm}

\item[$\bullet$]  For $p\in \{0,\dots,d\}$ and    $\pi^{(p)}_h$, p.~\pageref{page.pihp}

  \vspace{-0.1cm}

 \item[$\bullet$] $ \widetilde v_{k,\ell} $ and   $ \chi_{k,\ell} $, p.~\pageref{page.tilevkl}   (labeled with the lexicographic order in p.~\pageref{page.lexico})

   \vspace{-0.1cm}

 \item[$\bullet$]  $\Phi_j$, p.~\pageref{page.Phij} and~\pageref{page.Phij2}

  \vspace{-0.1cm}

 \item[$\bullet$] $w_j$, p.~\pageref{page.wj} and~\pageref{page.wj2}

  \vspace{-0.1cm}

   \item[$\bullet$] $\theta_j$ and $\tilde \phi_j$, p.~\pageref{page.tildephij} and~\pageref{page.tildephij2}

  \vspace{-0.1cm}

 \item[$\bullet$]  $u^{(1)}_{j,wkb}$ p.~\pageref{page.u1jwkb} and~\pageref{page.u1jwkb2}  
 
  \vspace{-0.1cm}

 \item[$\bullet$] $c_j(h)$, p.~\pageref{page.cj} and~\pageref{page.cj2}

  \vspace{-0.1cm}

 \item[$\bullet$] $\widetilde \phi_{j,wkb}$, p.~\pageref{page.phiwkb}
 
  \vspace{-0.1cm}

 \item[$\bullet$] $\widetilde u_k$ and  $\widetilde \psi_{j}$, page \pageref{page.qm}
 
  \vspace{-0.1cm}

 \item[$\bullet$] $\lambda_{2,h}$, p.~\pageref{page.lambda2h}
 
   \vspace{-0.1cm}

 \item[$\bullet$] $B_j$, p.~\pageref{page.bj}

   \vspace{-0.1cm}

 \item[$\bullet$] $p_{j,k}$, $C_{j,k}$,  and $\ve_{j,k}$, p.~\pageref{page.pjk}

    \vspace{-0.1cm}

   \item[$\bullet$]  $S$ and  $S_{j,k}$, p.~\pageref{page.sjk}

    \vspace{-0.1cm}

   \item[$\bullet$]  $ \widetilde S$ and  $ \widetilde S_{j,k}$,  p.~\pageref{page.tildesjk}

   \vspace{-0.1cm}

 \item[$\bullet$] $D$, $D_{k,k}$, and  $q_k$ p.~\pageref{page.dkk}

   \vspace{-0.1cm}

 \item[$\bullet$] $\widetilde C$ and  $ \widetilde C_{j,k}$, p.~\pageref{page.tildec}

   \vspace{-0.1cm}

 \item[$\bullet$] $\ft S_k$, p.~\pageref{page.ftsk}
 
 \vspace{-0.1cm}

 \item[$\bullet$] $\mu_i(T)$, p.~\pageref{page.mui}
  
   \vspace{-0.1cm}

 \item[$\bullet$] $C_0$, $C_1$, p.~\pageref{page.c0c1}
 
   \vspace{-0.1cm}

 \item[$\bullet$] $\lambda_{k,h}$, p.~\pageref{page.lambdakh}

    \vspace{-0.1cm}

 \item[$\bullet$] $\ve_h$, p.~\pageref{page.veh}
   
  \vspace{-0.1cm}

 \item[$\bullet$] $\widetilde  \pi^{(0)}_h$, p.~\pageref{page.tildephio}
 
 \vspace{-0.1cm}

 \item[$\bullet$] $\kappa_{ji}$ p.~\pageref{page.kji}
 
  \vspace{-0.1cm}

 \item[$\bullet$]  $Z_j$ and   $\psi_j$,  p.~\pageref{page.zjpsij}
 
  \vspace{-0.1cm}

 \item[$\bullet$]  $v_h$,  p.~\pageref{page.vh}
 
  \vspace{-0.1cm}

 \item[$\bullet$]  $\ft C_r$ ($r>0$),  p.~\pageref{page.cr}
 
 \end{itemize}
\end{multicols}

\bibliography{biblio_schuss} 
\bibliographystyle{plain}

\end{document}